\documentclass[11pt]{amsart}
\usepackage[letterpaper,margin=1in]{geometry}
\usepackage[T1]{fontenc}

\usepackage[dvipsnames]{xcolor}

\newcommand\myshade{85}
\colorlet{mylinkcolor}{Red}
\colorlet{mycitecolor}{Cerulean}
\colorlet{myurlcolor}{Plum}

\usepackage{lmodern}

\usepackage{quiver}

\usepackage[utf8]{inputenc}

\usepackage[inline,shortlabels]{enumitem}

\usepackage{amssymb,mathrsfs,stmaryrd}
\usepackage{mathtools}
\usepackage{thmtools,thm-restate}
\usepackage{colonequals}
\usepackage{tikz}

\usetikzlibrary{cd}

\usepackage[pdfusetitle,pagebackref,linktocpage]{hyperref}
\usepackage{verbatim}

\usepackage[colorinlistoftodos,textsize=scriptsize]{todonotes}
\usepackage{marginnote}

\usepackage[capitalize,noabbrev]{cleveref}

\hypersetup{
	linkcolor  = mylinkcolor!\myshade!black,
	citecolor  = mycitecolor!\myshade!black,
	urlcolor   = myurlcolor!\myshade!black,
	colorlinks = true
}

\mathchardef\mhyphen="2D

\usepackage[normalem]{ulem} 

\usepackage{relsize}
\usepackage[bbgreekl]{mathbbol}
\usepackage{amsfonts}
\DeclareSymbolFontAlphabet{\mathbb}{AMSb} 
\DeclareSymbolFontAlphabet{\mathbbl}{bbold}

\newcommand{\suchthat}{\;\ifnum\currentgrouptype=16 \middle\fi\vert\;}

\newcommand\restr[2]{{\left.\kern-\nulldelimiterspace#1\vphantom{\big|}\right|_{#2}}}

\theoremstyle{plain}
\newtheorem{theorem}{Theorem}[section]

\newtheorem{proposition}[theorem]{Proposition}
\newtheorem{conjecture}[theorem]{Conjecture}
\newtheorem{lemma}[theorem]{Lemma}

\newtheorem{corollary}[theorem]{Corollary}

\theoremstyle{definition}
\newtheorem{definition}[theorem]{Definition}
\newtheorem{example}[theorem]{Example}
\newtheorem{construction}[theorem]{Construction}

\newtheorem{convention}[theorem]{Convention}

\newtheorem{remark}[theorem]{Remark}

\definecolor{mb}{rgb}{0.36, 0.54, 0.66}

\newcommand{\Spf}{{\mathrm{Spf}}}
\newcommand{\Spa}{{\mathrm{Spa}}}
\newcommand{\st}{{\mathrm{st}}}
\newcommand{\dR}{{\mathrm{dR}}}
\newcommand{\pe}{{\mathrm{pro\acute{e}t}}}
\newcommand{\Ainf}{{\mathbb{A}_{\mathrm{inf}}}}
\newcommand{\Acrys}{{\mathbb{A}_{\mathrm{crys}}}}
\newcommand{\Ast}{{\mathbb{A}_{\mathrm{st}}}}

\newcommand{\OA}{{\mathcal{O}\mathbb{A}}}
\newcommand{\OB}{{\mathcal{O}\mathbb{B}}}

\newcommand{\crys}{{\mathrm{crys}}}

\newcommand{\Spec}{{\mathrm{Spec}}}
\newcommand{\et}{{\mathrm{\acute{e}t}}}

\newcommand{\FEt}{\mathsf{F\Acute{E}t}}

\newcommand{\Gal}{{\mathrm{Gal}}}

\newcommand{\Isoc}{\mathrm{Isoc}}
\newcommand{\wIsoc}{\mathrm{wIsoc}}
\newcommand{\Shv}{{\mathrm{Shv}}}
\newcommand{\Perfd}{{\mathrm{Perfd}}}

\newcommand{\rAinf}{{\mathrm{A_{inf}}}}    
\newcommand{\rAst}{{\mathrm{A_{st}}}}
\newcommand{\rhAst}{{\mathrm{\widehat{A}_{st}}}}
\newcommand{\rAcrys}{{\mathrm{A_{crys}}}}
\newcommand{\lcrys}{{{\mathrm{log}\textrm{-}\mathrm{crys}}}}
\newcommand{\lcris}{{{\textrm{$\ell$-}\mathrm{crys}}}}
\newcommand{\Vect}{{\mathrm{Vect}}}
\newcommand{\Loc}{{\mathrm{Loc}}}
\newcommand{\gr}{{\mathrm{gr}}}

\newcommand{\IA}{\mathbb{A}}
\newcommand{\IB}{\mathbb{B}}
\newcommand{\IC}{\mathbb{C}}

\newcommand{\IN}{\mathbb{N}}
\newcommand{\IP}{\mathbb{P}}
\newcommand{\IQ}{\mathbb{Q}}
\newcommand{\IZ}{\mathbb{Z}}

\newcommand{\sE}{\mathcal{E}}
\newcommand{\sF}{\mathcal{F}}
\newcommand{\sO}{\mathcal{O}}

\newcommand{\Fil}{\mathrm{Fil}}

\newcommand{\tensor}{\otimes}

\newcommand{\fp}{\mathfrak{p}}

\newcommand{\<}{\left<}
\renewcommand{\>}{\right>}

\def\wt{\widetilde}
\def\what{\widehat}

\newcommand{\sC}{\mathcal{C}}
\newcommand{\sD}{\mathcal{D}}
\newcommand{\sI}{\mathcal{I}}
\newcommand{\sK}{\mathcal{K}}
\newcommand{\sL}{\mathcal{L}}
\newcommand{\sM}{\mathcal{M}}

\newcommand{\sR}{\mathcal{R}}
\newcommand{\sX}{\mathcal{X}}

\newcommand{\fm}{\mathfrak{m}}
\newcommand{\fn}{\mathfrak{n}}

\newcommand{\into}{\hookrightarrow}

\newcommand{\Aut}{\mathrm{Aut}}

\newcommand{\sfC}{\mathsf{C}}
\newcommand{\sfN}{\mathsf{N}}

\newcommand{\sfF}{\mathsf{F}}

\renewcommand{\bf}{{\{f\}}}
\newcommand{\bg}{{\{g\}}}

\newcommand{\sto}{\stackrel{\sim}{\to}}

\newcommand{\fS}{\mathfrak{S}}

\newcommand{\p}{\partial}
\newcommand{\ctensor}{\what{\tensor}}
\newcommand{\Coh}{\mathrm{Coh}}

\newcommand{\logu}{\log(u/[\pi^\flat])}

\newcommand{\shM}{\mathcal{M}}

\newcommand{\ket}{{\mathrm{k}\et}}

\title{Pointwise criteria of $p$-adic local systems}
	\vspace{-1ex}
	\author{Haoyang Guo}
        \author{Ziquan Yang}

\begin{document}
	
\begin{abstract}
		Given a $\mathbb{Z}_p$-linear local system over a smooth rigid space, we show that it is crystalline (resp. semi-stable) with respect to any smooth (resp. semi-stable) integral model if and only if its restrictions at many classical points are crystalline (resp. semi-stable) representations.
		To this end, we introduce a crystalline Riemann--Hilbert functor, and give several applications, including a semi-stable comparison theorem in the relative setting.
\end{abstract}

	\maketitle
	\tableofcontents
	
\section{Introduction}
\label{sec intro}

\subsection{Pointwise criteria}

Let $K$ be a complete discretely valued $p$-adic field of characteristic $0$ with ring of integers $\mathcal{O}_K$ and perfect residue field $k$, and let $\Gal_K$ be the absolute Galois group of $K$.
Given a continuous representation of $\Gal_K$ on a finite dimensional $\mathbb{Q}_p$-vector space $V$, a fundamental question in arithmetic geometry is to characterize those representations which ``arise from geometry''; namely those which are subquotients of the $p$-adic \'etale cohomology of algebraic varieties $Y$ over the field $K$.
Celebrated results in $p$-adic Hodge theory, which used to be conjectures of Fontaine and Fontaine--Jannsen, state that $V$ is a \emph{de Rham} representation if it arises from geometry, and is moreover \emph{crystalline} (resp. \emph{semi-stable}) if the variety $Y$ admits \emph{good reduction} (resp. \emph{semi-stable reduction}), i.e., $Y$ extends to a smooth (resp. semi-stable) scheme over $\sO_K$.

It is then natural to consider these notions in the relative setting, in which a $p$-adic $\Gal_K$-representation is generalized by a \textit{$p$-adic local system} over a smooth rigid space $X_\eta$ over $K$. 
Moreover, the notions of being de Rham, crystalline, or semi-stable make sense for $p$-adic local systems as well. There are two fundamental results which tell us when a $p$-adic local system is de Rham: 

\begin{enumerate}[label=\upshape{(\Roman*)}]
    \item In his foundational work \cite{Sch13}, Scholze established a relative de Rham comparison isomorphism, which implies that given a family of smooth proper rigid varieties $f_\eta: Y_\eta \to X_\eta$, the local system $R^if_{\eta, *}\mathbb{Q}_p$ over $X_\eta$ is de Rham.
    \item Liu--Zhu's celebrated rigidity theorem \cite{LZ17} tells us that a $p$-adic local system on a connected rigid space $X_\eta$ is de Rham if and only if its restriction at any single classical point\footnote{By a \emph{classical point} of a rigid space, we mean a point that corresponds to a map $\Spa(K')\to X_\eta$, where $K'/K$ is a finite extension of $K$ and in particular has perfect residue field.
	When $X_\eta$ is an algebraic variety over $K$, a classical point is the same as a closed point.} $x_\eta$ is de Rham---this is a pointwise criterion for de Rham local systems.
\end{enumerate}

In this article, we aim to investigate how to extend Scholze and Liu--Zhu's results to crystalline and semistable local systems. 
It is well known that the naive analogue of Liu--Zhu's rigidity theorem is false for either of the two notions (\cite[Rmk~1.4]{LZ17}). 
In fact, for a local system $T$ over a rigid space $X_\eta$ with good reduction, even if it
is crystalline at a set of points whose reductions are Zariski
dense on the special fiber, $T$ may still not be crystalline or semistable (cf. \cref{counterexample of Liu-Zhu}). 
However, the situation becomes different if we consider more points. 
In particular, we have the following.

\begin{theorem}[Pointwise Criteria, cf. \Cref{thm:PC for semi-stable reduction}]
	\label{intro:thm pc}
	Let $X_\eta$ be a rigid space over $K$ that has a smooth (resp. semi-stable) $p$-adic integral model $X$. 
    Let $T$ be a $\mathbb{Z}_p$ local system on $X_\eta$.

    Then $T$ is crystalline (resp. semi-stable) with respect to $X$ if and only if the Galois representation associated to the restriction of $T$ at each classical point is crystalline (resp. semi-stable). 
\end{theorem}

Along the way, we also gave an analogous pointwise criterion for unramified $\ell$-adic local systems, see \cref{Thm:l-adic PC}.

\begin{remark}[Effective version]
	The above results are stated in their cleanest forms.
	In fact, it suffices to test crystallinity or semi-stability at any \emph{effective} subset $\mathcal{C}$ of classical points in $X_\eta$, rather than at all classical points.
        Roughly speaking, a subset $\mathcal{C}$ is effective if the inertia groups of these points topologically generate the global inertia group $\ker(\pi_1^\mathrm{alg}(X_\eta) \to \pi_1^\et(X_s))$, and the same property is preserved after replacing $X_\eta$ by certain types of \'etale morphisms.
        For example, given a smooth $p$-adic formal scheme $X$ over $\mathcal{O}_K$ and any Zariski dense subset of closed points in the special fiber $X_s$, the set of classical points $\mathrm{red}^{-1}(\mathcal{C}_s)$ is effective in $X_\eta$, where $\mathrm{red}:X_\eta\to X_s$ is the reduction map of the rigid space.
	We refer the reader to \Cref{def:eff set} for the precise definition.
    We would like to mention that it is a natural question if the set of special points form an effective subset in a Shimura variety with good reduction, which we do not have an answer.
\end{remark}

\begin{remark}
    One may naturally compare \cref{intro:thm pc} with a result of Shimizu \cite[Thm~1.6]{Shi22}, where the latter implies that if the restriction of the $p$-adic local system $T$ at a classical point $x_\eta$ is potentially crystalline or semistable, then the same holds in a small analytic neighborhood of $x_\eta$. 
As far as the authors are aware of, our result cannot be deduced from Shimizu's.
One major reason is that the union of such open neighborhoods at all classical points is in general far from being a cover of the entire $X_\eta$ in an appropriate sense of rigid analytic geometry.
\end{remark}

Here we also remark on a pleasant consequence of \cref{intro:thm pc}. 
Unlike that of de Rham local systems, the definitions of crystalline and semistable local systems a priori make references to the choice of a smooth or semi-stable integral model of $X_\eta$, despite the fact the local systems themselves are objects defined on $X_\eta$. 
However, the property of being crystalline or semi-stable at every classical point clearly only depends on the generic fiber.
As a consequence, we see the crystallinity and the semi-stability are independent of the choice of models (see also \cite[Cor.\ 5.5]{DLMS2}). In fact, along the same lines one deduces a stronger consequence. Namely, both the crystallinity and the semi-stability are preserved under pullback along a map of rigid spaces, even if the map does not extend integrally.
\begin{corollary}[Pullback preservation]
	\label{intro:cor pullback}
	Let $f_\eta:X'_\eta \to X_\eta$ be a map of two rigid spaces, and assume both $X'_\eta$ and $X_\eta$ admit good (resp. semi-stable) reductions.
	Let $T$ be a crystalline (resp. semi-stable) local system over $X_\eta$.
	Then $f_\eta^{-1}T$ is crystalline (resp. semi-stable).
\end{corollary}
Another immediate application is that Fontaine's conjecture in \cite{Fon97}, originally about crystalline and semi-stable representations and was proved by Liu \cite{Liu07}, extends naturally to the relative setting.
Without much additional effort, the latter in particular implies that the locus of crystalline (resp. semi-stable) local systems with fixed Hodge--Tate weights defines a closed subscheme in the universal deformation scheme, generalizing the fundamental work of Kisin (\cite{Kis08}) and Liu (\cite{Liu07}) to the relative setting.
\begin{corollary}[Fontaine's conjecture in the relative setting]
\footnote{Here we also kindly mention that for crystalline local systems, this was showed recently by Moon in \cite[Thm.\ 1.2]{Moo23}, by extending Liu's arguments to the setting of $p$-adic valuation rings with imperfect residue field, and then using the purity result of crystalline local systems.}
    Let $X_\eta$ be a rigid space over $K$ that has a smooth (resp. semi-stable) $p$-adic integral model $X$. 
    Let $T$ be a $\mathbb{Z}_p$ local system on $X_\eta$, and let $h \in \mathbb{N}$.

    Assume that for each $n\in \mathbb{N}$, the torsion local system $T/p^nT$ is the reduction of a crystalline (resp. semi-stable) local system $T_n$ with Hodge--Tate weights in $[0,h]$.
    Then $T$ is crystalline (resp. semi-stable) with Hodge--Tate weights in $[0,h]$.
\end{corollary}

\subsection{Relative semi-stable comparison}\label{sec: intro rel Cst}
One of the main conceptual contributions of \cref{intro:thm pc} is the idea that Fontaine and Fontaine--Jannsen's crystalline and semi-stable comparison conjectures in the relative setting should be consequences of their absolute versions, namely over a $p$-adic field with \textit{perfect} residue field.
The crystalline version in the relative setting has recently been established by the first named author and Reinecke in \cite{GR22} via prismatic methods. 
However, despite been widely expected, the semi-stable comparison is much more technical and its relative version in full generality so far has not been justified.
In fact, there has been little detailed exposition of the notion of semi-stable local systems until the very recent work \cite{DLMS2}, though the idea for such a notion dates back to \cite{Fal89} and \cite{AI12}. 
Here we also mention that unlike smooth morphisms, ``semi-stable morphisms'' lack a standard definition.

In this paper, we use \cref{intro:thm pc} to provide a working prototype of semi-stable comparison theorems in the relative setting, shed light on the optimal generality for such a theorem, and provide a new strategy which simplifies the entire program to the absolute setting.

\begin{theorem}[Semi-stability of \'etale cohomology sheaf, cf. \Cref{thm:relative Cst}]
    \label{intro:thm:relative Cst}
    Let $X$ be a semi-stable $p$-adic formal scheme over $\sO_K$ with the standard log structure $M_X$, let $(Y, M_Y)$ be a fine and saturated $p$-adic formal log scheme over $\sO_K$.
    Let $f:(Y, M_Y) \to (X, M_X)$ be a proper and log smooth morphism with Cartier type mod $\pi$ reduction, and let $f_\eta$ be the induced morphism between generic fibers.
    
    Assume one of the following conditions is true. 
    \begin{enumerate}[label=\upshape{(\alph*)}]
        \item\label{intro:thm: rel CN17} The map $f$ is algebrizable.
        \item\label{intro:thm: rel CK19} The underlying map between formal schemes of $f$ is pointwise weakly semi-stable. 
    \end{enumerate}
    Then the higher direct image $R^i f_{\eta, \ket *} \IQ_p$ is a semi-stable local system over $X_\eta$. 
\end{theorem}

The subscript ``$\ket$'' stands for ``Kummer-\'etale''. The condition in \ref{intro:thm: rel CK19} is defined below. Note that in case \ref{intro:thm: rel CK19} $M_Y$ is in fact trivial on $Y_\eta$, and the underlying morphism between rigid spaces of $f_\eta$ is smooth and proper, so we have $R^i f_{\eta, \ket *} \IQ_p = R^i f_{\eta, \et *} \IQ_p$.

\begin{definition}
We say that a morphism $f : Y \to X$ between $p$-adic formal schemes which are flat and topologically of finite type over $\sO_K$ is \textit{pointwise weakly semi-stable} if for every classical point $x_\eta \in X_\eta$, the fiber $Y_x$ at the integral extension $x \in X(\sO_{K(x_\eta)})$ of $x_\eta$ (\cref{prop: extend classical points}) is weakly semi-stable (\cref{def st models}).
\end{definition}

Continuing with the theme in \Cref{intro:thm:relative Cst}, one naturally asks if the $F$-isocrystal on the special fiber of the base $(X_k, M_{X_k})$ associated to $R^i f_{\eta, \ket *} \IQ_p$ is precisely the relative log crystalline cohomology of the special fiber $f_k$ of $f$. Let us call it the ``association statement''. 
We give a positive answer when the base admits good reduction, by combining Scholze's de Rham comparison theorem in \cite{Sch13}, the crystalline Riemann--Hilbert functor, and a relative Hyodo--Kato comparison theorem, where the latter two will be introduced soon.

\begin{theorem}[Association statement]
\label{intro:thm:association}
    In the setting of \Cref{intro:thm:relative Cst}.\ref{intro:thm: rel CK19}, 
    assume $X$ is in addition smooth over $\sO_K$. 
    Then the local system $R^i f_{\eta, \et *} \IQ_p$ is naturally associated to the (log) $F$-isocrystal $R^i f_{k, \crys, *} (\sO_{(Y_k, M_{Y_k})/W}[1/p])$ on $(X_k, M_{X_k})_\lcrys$, where $\sO_{(Y_k, M_{Y_k})/W}$ is the structure sheaf on the log-crystalline site $((Y_k, M_{Y_k})/W)_\lcrys$. 
\end{theorem}

Here we mention that the two conditions required in \Cref{intro:thm:relative Cst}.\ref{intro:thm: rel CN17} and \ref{intro:thm: rel CK19} together with the smoothness assumption on $X$ in \cref{intro:thm:association} are only artifacts of our proof, and we expect them to be unnecessary.
We refer the reader to \cref{rmk: assumptions in rel Cst} for a detailed discussion.

Semi-stable families considered in \cref{intro:thm:relative Cst} appear naturally in practice.
For example, when we compactify moduli spaces of varieties of a fixed type (e.g., curves, abelian varieties, K3 surfaces, etc.), the boundary of the compactification is often given by semi-stable degenerations of the smooth fibers.
Then, semi-stable families naturally arise when we deform the boundary from characteristic $p$ into the interior in characteristic $0$. 
To illustrate the phenomenon, let us consider the moduli space of pointed stable curves. 
Namely, we denote by $\sM_{g, n}$ (resp. $\overline{\sM}_{g, n}$) the moduli stack over $\IZ$ of $n$-marked smooth (resp. stable) curves of genus $g$, i.e., $\overline{\sM}_{g, n}$ is the Deligne-Mumford compactification $\sM_{g, n}$. Let $\sD \colonequals \overline{\sM}_{g, n} \smallsetminus \sM_{g, n}$ be the boundary (equipped with a reduced substack structure).  

\begin{proposition}
    Let $S$ be a semi-stable $p$-adic formal scheme over $\sO_K$. 
    Suppose that $f : C \to S$ is a stable curve given by a morphism $\gamma : S \to \overline{\sM}_{g, n}$ such that the special fiber $\gamma_k$ factors through $\sD_k$ but the generic fiber $\gamma_\eta$ factors through the analytification of $\sM_{g, n, K}$. Let $\sigma_1, \cdots, \sigma_n: S \to C$ be the marked points and let $C^\circ := C \smallsetminus \cup_{i = 1}^n \sigma_i(S)$ be the open complement. 
    
    Then $R^1 f^\circ_{\eta, \et *} \IQ_p$ is a semi-stable $p$-adic local system on $S_\eta$, where $f^\circ$ is the restriction of $f$ to the open curve $C^\circ$. 
\end{proposition}
Note that since $\overline{\shM}_{g, n}$ is smooth over $\IZ$ and $\sD \subseteq \shM_{g, n}$ is a relative normal crossing divisor (\cite[Thm.~5.2]{DM69}), it is very easy to construct a family $\gamma$ as above such that $\gamma_k$ is a smooth cover of $\sD_k$ (see \cref{ex: lift boundary}).

\subsection{Crystalline Riemann--Hilbert functors}
To prove the pointwise criteria, we introduce crystalline analogues of the Riemann--Hilbert functor, which are compatible with the $p$-adic Riemann--Hilbert functor $D_\dR$ studied in Liu--Zhu, and extend the functors $D_\crys$ and $D_\st$ from the case of a $p$-adic field to the relative setting. 
The construction of these functors is of independent interest and is, in itself, a major contribution of the article.
We assume that $X$ is a smooth $p$-adic formal scheme over $\mathcal{O}_K$ in this subsection, and let $T$ be a $p$-adic local system on the rigid generic fiber $X_\eta$.

We first recall the simplest setting when $X = \Spf(\sO_K)$ is a point.
Under the assumption, a $\mathbb{Z}_p$-local system $T$ of rank $d$ is equivalent to a continuous representation of $\Gal_K$ on a free $\mathbb{Z}_p$-module of rank $d$.
In addition, the local system $T$ is crystalline (resp. the semi-stable) when the associated $\varphi$-module $D_\crys(T)$ (resp. $(\varphi,N)$-module $D_\st(T)$) is of dimension $d$. 
Here the $\varphi$-module $D_\crys(T)$ can be recovered by $D_\st(T)$ by taking the kernel of the monodromy operator $N$ on the latter.

Now let $X$ be a smooth $p$-adic formal scheme over $\mathcal{O}_K$, and let $T$ be a $\mathbb{Z}_p$-local system over $X_\eta$, such that the restriction of $T$ at every classical point is crystalline. 
When $X$ is of positive relative dimension over $\sO_K$, Faltings's definition of the crystallinity  (\cite{Fal89}; cf. \Cref{def crys/st Faltings}) only requires the existence of some $F$-isocrystal that is \emph{associated} to the given local system $T$. 
Although this definition agrees with the preceeding one when $X = \Spf(\sO_K)$, the associated $F$-isocrystals are not functorially assigned and a priori may not be unique.

In order to resolve this lack of functoriality, we introduce a higher dimensional extension of $D_\crys$ and $D_\st$ functors.
To explain the construction, we setup the notations first: Let  $K_0=W(k)[1/p]$ be the maximal unramified subextension in $K$, and let $X_s$ be the reduced special fiber of $X$, regarded as a smooth variety over the residue field $k$.
For a map of $p$-adic formal schemes $f:X'\to X$ over $\mathcal{O}_K$, we let $f_s:X'_s\to X_s$ be the induced map of the reduced special fibers.
We use $\Loc_{\mathbb{Z}_p}(X_\eta)$ to denote the category of $\mathbb{Z}_p$ local systems over $X_\eta$, and let $\Isoc^\varphi(X_{s,\crys})$ be the category of $F$-isocrystals over the crystalline site $(X_s/W)_\crys$.
In addition, we let $\Isoc^{\varphi,N}(X_{s,\crys})$ be the category of data $(\mathcal{E},\varphi_\mathcal{E}, N_\mathcal{E})$, where $(\mathcal{E},\varphi_\mathcal{E})$ is an $F$-isocrystal over $(X_s/W)_\crys$, and $N_\mathcal{E}$ is a nilpotent endomorphism of the isocrystal $\mathcal{E}$, satisfying the formula $N_{\varphi^*\mathcal{E}}=p\varphi^* N_\mathcal{E}$.
Then there are natural functors of categories
\[
\Isoc^\varphi(X_{s,\crys}) \longrightarrow \Isoc^{\varphi,N}(X_{s,\crys}) \xrightarrow{\ker(N)} \Isoc^{\varphi}(X_{s,\crys}),
\]
where the first functor is defined by equipping an $F$-isocrystal with the zero endomorphism, and the second functor is defined by taking the kernel of the $N_{\mathcal{E}}$.
Note that the first functor is fully faithful, and the composition is the identity.

Next, we may sheafify the above notions with respect to the Zariski topology of $X_s$, and define a \emph{weak isocrystal} to be a sheaf of isocrystals over $(X_s)_{\mathrm{Zar}}$ with \textit{injective} transition morphisms (cf. \Cref{def:weak_F-isoc}). A usual isocrystal is a weak isocrystal whose transition morphisms are all isomorphisms. 
We let $\wIsoc^\varphi(X_{s,\crys})$ and $\wIsoc^{\varphi,N}(X_{s,\crys})$ be the categories of weak $F$-isocrystals with additional structures that are analogous to the above. 
Now we can state the main construction of the article.

\begin{theorem}[Crystalline Riemann--Hilbert functors (\Cref{thm:gluing of D functors})]
	\label{intro:thm:RH}
	There are natural functors $\mathcal{E}_\crys$ and $\mathcal{E}_\st$ fitting into the following commutative diagrams
 \[
 \begin{tikzcd}
 & \wIsoc^{\varphi,N}(X_{k,\crys}) \arrow[dd, "\ker(N)"] \\
 \Loc_{\mathbb{Z}_p}(X_\eta) \ar[ru, "\mathcal{E}_\st"] \ar[rd, "\mathcal{E}_\crys"]& \\
 & \wIsoc^{\varphi}(X_{s,\crys}),
 \end{tikzcd}
 \quad
 \begin{tikzcd}
 & (\mathcal{E}_{\st,T}, \varphi_{\mathcal{E}_{\st,T}}, N_{\mathcal{E}_{\st,T}}) \arrow[dd, mapsto] \\
 T \arrow[ru, mapsto] \arrow[rd, mapsto] &\\
 & (\mathcal{E}_{\crys, T}, \varphi_{\mathcal{E}_{\crys,T}}) = \ker(N_{\mathcal{E}_{\st,T}}),
 \end{tikzcd}
 \]
    satisfying the following properties:
    \begin{enumerate}[label=\upshape{(\roman*)}]
    	\item\label{intro:thm:RH rank}  \emph{(Rank inequalities)} The weak isocrystal $\mathcal{E}_{\crys,T}$ (resp. $\mathcal{E}_{\st,T}$) is of rank $\leq \mathrm{rank}(T)$, and the local system $T$ is crystalline (resp. semi-stable) with respect to $X$ if and only if $\mathrm{rank}(\mathcal{E}_{\crys,T})$ is constant and is equal to $\mathrm{rank}(T)$ (resp. $\mathrm{rank}(\mathcal{E}_{\st,T})=\mathrm{rank}(T)$).
    	\item\label{intro:thm:RH pullback} \emph{(Injectivity of base change morphism)} Assume $f:Y\to X$ is a map of $p$-adic formal schemes such that $Y$ is smooth over $\mathcal{O}_{K'}$ for a finite extension $K'/K$.
    	Then there are natural injections of weak $F$-isocrystals with nilpotent endomorphisms 
    	\[
    	f_s^*\mathcal{E}_{\crys,T} \into \mathcal{E}_{\crys,f_\eta^{-1}T}, \quad f_s^*\mathcal{E}_{\st,T} \into \mathcal{E}_{\st,f_\eta^{-1}T},
    	\]
    	which are functorial with respect to $T\in \Loc_{\mathbb{Z}_p}(X_\eta)$ and satisfy the cocycle condition with respect to compositions of maps between $p$-adic formal schemes.
        \item\label{intro:thm:RH smooth bc} \emph{(Smooth base change)} In the setting of \ref{intro:thm:RH pullback}, assume the map $f$ is $p$-completely smooth.
        Then the base change morphisms in \ref{intro:thm:RH pullback} are all isomorphisms.
    	\item\label{intro:thm:RH dR} \emph{(Compatibility with $D_\dR$)} By evaluating at the divided power thickening $(X,X_{p=0})$
    	\footnote{Here the thickening is an object in the crystalline site $(X_{p=0}/W)_\crys$, and the evaluation makes sense thanks to Dwork's trick; cf. \Cref{Dwork's trick}.},
    	there are functorial injections of flat connections over the generic fiber $X_\eta$ that are compatible with pullbacks:
    	\begin{equation}
    	    \label{eqn:intro:thm:RH dR}
         \mathcal{E}_{\crys,T}(X,X_{p=0})\otimes_{K_0} K \into \mathcal{E}_{\st,T}(X,X_{p=0})\otimes_{K_0} K \into D_\dR(T),
    	\end{equation}
    	where $D_\dR(-)$ is the $p$-adic Riemann--Hilbert functor of $T$ defined in \cite[Thm~3.9(v)]{LZ17}.
     \item\label{intro:thm:RH point} \emph{(Point case)} When $X = \Spf(\sO_K)$ is a point, there are isomorphisms $\mathcal{E}_{\crys,T}(X,X_{p=0}) \simeq D_\crys(T)$ and $\mathcal{E}_{\st,T}(X,X_{p=0}) \simeq D_\st(T)$ respecting Frobenius and nilpotent endomorphisms, such that (\ref{eqn:intro:thm:RH dR}) become the canonical injections $D_\crys(T) \tensor_{K_0} K \into D_\st(T) \tensor_{K_0} K \into D_\dR(T)$. 
    \end{enumerate}
\end{theorem}
Here we name our functors as the \emph{crystalline Riemann--Hilbert functors}, in light of the compatibility with \cite{LZ17} as in \Cref{intro:thm:RH}.\ref{intro:thm:RH dR}.

\begin{remark}[Relation to Tan--Tong \cite{TT19}]
    When the base ring $\sO_K$ is absolutely unramified, the evaluation of our functor $\mathcal{E}_\crys$ on $X$ is identical to $D_\crys(T)$ in \cite[Def.~3.12]{TT19}. 
    One of the main observations, which enables us to extend their construction to allow ramifications, is precisely the crystal interpretation of the construction. 
    Namely, by deformation theory, one can always find a Zariski cover $\{U_\alpha\}$ of $X$ such that each $U_\alpha$ admits an unramified model and apply $D_\crys(T)$ to the unramified models. 
    Although the construction depends on the choice of the unramified models, we can glue the outputs, not in the category of vector bundles (with additional structures) on rigid spaces as in \cite{TT19}, but in the category of $F$-isocrystals on the special fiber $X_s$. 
    As a consequence, although Tan-Tong's crystalline period sheaf $\sO \IB_\crys$, which is the crystalline analogue of Scholze's de Rham period sheaf $\sO \IB_\dR$, cannot be globalized, the resulting functor $\mathcal{E}_\crys$ can. 
\end{remark}

\begin{remark}[Breuil's equivalence and relative Hyodo--Kato isomorphism]
	Although the functor $\mathcal{E}_{\st,T}$ behaves like the classical $D_\st$ functor, the cautious reader might notice that the semi-stability is usually defined via $F$-isocrystals over the log-crystalline site (cf. \cite{AI12}, \cite{DLMS2}).
	In fact, for a smooth $p$-adic formal scheme $X$, we show in \Cref{equiv def log isoc} that there is a natural equivalence of categories
	\[
	\Isoc^{\varphi,N}(X_{s, \crys}) \simeq \Isoc^\varphi ((X_s, (0^\IN)^a)_\lcrys),
	\]
	where the latter is the category of $F$-isocrystals over the log crystalline site of the special fiber $(X_s, (0^\IN)^a)$ over $W$.
	One of the technical ingredients is a re-interpretation of the category of $\varphi$-modules over certain non-noetherian divided power polynomial (cf. \Cref{thm:Breuil}), which in the case of a point was shown by Breuil in \cite[\S 6]{Bre97}.
    As an application of the latter result, we in particular extend the Hyodo--Kato isomorphism of \cite{HK94} to the relative setting (cf. \Cref{thm:relative Hyodo-Kato}).
\end{remark}

\begin{remark}[Purity for arbitrary $p$-adic local systems]
		Let $\mathcal{E}$ be a weak $F$-isocrystal over a smooth connected variety $Z$ in characteristic $p$.
		There is a largest open subvariety $U_\mathcal{E}$ of $Z$, which we define as the \emph{pure locus} of $\mathcal{E}$, such that the restriction of $\mathcal{E}$ on the open subvariety $U_\mathcal{E}\subset Z$ is an actual $F$-isocrystal (cf. \Cref{def:various_notions_of_wIsoc}).
		For example, when $\mathcal{E}$ is an $F$-isocrystal over $Z$, its pure locus is $Z$ itself; and if $\mathcal{F}$ is an isocrystal on an open subscheme $V$ and if $i:V\to Z$ is the open immersion, the pure locus of $i_!\mathcal{F}$ is equal to $V$.
		In light of the purity of crystalline and semi-stable local systems, we make the following conjecture on arbitrary $p$-adic local systems.
		\begin{conjecture}
			Let $X$ be a smooth connected $p$-adic formal scheme over $\mathcal{O}_K$, let $\Spa(L)$ be the Shilov point of $X_\eta$, and let $T$ be a $\mathbb{Z}_p$-local system over $X_\eta$.
			Then we have the following equalities of integers
			\begin{align*}
				&\mathrm{rank}(\mathcal{E}_{\crys,T}|_{U_{\mathcal{E},_{\crys,T}}}) = \mathrm{rank}(D_{\crys,L}(T|_{\Spa(L)})), \\
				& \mathrm{rank}(\mathcal{E}_{\st,T}|_{U_{\mathcal{E},_{\st,T}}})  = \mathrm{rank}(D_{\st,L}(T|_{\Spa(L)})),
			\end{align*}
			where $D_{\crys,L}(T|_{\Spa(L)})$ (resp. $D_{\st,L}(T|_{\Spa(L)})$) is the $\varphi$-module (resp. $(\varphi,N)$-module) associated to the $\Gal(L)$-representation $T|_{\Spa(L)}$.\footnote{The reader may look at \cite[\S2.2]{Mor14} for an exposition of these functors in the imperfect residue field case.}
		\end{conjecture}
\end{remark}
 
\subsection{Descending the crystallinity}
Now we explain the strategy of the proof for the pointwise criteria in \Cref{intro:thm pc}. We let $T$ be a local system of rank $d$ on the Raynaud generic fiber $X_\eta$ of a smooth $p$-adic formal scheme $X$ over $\Spf(\sO_K)$.
We will consider the good reduction case in the introduction; the case when $X$ has semi-stable reduction follows quickly from the former, together with the recent advance of Du--Liu--Moon--Shimizu on the purity of semi-stable local systems \cite{DLMS2}.

We start by recalling the strategy in the simplest situation, when $X_\eta=\mathrm{Spa}(K)$ is a single point, and identify the local system $T$ with its corresponding $\Gal_K$-representation.
In this case, the crystallinity of $T$ (with respect to $K$) can usually be checked in the following two steps:
\begin{enumerate}
	\item First show that $T$ is a de Rham representation, and hence semi-stable when restricted to a finite totally ramified Galois extension $K'$ of $K$.
	\item Then show that the action of $\Gal(K'/K)$ on the associated $(\varphi,N)$-module $D_{\st, K'}(T)$ is trivial, and the monodromy operator $N$ is zero. 
\end{enumerate}
Here we emphasize that the second step makes an essential use of the functoriality of the construction $D_{\st, K'}(T)$ with respect to $T$, so that there is a canonical $\Gal(K'/K)$-action on $D_{\st, K'}(T)$, identifying $D_{\st,K}(T)$ as the invariant submodule:
\[
D_{\st,K}(T) \simeq D_{\st, K'}(T)^{\Gal(K'/K)}.
\]

We then return to the general setting for an arbitrary smooth $p$-adic formal scheme $X$.
In this case, thanks to the rigidity of the de Rham local system by Liu--Zhu (\cite{LZ17}), our pointwise assumption implies that the local system $T$ is de Rham.
So by taking the restriction of $T$ at the Shilov point $\Spa(L)$, we know (the corresponding representation of) $T|_{\Spa(L)}$ is de Rham.
Moreover, thanks to the work of Morita (\cite{Mor14}) and Ohkubo (\cite{Ohk13}), $T|_{\Spa(L)}$ is in fact a potentially semi-stable representation of $\Gal_L$.
This means that there is a finite Galois cover $L'/L$ such that the restriction $T|_{\Spa(L'/L)}$ is semi-stable.

To proceed, we want to relate the representation of $\Gal_{L'}$ with the local system over the original $p$-adic formal scheme $X$.
We first notice that by approximation arguments in \Cref{prop: tower} and \Cref{thm: pre-tower}, the finite extension $L'/L$ can be refined so that it comes from the map of Shilov points with respect to the following tower of $p$-adic formal schemes:
\[
X''=X_m\xrightarrow{g_m} \cdots \xrightarrow{g_1} X_1=X'_{\mathcal{O}_{K'}}  \xrightarrow{g_0} X_0=X' \xrightarrow{f} X,
\]
where
\begin{itemize}
	\item the map $f$ is $p$-completely \'etale;
	\item the field $K'$ is a finite Galois extension of $K$ that is totally ramified, and $g_0$ is the extension map;
	\item for each $1\leq i\leq m$, the map $g_i:X_i\to X_{i-1}$ is a primitive inseparable cover between two smooth $p$-adic formal schemes over $\mathcal{O}_{K'}$ (cf. \Cref{def: purely inseparable cover}).
\end{itemize}
Here the \emph{primitive inseparable cover} is introduced in \Cref{def: purely inseparable cover}, and roughly means a degree $p$ Galois cover of a rigid space (with good reduction) such that the induced map on the reduced special fiber is purely inseparable.
We can then spread out the semi-stability of the restriction $T|_{\Gal_L}$ to the entire $p$-adic formal scheme $X''$, thanks to the purity result in \cite{DLMS2} mentioned before.

Next, to show the original local system $T$ is crystalline, we need to descend the semi-stability from $X_i$ to $X_{i-1}$ using the pointwise assumption.
We let $x_\eta \in X_{i,\eta}$ be a classical point such that $y_\eta\colonequals g_{i,\eta}^{-1}(x_\eta)$ is a single ramified point.
Then the relative Galois group $G_i\colonequals \Gal(X_{i,\eta}/X_{i-1,\eta})$ is naturally isomorphic to the Galois group of the field extension, namely $\Gal(K(y_\eta)/K(x_\eta))$.
By taking the rings of integers, we in particular obtain a $G_i$-equivariant commutative diagram of $p$-adic formal schemes
\[
\begin{tikzcd}
	y\colonequals\Spf(\mathcal{O}_{K(y_\eta)}) \arrow[rr, "{G_i\text{-equivariant}}"] \arrow[d]  && X_i \arrow[d, "g"]\\
	x\colonequals\Spf(\mathcal{O}_{K(x_\eta)}) \arrow[rr] && X_{i-1}.
\end{tikzcd}
\]
Now by the functoriality of the construction $\mathcal{E}_{\st,T|_{X_i}}$, there is a natural action of $G_i$ on $\mathcal{E}_{\st,T|_{X_i}}$ in the categories of $F$-isocrystals.
Moreover, by the base change theorem of the crystalline Riemann--Hilbert functor in \Cref{intro:thm:RH}.\ref{intro:thm:RH pullback} and by \Cref{intro:thm:RH}.\ref{intro:thm:RH point}, the $G_i$ action on $\mathcal{E}_{\st,T|_{X_i}}$ restricts to the $\Gal(K(y_\eta)/K(x_\eta))$ action on $D_{\st, K(y_\eta)}(T|_{y_\eta})$.
However, by the assumption that $T|_{x_\eta}$ is crystalline, the latter Galois action is trivial.
As a consequence, thanks to the faithfulness of the restriction functor of $F$-isocrystals along the map of special fibers $y_s\to X_{i,s}$ (\cite[Thm.\ 4.1]{Ogu84}), we see the $G_i$ action on $\mathcal{E}_{\st,T|_{X_i}}$ is also trivial.
Finally, using the triviality of $G_i$ action, we can show that the semi-stability of $T|_{X_{i,\eta}}$ implies that of $T_{X_{i-1,\eta}}$,
\footnote{Here a technical caveat is the fact that the map of reduced special fibers $X_{i,s}\to X_{i-1,s}$ is only purely inseparable in general (instead of being an isomorphism, as happened in the case of a point), so additional descent arguments for vector bundles over period sheaves are needed in order to obtain the aforementioned implication of semi-stability.}
 and hence the semi-stability of $T$.
Similar argument can also be used to show that the pointwise crystalline assumption will enforce the semi-stable local system $T$ to be crystalline.
We refer the reader to \Cref{thm:descend along primitive insep} and \Cref{thm:descend along base extension} for details.

\begin{remark}
	The curious reader might wonder if the pointwise criteria follows quickly from the conjectural equivalence between the de Rhamness and the potential semi-stability for $p$-adic local systems.
	In fact, even assuming the conjecture, we are not aware of an argument without shrinking and descending as described above.
	This is essentially due to the complexity of semi-stable reductions of rigid spaces: given a finite \'etale cover $X'_\eta \to X_\eta$, it is not known if $X'_\eta$ admits a global semi-stable reduction without further refinement by open and (non-proper) \'etale morphisms.
\end{remark}

\subsection{Outline}
\label{sub:outline}
We begin with the basics of log-crystalline sites in \Cref{sec prep}, where we provide various structural results for (log) $F$-isocrystal over a smooth finite type $k$-scheme with the hollow log structure, and prove the Breuil's equivalence theorem and the Hyodo--Kato isomorphism in the relative setting.
Then we collect the constructions and the properties for the horizontal period sheaves and the period sheaves with connections in \Cref{sec period sheaves}.
In particular, we introduce the semi-stable period sheaves for a rigid space with good reduction.

After the above preparations, we proceed in \Cref{sec def local system} to introduce the crystalline and the semi-stable local systems.
Our first definition in \Cref{sub def local system Faltings} is in the spirit of Faltings and has no regularity assumption on the reduction of the rigid space.
In the special case when the rigid space admits a smooth affine reduction with a prescribed unramified model, we construct the $D_\crys$ and the $D_\st$ functors from $p$-adic local systems to $F$-isocrystals in \Cref{sub D functors}, and thus obtain a local formula of the associated $F$-isocrystals.

We then study the pullback behaviors of the aforementioned constructions in \Cref{sec pullback}.
We show in \Cref{sub general pullback} that the general notions of the crystallinity and the semi-stability are preserved by the pullback functor, for the generic fiber of a map of $p$-adic formal schemes $f:X\to Y$.
When the map $f$ is between affine smooth $p$-adic formal schemes with unramified models, we show in \Cref{sub pullback of D functors} that the above $D_\crys$ and the $D_\st$ functors admit natural pullback injections.
As an interlude, we introduce and study the notion of weak $F$-isocrystals in \Cref{sub: weak F-isoc}.
In \Cref{sec: global RH}, we prove that when interpreted as weak $F$-isocrystals, the $D_\crys$ and the $D_\st$ functors can be globalized at face value, yielding the crystalline Riemann--Hilbert functors for general smooth $p$-adic formal schemes.

Next in \Cref{sec nor} we establish the geometric preparations for the pointwise criteria.
We show in \Cref{sub towers} that any finite \'etale cover of the Shilov point of a rigid space can spreads out to a finite tower of standard covers of rigid spaces, up to shinking to a smaller neighborhood.
We also introduce the notion of the effective set of classical points, and in particular show the $\ell$-adic pointwise criterion in \Cref{sub:effective}.

Finally, assembling all the previous ingredients, we prove the pointwise criteria in \Cref{sec pc}, and give various applications.
We start with the primitive reductions in \Cref{sub pc first reduction}.
The main technical arguments are established in \Cref{sub pc descend}, where we show the descent of the crystallinity and the semi-stability.
Finally we provide several applications of the pointwise criteria in \Cref{sub application}, including the semi-stability of the higher direct image along a semi-stable proper smooth map.

\subsection{Acknowledgment}	
	A special thank goes to Sasha Petrov for his collaboration during the early stages of the project and for offering many important insights. We thank Fabrizio Andreatta, Bhargav Bhatt, Sasha Beilinson, Frank Calegari, {K\k{e}stutis} {\v{C}}esnavi{\v{c}}ius, Heng Du, Matt Emerton, H\'el\`ene Esnault, Asvin G., Michael Larsen, Shizhang Li, Tong Liu, Yong-Suk Moon, Kiran Kedlaya, Mark Kisin, Teruhisa Koshikawa, Emanuel Reinecke, Peter Scholze, Ananth Shankar, Zijian Yao and John Yin for helpful conversations or correspondence. During the preparation of the article, the first author was supported by Max-Planck-Institut f\"ur Mathematik and the University of Chicago, the second author was supported by Simons AMS travel fund.

\subsection{Conventions}
\label{subsec conv}
Given a ring $R$ and an ideal $I$, we use $R^\wedge_I$ to denote the $I$-adic completion of $R$.
When $A$ and $B$ are $R$-algebras, we use $A\widehat{\otimes}_R B$ to denote the $I$-complete tensor product $(A\otimes_R B)^\wedge_I$ whenever the topology is clear.
We use $R\langle x_i\rangle$ (resp. $R\{x_i\}$ ) to denote the convergent power series ring (resp. $I$-complete divided power polynomial ring) over $R$, 
namely the $I$-adic completion of the polynomial ring of variables $x_i$ over $R$ (resp. the divide power polynomial ring of variables $x_i$ over $R$).
\footnote{Here we warn the reader that our usage of the notation is different from the common convention in prismatic cohomology, where $R\{x_i\}$ means the free $\delta$-algebra over a given $\delta$-ring $R$, as introduced by Bhatt--Scholze \cite{BS22}.}
To lighten the notation, we also abbreviate the notion ``divided power'' using the term ``PD'' for many places in the article.

Given a pre-log structure $(R,N_R)$, we use $N_R^a$ to denote the associated log structure on $R$.
Moreover, for a (pre-)log structure $M$, we often abuse the same symbol to denote the associated monoid when the meaning is clear.

For a $p$-adic formal scheme $X$ over $\mathcal{O}_K$, where $K$ is a $p$-adic field with perfect residue field $k$, we use $X_{p=0}$ to denote the mod $p$ fiber of $X$, use $X_k$ to denote the base change $X\times_{\Spf(\mathcal{O}_K)} \Spec(k)$, and use $X_s$ to denote the reduced closed scheme $(X_k)_\mathrm{red}$.

For a scheme $Z$ over a perfect field $k$ in characteristic $p$, we let $Z_\crys$ be its absolute crystalline site, namely the crystalline site of $Z$ over $W(k)$, where the latter has the canonical divided power structure on the ideal $p$.
If $Z$ admits a log structure $M_Z$, we let $(Z,M_Z)_\lcrys$ be its absolute log-crystalline site, namely the log-crystalline of $(Z,M_Z)$ over $W(k)$, where the latter is equipped with the trivial log structure.

For a given $p$-adic discretely valued field $K$ with perfect residue field, we let $C$ be its complete algebraic closure.
We fix a compatible system $\zeta_{p^n}$ of $p^n$-th roots of unity in $\sO_C$ and let $\epsilon=(1,\zeta_p,\ldots)$ be the corresponding element of $\sO_C^\flat=\lim_{x\mapsto x^p} \sO_C/p$.
Let $\mathrm{A}_{\inf} = \mathrm{A}_{\inf}(\sO_C) \colonequals W(\sO_C^\flat)$ and	$q \colonequals [\epsilon] \in \mathrm{A}_{\inf}(\sO_C)$ be the Teichm\"uller lift.
We set $\mu \colonequals q - 1$ and $\xi \colonequals [p]_{q^{1/p}} = \frac{q-1}{q^{1/p}-1}$.
They satisfy the relation $\varphi(\mu)=\mu\cdot \varphi(\xi)$.
We call the quotient map $\mathrm{A}_{\inf} \to \mathrm{A}_{\inf}/\xi\simeq \sO_C$ the \emph{de Rham specialization map} and denote it by $\theta$ as usual.

For a $p$-adic topological $K$-algebra $A$, to simplify the notation, we let $\Spa(A)$ be the adic spectrum $\Spa(A,A^\circ)$, where $A^\circ$ is the subring of power-bounded elements.

\section{Log structure, and logarithmic $F$-isocrystals}
\label{sec prep}
Let $K$ be a complete discretely valued $p$-adic field with perfect residue field $k$, and let $\mathcal{O}_K$ be its ring of integers. In this section, we recall the basics of the log structures used in the article, and prove various structural results on $F$-isocrystals and their cohomology over (log-)crystalline sites.

\subsection{Log structures, log-crystalline site, and semi-stable reduction}
\label{sub review of log} We recall some basic definitions for readers' convenience and set up some notations and terminology which we shall use later, following \cite{Kat89} and \cite{HK94}. Let $X$ be a scheme.
A \textbf{pre-log structure} on $X$ is a sheaf of monoids $M$ on the \'etale site $X_\et$ endowed with a homomorphism $\alpha : M \to \sO_X$ with respect to the multiplication on $X$, where $\alpha$ is often assumed. 
We call it a \textbf{log structure} if $\alpha$ restricts to an isomorphism $\alpha^{-1}(\sO_X^*) \simeq \sO_X^*$. 
Given any pre-log structure $(M, \alpha)$ on $X$, we can always form the \textbf{associated} log structure $M^a$ by taking the pushout of $M$ along $\alpha^{-1}(\sO_X^\times) \to \sO_X^\times$. 
We shall often refer to a ``scheme with log structure'' simply as a ``log scheme''. By the \textbf{trivial log structure} on $X$, we mean $\sO_X^\times \into \sO_X$. For convenience we impose the following convention: \textit{When we view a scheme $X$ as a log scheme without specifying a log structure, the trivial log structure is assumed.}

 For a monoid $P$, let us write $P^{\mathrm{gp}}$ for the associated commutative group $\{ ab^{-1} \mid a, b \in P \}$, and we say that $P$ is \textbf{integral} if the natural map $P \to P^{\mathrm{gp}}$ is injective. Given $P$, denote by $P_X$ the constant sheaf of monoid over the scheme $X$ defined by $P$. 
 A log structure $M$ on $X$ is said to be \textbf{quasi-coherent} (resp. \textbf{coherent}) if \'etale-locally on $X$, there exists a monoid (resp. finitely generated monoid) $P$ and a homomorphism $P_X \to \sO_X$ such that $M \simeq (P_X \to \sO_X)^a$. 
 We say that $M$ is \textbf{fine} if it is integral (i.e., given by a sheaf of integral monoids) and coherent. When $(X, M)$ is a fine log scheme, such homomorphism $P_X \to \sO_X$ will be called a \textbf{chart for $M$}.

 For a morphism of schemes $f : X \to Y$ and a log structure $N$ on $X$, the \textbf{inverse image} $f^* N$ is defined to be $(f^{-1}(N))^a$. 
 A morphism between log schemes $f : (X, M) \to (Y, N)$ is defined to be a morphism $f : X \to Y$ between the underlying schemes together with a morphism $f^* N \to M$, such that the latter is compatible with the map of structure sheaves. 
 Now let $f : (X, M) \to (Y, N)$ be a morphism between log schemes. 
 We say the morphism is \textbf{strict} (resp. \textbf{exact}) if the map of monoids $f^*N\to M$ is an isomorphism (resp. $f^{-1} N = f^{-1}(N)^{\mathrm{gp}} \times_{M^{\mathrm{gp}}} M$).
 We also say that $f$ is a \textbf{closed immersion} (resp. \textbf{exact closed immersion}) if the underlying morphism $f : X \to Y$ is a closed immersion and the associated morphism $f^* N \to M$ is surjective (resp. bijective). 
 In particular, an exact closed immersion is strict. 
 One can also define the notions of (formally) \textbf{\'etale} or \textbf{smooth} morphisms between log schemes, and we refer the reader to the standard references (e.g., \cite[\S3.3]{Kat89}).
 When $(X, M)$ and $(Y, N)$ are both fine log schemes, then \'etale locally on $X$ we can always find a \textbf{chart for $f$}, i.e., a triple $(P_X \to \sO_X, Q_Y \to \sO_Y, Q \to P)$ where $P_X \to \sO_X$ and $Q_Y \to \sO_Y$ are charts for $M$ and $N$ respectively, and $Q \to P$ is a morphism between monoids for which the following diagram commutes
 \[\begin{tikzcd}
	{Q_X} & {P_X} \\
	{f^{-1}(N)} & M.
	\arrow[from=1-1, to=1-2]
	\arrow[from=1-1, to=2-1]
	\arrow[from=1-2, to=2-2]
	\arrow[from=2-1, to=2-2]
\end{tikzcd}\]
For our purposes, charts are mainly used to find \textit{exact localizations} for closed immersions, by which we mean diagrams like (\ref{diag: exact localization}) below. 
\begin{lemma}
\label{lem: exact localization}
    Let $\iota : (X, M) \to (Y, N)$ be a closed immersion between fine log schemes. Then \'etale locally on $X$, $\iota'$ factors as 
    \begin{equation}
        \label{diag: exact localization}
        \begin{tikzcd}
	& {(Y', N')} \\
	{(X, M)} & {(Y, N)}
	\arrow["{\textrm{\'etale}}", from=1-2, to=2-2]
	\arrow["\iota"', from=2-1, to=2-2]
	\arrow["\iota'", hook, from=2-1, to=1-2]
    \end{tikzcd}
    \end{equation}
    such that $\iota'$ is an exact closed immersion and the natural projection $(Y', N') \to (Y, N)$ is \'etale. 
    
    More concretely, if $(P_X, \to M, Q_Y \to N, h : Q \to P)$ is a chart for $\iota$ such that $h^{\mathrm{gp}}$ is surjective, then $(Y', N')$ can be constructed by 
    \begin{equation}
        \label{eqn: exact localization}
        Y' \colonequals Y \tensor_{\IZ[Q]} \IZ[Q'] \text{ and } N' \colonequals N|_{Y'},
    \end{equation}
    where $Q' \colonequals (h^{\mathrm{gp}})^{-1}(P)$. 
\end{lemma}
\begin{proof}
    See \cite[Prop.~4.1]{Kat89} and its proof.
\end{proof}

Next, we recall the definition of log-crystalline sites, up to the generality we need. Let $k$ be a perfect field of characteristic $p > 0$ and set $W \colonequals W(k)$ and $W_n \colonequals W/p^n W$ for each $n \in \IN$. Let $\gamma$ denote the standard PD structure on $W$ or $W_n$. First, we summarize in the following proposition the universal property and construction of \textbf{log PD-envelopes}.

\begin{proposition}
\label{prop: log PD-envelope}
    For each $n \in \IN$, consider the following two categories.
    \begin{enumerate}[label=\upshape{(\alph*)}]
        \item The category $\mathscr{C}$ of closed immersions of log schemes $\iota : (X, M) \into (Y, N)$ over $W_n$ such that $M$ is integral and coherent and $N$ is quasi-coherent. 
        \item The category $\wt{\mathscr{C}}$ of pairs $(\iota, \delta)$, where $\iota$ is an exact closed immersions $\iota : (X, M) \into (Y, N)$ of integral and quasi-coherent log schemes over $W_n$, and $\delta$ is a PD structure on the ideal of $Y$ defining $X$ which is compatible with the standard PD-structure $\gamma$. 
    \end{enumerate}
Then the natural forgetful functor $\mathsf{F} : \wt{\mathscr{C}} \to \mathscr{C}$ admits a right adjoint $\mathsf{D}$, which sends an object to what is called its log PD-envelope. Moreover, if $\iota : (X, M) \into (Y, N)$ is an object in $\mathscr{C}$ with $M$ fine and $N$ coherent, and $\iota' : (X, M) \into (Y',N')$ is an \'etale localization as in (\ref{diag: exact localization}), then the log PD-envelope $\mathsf{D}(\iota)$ is of the form $(\wt{\iota} : (X, M') \into (D, M_D), \delta)$, where $(\wt{\iota}, \delta)$ is the usual PD-envelope of $\iota'$ with respect to $\gamma$, $M' = M|_{\wt{X}}$ and $M_D = N'|_{D}$.
\end{proposition}
\begin{proof}
    This follows from \cite[\S1.3, Thm]{Bei} and \cite[Prop.~5.2]{Kat89}. Note that the Beilinson's construction generalizes Kato's and they agree whenever the latter applies. 
\end{proof}

Now we fix a log scheme $(X, M)$ over $k$ of finite type. The \textbf{(big absolute) log-crystalline site} $((X, M)/W_n)_\lcrys$ for $n \in \IN$ (resp. $((X, M)/W)_\lcrys$) is the category whose objects are 5-tuples of the form $(U, T, M_T, \iota, \delta)$ such that $U$ is a scheme over $X$, $(T, M_T)$ is a log scheme over $W_n$ (resp. $W_m$ for some $m$), and $(\iota: (U, M|_U) \into (T, M_T), \delta)$ is an object of $\wt{\mathscr{C}}$ as in \cref{prop: log PD-envelope}. 
As in \Cref{subsec conv} and to simplify notation, we use $(X, M)_\lcrys$ to abbreviate the notation $((X, M)/W)_\lcrys$, and impose the following convention.
\begin{convention}
    When we use $W$ or $W_n$ equipped with the trivial log structure as the base of log-crystalline site, we omit it the log structure from the notation. 
\end{convention}
Define the structure sheaf $\sO_{(X,M)/W_n}$ (or $\sO_{(X, M)/W}$) to be the functor which sends each object $(U, T, M_T, \iota, \delta)$ to $\Gamma(T, \sO_T)$. A \textbf{coherent crystal} on $\sO_{(X,M)/W_n}$ (or $\sO_{(X, M)/W}$) is a sheaf of coherent $\sO_{X/W}$-modules $\sE$ such that for every morphism $g : T' \to T$ in the given site, the induced map $g^* \sE(T) \to \sE(T')$ is an isomorphism. The category of \textbf{isocrystals}, which we shall denote by $\mathrm{Isoc}((X,M)_\lcrys)$, is defined as the $\IQ$-linearization of the category of coherent crystals over $(X,M)_\lcrys$. 
To define $F$-isocrystals, let us first call how the \textbf{absolute Frobenius} morphism $F_{(X, M)} : (X, M) \to (X, M)$ for a log scheme $(X, M)$ over $k$ is defined: On the underlying scheme it is just the standard absolute Frobenius morphism, and $F^{-1}_{(X, M)} M \to M$ is the multiplication by $p$ map when $F^{-1}_{(X, M)} M$ is canonically identified with $M$. An \textbf{$F$-isocrystal} consists of an isocrystal $\sE$ together with an isomorphism $\varphi: F^*_{(X,M)} \sE \stackrel{\sim}{\to} \sE$ in $\Isoc((X,M)_\lcrys)$. We denote the category of $F$-isocrystals by $\Isoc^\varphi((X, M)_\lcrys)$. 

For book-keeping purposes, let us also introduce the category $(X, M)^\wedge_\lcrys$ whose objects are inverse systems $\{ \mathcal{T}_n = (U_n, T_n, M_{T_n}, \iota_n, \delta_n) \in ((X,M)/W_n)_\lcrys \}_{n \in \IN}$ such that the mod $p^n$-reduction of $\mathcal{T}_m$ is $\mathcal{T}_n$ whenever $m \ge n$ and each $\mathcal{T}_n$ is flat over $W_n$. 
We identify such inverse systems with the exact closed immersion of formal $\sO_K$-schemes
$(U, M|_U) \into \varprojlim_n (T_n, M_{T_n})$
with log PD-structure $\varprojlim_n \delta_n$ (cf. \cite[Def.~2.2.1]{Shi20})
, and call such an object \textbf{pro-PD-thickening}. 
Given such $\{ \mathcal{T}_n \}$ and a sheaf $\sE$ of $\sO_{(X, M)/W}$-modules, we may evaluate $\sE$ on $\{ \mathcal{T}_n \}$ and obtain a sheaf of modules $\varprojlim \sE(T_n, M_{T_n})$ on the formal scheme $\varprojlim T_n$. 
The lemma below describes categorical products in $(X, M)^\wedge_\lcrys$. 

\begin{lemma}
\label{lem: products in lcrys site}
    Let $\Lambda$ be a finite set and let $\mathcal{T}_\lambda \colonequals (U_\lambda, T_\lambda, M_{T_\lambda}, \iota_\lambda, \delta_\lambda)$ be objects of $((X, M) / W_n)_\lcrys$ (resp. $(X, M)_\lcrys^\wedge$) indexed by $\lambda \in \Lambda$. Then the (categorical) product $\prod_{\lambda \in \Lambda} \mathcal{T}_\lambda$ is given by the log PD-envelope (resp. the $p$-complete log PD-envelope) $\mathcal{T}$ of the closed immersion 
    $$ \iota : \biggl(U \colonequals \prod_X U_\lambda, M|_U \biggr) \into \biggl(T \colonequals \prod_{W_n} T_\lambda, M_T \colonequals \Bigl(\prod M_{T_\lambda} \Bigr|_T \Bigr)^a \biggr). $$
\end{lemma}
\begin{proof} 
This can be read off from \cite[\S1.5~Prop.]{Bei} and the arguments therein. 
\end{proof}

We also recall the definition of semi-stable $p$-adic formal schemes over $\sO_K$.
\begin{definition}
\label{def st models}
We say that a $p$-adic formal scheme $Y$ over $\sO_K$ is \textit{semi-stable} (resp. \textit{weakly semi-stable}) if $Y$ can be covered in \'etale topology by open affines $U$ \'etale over $$\Spf(\sO_K \langle y_1, \cdots, y_r, y_{r+1}^\pm, \cdots, y_n^\pm \rangle / (y_1 \cdots y_r - \varpi))$$ where $\varpi$ is a uniformizer of $\sO_K$ (resp. $0 \neq \varpi$ is a non-unit in $\sO_K$). Here $r, n, \varpi$ may depend on $U$. 
\end{definition}

Now we review the log structures on the $p$-adic discrete valuation ring $\mathcal{O}_K$ and its Breuil--Kisin ring, which will also be the setup for the rest of the section.
\begin{example}\label{conv of log of base}
	We let  $K$  be  a complete discretely valued $p$-adic field, with $\mathcal{O}_K$ the ring of integers and $k$ the perfect residue field.
	Let $W=W(k)$ be the ring of Witt vectors, and let $\mathfrak{S}$ be the Breuil--Kisin ring $W(k)[\![u]\!]$.
	For a given uniformizer $\pi$ of $\mathcal{O}_K$, 
	we can equip the ring $\mathcal{O}_K$ with the log structure $M_{\mathcal{O}_K}$ associated to the pre-log structure $\mathbb{N}\to \mathcal{O}_K$, $n\mapsto \pi^n$.
	Equivalently, the log structure $M_{\mathcal{O}_K}$ is given by the monoid $\mathcal{O}_K\cap K^\times \to \mathcal{O}_K$.
	Moreover, it induces a log structure $M_k$ on the residue field $k$ by base change, namely the map of monoids $\mathbb{N}\to k$, sending $n$ onto $0^n$, and one can equip a hollow log structure $M_0$ on $W$ that is associated to $\mathbb{N}\to W$, $n\mapsto 0^n$.\footnote{Note that $M_0$ is different from $M_{\mathcal{O}_K}$ even when $W=\mathcal{O}_K$.
		This marks the difference between the \emph{crystalline log structure} and the \emph{de Rham log structure}, which traces back to Hyodo--Kato's comparison theorem of log analogue of crystalline-de Rham comparison.}
	On the other hand, the choice of $\pi$ induces a log structure $M_\mathfrak{S}$ on $\mathfrak{S}$ deforming $M_{\mathcal{O}_K}$, defined by $\mathbb{N}\to \mathfrak{S}$, $n\mapsto u^n$.
	The various (pre-)log structures can be summarized into the following commutative diagram
	\[
	\begin{tikzcd}
		& \mathbb{N} \arrow[ld, "n\mapsto \pi^n"'] \arrow[d, "n\mapsto u^n" description] \arrow[rd, "n\mapsto 0^n"] &\\
		\mathcal{O}_K & \mathfrak{S} \arrow[l, "u\mapsto \pi"] \arrow[r, "u\mapsto 0^n"'] & W.
	\end{tikzcd}
	\]

	Denote by $E(u)$ the Eisenstein polynomial which generates the kernel of $\mathfrak{S}\to \mathcal{O}_K$. Note that with the given log structures, the surjection $\mathfrak{S}\to \mathcal{O}_K$ is already an exact closed immersion. We let $D_\mathfrak{S}(E)$ be the $p$-complete PD-envelope of $(E(u))$, and use $M_\mathfrak{S}$ to denote the induced log structure on  $D_\mathfrak{S}(E)$. More explicitly, 
    \begin{equation}
    \label{eqn: formula for D}
        D =W\llbracket u \rrbracket [ \frac{E(u)^m}{m!}, \forall m\geq 1]^\wedge_p
    \end{equation}
    By the formal smoothness of the ring $\mathfrak{S}$ over $W$, the log PD-thickening $( D_\mathfrak{S}(E), \mathcal{O}_K/p, M_\mathfrak{S})$ is a weakly initial object in the absolute log-crystalline site $\bigl( \mathcal{O}_K/p, M_{\mathcal{O}_K/p} \bigr)_\lcrys$.
\end{example}

Our next example is the log structure on a framed $p$-adic formal scheme with semi-stable reduction.
\begin{example}
	\label{log of framed semi-stable reduction}
	Let $R$ be a $p$-complete algebra over $\mathcal{O}_K$.
	Assume $R$ is ($p$-completely) smooth over $\mathcal{O}_K\langle y_1,\ldots, y_s \rangle/(\prod_{j=1}^s y_j-\pi)$ for some integer $s\geq 1$.
	Then we can equip $R$ with the log structure $M_R$ associated with the map of monoids
	\[
	\mathbb{N}^s \to R, \quad (n_1,\ldots,n_s)\mapsto \prod y_j^{n_j}.
	\]
	This in particular enhances the map of rings $\mathcal{O}_K\to R$ to a map of log pairs $(\mathcal{O}_K, M_{\mathcal{O}_K}) \to (R,M_R)$.
	Here we note that in the special case when $s=1$ and thus $R$ is smooth over $\mathcal{O}_K$, the log structure $M_R$ is nothing but the base change of $M_{\mathcal{O}_K}$ along the map of rings $\mathcal{O}_K\to R$.

More canonically, by \cite[Claim.~1.6.1]{CK19} the log structure $M_R\to R$ is equal to the map of monoids 
	\[
	(R\cap (R[1/p])^\times) \longrightarrow R,
	\]
	which in particular can be globalized to a log structure on a $p$-adic formal scheme $X$, namely
	\[
	M_X: (\mathcal{O}_X \cap \lambda_*\mathcal{O}_{X_\eta}^\times) \longrightarrow \mathcal{O}_X,
	\]
	where $\lambda:X_{\eta,\mathrm{an}} \to X_\mathrm{Zar}$ is the canonical map of sites.
We call this latter construction the \emph{standard log structure} on a $p$-adic formal scheme $X$.
Here we also remark that by construction, the generic fiber $M_{X_\eta}$ is the trivial log structure on $X_\eta$.
\end{example}
	We then record the log structure on perfectoid rings.
\begin{example}[log structure on perfectoids]
	\label{log of perfectoid}
	Let $S$ be a $p$-torsionfree perfectoid ring over $\mathcal{O}_K$.
	We let $M_S$ be the monoid $S\cap S[1/p]^\times$ on $S$, this then forms a log structure on the ring $S$ which is compatible with $M_{\mathcal{O}_K}$ under the natural map $\mathcal{O}_K\to S$.
	More generally, assume $R$ is a topologically finite type $p$-adic formal scheme over $\mathcal{O}_K$ with the standard log structure $M_R$ as in \Cref{log of framed semi-stable reduction}.
	Then any map $R\to S$ can be naturally extended to a map of log pairs $(R,M_R)\to (S,M_S)$, since the image of $y_j$ in $S$ is also contained in $S[1/p]$ as well.	
\end{example}

For later usage, we give an explicit construction of the \v{C}ech nerve for the log PD-thickening associated with Breuil--Kisin ring. 
\begin{lemma}
	\label{Cech for log crys}
	Let $(\mathfrak{S},M_\mathfrak{S})$ be Breuil--Kisin ring with its log structure as defined in \Cref{conv of log of base}, and let $(D,\mathcal{O}_K/p,M_D)\colonequals (D_\mathfrak{S}(E)^\wedge_p,\mathcal{O}_K/p,M_\mathfrak{S})$ be the associated log pro-PD-thickening in the site $(\mathcal{O}_K/p, M_{\mathcal{O}_K})^\wedge_\lcrys$.
	Then the \v{C}ech nerve $(D^\bullet, \sO_K / p, M_{D^\bullet})$ of the covering object $(D, \sO_K/p, M_D)$ in $(\mathcal{O}_K/p, M_{\mathcal{O}_K})^\wedge_\lcrys$ is naturally isomorphic to the cosimplicial $p$-complete PD-polynomials with log structures
	\[
	([n]\longmapsto \bigl( D^n\colonequals D\{t_i,~1\leq i\leq n\},~\mathcal{O}_K/p, ~M_{D^n}\colonequals \{u^{a_0}u_1^{a_1}\cdots u_n^{a_n}~|~ a_i\in \mathbb{Z}, ~\sum a_i \geq 0\}^a \bigr).
	\]
    where $t_i \colonequals u/ u_i - 1$.
\end{lemma}
\begin{proof}
	By \cref{lem: products in lcrys site}, the $(n+1)$-th self coproduct of $(D, \mathcal{O}_K/p,M_D)$ in $(\mathcal{O}_K/p, M_{\mathcal{O}_K/p})^\wedge_\lcrys$ is the $p$-adic completion of the log PD-envelope for the surjection 
 \begin{equation}
 \label{eqn: BK pre-envelope}
     \bigl( (D^{\otimes_W n+1})^\wedge_p, (\prod_{i=0}^n M_\mathfrak{S})^a \bigr) \longrightarrow (\mathcal{O}_K/p, M_{\mathcal{O}_K/p}).
 \end{equation}
    Applying formula (\ref{eqn: exact localization}) in \cref{lem: exact localization}, one finds an exact localization of (\ref{eqn: BK pre-envelope}) given by
    \begin{equation}
        \label{eqn: BK exact localization}
        \bigl( (D^{\otimes_W n+1})^\wedge_p[ \frac{u}{u_i}, \frac{u_i}{u},~\forall 1\leq i\leq n] , \{u^{a_0}u_1^{a_1}\cdots u_n^{a_n}~|~ a_i\in \mathbb{Z}, \sum a_i \geq 0\}^a \bigr) \longrightarrow (\mathcal{O}_K/p , M_{\mathcal{O}_K/p}),
    \end{equation}
	the kernel of which is topologically generated by the elements $u u^{-1}_i-1$, for $1\leq i\leq n$.
	So the ring  $D^n$ is isomorphic to the $p$-complete PD-envelope of the above surjection, namely
	\[
	\bigl( W\llbracket u, u_1,\ldots, u_n\rrbracket[ \frac{E(u)^m}{m!}, \frac{E(u_i)^m}{m!}, \frac{(\frac{u}{u_i}-1)^m}{m!}~|~ \forall 1\leq i\leq n, m\geq 1] )^\wedge_p.
	\]
	This then is naturally isomorphic to the complete divided power polynomial ring of the variables $u/u_i-1$ for $1\leq i\leq n$ over the ring $D$ in \cref{conv of log of base}, namely the ring 
	\[
	D\{\frac{u}{u_i} -1,~1\leq i\leq n\} \simeq D\{t_i,1\leq i\leq n\},
	\]
	for variables $t_i$ representing $u_i/ u -1$.
	Here, the log structure of $D^n$ by construction is associated to the pre-log structure for the monoid $\{u^{a_0}u_1^{a_1}\cdots u_n^{a_n}~|~ a_i\in \mathbb{Z}, ~\sum a_i \geq 0\}$, which we denote as $M_{D^n}$.
\end{proof}

The previous lemma in particular allows us to define the \emph{monodromy operator} on the log PD-thickening $(D,M_D)$. 
\begin{construction}[Monodromy on $D$]
	\label{const: monodromy on D} We let $p_0$ and $p_1$ be the two coface maps from $(D,M_D)$ to $(D^1,M_{D^1})$ in the cosimplicial diagram of \Cref{Cech for log crys}.
	By construction, they are $W$-linear and are induced by sending the element $u\in D$ onto $u$ and $u_1$ in $D^1$ respectively.
	In particular, by taking their difference and translating the formula in terms of $u$ and $t_1$, we get 
	\[
	D\xrightarrow{p_0-p_1} t_1\cdot D\{t_1\},\quad u^m\mapsto u^m-u_1^m=u^m(1-\frac{1}{(1+t_1)^m}).
	\]
	By taking the image of $p_0-p_1$ in the quotient $t_1D\{t_1\}/(t_1^{[\geq2]})^\wedge_p D\{t_1\} \simeq t_1\cdot D \simeq D$,
	we get the \emph{monodromy} operator on $D$, namely the endomorphism
	\[
	N_D\colon D\longrightarrow D,\quad u^m\longmapsto mu^m.
	\]
	It in particular sends a power series $f(u)\in D\subset K_0\llbracket u \rrbracket$ onto the element $f'(u)u$.
	As a consequence, the kernel of the endomorphism $N_D$ is equal to $W$, and by induction the subset of nilpotent elements under $N_D$ is the polynomial subring $W[\log{u}]\subset D$.	
 	Moreover, since the Frobenius $\varphi_\mathfrak{S}$ sends $u$ onto $u^p$, one readily checks that $N_D$ and $\varphi_D$ satisfy the formula $N_D\varphi_D=p\varphi_DN_D$.
  
  Alternatively, regarding $D$ as a subring of $K_0\llbracket u \rrbracket$, one may view $N_D$ as characterized by the property that $\nabla (f(u)) = N_D(f(u)) d \log (u)$, where $\nabla :D \to D \otimes_{W\{u\}} \Omega^1_{(W\{ u \}, (u^\IN)^a)/W} = D\cdot  d \log (u)$ is the standard differentiation.
\end{construction}

\subsection{Logarithmic $F$-isocrystals}
\label{sub: log F isocrystals}
In this subsection, we discuss the various descriptions of the category of $F$-isocrystals and their logarithmic analogues for smooth finite type schemes in characteristic $p$.
Before the discussion, we recall our convention that unless otherwise mentioned, the base of any crystalline site is $\mathcal{O}_{K_0}=W(k)$, and the base of any log-crystalline site is $W(k)$ with the trivial log structure.

First, we recall an explicit local formula for $F$-isocrystals over the crystalline site.
\begin{proposition}\cite{Ogu84}\label{equiv def isoc}
	let $Z$ be a smooth scheme over $\mathcal{O}_K/p$, and let $Z_s$ be its reduced special fiber.
	The following categories are equivalent:
	\begin{enumerate}[label=\upshape{(\alph*)},series=innerlist]
		\item\label{item: a in Ogus} The category of $F$-isocrystals over the convergent site $(Z/W)_{\mathrm{conv}}$ in the sense of Ogus \cite[\S 2]{Ogu84}.
		\item\label{item: b in Ogus} The category of $F$-isocrystals over the crystalline site $Z_{s,\crys}$.
	\end{enumerate}
In addition, if there is a smooth $p$-adic formal scheme $Z_0$ over $W$ that admits a Frobenius lift $\varphi_{Z_0}$ such that $Z\simeq Z\times_{\Spf(W)} \Spf(\mathcal{O}_K)$, then the above categories are also equivalent to the following
    \begin{enumerate}[label=\upshape{(\alph*)},resume=innerlist]
		\item The category $\mathrm{Vect}^{\varphi, \nabla_{Z_0}}(\mathcal{O}_{Z_{0,\eta}})$ of flat connections $\sF$ over the generic fiber $(Z_0)_\eta$ with a horizontal $\varphi_{Z_0}$-linear Frobenius endomorphism whose linearization $\varphi_{\sF}\colon\varphi_{Z_0}^* \sF \to \sF$ is an isomorphism.
	\end{enumerate}
\end{proposition}

\begin{proof}
	Similar to \Cref{Dwork's trick}, the equivalence between the first two categories were proved by Ogus in \cite[Prop.\ 2.18, Ex.\ 2.7.3]{Ogu84}.
	Assume the existence of the lift, then the equivalence between the last two categories is a classical result. 
 It is well known (e.g., by \cite[\href{https://stacks.math.columbia.edu/tag/07JH}{Tag 07JH}]{stacks-project}) that the category of isocrystals on $Z_{s, \crys}$ is equivalent to that of vector bundles on $(Z_0)_\eta$ equipped with an integrable and quasi-nilpotent connection. Then \cite[Thm 2.4.2]{Ber96} tells us that the quasi-nilpotence is automatic when a Frobenius structure is present.
\end{proof}

Our goal is to establish the analogues of the above facts in the log-crystalline setting. 
Let us start with the following.

\begin{proposition}
	\label{Dwork's trick}
	Let $Z$ be a finite type scheme over $\mathcal{O}_K/p$, and let $Z_s$ be the base change $Z\times_{\Spec(\mathcal{O}_K/p)} \Spec(k)$.
	\begin{enumerate}[label=\upshape{(\roman*)}]
		\item\label{Dwork's trick crys} The natural pullback functor along $Z_{s, \crys} \to Z_{\crys}$ induces an equivalence of categories of $F$-isocrystals
		\[
		\Isoc^\varphi(Z_{\crys}) \simeq \Isoc^\varphi(Z_{s,\crys}).
		\]
		
		\item\label{Dwork's trick log-crys} Assume $Z$ admits a log structure $M_Z$, and let $M_s$ be its base change along $\Spec(k)\to \Spec(\mathcal{O}_K/p)$.
		The natural pullback functor along $(Z_s,M_s)_{\lcrys} \to (Z,M)_{\lcrys}$ induces an equivalence of categories of $F$-isocrystals
		\[
		\Isoc^\varphi((Z,M)_{\lcrys}) \simeq \Isoc^\varphi((Z_s,M_s)_{\lcrys}).
		\]
	\end{enumerate}
\end{proposition}
\begin{proof}
	This is known as Dwork's trick, and part \ref{Dwork's trick crys} follows from the proof of \cite[Prop.~2.18]{Ogu84}.
	For part \ref{Dwork's trick log-crys}, it can be proved in the same way by replacing the crystalline site of loc.~cit. by the log-crystalline site (cf. \cite[Rmk~B.23]{DLMS2}).
	For readers' convenience, we explain the construction of the functor that is inverse to the pullback. 
	
	Let $n\in \mathbb{N}$ be an integer so that the $n$-th power of the Frobenius on $\mathcal{O}_K/p$ factors through the natural surjection $\mathcal{O}_K/p\to k$.
	By the functoriality of Frobenius morphisms, we then have a commutative diagram of schemes over $\mathcal{O}_K/p$ that is compatible with their log structures
	\[
	\begin{tikzcd}
		Z \arrow[rr,"F_Z^n"] \arrow[rd, "f"] & &Z\\
		& Z_s, \arrow[ru, hook,"\iota"] &
	\end{tikzcd}
	\]
	where $\iota:Z_s\to Z$ is the natural closed immersion.
	Now we define a functor from $\Isoc^\varphi((Z_s,M_s)_{\lcrys})$ to $\Isoc^\varphi((Z,M)_{\lcrys})$, by sending $(\mathcal{F},\varphi_\mathcal{F})$ onto its pullback $(f^*\mathcal{F},\varphi_{f^*\mathcal{F}})$ along the map of sites $f:(Z_s, M_s)_{\lcrys} \to (Z,M)_{\lcrys}$.
	
	To check that this forms an inverse to the pullback functor along the closed immersion $\iota$, we notice that for an $F$-isocrystal $(\mathcal{E},\varphi_{\mathcal{E}})$ over $\Isoc^\varphi((Z,M)_{\lcrys})$, 
	the Frobenius structure $\varphi_\mathcal{E}$ gives an isomorphism of $F$-isocrystals $F_Z^{n,*} \mathcal{E}\simeq \mathcal{E}$ in $\Isoc^\varphi((Z,M)_{\lcrys})$.
	Notice that by the commutative diagram above, we have
	\[
	F_Z^{n,*} \mathcal{E} \simeq f^*\mathcal{F},
	\]
	where $\mathcal{F}\colonequals \iota^*\mathcal{E}$ is the restriction.
	So we get $f^*\mathcal{F}\simeq \mathcal{E}$, which is functorial with respect to $\mathcal{E}$.
	On the other hand, by the functoriality of Frobenius morphisms, the composition $\iota^*\circ f^*$ is equivalent to the pullback functor $F_{Z_s}^{n,*}$ for the $n$-th power Frobenius $F_{Z_s}^n:Z_s\to Z_s$.
	Hence thanks to the Frobenius structure on isocrystals again, we see the functor $\iota^*\circ f^*$ is an equivalence on $F$-isocrystals.        
\end{proof}

By the above proposition, we can simplify the study of log F-isocrystals on a finite type log scheme over $\sO_K/p$ to the study of log F-isocrystals on a log $k$-variety.

\begin{definition}
\label{def: isocrystals with N}
    Let $Z$ be a $k$-variety. Let $F_Z$ be the absolute Frobenius endomorphism on $Z$. Let $\Isoc^{N}(Z_\crys)$ be the category of pairs $(\sE, N)$ where $(\sE, \varphi_\sE)$ is an F-isocrystal in $ \Isoc^\varphi(Z_\crys)$ and $N$ is an endomorphism of $\sE$.
    Let $\Isoc^{\varphi, N}(Z_\crys)$ be the category of triples $(\sE, \varphi_{\sE}, N_{\sE})$ where $(\sE, \varphi_\sE)$ is an object of $\Isoc^\varphi(Z_\crys)$ and $N_{\sE}$ is an endomorphism of the isocrystal $\sE$, with an induced endomorphism $N_{F_Z^*\sE}\colonequals pF_Z^*N_\sE$ acting on $F_Z^*\sE$.
    We often call $N$ the \textbf{monodromy operator} on $\sE$.
\end{definition}

Note that by definition, the isomorphism of isocrystals $\varphi_\sE:F_Z^*\sE \simeq \sE$ is compatible with the operators $N_{F_Z^*\sE}$ and $N_\sE$.
\begin{proposition}
    For every $(\sE, \varphi_\sE, N_\sE) \in \Isoc^{\varphi, N}(Z_\crys)$, the endomorphism $N_\sE$ is nilpotent. 
\end{proposition}
\begin{proof}
    By \cite[Thm~4.1]{Ogu84} it suffices to show that for some closed point $z$ on each connected component of $Z$, the fiber of $N$ at $z$ is nilpotent. 
    Therefore, we reduce to the case when $Z = \Spec(k)$ is a point itself, so that $\sE$ is nothing but a $K_0$-vector space and $\varphi_\sE$ is a $\sigma$-linear automorphism of $\sE$. 
    We let $\widetilde{\varphi}_\sE$ be the semi-linear endomophism of $\sE$ induced from the isomorphism $\varphi_\sE$.
    Then the condition on the monodromy operator translates to an equality of semi-linear endomorphisms $N_\sE \widetilde{\varphi}_\sE  = p \widetilde{\varphi}_\sE N_\sE$, which implies that $N_\sE \in \mathrm{End}_{K_0}(\sE)$ has no nonzero eigenvalues. 
    Therefore, $(N_\sE)^{\dim \sE} = 0$ by the existence of Jordan normal form. 
\end{proof}

\begin{construction}
\label{const: decompose connections}
    Let $P$ be a topologically finite type $p$-complete and $p$-torsionfree $W$-algebra equipped with the trivial log structure. 
    Set $D_P \colonequals (P \tensor_{W} D)^\wedge_p$.
    As $(D, (u^\IN)^a)$ is (formally) log smooth over $W$, so is $(D_P, (u^\IN)^a)$ over $P$. Hence the conormal sequence below is short exact (see e.g., \cite[IV~Thm~3.2.3]{OgusBook}):
    \begin{equation}
        \label{eqn: conormal sequence}
        0 \to D_P \what{\tensor}_P \Omega^1_{P/W} \to \Omega^1_{(D_P, (u^\IN)^a)/W} \to \Omega^1_{(D_P, (u^\IN)^a)/ P} \to 0.
    \end{equation}
    Since the formation of logarithmic diffentials commutes with base change, there is a canonical isomorphism $\Omega^1_{(D_P, (u^\IN)^a)/ P} \simeq D_P \what{\tensor}_D \Omega^1_{(D, (u^\IN)^a)/W} $. 
    The section $D_P \to P$ given by $u \mapsto 0$ of the natural map $P \to D_P$ provides a splitting of (\ref{eqn: conormal sequence}), or equivalently a decomposition 
    \begin{equation}
    \label{eqn: decompose Omega}
        \Omega^1_{(D_P, (u^\IN)^a)/W} \simeq \bigr(D_P \what{\tensor}_P \Omega^1_{P/W} \bigr) \oplus \bigr(D_P \what{\tensor}_D \Omega^1_{(D, (u^\IN)^a)/W} \bigr) = \bigr(D_P \what{\tensor}_P \Omega^1_{P/W} \bigr) \oplus \bigr(D_P \cdot d\log(u) \bigr).
    \end{equation}
    The decomposition above and the projection formula implies that given a quasi-coherent $D_P$-module $\sM$, to give a continuous connection $\nabla_\sM$ on $\sM$ that has the log-pole at $\{u=0\}$, it is equivalent to give the connections $$\nabla_{\sM, P} :  \sM \to \sM \what{\tensor}_{P} \Omega^1_{P/W} \quad \textrm{and} \quad \nabla_{\sM, u} : \sM \to \sM \cdot d \log(u), $$
    where $\sM$ is viewed as a module over $P$ and $D$ respectively\footnote{This means that we suppressed the notations for the direct image functor after applying the projection formula.}. 
    Let $\nabla_{\sM, \p_u}$ denote the $W$-linear endomorphism on $\sM$ such that 
    \begin{equation}
        \label{eqn: divide dlog(u)}
        \nabla_{\sM, u}(m) = \nabla_{\sM, \p_u}(m) \cdot d \log(u), ~\forall m \in \sM. 
    \end{equation}
    Direct computation shows that $\nabla_\sM^2 = 0$ if and only if $\nabla_{\sM, P}^2 = 0, \nabla_{\sM, u}^2 = 0$, and $\nabla_{\sM, P}$ commutes with $\nabla_{\sM, u}$ (i.e., the following diagram commutes).
    \begin{equation}
    \label{diag: two nablas commute}
        \begin{tikzcd}
	\sM & {\sM \what{\tensor}_P \Omega^1_{P/W}} \\
	\sM & {\sM \what{\tensor}_P \Omega^1_{P/W}}
	\arrow["{\nabla_{\sM, P}}", from=1-1, to=1-2]
	\arrow["{\nabla_{\sM, \p_u}}"', from=1-1, to=2-1]
	\arrow["{\nabla_{\sM, \p_u} \tensor \mathrm{id}}", from=1-2, to=2-2]
	\arrow["{\nabla_{\sM, P}}", from=2-1, to=2-2]
\end{tikzcd}
    \end{equation}
\end{construction}

\begin{construction}
    \label{const: from log to non log}
    Let $Z$ be a $k$-variety. There is a natural fully faithful functor of sites $Z_\crys \to (Z, (0^\IN)^a)_\lcrys$ given be sending a PD-thickening $(T, U) \in Z_\crys$ to $(T, U, (0^\IN)^a)$. This induces a natural forgetful functor $ \mathsf{F} : \Isoc((Z, (0^\IN)^a)_\lcrys) \to \Isoc(Z_\crys)$. 
    Now we construct a \textbf{residue functor} $\mathrm{Res}$ which equips a monodromy operator $N_{\sfF(\sE)}$ on $\sfF(\sE)$ for each $\sE \in \Isoc((Z, 0^\IN)^a)_\lcrys$ which will fit into the following commutative diagram:
    \[\begin{tikzcd}
	& {\Isoc^N(Z_\crys)} \\
	{\Isoc((Z,(0^\IN)^a)_\lcrys)} & {\Isoc(Z_\crys)}.
	\arrow["{\textrm{forget $N$}}", from=1-2, to=2-2]
	\arrow["{\mathrm{Res}}", from=2-1, to=1-2]
	\arrow["{\mathsf{F}}", from=2-1, to=2-2]
    \end{tikzcd}\]
    Fix such an isocrystal $\sE$. 
    Our goal is to construct for each pro-PD-thickening $(T, U) \in Z_\crys$ an endomorphism $N_\sE (T)$ of $\sE(T, U, (0^\IN)^a)$ which is functorial in $(T, U)$. 

    We may assume $T=\Spf(P)$ is affine and perform \cref{const: decompose connections} for $\sM=\sE(\Spf(D_P),U,(u^\IN)^a)$.
    Note that the exact surjection of the log pairs $(D_P, (u^\IN)^a) \to (P, (0^\IN)^a)$ that is given by $u \mapsto 0$ induces a morphism between objects in $(Z, (0^\IN)^a)_\lcrys$, so that by the crystalline nature of $\sE$ there is a canonical isomorphism $\sM \otimes_{D_P} P \simeq \sE(T, U, (0^\IN)^a)$.
    By the Leibniz rule, for any $f(u) \in D \subseteq K_0[\![u ]\!]$ and $m \in \sM$, we have
    $$ \nabla_{\sM, \p_u} ( f(u) m) = u f'(u) m + f(u) \nabla_{\sM, u} (m). $$
    This implies that $\nabla_{\sM, \p_u}(u \sM) \subseteq u \sM$. Therefore, $\nabla_{\sM, \p_u}$ induces an endomorphism on $$ \sM \cdot d \log (u) / \sM \cdot du \simeq \sM / u \sM \simeq \sE(T, U, (0^\IN)^a),$$ which is the sought after $N_{\sfF(\sE)}(T,U)$. The commutativity of the diagram (\ref{diag: two nablas commute}) implies that $N_{\sfF(\sE)}(T,U)$ is $P$-linear. 
    Putting $N_{\sfF(\sE)}(T,U)$ together as $(T, U)$ runs through objects in $Z_\crys$, we obtain the desired endomorphism $N_{\sfF(\sE)}$ on $\sfF(\sE)$. 
    The construction is functorial, yielding the functor $\mathrm{Res} : \sE \mapsto (\sfF(\sE), N_{\sfF(\sE)})$.
\end{construction}

The functor $\sfF$ clearly preserves Frobenius structures when they are present, and naturally induces a functor $\Isoc^\varphi : \Isoc^\varphi ((Z, (0^\IN)^a)_\lcrys) \to \Isoc^{\varphi}(Z_\crys)$ which is denoted by $\sfF$ as well. 
In addition, the residue functor $\mathrm{Res}$ can be upgraded to a functor $\Isoc^\varphi((Z, (0^\IN)^a)_\lcrys) \to \Isoc^{\varphi, N}(Z_\crys)$, as we shall see below.

\begin{proposition}
\label{prop:from_log_F_isoc_to_varphi-N-isoc}
    Let $(\sE, \varphi_\sE) \in \Isoc^\varphi((Z, (0^\IN)^a)_\lcrys)$, let $(\sE', \varphi_{\sE'})$ be the $F$-isocrystal $\sfF(\sE, \varphi_\sE)$ over the crystalline site $Z_\crys$, and let $(\sE', N_{\sE'})$ be the isocrystal with the monodromy operator $\mathrm{Res}(\sE)$. Then the triple $(\sE', \varphi_{\sE'}, N_{\sE'})$ is naturally an object in $\Isoc^{\varphi, N}(Z_\crys)$, yielding a functor
    \[
    \mathrm{Res}\colon \Isoc^\varphi((Z, (0^\IN)^a)_\lcrys) \longrightarrow \Isoc^{\varphi, N}(Z_\crys).
    \]
\end{proposition}

\begin{proof}
    We let $\sE'$ be the underlying isocrystal of $\mathsf{F}(\sE)$ over $Z_\crys$. 
    Note that $F_Z^* \sE$ is an isocrystal on $(Z, (0^\IN)^a)_\crys$ in its own and $\sfF(F_Z^* \sE) = F_Z^* \sE'$. 
    Let $N_{F_Z^* \sE'}$ be the monodromy operator on $F_Z^* \sE'$ such that $\mathrm{Res}(F_Z^* \sE) = (F_Z^* \sE, N_{F_Z^* \sE'})$. 
    Now we compare the following two diagrams. 
    \begin{equation}
    \label{diag: two diagrams for monodromy}
        \begin{tikzcd}
	{F_Z^* \sE'} & {\sE'} && {F_Z^* \sE'} & {\sE'} \\
	{F_Z^* \sE'} & {\sE'} && {F_Z^* \sE'} & {\sE'}
	\arrow["{\varphi_{\sE'}}", from=1-1, to=1-2]
	\arrow["{N_{F_Z^* \sE'}}"', from=1-1, to=2-1]
	\arrow["{N_\sE'}", from=1-2, to=2-2]
	\arrow["{\varphi_{\sE'}}", from=1-4, to=1-5]
	\arrow["{pF_Z^* N_{\sE'}}"', from=1-4, to=2-4]
	\arrow["{ N_{\sE'}}", from=1-5, to=2-5]
	\arrow["{\varphi_{\sE'}}", from=2-1, to=2-2]
	\arrow["{\varphi_{\sE'}}", from=2-4, to=2-5]
        \end{tikzcd}
    \end{equation}
    The first diagram commutes by the functoriality of $\mathrm{Res}$, and our goal is to show that the second one also commutes. Therefore, we reduce to showing that $N_{F_Z^* \sE'} = p F_Z^* N_{\sE'}$.

    Zariski locally we can always embed $Z$ into a smooth formal scheme over $W$, and by taking PD-envelope, we find an object $(T=\Spf(P), U) \in Z_\crys$ which admits a Frobenius endomorphism $\varphi_P$ compatible with that of $Z$. Moreover, such $(T, U)$ is weakly final, so it suffices to show that the above equality is satisfied once evaluated at $(T,U)$. Now perform \cref{const: decompose connections}. The object $D_P$ is equipped with a diagonal Frobenius action $\varphi_{D_P} \colonequals \varphi_P \what{\tensor}_W \varphi_D$. 
    Let $\sM \colonequals \sE(D_P, U, (u^\IN)^a)$. 
    Then by \Cref{const: from log to non log}, the monodromy operator $N_{F_Z^* \sE'}$ is computed by taking the residue of the pullback (log) connection $\varphi_{D_P}^* (\nabla_{\sM})$ on $\varphi_{D_P}^* \sM = (F_Z^* \sE')(D_P)$. Hence the equation $N_{F_Z^* \sE'}(T,U) = p F_Z^* N_{\sE'}(T,U)$ follows from the fact that $\varphi_D^* (d \log(u)) = d \log(u^p) = p \cdot d \log(u)$. 
\end{proof}

\begin{lemma}
	\label{prop:equiv def of crys vs pullback new}
    Let $f : Z' \to Z$ be a morphism between $k$-varieties. Then the following diagrams of categories commute.
    \begin{equation*}
    \label{diag: equiv def of crys vs pullback}
        \begin{tikzcd}
	{\Isoc((Z, (0^\IN)^a)_\lcrys)} & {\Isoc^N(Z_\crys)} & {\Isoc^\varphi((Z, (0^\IN)^a)_\lcrys)} & {\Isoc^{\varphi, N}(Z_\crys)} \\
	{\Isoc((Z', (0^\IN)^a)_\lcrys)} & {\Isoc^N(Z'_\crys)} & {\Isoc^\varphi((Z', (0^\IN)^a)_\lcrys)} & {\Isoc^{\varphi, N}(Z_\crys)}
	\arrow["{\mathrm{Res}}", from=1-1, to=1-2]
	\arrow["{f^*}"', from=1-1, to=2-1]
	\arrow["{f^*}", from=1-2, to=2-2]
	\arrow["{\mathrm{Res}}", from=1-3, to=1-4]
	\arrow["{f^*}"', from=1-3, to=2-3]
	\arrow["{f^*}"', from=1-4, to=2-4]
	\arrow["{\mathrm{Res}}", from=2-1, to=2-2]
	\arrow["{\mathrm{Res}}", from=2-3, to=2-4]
        \end{tikzcd}
    \end{equation*}
\end{lemma}
\begin{proof}
    This is clear from the construction of $\mathrm{Res}$. 
\end{proof}

Motivated by the above results, we make the following construction. 

\begin{definition}
\label{const: Frob flat modules} Let $P$ be a topologically finite type $p$-complete and $p$-torsionfree $W$-algebra equipped with the trivial log structure. Set $D \colonequals W\{u\}$ and $D_P \colonequals (D \tensor_W P)^\wedge_p$. 
Assume $\varphi_P$ is an endomorphism of $P$ which lifts the absolute Frobenius morphism of the mod $p$ reduction. 
Define the Frobenius action $\varphi_{D_P}$ on $D_P$ by $\varphi_D \tensor \varphi_P$. 
Then we define the following categories of Frobenius modules with various additional structures.  
    \begin{enumerate}[label=(\alph*)]
    \item The category $\mathrm{Coh}^\varphi(D_P[1/p])$ (resp. $\Coh^\varphi(P[1/p])$) of coherent $D_P[1/p]$-modules $\sM$ (resp. coherent $P[1/p]$-modules $\sM_0$) together with a Frobenius structure $\varphi_\sM : \varphi^*_{D_P} \sM \stackrel{\sim}{\to} \sM$ (resp. $\varphi_{\sM_0} : \varphi_P^* \sM_0 \stackrel{\sim}{\to} \sM_0$). 
    \item The category $\mathrm{Coh}^\nabla(D_P[1/p])$ of coherent $D_P[1/p]$-modules $\sM$ together with a continuous quasi-nilpotent\footnote{We are referring to condition (iii) in \cite[Thm~6.2(b)]{Kat89}.} integrable log connection $\nabla_\sM : \sM \to \sM \ctensor_{D_P} \Omega^1_{(D_P, (u^\IN)^a)/W}$. 
    \item The category $\mathrm{Coh}^{\varphi, \nabla}(D_P[1/p])$ of triples $(\sM, \varphi_\sM, \nabla_\sM)$ where $(\sM, \nabla_\sM) \in \mathrm{Coh}^\nabla(D_P[1/p])$ and $\varphi_\sM : \varphi_{D_P}^* \sM \stackrel{\sim}{\to} \sM$ is a horizontal Frobenius structure on $\sM$. 
    \item The category $\mathrm{Coh}^{\nabla_P,\nabla_u}({D_P}[1/p])$ of coherent ${D_P}[1/p]$-modules $\sM$ equipped with two commuting (cf. Diagram~\ref{diag: two nablas commute}) quasi-nilpotent continuous integrable log connections $\nabla_{\sM, P}$ and $\nabla_{\sM, u}$ on $\sM$ when $\sM$ is viewed as a module over $P$ and $(D, (u^\IN)^a)$ respectively.
    \item The category $\mathrm{Coh}^{\varphi, \nabla_P,\nabla_u}({D_P}[1/p])$ of tuples $(\sM, \varphi_{\sM}, \nabla_{\sM, P}, \nabla_{\sM, u})$, such that the triple $(\sM, \nabla_{\sM, P}, \nabla_{\sM, u})$ belongs to $\mathrm{Coh}^{\nabla_P,\nabla_u}({D_P}[1/p])$, and $\varphi_\sM : \varphi_{D_P}^* \sM \stackrel{\sim}{\to} \sM$ is a Frobenius structure on $\sM$ which is horizontal with respect to both $\nabla_{\sM, P}$ and $\nabla_{\sM, u}$. 
    \item The category $\mathrm{Coh}^{\varphi, \nabla_P, N}(P[1/p])$ of coherent $P[1/p]$-modules $\sM_0$ with quasi-nilpotent integrable connection $\nabla_{\sM_0,P}$, horizontal Frobenius action $\varphi_{\sM_0} : \varphi_{P}^* \sM_0 \stackrel{\sim}{\to} \sM_0$ and horizontal endomorphism $N_{\sM_0}$ of $\sM_0$ such that $N_{\varphi_P^*\sM_0} = p\varphi_P^* N_{\sM_0}$. 
\end{enumerate}
\end{definition}

\begin{remark}
\label{rmk: Frob modules}
\begin{enumerate}[label=\upshape{(\alph*)}]
    \item By \cref{const: decompose connections}, there are natural identifications $\Coh^{\nabla}({D_P}[1/p]) \simeq \Coh^{\nabla_P,\nabla_u}({D_P}[1/p])$ and $\Coh^{\varphi, \nabla}({D_P}[1/p]) \simeq \Coh^{\varphi, \nabla_P,\nabla_u}({D_P}[1/p])$. 
    We remark that if a log connection $\nabla_\sM$ on a ${D_P}$-module $\sM$ decomposes into $\nabla_{\sM, P}\oplus \nabla_{\sM, u}$ as in \cref{const: decompose connections}, then $\nabla_\sM$ is quasi-nilpotent if and only if both $\nabla_{\sM, P}$ and $\nabla_{\sM, u}$ are quasi-nilpotent.
    \item As in \cref{const: from log to non log} and proof of \Cref{prop:from_log_F_isoc_to_varphi-N-isoc}, given $(\sM, \nabla_\sM) \in \Coh^{\nabla_P,\nabla_u}({D_P}[1/p])$, we can attach a natural monodromy operator $N_{\sM_0}$ on $\sM_0 = \sM/u \sM$ by reducing $\nabla_{\sM, u}$ mod $u$. 
    Hence there is a natural functor $\Coh^{\varphi, \nabla_P,\nabla_u}({D_P}[1/p]) \to \Coh^{\varphi, \nabla_P, N}(P[1/p])$ induced by $u \mapsto 0$. 
\end{enumerate}
\end{remark}

The following equivalence on Frobenius equivariant flat connections over two different base rings is inspired by Breuil's work \cite[\S 6]{Bre97} and can be used to prove the higher dimensional version of loc.\ cit..
Roughly speaking, the statement can be thought as a pro-analogue of Dwork's trick.

\begin{theorem}
	\label{thm:Breuil}
	Consider the set up in \cref{const: Frob flat modules}.  
	We view $P$ (resp. $D = W\{ u \}$) as a ${D_P}$-algebra (resp. $W$-algebra) via the surjection induced by $u \mapsto 0$. Then the base change functor $\Coh^{\varphi}({D_P}[1/p]) \to \Coh^\varphi(P[1/p])$ induced by $u \mapsto 0$ is an equivalence of categories, with the quasi-inverse of which given by $$(\sM_0, \varphi_{\sM_0}) \mapsto (D \ctensor_W \sM_0\simeq D_P[1/p]\otimes_P \sM_0, \varphi_D \tensor \varphi_{\sM_0}). $$
    Equivalently, we have the following two statements. 
    \begin{enumerate}[label=\upshape{(\roman*)}]
        \item\label{eqn: Breuil 0.0} For every $(\sM, \varphi_\sM) \in \mathrm{Coh}^{\varphi}({D_P}[1/p])$ with mod $u$ reduction $(\sM_0, \varphi_{\sM_0})$, there exists a unique Frobenius-equivariant section $s$ to the reduction map $\sM \to \sM_0$. 
        \item\label{eqn: Breuil 0.5} In Part \ref{eqn: Breuil 0.0}, the morphism $D \ctensor_W \sM_0\to \sM$ induced by $s$ is a Frobenius-equivariant isomorphism, with the Frobenius action on the domain given by $\varphi_D \tensor \varphi_{\sM_0}$. 
    \end{enumerate}
\end{theorem}

\begin{proof}
As the Tate twists on both categories are compatible with the functor, it suffices to show the equivalence for the subcategories consisting of effective objects (namely those Frobenius modules where the inverse of the Frobenius structures can be defined on a finitely generated submodule over $D_P$ (resp. $P$)).
    So we let $\sL \subseteq \sM$ be a ${D_P}$-lattice such that $\varphi^{-1}_\sM(\sL) \subseteq \varphi_{D_P}^* \sL$.

    To begin, we produce a sequence of maps $\{ \psi_n:\varphi_{T}^{n,*} \sM \to \sM \}$ as below.
    Let $\iota: \varphi_{P}^* \sM\to \varphi_{{D_P}}^*\sM= (\varphi_{P}^*\sM)\otimes_{D,\varphi_D} D$ be the $P[1/p]$-linear map which sends an element $x$ to $x\otimes 1$, and set $\psi_1 \colonequals \iota \circ \varphi_\sM$. 
    Note that $\psi_1$ is semi-linear with respect to $\varphi_D$.
    Then we have the following commutative diagram: 
	\begin{equation}
		\label{eq:Breuil diagram 1}
		\begin{tikzcd}
	{\varphi_{P}^*\sM} & {\varphi_{{D_P}}^* \sM} & \sM \\
	& {\varphi_{P}^*\sM_0} & {\sM_0}
	\arrow["\iota", from=1-1, to=1-2]
	\arrow["{\psi_1}", curve={height=-18pt}, from=1-1, to=1-3]
	\arrow["{u \mapsto 0}"', from=1-1, to=2-2]
	\arrow["{\varphi_\sM}","\sim"', from=1-2, to=1-3]
	\arrow["{u\mapsto 0}", from=1-2, to=2-2]
	\arrow["{u\mapsto 0}", from=1-3, to=2-3]
	\arrow["{\varphi_{\sM_0}}","\sim"', from=2-2, to=2-3]
		\end{tikzcd}
	\end{equation}
	Now we iterate the construction by constructing the diagram below, where the map $\varphi^{i, *}_P \sM \to \varphi^{i - 1, *}_P \sM$ is given by pulling back $\psi_1$ along the map $\varphi^{i - 1}_P$. 
	\begin{equation}
		\label{eq:Breuil diagram 2}
		\begin{tikzcd}
					\varphi_{P}^{n,*} \sM \arrow[r] \arrow[d, "u\mapsto 0"'] & \varphi_{P}^{n-1,*} \sM \arrow[d, "u\mapsto 0"'] \ar[r] & \cdots \ar[r] & \sM \arrow[d, "u\mapsto 0"'] \\
			\varphi_{P}^{n,*} \sM_0 \arrow[r] & \varphi_{P}^{n-1,*} \sM_0 \arrow[r] &\cdots \arrow[r] & \sM_0,
		\end{tikzcd}
	\end{equation}
    Let $\psi_n:\varphi_{P}^{n,*} \sM \to \sM$ be the composite of the maps on the top row. 
    
    Now we produce the section map.
    Given any $x_0\in \sL_0 \colonequals \sL/u \sL$, we let $x_n$ be its unique preimage in $\varphi_P^{n,*}\sL_0\subset \varphi_{P}^{n,*} \sM_0$.
    Then we note the following lemma. 

    \begin{lemma}
    \label{lem: for Breuil}
       Let $n$ be a positive integer and let $y_n, z_n$ be any two lifts of the element $x_n\in \varphi_P^{n,*}\sL_0$ in $\varphi_P^{n, *} \sL$.
       Then $\psi_n(y_n) - \psi_n(z_n) \in p^n! \sL$. 
    \end{lemma}
    \begin{proof}
        Consider the top row of Diagram~\ref{eq:Breuil diagram 2}.  
        As both $y_n, z_n$ are lifts of the element $x_n$ along the mod $u$ surjection $\varphi_P^{n,*}\sM\to \varphi_P^{n,*}\sM_0$, we have the difference element $y_n - z_n$ is contained in $u \sL$. 
        By the $\varphi_D^n$-semi-linearity of the map $\psi_n$, we then have 
        \begin{align}
    	\label{eq:Breuil difference}
    	\psi_n(y_n)-\psi_n(y_n') \in u^{p^n}\sL=p^n!\cdot \frac{u^{p^n}}{p^n!}\sL \subseteq p^n! \sL,
        \end{align}
        where the last inclusion follows from the construction that $u^{p^n}/ p^n ! \in D$. 
    \end{proof}

    For each $n\in \IN$, let $y_n \in \varphi_P^{n, *} \sM$ be a sequence of lifts of $x_n$. It follows from the above lemma that $\psi_n(y_n)$ converges in the $p$-adic topology of $\sL$ and its limit is independent of the choice of $y_n$'s.
    Therefore, we obtain a well-defined map $\sL_0 \to \sL$ by sending $x_0$ to $\lim_n \psi_n(y_n)$. By inverting $p$, we obtain a Frobenius-equivariant section map $s : \sM_0 \to \sM$. 
    One readily checks that the construction of $s$ is independent of the choice of the lattice $\sL$ in the beginning. 

    To see that the section map $s$ is the unique Frobenius-equivariant section, we note that if $s'$ is another such a section, then Diagram~(\ref{eq:Breuil diagram 2}) can be extended as below: 
        \[\begin{tikzcd}
	{\varphi_{P}^{n,*} \sM} & {\varphi_{P}^{n-1,*} \sM} & \cdots & \sM \\
	{\varphi_{P}^{n,*} \sM_0} & {\varphi_{P}^{n-1,*} \sM_0} & \cdots & {\sM_0,}
	\arrow[from=1-1, to=1-2]
	\arrow[from=1-1, to=2-1]
	\arrow[from=1-2, to=1-3]
	\arrow[from=1-2, to=2-2]
	\arrow[from=1-3, to=1-4]
	\arrow[from=1-4, to=2-4]
	\arrow["{\varphi^{n, *}_{P}(s')}"', curve={height=12pt}, from=2-1, to=1-1]
	\arrow[from=2-1, to=2-2]
	\arrow["{\varphi^{n-1, *}_{P}(s')}"', curve={height=12pt}, from=2-2, to=1-2]
	\arrow[from=2-2, to=2-3]
	\arrow[from=2-3, to=2-4]
	\arrow["{s'}"', curve={height=12pt}, from=2-4, to=1-4]
        \end{tikzcd} \]
    Given $x \in \sM_0$ and the preimages $x_n \in \varphi^{n, *}_{P} \sM_0$ as before, we may take $y_n$ to be ${\varphi^{n, *}_{P}(s')}(x_n)$. 
    Then by construction, we get $s(x) = \lim_n \psi_n(y_n)$. This completes the proof of \ref{eqn: Breuil 0.0}. 

    Now we prove \ref{eqn: Breuil 0.5}. 
    As above, we let $\sL$ be a lattce that is preserved by $\varphi_{\sM}^{-1}$, and write $(\sM', \varphi_{\sM'})$ for $((D \tensor_W \sL_0)^\wedge_p[1/p], \varphi_D \tensor \varphi_{\sM_0})$, with $\sL'=(D\otimes_W \sL_0)^\wedge_p$ the induced lattice of $\sM'$.
    We let $h$ be the natural map $\sL' \to \sL$ induced by $s$, which naturally extends to a map $\widetilde{h}:\sM'\to \sM$.
    Since the element $u$ is contained in the Jacobson ideal of the ring $D_P$ and both $\sL$ and $\sL'$ are finitely generated over $D_P$, we know by Nakayama's lemma that the map $h$ and hence $\widetilde{h}$ are surjections.
	In addition, since the composition of $h$ with the map $\sL \to \sL_0$ is a surjection of $P$-modules, the kernel module $\ker(h)$ is contained in $u \sL'$.
 We want to show that in fact $\ker(h) = 0$. 
	Indeed, consider the following commutative diagram.
    \[\begin{tikzcd}
	 {\varphi_{D_P}^* \sM'} & {\sM'} \\
{\varphi_{D_P}^* \sM} & \sM
	\arrow["{\varphi_{\sM'}}", "\sim"', from=1-1, to=1-2]
	\arrow["{\varphi_{D_P}^* \widetilde{h}}"', from=1-1, to=2-1]
\arrow["\widetilde{h}"', from=1-2, to=2-2]
	\arrow["{\varphi_\sM}","\sim"', from=2-1, to=2-2]
    \end{tikzcd}\]
    On the one hand, by taking the Frobenius pullback, we know that $\ker(\varphi_{D_P}^* h)$ is contained in $u^p\sL'$.
    So by inverting the element $p$ and the commutative diagram above, we see that $\ker(\widetilde{h})$ is contained in $u^p\sM'$.
    On the other hand, by the explicit construction of $\sM'$ and $\sL'$, we know $u^p\sM'\cap \sL'=u^p\sL'$.
    Hence we have $\ker(h)=\ker(\widetilde{h})\cap \sL'\subset u^p\sL'$.
    Iterating this process, we must have $\ker(h) \subseteq u^{p^m} \sL'$ for each $m \in \IN$. 
    However, as $\sL'$ is $u$-separated, we must have $\ker(h) = 0$. 
\end{proof}

\begin{corollary}
\label{cor: Breuil}
    In the setting of \cref{const: Frob flat modules}, the base change functor $\Coh^{\varphi, \nabla_P,\nabla_u}({D_P}[1/p]) \to \Coh^{\varphi, \nabla_P, N}(P[1/p])$ induced by $u \mapsto 0$ is an equivalence of categories. The quasi-inverse functor is given by sending $(\sM_0,\varphi_{\sM_0},\nabla_{\sM_0},N_{\sM_0}) \in \Coh^{\varphi, \nabla_P, N}(P[1/p])$ to the object $(\sM, \varphi_{\sM}, \nabla_{\sM,P},\nabla_{\sM,u})$ defined by 
    \begin{equation}
    \label{eqn: quasi-inverse of mod u functor}
         \bigl(D \ctensor_W \sM_0,~ \varphi_D \tensor \varphi_{\sM_0},~ \mathrm{id}_D \tensor \nabla_{\sM_0},~ [\mathrm{id}_D \tensor N_{\sM_0} + N_D \ctensor \mathrm{id}_{\sM_0}] \cdot d\log(u) \bigr)
    \end{equation}
    in $\Coh^{\varphi, \nabla_P,\nabla_u}({D_P}[1/p])$.
\end{corollary}
\begin{proof}
    First, one checks by routine computations that the formula (\ref{eqn: quasi-inverse of mod u functor}) indeed defines an object of $\Coh^{\varphi, \nabla_P,\nabla_u}({D_P}[1/p])$ (cf. \cref{rmk: Frob modules}). 
    We remark that the condition that $ N_{\varphi_P^*\sM_0} = p\varphi_{P} ^* N_{\sM_0}$ ensures that the log-connection $\varphi_\sM$, defined by putting $\nabla_{\sM,P}$ and $\nabla_{\sM,u}$ together, is horizontal with respect to the linear isomorphism $\nabla_{\sM,P}:\varphi_{D_P}^*\sM\to \sM$.
    
    Suppose that $(\sM, \varphi_\sM, \nabla_{\sM, P}, \nabla_{\sM, u})$ is an object of $\Coh^{\varphi, \nabla_P,\nabla_u}({D_P}[1/p])$ with mod $u$ reduction $(\sM_0,\varphi_{\sM_0},\nabla_{\sM_0},N_{\sM_0})$. 
    Then \cref{thm:Breuil} tells us that there is a unique Frobenius-equivariant isomorphism $\sM \simeq D \ctensor_W \sM_0$. 
    It remains to show that under this isomorphism, both $\nabla_{\sM, P}$ and $\nabla_{\sM, u}$ have to be given by the formula in (\ref{eqn: quasi-inverse of mod u functor}).  
    Note that by putting $\nabla_{\sM, P}$ and $\nabla_{\sM, u}$ together to form a log connection $\nabla_\sM$ on $\sM$, we obtain an equivalence of categories $\Coh^{\varphi, \nabla_P,\nabla_u}({D_P}[1/p]) \simeq \Coh^{\varphi, \nabla}({D_P}[1/p])$  (cf. \cref{const: decompose connections}). 
    So it suffices to show that $\nabla_\sM$ is the unique log connection for which $\varphi_\sM$ is horizontal and whose mod $u$ reduction gives $\nabla_{\sM_0}$ and $N_{\sM_0}$. 
    Indeed, if $\nabla'_\sM$ is another such log connection, then $\nabla_\sM - \nabla_\sM'$ is a Frobenius equivariant ${D_P}$-linear morphism $\sM \to \sM \tensor \Omega^1_{({D_P}, (u^\IN)^a)/W}$ whose mod $u$ reduction is $0$, where we equip $\sM \tensor \Omega^1_{D_P/W}$ with the  the Frobenius structure 
 defined by $\varphi_\sM \tensor \mathrm{id}$.
    So the map $\nabla_\sM - \nabla'_\sM$ is a morphism between two objects in $\Coh^{\varphi}({D_P}[1/p])$ which vanishes under the reduction functor $\Coh^\varphi({D_P}[1/p]) \to \Coh^\varphi(P[1/p])$. 
    Since this reduction functor is in particular fully faithful, we see the map $\nabla_\sM - \nabla'_\sM$ must be zero. 
\end{proof}

\begin{theorem} 
\label{thm: log crys = crys + N}
Suppose that $Z$ is a smooth $k$-variety. Then the residue functor $$\mathrm{Res} : \Isoc^{\varphi}((Z, (0^\IN)^a) \to \Isoc^{\varphi, N}(Z_\crys)$$ is an equivalence of categories. Moreover, every object of $\Isoc^{\varphi}((Z, (0^\IN)^a)$ is locally free. 
\end{theorem}
\begin{proof}
    We may assume that $Z$ is affine and lifts to some smooth formal scheme $\Spf(P)$ over $W$. 
    Moreover, assume that $F_Z$ lifts to $\varphi_P$. 
    Apply \cref{const: decompose connections} and then we obtain the following commutative diagram. 
    \begin{equation}
		\label{eq:diagram of cat of isoc}
		\begin{tikzcd}
	{\Isoc^\varphi \bigl( (Z, (0^\IN)^a)_\lcrys \bigr)} && { \Isoc^{\varphi,N} (Z_{\crys})} \\
	{\mathrm{Coh}^{\varphi, \nabla_P,\nabla_u}({D_P}[1/p])} && {\Coh^{\varphi,\nabla_P,N}(P[1/p])}
	\arrow["\mathrm{Res}", from=1-1, to=1-3]
	\arrow[from=1-1, to=2-1]
	\arrow[from=1-3, to=2-3]
	\arrow["u \mapsto 0", from=2-1, to=2-3]
        \end{tikzcd}
	\end{equation}
    The left vertical arrow is given by the composition 
    \[\begin{tikzcd}
	{\Isoc^\varphi \bigl( (Z, (0^\IN)^a)_\lcrys \bigr)} & {\Coh^{\varphi, \nabla}({D_P}[1/p])} & {\Coh^{\varphi, \nabla_P,\nabla_u}({D_P}[1/p])}
	\arrow["\sim", from=1-1, to=1-2]
	\arrow["\sim", from=1-2, to=1-3]
    \end{tikzcd}\]
    where the first equivalence is provided by \cite[Thm~6.2]{Kat89} and the second by \cref{const: decompose connections}. 
    The right vertical arrow in Diagram \ref{eq:diagram of cat of isoc} is an equivalence, provided by \cref{equiv def isoc}. 
    The bottom horizontal arrow is an equivalence, thanks to \cref{cor: Breuil}. Therefore, $\mathrm{Res}$ must also be an equivalence. 

    Now given $\sE \in \Isoc^\varphi ((Z, (0^\IN)^a)_\lcrys)$, it is locally free if and only if the ${D_P}[1/p]$-module $\sM \colonequals \sE({D_P}, Z, (u^\IN)^a)$ is locally free, as $(\Spf({D_P}), Z, (u^\IN)^a)$ is a weakly final object in $(Z, (0^\IN)^a)_\lcrys$. 
    Since $Z$ is smooth, the isocrystal $\sfF(\sE)$ on $Z_\crys$ is automatically locally free (cf. \cite[Lem.~4.9]{Xu19}). 
    Therefore, the $P[1/p]$-module $\sM_0 \colonequals \sfF(\sE)(\Spf(P), Z)$ is locally free. 
    Now note that $\sM_0 = \sM / u \sM = \sE(\Spf(P), Z, (0^\IN)^a)$. 
    Therefore, $\sM$ must also be locally free as $\sM \simeq D_P[1/p] \tensor_{P[1/p]} \sM_0$, thanks to the equivalence by \cref{thm:Breuil}. 
\end{proof}

\begin{remark}
\label{rmk: non loc free example}
    The log version of the local-freeness assertion \cite[Lem.~4.9]{Xu19} is not true without assuming the existence of Frobenius action. 
    Moreover, it is not true that $\Isoc((Z, (0^\IN)^a)_\lcrys) \to \Isoc^N(Z_\crys)$ is an equivalence even when $Z$ is smooth. Here is an example: Let $Z = \Spec(k)$ be a point. Set $\sM \colonequals W$, viewed as a $D = W\{ u \}$-module via $u \mapsto 0$. 
    Then we may define a ``sky-scraper'' log-connection on the $D$-module $\sM$ by equipping it with the zero log connection $\nabla_{\sM} = 0$. 
    The resulting isocrystal $\sE$ on $(\Spec(k), (0^\IN)^a)_\lcrys$ is clearly not locally free. 
    On the other hand, both $\sE$ and the trivial isocrystal $\sO_{(Z, (0^\IN)^a)/W}[1/p]$ on $(Z, (0^\IN)^a)_\lcrys$ have the isomorphic image in $\Isoc^N(Z_\crys)$. 
    
    However, the isocrystal $\sE$ cannot be equipped with a Frobenius structure, because $K_0$ is not isomorphic to $\varphi^*_{W \{ u \}} K_0 \simeq K_0 [u]/u^p$ as a $W\{ u \}$-module. 
\end{remark}

Combining all the previous results established in this subsection, we obtain the following equivalent descriptions of $F$-isocrystals over the log-crystalline site, extending \Cref{equiv def isoc} to the log setting.
\begin{corollary}
        \label{equiv def log isoc}
	Let $Z$ be a smooth scheme over $\mathcal{O}_K/p$, let $M_Z$ be the log structure $(\pi^\IN)^a$, and let $(Z_s,M_s)$ be the reduced special fiber.
	The following categories are equivalent:
	\begin{enumerate}[label=\upshape{(\alph*)},series=innerlist]
		\item The category $\Isoc^\varphi ((Z,M_Z)_\lcrys)$ of $F$-isocrystals over the log-crystalline site $(Z,M_Z)_\lcrys$. 
		\item The category $\Isoc^\varphi ((Z_s,M_{Z_s})_\lcrys)$ of $F$-isocrystals over the log-crystalline site $(Z_s,M_{Z_s})_\lcrys$. 
		\item\label{equiv def log isoc, with monodromy} The category $\Isoc^{\varphi,N} (Z_{s,\crys})$ of $F$-isocrystals with nilpotent endomorphisms over the crystalline site $Z_{s,\crys}$.
	\end{enumerate}
	In addition, if there is a smooth $p$-adic formal scheme $Z_0$ over $W$ that admits a Frobenius lift $\varphi_{Z_0}$ such that $Z\simeq Z_0 \tensor_W (\sO_K/p)$, then the above categories are also equivalent to the following
	\begin{enumerate}[label=\upshape{(\alph*)},resume=innerlist]
		\item\label{equiv def log isoc, conn} The category $\mathrm{Vect}^{\varphi,\nabla_{Z_0},N}(Z_{0,\eta})$ of the tuples $(\mathcal{F},\varphi_\mathcal{F},\nabla_{\mathcal{F},Z_0}, N_\mathcal{F})$, where $(\mathcal{F},\varphi_\mathcal{F},\nabla_{\mathcal{F},Z_0})$ is an object of $\Vect^{ \varphi, \nabla_{Z_0}}(Z_{0, \eta})$ and $N_\sF$ is a nilpotent horizontal endomorphism on $\sF$ such that $N_{\varphi_{Z_0}^*\mathcal{F}} (\colonequals \varphi_\sF^{-1} \circ N_\sF \circ \varphi_\sF $) is equal to $p\varphi_{Z_0}^*N_\mathcal{F}$.
  \end{enumerate}
  In addition, the canonical functor $$\Isoc^\varphi(Z_\crys) \to \Isoc^\varphi((Z, (\pi^\IN)^a)_\lcrys)$$is fully faithful with essential image consisting of objects such that $N = 0$ in \ref{equiv def log isoc, with monodromy} and \ref{equiv def log isoc, conn}. 
\end{corollary}

Finally, we record a tautological lemma about the compatibility of the pullback functor with respect to the evaluation.
\begin{lemma}
\label{prop:equiv def of crys vs pullback}
	Let $f:Z_s\to Z_s'$ be a map of $k$-varieties, and let $f_0:Z_0\to Z_0'$ be a map of $p$-adic formal schemes over $W$ with compatible Frobenius lifts.
	The pullback of $F$-isocrystals along $f$ is compatible with the coherent pullback along $f_0$ under the equivalences in \Cref{equiv def log isoc}
\end{lemma}

\subsection{Relative log crystalline cohomology}
\label{sec:relative_log_crys_coh}

In this subsection, we establish a crystal property of the log-crystalline cohomology in a relative setting, and in particular prove a relative version of the Hyodo--Kato comparison theorem.
We start with a remark commenting on the relationship between various constructions over the big and the small log-crystalline sites.
\begin{remark}
\label{rmk: comparison between sites}
In many standard references, e.g., \cite{Shi07, Kat89, HK94}, one considers the small log crytalline site, whereas in the present paper we use the big (absolute) log crystalline site throughout. 
It is well-known that the notion of crystals and their cohomology are independent of the choice of the big or the small crystalline site (see for example the arguments in \cite[Cor.\ 2.2.8, Prop.\ 3.1.7]{Guo21}), for the reader's convenience let us briefly explain it as below.

Let $X$ be a finite type $k$-scheme with a fine log structure $M$. Let $(X, M)_\lcris$ be small log-crystalline site, consists of those objects $(U, T, M_T, \iota, \delta)$ in the big site $(X, M)_\lcrys$ such that $U \to X$ is \'etale and $M_T$ is fine (i.e., integral and coherent).
Zariski locally, $(X,M)$ admits a closed immersion into a log-smooth scheme over $W(k)$.
We let $(\mathcal{D},M_\mathcal{D})$ be the log pro-PD-envelope of the closed immersion, which is an object in both the small and the big sites.
So by the log-smoothness, the object $(\mathcal{D},M_\mathcal{D})$ is a covering object of both $(X, M)_\lcrys$ and $(X, M)_\lcris$, up to shrinking $X$ by an open subscheme.
In addition, the \v{C}ech nerves of the covering object $(\mathcal{D},M_\mathcal{D})$ within the two sites are by construction identical to each other, as the embedding $(X, M)_\lcrys \into (X, M)_\lcris$ preserves finite limits.
As a consequence, as in the classical translation between the (iso)crystal and the (iso)-stratifications (cf. \cite[(4.3.9)]{Shi07}), we see the category of (iso)crystals over $(X, M)_\lcrys$ is equivalent to the category of those over $(X, M)_\lcris$, identifying the cohomology of (iso)crystals computed in the big and the small sites.
\end{remark}

In the following, we let $f_{\crys}$ denote the morphism of the small log-crystalline topoi $(Y, M_Y)_\lcris \to (X, M_X)_\lcris$, for a map of fine log $k$-varieties $f : (Y, M_Y) \to (X, M_X)$.
\begin{proposition}
\label{prop: relative log crys}
    Let $X$ be a smooth variety over $k$ equipped with the log structure $M_{X} \colonequals (0^\IN)^a$. 
    Let $(Y, M_Y)$ be a fine log $k$-variety and $f : (Y, M_Y) \to (X, M_X)$ be a proper and log-smooth morphism of Cartier type. Let $i\in \mathbb{N}$.
    \begin{enumerate}[label=\upshape{(\roman*)}]
        \item\label{log crys coh BC} Let $g : X' \to X$ be a morphism between smooth $k$-varieties and $f' : (Y', M_{Y'}) \to (X', M_{X'} \colonequals (0^\IN)^a)$ be the pullback of the map $f$ along the map $g$. 
        Then the canonical base change morphism $$g^* R^i f_{\crys *}(\sO_{(Y, M_{Y})/W}[1/p]) \to R^i f'_{\crys *}(\sO_{(Y', M_{Y'})/W}[1/p]) $$
        is an isomorphism. 
        \item\label{is an F-isocrystal} The higher direct image sheaf $R^i f_{\crys *}(\sO_{(Y, M_Y)/W}[1/p])$ is an $F$-isocrystal, i.e., an object of $\Isoc^\varphi((X, M_X)_\lcris)$. 
    \end{enumerate}
\end{proposition}
\begin{proof} 
    We first mention that the evaluation of the coherent crystalline sheaf $\sF \colonequals R^i f_{\crys *}(\sO_{(Y, M_Y)/W})$ on an object $(T, U, M_T) \in (X, M_X)_\lcris$ can be computed by localization. 
    Namely, consider the functor of topoi $$f_{Y_U/T, \crys *} : ((Y|_U, M_Y|_{Y_U})/(T, M_T))^\sim_\lcris \to T^\sim_{\mathrm{zar}}$$
    introduced in \cite[Def.~1.4]{Shi07}. Then we have a commutative diagram 
    \[\begin{tikzcd}
	{(Y, M_Y)^\sim_\lcris} & {((Y|_U, M_Y|_{Y_U})/(T, M_T))^\sim_\lcris} \\
	{(X, M_X)_\lcris^\sim} & {T^\sim_{\mathrm{zar}}}
	\arrow["{j_{T}^*}", from=1-1, to=1-2]
	\arrow["{f_{\crys*}}"', from=1-1, to=2-1]
	\arrow["{f_{Y_U/T, \crys*}}", from=1-2, to=2-2]
	\arrow["{e_T}", from=2-1, to=2-2]
    \end{tikzcd}\]
    where $j_T^*$ (resp. $e_{T*}$) is the obvious restriction (resp. evaluation) functor. Since both $e_{T*}$ and $j_T^*$ are exact, and $j^*_T(\sO_{(Y, M_Y)/W}) = \sO_{(Y_U, M_Y|_{Y_U})/(T, M_T)}$, we obtain a canonical isomorphism 
    \begin{equation}
        \label{eqn: localization}
        \sF(T, U, M_T) = e_{T*}(R^i f_{\crys *}(\sO_{(Y, M_Y)/W})) \simeq R^i f_{Y_U/T, \crys *}(\sO_{(Y_U, M_Y|_{Y_U})/(T, M_T)}).
    \end{equation}
    Here we note that as the map $f$ is assumed to be smooth, the right hand side of (\ref{eqn: localization}), after inverting $p$, can be computed by log-convergent cohomology \cite[Thm~2.36]{Shi07}. 
    We shall repeatedly use this fact below. 

    We then produce an $F$-isocrystal $\sE$ on $(X, M_X)_\lcris$ and will eventually show that it coincides with $\sF[1/p]$ as sheaves over the small log-crystalline site. 
    Without loss of generality, let us assume that $X = \Spec(R_s)$ is affine and admits a lifting $\Spf(R_0)$ over $W$.
    To construct $\sE$, we take $({D_P}, X, M_{D_P})$ to be the covering object $(\Spf(R_0 \{ u \}), X, (u^\IN)^a)$ of $(X, M)_\lcris$ and consider its \v{C}ech nerve $({D_P}^\bullet, X, M_{{D_P}^\bullet})$. 
    Note that for every $n\in \mathbb{N}$, each structural morphism in the \v{C}ech nerve $({D_P}^n, X, M_{{D_P}^n}) \to ({D_P},X,M_{D_P})$ is analytically flat. 
    Therefore, by \cite[Cor.~3.10]{Shi07} (using the conditions (1) and (2) of loc.\ cit.), for each projection morphism $p_i : ({D_P}^1, X, M_{{D_P}^1}) \to ({D_P}, X, M_{D_P})$ there is a canonical isomorphism $p_i^* \sF({D_P}, X, M_{D_P})[1/p] \simeq \sF({D_P}^1, X, M_{{D_P}^1})[1/p]$.
    It is tautological that the resulting transition map $p_0^* \sF({D_P}, X, M_{D_P})[1/p] \simeq p_1^* \sF({D_P}, X, M_{D_P})[1/p]$ satisfies the cocycle condition. Hence $\sF({D_P}, X, M_{D_P})[1/p]$ together with this transition map defines an isocrystal $\sE$ on $(X, M)_\lcris$. 
    The fact that $\sE$ has the structure of an $F$-isocrystal is the consequence of \cite[Prop.~2.24]{HK94}. By \Cref{thm: log crys = crys + N}, the existence of Frobenius structure forces $\sE$, and hence $\sF({D_P}, X, M_{D_P})[1/p]$, to be locally free. 

    Now, let the maps $g$ and $f'$ be as in \ref{log crys coh BC}. 
    Let $(T', X', M_{T'})$ be an object of $(X', (0^\IN)^a)_\lcris$, which is also an object of the big log crystalline site $(X, (0^\IN)^a)_\lcrys$. 
    Since $({D_P}, X, M_{D_P})$ is a covering object of $(X, (0^\IN)^a)_\lcrys$ as well, the map $g$ lifts to a morphism $\wt{g} : (T', M_{T'}) \to ({D_P}, M_{D_P})$ between the log schemes. 
    Moreover, as the ${D_P}[1/p]$-module $\sF({D_P}, X, M_{D_P})[1/p]$ is locally free, by \cite[Cor.~3.10]{Shi07} (using the conditions (1)' and (2) of loc.\ cit.), the canonical base change map below is an isomorphism
    \begin{equation}
        \label{eqn: log crys coh BC 2} 
        \wt{g}^* \sF({D_P}, X, M_{D_P})[\frac{1}{p}] \simeq R^i f'_{Y'/T', \crys *}(\sO_{(Y', M_{Y'})/(T', M_{T'})}[\frac{1}{p}]).
    \end{equation}
    This in particular implies the base change formula of the log-crystalline cohomology in Part \ref{log crys coh BC}.
    
    Suppose now that $X'$ is \'etale over $X$. Then $(T', X', M_{T'})$ is an object of $(X, M_X)_\lcris$.
    In particular, the left (resp. right) hand side of \ref{eqn: log crys coh BC 2} is canonically identified with $\sE(T', X', M_{T'})$ (resp. $\sF(T', X', M_{T'})[1/p]$). This implies that $\sE = \sF[1/p]$ and hence we obtain Part \ref{is an F-isocrystal}.
\end{proof}

We now prove the relative version of the Hyodo--Kato isomorphism over a smooth $p$-adic formal scheme.

\begin{theorem}
\label{thm:relative Hyodo-Kato}
    Let $X$ be a smooth $p$-adic formal scheme over $\sO_K$ equipped with the standard log structure $M_X$.
    Let $(Y, M_Y)$ be a fine log $p$-adic formal scheme, and let $f : (Y, M_Y) \to (X, M_X)$ be a proper and log-smooth morphism with Cartier type reduction.
    Suppose that there is a smooth $p$-adic formal scheme $X_0$ over $W$ such that $X\simeq X_0 \tensor_K \sO_K$.
    \begin{enumerate}[label=\upshape{(\roman*)}]
        \item\label{thm:relative Hyodo-Kato map} For each choice of uniformizer $\pi \in \sO_K$ there is an isomorphism of flat connections
    \begin{equation}
        \label{eqn: HK iso}
        (R^i f_{s, \crys} \sO_{(Y_s, M_{Y_s})/W})(X_{0}, (0^\IN)^a)\otimes_{W} K \simeq R^i f_{\eta *} \Omega^\bullet_{(Y_\eta, M_{Y_\eta})/X_\eta}.
    \end{equation}
    \item\label{thm:relative Hyodo-Kato point} For every point $x_0 \in X_0(W)$ with the special fiber $x_s \in X_s(k)$ and the base extension $x \colonequals x_0 \tensor_W \mathcal{O}_K \in X(\sO_K)$, the base change of the isomorphism (\ref{eqn: HK iso}) along the closed immersion $x_0\to X_0$, namely the isomorphism
    \begin{equation}
        \label{eqn: HK iso 2}
        \mathrm{H}^i_\lcris((Y_{x_s}, M_{Y_{x_s}})/(W, (0^\IN)^a)) \tensor_W K \simeq \mathrm{H}^i_\dR((Y_{x_\eta}, M_{Y_{x_\eta}})/ K),
    \end{equation}
    is identical to the Hyodo--Kato isomorphism map $\rho_\pi$ in \cite[Thm.~5.1]{HK94}. 
    \end{enumerate} 
\end{theorem}
\begin{proof}
    Set $\sE \colonequals R^i f_{s, \crys} \sO_{(Y_s, M_{Y_s})/W}[1/p]$. Then Part \ref{is an F-isocrystal} of \cref{prop: relative log crys} says that $\sE$ is an $F$-isocrystal on $(X_s, (0^\IN)^a)_\lcris$, and Part \ref{log crys coh BC} says that the restriction $x_s^* \sE$ which is an $F$-isocrystal on $(\Spec(k), (0^\IN)^a)_\lcris$ (or equivalently a $(\varphi, N)$-module over $K_0$) is indeed canonically isomorphic to $\mathrm{H}^i_\lcris((Y_{x_s}, M_{Y_{x_s}})/(W, (0^\IN)^a))[1/p]$. 

    To prove the proposition we closely follow the argument of \cite[Thm~5.1]{HK94}. 
    Without loss of generality, let us assume that $X = \Spf(R)$, $X_s = \Spec(R_s)$ and $X_0 = \Spf(R_0)$ are all affine. 
    Let $\varphi_{R_0}$ be the lift of the absolute Frobenius $F_{R_s}$ of $R_s$ and $\varphi_{R_0 \{ u \}}$ be the endomorphism of $R_0\{ u \}$ which acts as $\varphi_{R_0}$ on $R_0$ and sends $u$ to $u^p$. 
    Similarly, let $\varphi_\fS$ be the Frobenius endomorpshim on $\fS = W[\![u]\!]$ which sends $u$ to $u^p$ and acts as $\sigma$ on $W$. 
    Let $D$ be the log-PD envelope of the exact surjection $(\fS, (u^\IN)^a) \to (\sO_K, (\pi^\IN)^a)$ given by sending $u$ to $\pi$ (cf. \Cref{conv of log of base}) and set $D_{R_0} \colonequals (D \tensor_W R_0)^\wedge_p$. 
    Choose $r \in \IN$ sufficiently large such that $p \mid \pi^r$ in $\sO_K$. Then we obtain a commutative diagram 
    \[\begin{tikzcd}
	\fS & \fS \\
	k & {\sO_K/p}
	\arrow["{\varphi_\fS^r}", from=1-1, to=1-2]
	\arrow["{\mod (p, u)}"', from=1-1, to=2-1]
	\arrow["{\mod (p, E(u))}", from=1-2, to=2-2]
	\arrow[hook, from=2-1, to=2-2]
    \end{tikzcd}.\]
    By taking the log PD-envelops of the vertical surjections, we obtain a morphism $W\{ u \} \to D$, which we denote by $g$, following the notation of \textit{loc. cit.} Note that $g$ induces a morphism $R_0\{ u \} \to D_{R_0}$ by base change, which we denote by the same letter. 
    
    By \Cref{thm:Breuil}, there exists a unique Frobenius equivariant isomorphism  
    $$ s : \sE(R_0\{ u \}, R_s, (u^\IN)^a) \simeq \sE(R_0, R_s, (0^\IN)^a) \tensor_{R_0} R_0\{ u \}. $$
    When $Y = \Spf(\sO_K)$ is just a point, the above becomes an isomorphism 
    $$s : \mathrm{H}^i_\lcris((Y_s, M_{Y_s})/(W, (0^\IN)^a)) \tensor_W W\{ u \} \stackrel{\sim}{\to} \mathrm{H}^i((Y_s, M_{Y_s})/(W\{u\}, (u^\IN)^a)).$$
    Alternatively, one obtains such an isomorphism by applying \cite[Prop.~4.13]{HK94}. More precisely, the result in \textit{loc. cit.} produces a Frobenius-equivariant isomorphism between two complexes of sheaves on $Y_{s, \et}$, which gives an isomorphism as above after we take $\mathrm{H}^i(Y_{s, \et}, -)$. 

    We define an isomorphism $h_\pi$ using the following commutative diagram of the isomorphisms. 
    \begin{equation}
    \label{diag: construct HK}
        \begin{tikzcd}
	{\bigr(\sE(R_0, R_s, (0^\IN)^a) \tensor_{R_0} R_0\{u\} \bigr) \tensor_g D_{R_0}} && {\sE(R_0 \{ u \}, R_s, (u^\IN)^a) \tensor_g D_{R_0}} \\
	{\sE(R_0, R_s, (0^\IN)^a) \tensor_{\varphi^r_{R_0}} D_{R_0}} && {\sE(D_{R_0}, R_s, (u^\IN)^a)} \\
	{(\varphi^r_{R_0})^*\sE(R_0, R_s, (0^\IN)^a) \tensor_{R_0} D_{R_0}} && {\sE(R_0, R_s, (0^\IN)^a) \tensor_{R_0} D_{R_0}}
	\arrow["{s \tensor_g D_{R_0}}", from=1-1, to=1-3]
	\arrow[equal, from=1-1, to=2-1]
	\arrow["{(\ast)}"', from=1-3, to=2-3]
	\arrow[equal, from=2-1, to=3-1]
	\arrow["{{\varphi^r \tensor 1}}", from=3-1, to=3-3]
	\arrow["{h_\pi}", from=3-3, to=2-3, dashed]
    \end{tikzcd}
    \end{equation}
    Here in the diagram $\varphi_\sE^r : (\varphi^r_{R_0})^* \sE(R_0, R_s, (0^\IN)^a) \stackrel{\sim}{\to} \sE(R_0, R_s, (0^\IN)^a)$ is the Frobenius isomorphism given by the $F$-isocrystal structure of $\sE$. The isomorphism $(\ast)$ is induced by $g$ and the crystalline nature of $\sE$, viewed as a morphism $(R_0\{ u \}, R_s, (u^\IN)^a) \to (D_{R_0}, R_s, (u^\IN)^a)$ between objects in $(X_s, (0^\IN)^a)_\lcris$. 
    One checks that the definition of $h_\pi$ is independent of the choice of $r$. We then define an isomorphism $\sE(X_0, X_s, (0^\IN)^a) \tensor_{K_0} K \simeq \sE(X, X_{p = 0}, M_X = (\pi^\IN)^a)$ by the $K$-linearization of the following maps
    \begin{equation}
    \label{eqn: Breuil to evaluate E}
        \sE(X_0, X_s, (0^\IN)^a) \into \sE(X_0, X_s, (0^\IN)^a) \tensor_{R_0} D_{R_0} \stackrel{h_\pi}{\to} \sE(D_{R_0}, R_s, (u^\IN)^a) \stackrel{u \mapsto \pi}{\to} \sE(X, X_s, (\pi^\IN)^a).
    \end{equation}   
    
    Now, using \cite[Thm~6.4]{Kat89} (see also the paragraph above \cite[Thm~2.36]{Shi07}), one obtains a canonical isomorphism 
    $$ R^i f_{Y_{p = 0}/ X, \crys *}(\sO_{(Y_{p = 0}, M_{Y_{p = 0}})/(X, M_X)}) \simeq R^i f_{*} \Omega^\bullet_{(Y,M_Y)/(X, M_X)} $$
    where $\Omega^\bullet_{(Y,M_Y)/(X, M_X)}$ is the (relative) log de Rham complex of the morphism $f$. After inverting $p$, one obtains an isomorphism 
    \begin{equation}
        \label{eqn: HK iso 3}
        \sE(X, X_s, M_X) \simeq R^i f_{\eta *} \Omega^\bullet_{(Y_\eta, M_{Y_\eta})/X_\eta},
    \end{equation}
    which finishes the proof of Part \ref{thm:relative Hyodo-Kato map}. 
    
    For Part \ref{thm:relative Hyodo-Kato point}, we note that for the given point $x_0\in X_0(W)$, there is a commutative diagram of log schemes 
    \begin{equation}
        \label{diag: BC to point for HK iso}
        \begin{tikzcd}
            (\Spf(W), (0^\IN)^a) \arrow[d, "x_0"] & (\Spf(\sO_K), (\pi^\IN)^a) \arrow[d, "x"] \arrow[l] & (X_s, (0^\IN)^a) \arrow[d, "x_s"] \arrow[l, hook] \\
            (X_0,(0^\IN)^a ) & (X,M_X)  \arrow[l] & (X_s, (0^\IN)^a). \arrow[l, hook]
        \end{tikzcd}
    \end{equation}
    Therefore, one indeed obtains an isomorphism as in (\ref{eqn: HK iso 2}) from (\ref{eqn: HK iso}) via a base change.
    Then one can check readily that the construction of (\ref{eqn: HK iso 3}) agrees with the construction of $\rho_\pi$ in \cite[Thm.~5.1]{HK94}. 
    Here the essential point is that the diagram \cite[(5.2.1)]{HK94} about certain complexes of sheaves on $Y_{s, \et}$ gives rise to diagram (\ref{diag: construct HK}) after we take $\mathrm{H}^i(Y_{s, \et}, -)$, and we have chosen notations to facilitate the comparison. 
\end{proof}

\section{Period sheaves}
\label{sec period sheaves}
We recall the construction and the properties for the crystalline horizontal period sheaf and the crystalline period sheaf with connection in this section, and introduce their semi-stable analogues, for a rigid space with good reduction.

\subsection{Horizontal period sheaves}
\label{sub horizontal period sheaves}

We first recall the construction of horizontal period sheaves and study their basic properties. Before doing so, we first discuss horizontal period rings. We assume the basics of perfectoid algebras as in \cite{Sch13}. 
As before, we let $\sO_K$ be the ring of integers of some fixed $p$-adic field with a perfect residue field $k$, and recall our conventions in \cref{subsec conv}. 

\begin{definition}
\label{def roman horizontal period rings}
Let $S$ be a $p$-torsionfree perfectoid  $\sO_C$-algebra. Let $S^\flat$ be the tilt $\varprojlim_{x \mapsto x^p} S$. 
Equip $S$ with the standard log structure $M_S \colonequals S \cap S[1/p]^\times$.
\begin{enumerate}[label=\upshape{(\roman*)}]
    \item Set $\rAinf(S) \colonequals W(S^\flat)$, which is equipped with a canonical surjection $\theta : \rAinf(S) \to S$, sending the Teichm\"uller lift $[(x_0,x_1,\ldots)]\in W(S^\flat)$ onto the element $x_0$.
    Let $\mathrm{B}_\dR^+(S)$ be the completion of $\rAinf(S)[1/p]$ at $\ker(\theta)$, and $\mathrm{B}_\dR(S)$ be the localization $\mathrm{B}_\dR^+(S)[1/\mu]$. 
    Recall that the constant $\mu$ was introduced in \cref{subsec conv}. 
    \item Let $\rAcrys(S)$ be the $p$-completion of the PD-envelope of the surjection $\theta$. Set $\mathrm{B}_\crys^+(S) \colonequals \rAcrys(S)[1/p]$ and $\mathrm{B}_\crys(S) \colonequals \mathrm{B}_\crys^+(S)[1/\mu] = \rAcrys(S)[1/\mu]$. 
    \item Let $M_{\rAinf(S)}$ be the log structure on $\rAinf(S)$ associated to the monoid $\theta^{-1}(M_S)$, and let $M_{\rAcrys(S)}$ be the log structure induced by $M_{\rAinf(S)}$.
    \item The Frobenius operator $\varphi_{\rAinf(S)}$ induces Frobenius operators $\varphi_{\rAcrys(S)}, \varphi_{\mathrm{B}_\crys^+(S)}$ and $\varphi_{\mathrm{B}_\crys(S)}$ on the respective rings. 
    \end{enumerate}
\end{definition}

\begin{remark}
    	\label{rmk: log envelope is envelope}
    	The pro log PD envelope of the map $\theta$, when viewed as a surjection of log pairs $(\rAinf(S),M_{\rAinf(S)})\to (S,M_S)$, is precisely $(\rAcrys(S),M_{\rAcrys})$, because $\theta$ is by construction exact. 
        As a consequence, the triple $(\rAcrys(S),S/p,M_{\rAcrys})$ forms an object in $(\mathcal{O}_K/p, M_{\mathcal{O}_K/p})_\lcrys$.
\end{remark}

Recall from \Cref{conv of log of base} that we introduced the ring $D \colonequals D_\mathfrak{S}(E(u))^\wedge_p$, the $p$-complete PD-envelope for the surjection $\mathfrak{S}=W\llbracket u\rrbracket \to \mathcal{O}_K$. We use $M_D$ to denote the induced log structure of $M_\mathfrak{S}=(u^\IN)^a$ on $D$.
Then the triple $(D_\mathfrak{S}(E(u))^\wedge_p, \mathcal{O}_K/p, M_\mathfrak{S})$ forms a pro-object in the log crystalline site $(\mathcal{O}_K/p,M_{\mathcal{O}_K/p})_{\lcrys}$ that is weakly initial. 
Let $\varphi_\fS$ be the Frobenius operator on $\fS$ which acts as $\sigma$ on $W$ and sends $u$ to $u^p$.

Now we construct the semi-stable period rings.

\begin{construction}
    	\label{Explicit formula of Ast}
    	Let $S$ be a $p$-torsionfree perfectoid $\mathcal{O}_C$-algebra.
    	Define $\rhAst(S)$ to be the $p$-complete log PD envelope of the surjection 
    	\[
    	\bigl((\mathrm{A}_{\inf}(S), M_{\mathrm{A}_{\inf}(S)}) \otimes_W (\mathfrak{S}, M_\mathfrak{S}) \bigr)^\wedge_p \xrightarrow{\theta\otimes (\text{mod}~E(u))} (S,M_S).
    	\]
        Let $\varphi_{\rhAst(S)}$ be the Frobenius operator on $\rhAst(S)$ induced by $\varphi_{\Ainf(S)} \tensor \varphi_{\fS}$ on the left hand side above. 
        To compute $\rhAst(S)$ explicitly, we first form the exact localization of the the source by slightly adapting \Cref{lem: exact localization}.
    	By \Cref{def roman horizontal period rings}, the map of monoids $M_{\Ainf(S)}\to M_S$ induces an equivalence of associated log structures on the ring $S$.
    	Moreover, since the image of $u\in M_{\mathfrak{S}}$ in $S$ coincides with that of $[\pi^\flat]\in \mathrm{A}_{\inf}(S)$, the $p$-complete exact localization for the map of log pairs above is given by 
    	\[
    	\bigl(\rAinf(S)\otimes_W \mathfrak{S}[\frac{u}{[\pi^\flat]}, \frac{[\pi^\flat]}{u}]\bigr)^\wedge_p.
    	\]
    	To get the ring $\rhAst(S)$, we then take the $p$-complete divided power envelope of the surjection
    	\[
    	\bigl( \rAinf(S)\otimes_W \mathfrak{S}[\frac{u}{[\pi^\flat]}, \frac{[\pi^\flat]}{u}] \bigr)^\wedge_p \longrightarrow S.
    	\]
    	Notice that since the kernel of the above map is generated by $\ker(\theta)$ and $u/[\pi^\flat]-1$, we can rewrite the ring $\rhAst(S)$ as the $p$-complete divided power polynomial for the variable $u/[\pi^\flat]-1$, namely
    	\[
    	\rhAst(S)\simeq \rAcrys(S)\langle \frac{(u/[\pi^\flat]-1)^n}{n!} , n\in \IN\rangle = \rAcrys(S) \{\frac{u}{[\pi^\flat]}-1\}.
    	\]
    	Here the log structure $M_{\rhAst}(S)$ is by construction associated to the map of monoids
    	\[
    	\left( M_{\rAinf(S)} \times u^\mathbb{N} \right) \coprod_{[\pi^\flat]^\mathbb{N} \times u^\mathbb{N}} N^{(1)} \longrightarrow \rhAst(S),
    	\]
    	with $N^{(1)}\colonequals \{u^x\cdot(\pi^\flat)^y~|~x+y\geq 0\}$ being the submonoid of $u^\mathbb{Z} \times(\pi^\flat)^\mathbb{Z}$. 
    	
    	We also note that by construction and \Cref{rmk: log envelope is envelope}, the associated log divided power thickening $(\rhAst(S),S/p, M_{\rhAst(S)})$ is naturally identified with the coproduct of $(D_\mathfrak{S}(E(u)), \mathcal{O}_K/p, M_\mathfrak{S})$ and $(\rAcrys(S), S/p, M_{\rAcrys(S)})$ in the log crystalline site $(\mathcal{O}_K/p,M_{\mathcal{O}_K/p})_\lcrys$, fitting into a diagram of pro log PD thickenings
    	\[
    	\begin{tikzcd}
    		& (D_\mathfrak{S}(E(u)), \mathcal{O}_K/p, M_\mathfrak{S}) \ar[d]\\
    		(\rAcrys(S), S/p, M_{\rAcrys(S)}) \ar[r]& (\rhAst(S),S/p, M_{\rhAst(S)}).
    	\end{tikzcd}
        \]
\end{construction}

Similar to \Cref{const: monodromy on D}, we can define a monodromy operator on $\rhAst(S)$ that is linear with respect to $\rAcrys(S)$. 
\begin{construction}[Monodromy operator on $\rhAst$]
	\label{const: monodromy on hAst}
	Let $S$ be a $p$-torsionfree perfectoid $\mathcal{O}_C$-algebra.
	By \Cref{const: monodromy on D}, the \v{C}ech nerve of the covering object $(D,\mathcal{O}_K/p,M_D)$ in $(\mathcal{O}_K/p, M_{\mathcal{O}_K/p})_\lcrys$ is isomorphic to the cosimplicial complete log divided power polynomials
	\[
	(D^\bullet, \mathcal{O}_K/p, M_{D^\bullet}) = (D\{ t_1,\ldots,t_\bullet \} , \mathcal{O}_K/p, \{u^{a_0}\cdots u_\bullet^{a_\bullet}~|~\sum a_i \geq 0\}^a),
	\]
	where $t_i$ is the formal variable corresponding to the element $u/u_i-1$.
	On the other hand, by \Cref{rmk: log envelope is envelope}, we can associate the perfectoid ring $S$ with the log pair $(\rAcrys(S), M_{\rAcrys(S)})$, such that the triple $(\rAcrys(S), S/p, M_{\rAcrys(S)})$ forms an object in $(\mathcal{O}_K/p, M_{\mathcal{O}_K/p})_\lcrys$.
	In particular, we can form the coproduct of the cosimplicial objects $(D^\bullet, \mathcal{O}_K/p, M_{D^\bullet})$ with $(\rAcrys(S), S/p, M_{\rAcrys(S)})$ in the log crystalline site $(\mathcal{O}_K/p, M_{\mathcal{O}_K/p})_\lcrys$, to get the 
	cosimplicial object $(\rhAst(S)^\bullet, S/p, M_{\rhAst(S)^\bullet})$, which fits into a diagram 
	\[
	\begin{tikzcd}
		& (D^\bullet, \mathcal{O}_K/p, M_{D^\bullet}) \ar[d] \\
		(\rAcrys(S), S/p, M_{\rAcrys(S)}) \ar[r] & (\rhAst(S)^\bullet, S/p, M_{\rhAst(S)^\bullet}).
	\end{tikzcd}
    \]
	Moreover, by \Cref{Explicit formula of Ast}, the zero-th term $(\rhAst(S)^0, S/p, M_{\rhAst(S)^0})$ is naturally isomorphic to \[
	(\rhAst(S),S/p, M_{\rhAst(S)})\simeq (\rAcrys(S) \{\frac{u}{[\pi^\flat]}-1\}, S/p, M_{\rhAst(S)}).\]
	The explicit construction of $(D^\bullet, \mathcal{O}_K/p, M_{D^\bullet})$ in \Cref{Cech for log crys} further provides an identification 
	\[
	(\rhAst(S)^\bullet, S/p, M_{\rhAst(S)^\bullet}) \simeq (\rhAst(S)\{t_1,\ldots, t_\bullet\}, S/p, M_{\rhAst(S)^\bullet}),
	\]
	such that the two coface maps $p_0, p_1:\rhAst(S) \to \rhAst(S)^1\simeq \rhAst(S)\{t_1\}$ send a power series $f(u) \in \rhAst(S)= \rAcrys\{u/[\pi^\flat]-1\}$ onto the elements $f(u)$ and $f(\frac{u}{t_1+1})$ respectively in $\rhAst(S)\{t_1\}=\rAcrys(S)\{u/[\pi^\flat]-1,t_1\}$.
	
	Now similar to \Cref{const: monodromy on D}, since the difference map $p_0-p_1$ has image in $t_1\cdot \rhAst(S)\{t_1\}$, we define the $\rAcrys(S)$-linear endomorphism $N_{\rhAst(S)}$ as the composite of the following maps: 
    \[
    \rhAst(S) \stackrel{p_0 - p_1}{\longrightarrow} t_1\cdot \rhAst(S)\{t_1\} \stackrel{\mod t_1^{[\geq 2]}}{\longrightarrow} t_1 \cdot \rhAst(S) \simeq \rhAst(S)
    \]

    When $\rhAst(S)$ is viewed as a subring of the power series ring $\mathrm{B}^+_\crys(S)[\![u]\!]$, the endomorphism $N_{\rhAst(S)}$ is given by the formula $f(u) \mapsto u f'(u)$, similar to the endomorphism $N_D$ in \cref{const: monodromy on D}.
	In particular, the kernel of $N_{\rhAst(S)}$ is $\rAcrys(S)$, and by induction, the submodule of nilpotent elements is given by $\rAcrys(S)[\log{\frac{u}{[\pi^\flat]}}]$, and is isomorphic to the polynomial of one variable over $\rAcrys(S)$.    
	Moreover, by explicit calculations, one check immediately that Frobenius endomorphism $\varphi_{\rhAst}$ and monodromy operator $N_{\rhAst(S)}$ satisfy the formula that $N_{\rhAst}\varphi_{\rhAst}=p\varphi_{\rhAst}N_{\rhAst}$.
\end{construction}

\begin{definition}
    \label{def: Ast}
    Let $S$ be a $p$-torsionfree perfectoid $\mathcal{O}_C$-algebra.
    We define $(\rAst(S),M_{\rAst(S)})$ to be the log pair given by the ring $\rAcrys(S)[\log{\frac{u}{[\pi^\flat]}}]$ together with the log structure induced by $M_{\rAinf(S)}$.
    It is equipped with an $\rAcrys(S)$-linear nil endomorphism $N_{\rAst(S)}$ that vanishes on $\rAcrys(S)$, together with a natural Frobenius endomorphism $\varphi_{\rAst}$ that is compatible with $\varphi_{\rhAst}$, satisfying $N_{\rAst}\varphi_{\rAst}=p\varphi_{\rAst}N_{\rAst}$.
\end{definition}

\begin{remark}
	\label{rmk: descending a sheaf from C to K}
    The subring $\rAst(S) \subseteq \rhAst(S)$ is $\Gal_K$-invariant because $\sigma \in \Gal_K$ acts on $\rhAst$ by  
    $$ \log{\frac{u}{[\pi^\flat]}} \mapsto \log{\frac{u}{[\pi^\flat]}} - \chi(g)\cdot \log{\epsilon}, $$
    where $\chi$ denotes the cyclotomic character. 
    Note that $\chi(g)\cdot \log{\epsilon} \in \rAcrys(\mathcal{O}_C)$. 
\end{remark}

Finally, we summarize the relation of various horizontal period rings in the following.
\begin{proposition}
	\label{prop: relation of horizontal period rings}
	Let $S$ be a $p$-torsionfree perfectoid $\mathcal{O}_C$-algebra.
	Then there are natural maps of horizontal period rings 
	\[
	\rAinf(S) \longrightarrow \rAcrys(S) \longrightarrow \rhAst(S) \longrightarrow \mathrm{B}^+_\dR(S),
	\]
	that are functorial with respect to $S$, and are compatible with Frobenius structures and log structures whenever defined.
	Moreover, the first two maps are injective, and the last map induces an injection from $\rAst(S)$ into $\mathrm{B}^+_\dR(S)$.	
\end{proposition}
\begin{proof}
The first map exists as $\rAcrys(S)$ is the $p$-complete divided power envelope of $\rAinf(S)$ with respect to $\theta$ (cf. \Cref{def roman horizontal period rings}).
The injectiveness of the first map follows from the fact that the ideal $\ker(\theta)$ is generated by a non-zero-divisor \cite[Lem.~6.3]{Sch13}.
The second map together with its injectiveness follows from the explicit formula \Cref{Explicit formula of Ast}.
To construct the third map, we first notice that by construction, the map $\rAinf(S)\to \mathrm{B}^+_\dR(S)$ naturally extends to an injection $\rAcrys(S)\to \mathrm{B}^+_\dR(S)$ (\cite[Prop.~2.23]{TT19}), which further extends to a map from $\rAcrys(S)$-polynomial ring
\[
f:\rAcrys(S)[u]\longrightarrow \mathrm{B}^+_\dR(S),\quad u\longmapsto [\pi^\flat],
\]
where $\ker(f)=(u-[\pi^\flat])$.
Notice that since the canonical surjection $\theta:\mathrm{B}^+_\dR(S)\to S[1/p]$ sends $[\pi^\flat]$ onto the invertible element $\pi$, the map $f$ naturally extends to a map
\[
g:\rAcrys(S)[\frac{u}{[\pi^\flat]}-1] \longrightarrow \mathrm{B}^+_\dR(S).
\]
Moreover, the image of $u/[\pi^\flat]-1$ under the map $g$ admits all $p$-adic divided power power series. 
Hence the map $g$ naturally extends to a map 
\[
\rhAst(S) \simeq \rAcrys(S)\{\frac{u}{[\pi^\flat]}-1\} \longrightarrow \mathrm{B}^+_\dR(S).
\]
The latter in particular sends the element $\log{\frac{u}{[\pi^\flat]}}$ onto $\log{\frac{\pi}{[\pi^\flat]}}\in \mathrm{B}^+_\dR(S)$.
Finally, to check the induced map $\rAst(S)\to \mathrm{B}^+_\dR(S)$ is injective, by the explicit formula in \Cref{const: monodromy on hAst}, it suffices to check that the element $\log{\frac{\pi}{[\pi^\flat]}}\in \mathrm{B}^+_\dR(S)$ is transcendental over the subring $\rAcrys(S)$.
This then follows from \cite[Prop.~2.27]{Shi22}, where the ring $\mathrm{B}^+_\mathrm{max}(S)$ (which in loc.~cit. is denoted as $\mathbb{B}^+_\mathrm{max}(\Lambda)$ for $\Lambda=S$) contains our $\rAcrys(S)$ as a subring (\cite[Lem.~2.34]{Shi22}). 
\end{proof}

Now we are ready to sheafify the above constructions. 
We fix a $p$-torsionfree $p$-adic formal scheme $X$ that is topologically of finite type over $\mathcal{O}_K$, together with the standard log structure $M_X$ as in \Cref{log of framed semi-stable reduction}.
In particular, the log structure $M_{X_\eta}$ on the generic fiber $X_\eta$ is trivial.
Here we emphasize that $X$ is not assumed to be regular.
	
We implicitly assume the equivalence for the category of topoi $\Shv(X_{\eta, \pe}) \simeq \Shv(\Perfd/X_{\eta,\pe})$, where $\Perfd/X_{\eta,\pe}$ is the site defined over the category of affinoid perfectoid objects over $X_\eta$ and is equipped with the pro-\'etale topology (cf.~\cite[Lem.~4.6, Prop.~4.8]{Sch13}).
We first recall various structure sheaves on $X_{\eta, \pe}$. Let $\nu : X_{\eta, \pe} \to X_{\eta, \et}$ be the natural projection. Let $\sO_{X_\eta} \colonequals \nu^* \sO_{X_{\eta, \et}}$ be the uncompleted structure sheaf on $X_{\eta, \pe}$, with subring of integral elements $\sO^+_{X_\eta}$. Then we define 
    $$\what{\sO}^+_{X_{\eta}} \colonequals \varprojlim_{n \in \IN} \sO^+_{X_\eta} / p^n, ~\what{\sO}_{X_{\eta}} = \what{\sO}^+_{X_{\eta}}[1/p], \text{ and } \sO^{\flat+}_{X_\eta} = \varprojlim_{x \mapsto x^p} \sO^+_{X_\eta} / p.$$ For an object $V \in \Perfd/X_{\eta, \pe}$ with $\widehat{V} = \Spa(S[1/p],S)$, we have (\cite[Lem.~4.10, 5.10]{Sch13})
    $$ \what{\sO}^+_{X_\eta}(\what{V}) = S, ~\what{\sO}_{X_\eta}(\what{V}) = S[1/p], \text{ and }\sO^{\flat+}_{X_\eta}(\what{V}) = S^\flat.  $$
We also define a natural log structure on  $\what{\sO}^+_{X_\eta}$ by $M_{\what{\sO}^+_{X_\eta}} \colonequals \what{\sO}^+_{X_\eta} \cap (\what{\sO}_{X_\eta})^\times$, which by construction is compatible with $M_X$.

\begin{definition}
\label{def: horizontal period sheaves}
    The sheaf of rings $\IA_{\inf}$ together with its log structure $M_{\IA_{\inf}}$ is defined as follows: Given any object $V \in \Perfd/X_{\eta,\pe}$ with $\widehat{V} = \Spa(S[1/p],S)$, we set (cf. \cite[Def.~6.1,~Thm.~6.5]{Sch13})
$$ (\IA_{\inf}, M_{\IA_{\inf}})(V) = (\rAinf(S), M_{\rAinf(S)}). $$
    In a completely analogous way, we define sheaves $$? = \IB^+_\dR, \IB_\dR, \IA_\crys, \IB^+_\crys, \IB_\crys, \IA_\st, \what{\IB}^+_\st, \what{\IB}_\st,  \IB^+_\st, \IB_\st$$ on $X_{\eta, \pe}$ from corresponding constructions on $p$-torsionfree perfectoid $\sO_C$-algebras.\footnote{These constructions a priori only define pro-\'etale sheaves over $X_{\eta,C}$. But the resulting sheaves carry natural $\Gal_K$-actions, with which they are naturally viewed as pro-\'etale sheaves on $X_\eta$.}
    We denote the log structures on $?$ by $M_?$, the Frobenius operators on $? = \IA_\crys, \IB^+_\crys, \IB_\crys, \IA_\st,  \what{\IB}^+_\st, \what{\IB}_\st, \IB^+_\st, \IB_\st$ by $\varphi_?$, and monodromy operator on $? = \IA_\st,  \what{\IB}^+_\st, \what{\IB}_\st, \IB^+_\st, \IB_\st$ by $N_?$. 
\end{definition}
\begin{remark}
	\label{rmk: log of X and perfectoid}
	Recall that the log structure $M_{X_\eta}$ over the rigid space $X_\eta$ is trivial by assumption.
	So for any affine open subscheme $\Spf(R)\subseteq X$ and any affinoid perfectoid space $\Spa(S[1/p],S)$ over $\Spf(R)_\eta$, the structure map $R\to S$ canonically lifts to a map of log pairs $(R,M_R)\to (S,M_S)$.
	In particular, we get the following natural diagram of ringed sites with log structures
	\[
	\begin{tikzcd}
		\bigl(X_\mathrm{Zar},(\mathcal{O}_X, M_X) \bigr) \longleftarrow \bigl(X_{\eta,\pe}, (\widehat{\mathcal{O}}^+_{X_\eta}, M_{\widehat{\mathcal{O}}^+_{X_\eta}}) \bigr) \longrightarrow \bigl( (X_{\eta, \pe}, (\mathbb{A}_{\inf}, M_{\mathbb{A}_{\inf}}) \bigr).
	\end{tikzcd}
	\]
\end{remark}

    \cref{prop: relation of all period sheaves} sheafifies to the following statement. 

\begin{corollary}
\label{cor: relation of horizontal period sheaves}
Let $(X,M_X)$ be a topological finite type $p$-adic formal scheme over $(\mathcal{O}_K,M_{\mathcal{O}_K})$, such that $M_{X_\eta}$ is trivial.
There are natural injections of period sheaves over $X_{\eta,\pe}$ that are compatible with Frobenius structures and log structures whenever defined:
\[
\Ainf \longrightarrow \Acrys \longrightarrow \Ast\simeq \Acrys[\log{\frac{u}{[\pi^\flat]}}] \longrightarrow \mathbb{B}^+_\dR.
\]
Similar statements hold for their localizations after we inverting $p$ or $\mu$.

\end{corollary}

\subsection{Period sheaves with connection}
\label{sub period sheaves with conn}
		In this subsection, we recall the construction of the crystalline and the de Rham period sheaves with connection, and introduce their semi-stable analogue,
		for a smooth and affine $p$-adic formal scheme.
		This will be used to define an explicit functor sending crystalline and semi-stable local systems to their associated $F$-isocrystals.
		
		We start by fixing the setup for the $p$-adic formal scheme in this subsection.
		\begin{convention}
			\label{conv of smooth affine}
			We fix a smooth and affine $p$-adic formal scheme $X=\Spf(R)$ over $\mathcal{O}_K$.
			By deformation theory, there is a $p$-adic formal scheme $X_0=\Spf(R_0)$ that is smooth affine over $W$, together with a base change isomorphism $R_0\otimes_W \mathcal{O}_K\simeq R$.
			Moreover, by the complete smoothness of $R_0$ over $W$, the Frobenius endomorphism of $R_k\colonequals R\otimes_{\mathcal{O}_K} k= R_0\otimes_W k$ lifts to an endomorphism $\varphi_{R_0}:R_0\to R_0$ that is semi-linear with respect to $\varphi_W$.
			As in \Cref{log of framed semi-stable reduction}, we let $M_X$ and $M_{X_0}$ be the standard log structures on $X$ and $X_0$ respectively.
            Here we note that $M_X$ and $M_{X_0}$ restricts to the same log structure $M_{X_s} = (0^\IN)^a$ on the reduced special fiber $X_s$.
		    Then we have a natural isomorphism of log pairs $(X_0,M_{X_0}) \times_{(\Spf(W),M_W)} (\Spf(\mathcal{O}_K), M_{\mathcal{O}_K}) \simeq (X,M_X)$. 
			We also let $X_\eta$ be the generic fiber of $X$, where the induced log structure $M_{X_\eta}$ is trivial; similarly for $X_{0,\eta}$ and $M_{X_{0,\eta}}$.

            In addition, we say $X$ is \emph{small} if there is a $p$-completely \'etale morphism onto a $p$-adic torus $\Spf(\mathcal{O}_K\langle x_1^{\pm1},\ldots,x_r^{\pm1} \rangle)$.
		Here we note that for a general smooth $p$-adic formal scheme, the collection of its small affine open sub formal schemes form a basis for the Zariski topology.
		\end{convention} 
		
		We now recall the construction of crystalline period sheaves with connection, essentially following \cite{TT19} with some adjustments due to the ramification.
		\begin{definition}
			\label{def OAcrys}
			Assume the setup of smooth affine $X$ in \Cref{conv of smooth affine}.
			\begin{enumerate}[label=\upshape{(\roman*)}]
				\item The sheaf $\OA_{\crys,R_0}$ in $\Shv(X_{\eta,\pe})$ with its log structure $M_{\OA_{\crys,R_0}}$, connection $\nabla$, and Frobenius structure $\varphi_{\OA_{\crys,R_0}}$, is given as follows
				\begin{itemize}
					\item The sheaf $\OA_{\crys,R_0}$ is the $p$-completion of the PD-envelope of the surjection
					\[
					\bigr(\mathbb{A}_{\inf} \otimes_W R_0 \bigr)^\wedge_p \xrightarrow{\theta\otimes \iota} \widehat{\mathcal{O}}^+_{X_\eta},
					\]
					where $\iota$ is the composed structure map $R_0\to R\to \widehat{\mathcal{O}}^+_{X_\eta}$.
					\item The log structure $M_{\OA_{\crys,R_0}}$ on $\OA_{\crys,R_0}$ is induced from the tensor product log structure of $\bigl( (\Ainf, M_{\Ainf})\otimes_W (R_0,M_{R_0}) \bigr)^\wedge_p$.
					\item The connection $\nabla:\OA_{\crys,R_0} \to \OA_{\crys,R_0}\otimes_{R_0} \Omega^1_{R_0/W}$ is the one induced from the differential map $R_0\to \Omega^1_{R_0/W}$.
					\item The Frobenius structure $\varphi_{\OA_{\crys,R_0}}$ is induced from the tensor product Frobenius $\varphi_{\Ainf}\otimes_{\varphi_W} \varphi_{R_0}$ on $(\Ainf\otimes_W R_0)^\wedge_p$.
					It is horizontal with respect to $\nabla$.
				\end{itemize}
				\item The sheaf $\OB^+_{\crys,R_0}$ in $\Shv(X_{\eta,\pe})$ is defined by inverting the element $p$ on $\OA_{\crys,R_0}$,  with the induced log structure, the connection, and the Frobenius structure.
				\item The sheaf $\OB_{\crys,R_0}$ in $\Shv(X_{\eta, \pe})$ is given on an object $V \in (\Perfd/X_{\eta,C})_\pe$ with $\widehat{V} = \Spa(S[1/p],S)$, by $\OB_{\crys,R_0}(S[1/p],S)=\OA_{\crys,R_0}(S[1/p],S)[1/\mu]$. 
				It is equipped with the induced log structure, the connection, and the Frobenius structure.
			\end{enumerate}
		\end{definition}
	\begin{remark}
		\label{OAcrys of TT}
		By construction (cf. \cite[Def.~2.9]{TT19}), the sheaf $\OA_{\crys,R_0}\in \Shv(X_{\eta,\pe})$ is nothing but the pullback (or restriction) of the sheaf $\OA_{\crys}\in \Shv((X_{0,\eta})_\pe)$ of loc.\ cit. along the natural \'etale covering map of rigid spaces $X_\eta \to X_{0,\eta}$.
		In particular, various structural results of $\OA_{\crys}$ in loc.~cit. applies naturally to our $\OA_{\crys,R_0}$.
	\end{remark}
	By \Cref{def: horizontal period sheaves} and \Cref{def OAcrys}, there is a natural map of log pairs of sheaves 
	\[
	(\Acrys, M_{\Acrys}) \longrightarrow (\OA_{\crys,R_0}, M_{\OA_{\crys,R_0}}),
	\]
	which is compatible with Frobenius structure and has image in $(\OA_{\crys,R_0})^{\nabla=0}$.
	The following results of Tan--Tong and Guo--Li give a more concrete description of the relation.
	\begin{proposition}
		\label{prop: Acrys and OAcrys}
		Let $X$ be as in \Cref{conv of smooth affine}.
		\begin{enumerate}[label=\upshape{(\roman*)}]
			\item\cite[Cor.~2.17]{TT19} 
			\label{prop: Acrys and OAcrys conn} The connection $\nabla$ of $\OA_{\crys,R_0}$ naturally induces an isomorphism in $\Shv(X_{\eta,\pe})$:
			\[
			\Acrys \simeq (\OA_{\crys,R_0})^{\nabla=0}.
			\]
			\item\cite[Prop.~2.13]{TT19}  
			\label{prop: Acrys and OAcrys explicit} Assume $X$ is $p$-completely \'etale over $\Spf(\mathcal{O}_K\langle x_1^{\pm1},\ldots, x_r^{\pm1} \rangle)$, and let $V$ be the affinoid perfectoid pro-\'etale cover of $X_\eta$ given by $\Spf(\mathcal{O}_C\langle x_i^{\pm1/p^\infty}\rangle)_\eta\times_{\Spf(\mathcal{O}_K\langle x_i^{\pm1}\rangle)_\eta} X_\eta$.
			When restricted onto $X_{\eta,\pe}|_V$, the natural map $\Acrys \to \OA_{\crys,R_0}$ induces an isomorphism of sheaves of rings 
			\[
			 \Acrys|_V \{ v_i\} \xrightarrow{\sim} \OA_{\crys,R_0}|_V, \text{ given by } v_i \mapsto \frac{ 1\otimes x_i}{[x_i^\flat]\otimes 1} -1\in \OA_{\crys,R_0}(V)
			\]
			for a fixed choice of $p$-power roots $x_i^\flat$ of $x_i$.
   		\item\cite[Prop.\ 3.16.(4), Thm.\ 3.21]{GL21}
		\label{prop: Acrys and OAcrys et} Let $R'_0$ be a $p$-completely \'etale algebra over $R_0$, and let $X'$ be the base change of $X$ along $R_0\to R_0'$.
		We let $\OA_{\crys,R'_0}$ be the sheaf in $\Shv(X'_{\eta,\pe})$ defined for the data $(X', R_0')$.
		Then we have a natural isomorphism of sheaves over $X'_{\eta,\pe}$ which respects the Frobenius structures and connections
		\[
		\OA_{\crys,R_0}|_{X'_\eta} \simeq \OA_{\crys,R'_0}.
		\]
		\end{enumerate}
	\end{proposition}
	
	We then recall the construction of $\OB_\dR$ of Scholze \cite[Def.~6.13]{Sch13} (see also \cite{Sch16}) and Brinon \cite{Bri08}. 
 
	\begin{definition}
		\label{def OBdR}
		Let $X$ be a topologically finite type $p$-adic formal scheme over $\mathcal{O}_X$ with smooth generic fiber.
		\begin{enumerate}[label=\upshape{(\roman*)}]
			\item The sheaf $\OB^+_\dR$ in $\Shv(X_{\eta,\pe})$ with its connection is given as follows
			\begin{itemize}
				\item The sheaf $\OB^+_\dR$ is the formal completion of the surjection
				\[
				(\mathbb{A}_{\inf}\otimes_W \mathcal{O}_{X})^\wedge_p[1/p] \xrightarrow{\theta\otimes \iota} \widehat{\mathcal{O}}_{X_\eta},
				\]
				where $\iota$ is the structure map $\mathcal{O}_X \to \mathcal{O}^+_{X_\eta} \to \widehat{\mathcal{O}}^+_{X_\eta} \to \widehat{\mathcal{O}}_{X_\eta}$.
				\item The connection $\nabla:\OB^+_\dR \to \OB^+_\dR\otimes_{\mathcal{O}_{X_\eta}} \Omega^1_{X_\eta/K}$ is the one induced from the continuous differential map $\mathcal{O}_{X_\eta} \to \Omega^1_{X_\eta/K}$.
			\end{itemize}
			\item The sheaf $\OB_\dR$ in $\Shv(X_{\eta, \pe})$ is given on an object $V \in (\Perfd/X_{\eta,C})_\pe$ with $\widehat{V} = \Spa(S[1/p],S)$, by $\OB_\dR(S[1/p],S)=\OA_\dR(S[1/p],S)[1/\mu]$.
			It is equipped with the induced connection.
		\end{enumerate}
	\end{definition}
	We also have the following structural results relating $\mathbb{B}^+_\dR$ and $\OB^+_\dR$ together.
	\begin{proposition}
		\label{prop: BdR and OBdR}
		Let $X$ be a $p$-adic topologically finite type formal scheme over $\mathcal{O}_X$ with smooth generic fiber.
		\begin{enumerate}[label=\upshape{(\roman*)}]
			\item\cite[Cor.~6.13]{Sch13} 
			 \label{prop: BdR and OBdR conn} The connection $\nabla$ of $\OB^+_\dR$ naturally induces an isomorphism in $\Shv(X_{\eta,\pe})$: 
			\[
			\mathbb{B}^+_\dR \simeq (\OB^+_\dR)^{\nabla=0}.
			\]
			\item\cite[Prop.~6.10]{Sch13} 
			\label{prop: BdR and OBdR explicit} Assume $X_\eta$ is \'etale over $\Spa(K\langle x_1^{\pm1},\ldots, x_r^{\pm1} \rangle)$, and let $V$ be the affinoid perfectoid pro-\'etale cover of $X_\eta$ given by the base change $\Spa(C\langle x_i^{\pm1/p^\infty}\rangle)\times_{\Spa(K\langle x_i^{\pm1}\rangle)} X_\eta$.
			When restricted onto $X_{\eta,\pe}|_V$, the natural map  $\IB_\dR^+ \to \sO \IB^+_\dR$ induces an isomorphism of sheaves of rings 
			\[
			\mathbb{B}^+_\dR \llbracket v_i \rrbracket \simeq \OB^+_\dR, \text{ given by } v_i \mapsto \frac{ 1\otimes x_i}{[x_i^\flat]\otimes 1} -1\in \OB^+_\dR(V)
			\]
			for a fixed choice of $p$-power roots $x_i^\flat$ of $x_i$.
			\item\cite[Cor.~2.24, Cor.~2.25]{TT19} 
			\label{prop: BdR and OBdR with crys}Assume $X$ is as in \Cref{conv of smooth affine}.
			There is a natural commutative diagram of injections for sheaves of rings in $\Shv(X_{\eta,\pe})$ that are compatible with the connections:
			\[
			\begin{tikzcd}
				\Acrys \arrow[r, hook] \arrow[d, hook] & \OA_{\crys,R_0} \arrow[d, hook]\\
				\mathbb{B}^+_\dR \arrow[r, hook] & \OB^+_\dR.
			\end{tikzcd}
			\]
			It is also compatible with the explicit identifications in \Cref{prop: Acrys and OAcrys}.\ref{prop: Acrys and OAcrys explicit} and \Cref{prop: BdR and OBdR}.\ref{prop: BdR and OBdR explicit} when restricted onto the affinoid perfectoid object $V$.
		\end{enumerate}
	\end{proposition}
	
	Finally, we introduce the semi-stable period sheaves with connections.
	\begin{definition}
		\label{def OAst}
		Let $X$ be as in \Cref{conv of smooth affine}.
		\begin{enumerate}[label=\upshape{(\roman*)}]
			\item The sheaf $\OA_{\st,R_0}$ in $\Shv(X_{\eta,\pe})$ with its log structure $M_{\OA_{\st,R_0}}$, connection $\nabla$, Frobenius structure $\varphi_{\OA_{\st,R_0}}$, and monodromy operator $N_{\OA_{\st,R_0}}$, is given as follows
			\begin{itemize}
				\item The sheaf $\OA_{\st,R_0}$  is defined as the tensor product 
				\[
				\OA_{\st,R_0} \colonequals \OA_{\crys,R_0} \otimes_{\Acrys} \Ast,
				\]
				\item The log structure $M_{\OA_{\st,R_0}}$ is defined as the induced log structure of $M_{\OA_{\crys,R_0}}$ along the canonical map $\OA_{\crys,R_0}\to \OA_{\st,R_0}$.
			\item The connection $\nabla:\OA_{\st,R_0}\to \OA_{\st,R_0}\otimes_{R_0} \Omega^1_{R_0/W}$ is the base change of $\nabla:\OA_{\crys,R_0}\to \OA_{\crys,R_0}\otimes_{R_0} \Omega^1_{R_0/W}$ along the horizontal map $\Acrys\to \Ast$.
			\item The Frobenius structure $\varphi_{\OA_{\st,R_0}}$ is defined as the tensor product $\varphi_{\OA_{\crys,R_0}}\otimes_{\varphi_{\Acrys}} \varphi_{\Ast}$.
			\item The monodromy operator $N_{\OA_{\st,R_0}}$ is defined as the base change of the $\Acrys$-linear endomorphism $N_{\Ast}$ of $\Ast$ along the map $\Acrys\to \OA_{\crys,R_0}$.
		\end{itemize}
		\item The sheaf $\OB^+_{\st,R_0}$ in $\Shv(X_{\eta,\pe})$ is defined by inverting the element $p$ at $\OA_{\st,R_0}$,  with the induced log structure, the connection, the Frobenius structure, and the monodromy operator.
        \item The sheaf $\OB_{\st,R_0}$ in $\Shv(X_{\eta, \pe})$ is given on an object $V \in (\Perfd/X_{\eta,C})_\pe$ with $\widehat{V} = \Spa(S[1/p],S)$, by $\OB_{\st,R_0}(S[1/p],S)=\OA_{\st,R_0}(S[1/p],S)[1/\mu]$.
It is equipped with the induced log structure, the connection, the Frobenius structure, and the monodromy operator.
	\end{enumerate}
\end{definition}
As quick consequences of the explicit construction in \Cref{const: monodromy on hAst}, we get the following properties of $\OA_{\st,R_0}$. 

\begin{corollary}
	\label{cor: property of OAst}
	Let $X$ be as in \Cref{conv of smooth affine}. 
	The restriction of $\OA_{\st,R_0}$ to $(X_{\eta, C})_\pe$ is $\Gal_K$-equivariantly identified with the polynomial ring of one variable $\OA_{\crys,R_0}[\logu]$ over $\OA_{\crys,R_0}$, such that the following properties hold. 
	\begin{enumerate}[label=\upshape{(\roman*)}]
		\item The Frobenius structure $\varphi_{\OA_{\st,R_0}}$ acts as $\varphi_{\OA_{\crys,R_0}}$ on $\OA_{\crys, R_0}$ and sends $\logu$ to $p \logu$. 
		\item The monodromy operator $N_{\OA_{\st,R_0}}$ is the unique $\OA_{\crys,R_0}$-linear endomorphism which vanishes on the subring $\OA_{\crys,R_0}$, and sends $\logu^i$ to $i \logu^{i - 1}$. 
		In particular $N_{\OA_{\st,R_0}}$ is a nil operator.
		\item Both $\varphi_{\OA_{\st,R_0}}$ and $N_{\OA_{\st,R_0}}$ are both horizontal with respect to the connection $\nabla$, and satisfy the formula that $$N_{\OA_{\st,R_0}} \varphi_{\OA_{\st,R_0}} = p \varphi_{\OA_{\st,R_0}} N_{\OA_{\st,R_0}}. $$
  		\item\label{cor: property of OAst et} Let $R'_0$ be a $p$-completely \'etale algebra over $R_0$, and let $X'$ be the base change of $X$ along $R_0\to R_0'$.
		We let $\OA_{\st,R'_0}$ be the sheaf in $\Shv(X'_{\eta,\pe})$ defined for the data $(X', R_0')$.
		Then we have a natural isomorphism of sheaves over $X'_{\st,\pe}$ identifying all the structures
		\[
		\OA_{\st,R_0}|_{X'_\eta} \simeq \OA_{\st,R'_0}.
		\]
	\end{enumerate}
    Moreover, assume $X$ is $p$-completely \'etale over $\Spf(\mathcal{O}_K\langle x_1^{\pm1},\ldots, x_r^{\pm1} \rangle)$, and let $V$ be the affinoid perfectoid pro-\'etale cover of $X$ given by  $\Spf(\mathcal{O}_C\langle x_i^{\pm1/p^\infty}\rangle)_\eta\times_{\Spf(\mathcal{O}_K\langle x_i^{\pm1}\rangle)_\eta} X_\eta$. 
    \begin{enumerate}[label=\upshape{(\roman*)},resume]
        \item\label{cor: property of OAst explicit} The restriction $\OA_{\st,R_0}|_V$ admits an isomorphism with $\Acrys|_V\{v_i\}[\logu]$, compatible with that in \Cref{prop: Acrys and OAcrys}.\ref{prop: Acrys and OAcrys explicit}.
    \end{enumerate}
\end{corollary}
As in \Cref{prop: relation of horizontal period rings}, the various period sheaves with connections naturally relate to each other.
This is summarized in the following statement.
\begin{proposition}
	\label{prop: relation of all period sheaves}
	Let $X$ be as in \Cref{conv of smooth affine}.
	There is a natural commutative diagram of injections of period sheaves in $\Shv(X_{\eta,\pe})$
	\[
	\begin{tikzcd}
		\Acrys \arrow[r, hook] \arrow[d, hook] & \Ast \arrow[r, hook] \arrow[d, hook] & \Ast\otimes_W \mathcal{O}_K \arrow[r, hook] \arrow[d, hook] & \mathbb{B}^+_\dR \arrow[d, hook] \\
		\OA_{\crys,R_0} \arrow[r, hook] & \OA_{\st,R_0} \arrow[r, hook] & \OA_{\st,R_0}\otimes_W \mathcal{O}_K \arrow[r, hook] & \OB^+_\dR
		\end{tikzcd}
	\]
    with the following properties: 
    \begin{enumerate}[label=\upshape{(\roman*)}]
        \item The arrows in the bottom row respect connections, and the top row is obtained from the second row by taking horizontal sections. 
        \item The first square is a pushout diagram, and the arrows respect the Frobenius structures. In addition, the first column is obtained from the second by taking the $N = 0$ sections.
    \end{enumerate}
    In addition, there is an entirely similar conclusion after we invert $p$ or $\mu$ in the above diagram.
\end{proposition}
\begin{proof}
	Let us first forget the third column and consider the rest of the diagram.
	By \Cref{def OAst}, the sheaf of rings $\OA_{\st,R_0}$ is defined as the pushout of $\OA_{\crys,R_0}$ and $\Ast$ over $\Acrys$.
	Moreover, both $\OA_{\crys,R_0}$ and $\Ast$ admit maps to $\OB^+_\dR$, through \Cref{prop: BdR and OBdR}.\ref{prop: BdR and OBdR with crys} and \Cref{cor: relation of horizontal period sheaves} separately.
	So to get a map from $\OA_{\st,R_0}$ to $\OB^+_\dR$, it suffices to check that the above two maps are compatible after restricted to $\Acrys$.
	This, together with the injectiveness, can then be done pro-\'etale locally, and follows from the explicit descriptions of $\OA_{\st,R_0}$ and $\OB^+_\dR$ as in \Cref{cor: property of OAst}.\ref{cor: property of OAst explicit} and \Cref{prop: BdR and OBdR}.\ref{prop: BdR and OBdR explicit} respectively.
	The rest of the claims follow from \Cref{cor: property of OAst}.
	
	As the sheaves $\mathbb{B}^+_\dR$ and $\OB^+_\dR$ are $\mathcal{O}_K$-linear, the natural map from $\Ast$ and $\OA_{\st,R_0}$ naturally extends to their $\mathcal{O}_K$-linear extension.
	To show the injectiveness of the maps $\Ast\otimes_W \mathcal{O}_K\to \mathbb{B}^+_\dR$ and $\OA_{\st,R_0}\otimes_W \mathcal{O}_K\to \OB_\dR$, by \Cref{cor: property of OAst}.\ref{cor: property of OAst explicit} and \Cref{prop: BdR and OBdR}.\ref{prop: BdR and OBdR explicit}  again it suffices to consider the map of horizontal period sheaves.
	This then follows from \cite[Prop.~2.27]{Shi22}, where $\mathbb{B}^+_\mathrm{max}$ contains $\mathbb{B}^+_\crys$ as a subring (\cite[Lem.~2.34]{Shi22}).
\end{proof}
\begin{remark}
	As in \cref{sub horizontal period sheaves}, we can consider the complete semi-stable period sheaf with connection $\mathcal{O}\widehat{\mathbb{A}}_{\st,R_0}$, defined as the complete coproduct of the log PD-thickenings $(\Acrys, \mathcal{O}^+_{X_\eta}/p, M_{\Acrys})$ and $(D_{R_0\llbracket u \rrbracket}(E), R/p, (u^\mathbb{N})^a)$ in $(R/p,M_{R/p})_\lcrys$. 
	Here $R_0\llbracket u \rrbracket$ is the higher dimensional Breuil--Kisin ring associated with the $W$-model $R_0$.
	Just like how we construct $\IA_\st$ in \cref{sub horizontal period sheaves}, we can then take the submodule of nilpotent elements under the monodromy operator of $\mathcal{O}\widehat{\mathbb{A}}_{\st,R_0}$.
	One can check that this latter construction coincides with $\OA_{\st,R_0}$ in \Cref{def OAst}.
\end{remark}

Finally, we consider the pro-\'etale cohomology for period sheaves with connections.
In the following, we let $\nu:X_{\eta,\pe}\to X_{\eta, \et}$ be the natural map of sites,
and by abuse of notation we denote the restriction of $\mathcal{O}_{X_{0,\eta}}$ to $X_\eta$ still by $\sO_{X_{0, \eta}}$ when this implication is clear. 
\begin{theorem}[Zero-th cohomology of period sheaves]
	\label{thm: coh of period sheaves}
	Let $X$ be as in \Cref{conv of smooth affine}. 
	\begin{enumerate}[label=\upshape{(\roman*)}]
        \item\label{thm: coh of period sheaves dR} There is a natural isomorphism of sheaves of rings over $X_\et$ that identifies their connections:
		\[
		\mathcal{O}_{X_\eta} \longrightarrow \nu_* \OB_\dR.
		\]
		\item\label{thm: coh of period sheaves crys} There is a natural Frobenius equivariant isomorphism of sheaves of rings over $X_\et$ that identifies their connections:
		\[
		\mathcal{O}_{X_{0,\eta}} \longrightarrow \nu_* \OB_{\crys,R_0}.
		\]
		\item\label{thm: coh of period sheaves st} There is a natural Frobenius equivariant isomorphism of sheaves of rings over $X_\et$ that identifies their connections:
		\[
		\mathcal{O}_{X_{0,\eta}} \longrightarrow \nu_* \OB_{\st,R_0}.
		\]
	\end{enumerate}
\end{theorem}
\begin{proof}
	Part \ref{thm: coh of period sheaves dR} was shown in \cite[Cor.~6.19, Thm~7.6]{Sch13}.
	To get the other two parts, we first notice that by explicit constructions in \Cref{prop: Acrys and OAcrys}.\ref{prop: Acrys and OAcrys explicit} and \Cref{cor: property of OAst}.\ref{cor: property of OAst explicit}, as well as \Cref{prop: relation of all period sheaves}, there are natural injections
	\begin{equation}
        \label{eqn: coh of period sheaves eqn 1}
	    \mathcal{O}_{X_{0,\eta}} \hookrightarrow \OB_{\crys,R_0} \hookrightarrow \OB_{\st,R_0} \hookrightarrow \OB_\dR.
	\end{equation}
 Moreover, by \Cref{cor: property of OAst}.\ref{cor: property of OAst et} and the fact that $\nu_*$ is left exact, we get the following injections of sheaves on $X_{\eta,\et}$
 \begin{equation}
        \label{eqn: coh of period sheaves eqn 2}
	    \mathcal{O}_{X_{0,\eta}} \hookrightarrow \nu_* \OB_{\crys,R_0} \hookrightarrow \nu_* \OB_{\st,R_0} \hookrightarrow \nu_* \OB_\dR.
	\end{equation}
	The maps are compatible with connections and Frobenius structures whenever applicable.
	Moreover, since the sheaf of rings $\OB_\dR$ is $K$-linear, by tensoring (\ref{eqn: coh of period sheaves eqn 2}) with $K$ we obtain another sequence 
 \begin{equation}
        \label{eqn: coh of period sheaves eqn 3}
	    \mathcal{O}_{X_\eta}=\mathcal{O}_{X_{0,\eta}}\otimes_{K_0} K \hookrightarrow \nu_* \OB_{\crys,R_0}\otimes_{K_0}K \hookrightarrow \nu_*  \OB_{\st,R_0}\otimes_{K_0} K  \longrightarrow \nu_* \OB_\dR.
	\end{equation}
    
    The injectivity of all but the last arrow in (\ref{eqn: coh of period sheaves eqn 3}) is clear, and we claim that the last arrow is also injective.
    By the explicit constructions in \Cref{prop: Acrys and OAcrys}.\ref{prop: Acrys and OAcrys explicit}, \Cref{cor: property of OAst}.\ref{cor: property of OAst explicit}, and \Cref{prop: BdR and OBdR}.\ref{prop: BdR and OBdR explicit} we have isomorphisms
	\[
	\OB_{\crys,R_0} \otimes_{K_0} K \simeq \bigl(\mathbb{B}_\crys\otimes_{K_0} K \bigr) \{ v_i \},~\OB_{\st,R_0} \otimes_{K_0} K \simeq \bigl(\mathbb{B}_\st\otimes_{K_0} K \bigr) \{ v_i \},~~\OB_\dR \simeq \mathbb{B}_\dR \llbracket v_i \rrbracket,
	\]
	where the natural maps among them are compatible with the formal variables $v_i$ and are induced from $\mathbb{B}_\crys \to \mathbb{B}_\st \to \mathbb{B}_\dR$.
	So to show the claimed injectivity, it suffices to show the injectivity of the following maps
	\[\mathbb{B}_\crys \otimes_{K_0} K \longrightarrow \mathbb{B}_\st \otimes_{K_0} K \longrightarrow \mathbb{B}_\dR.
	\]
	This then follows for example from \cite[Prop.~2.27]{Shi22}, where $\mathbb{B}^+_\mathrm{max}$ contains $\mathbb{B}^+_\crys$ as a subring (\cite[Lem.~2.34]{Shi22}).

    Now, given that the morphisms in (\ref{eqn: coh of period sheaves eqn 3}) are all injective, we see that the algebra of global sections $\Gamma(X_\eta, \nu_* \OB_{\crys,R_0})$ is a $R_0[1/p]$-subalgebra of $\Gamma(X_\eta, \nu_* \OB_\dR) \simeq R[1/p]$ whose $K$-linearization is a subalgebra of $R[1/p]$ as well.
    Notice that since $R[1/p] = R_0[1/p] \tensor_{K_0} K$, the above implies that the inclusion $R_0[1/p]\to \Gamma(X_\eta, \nu_* \OB_{\crys,R_0})$ becomes an isomorphism after tensoring with $K$.
    As a consequence, thanks to the full faithfulness of the map $K_0\to K$, we see the ring $\Gamma(X_\eta, \nu_* \OB_{\crys,R_0})$ must be equal to $R_0[1/p]$. This proves \ref{thm: coh of period sheaves crys} and \ref{thm: coh of period sheaves st} follows similarly. 
\end{proof}
Here we note that by taking the submodules of horizontal sections, we get cohomology of horizontal period sheaves as below.
In the following, we let $\mathcal{O}_{\Spa(K_0)}$ be the preimage of the \'etale structure sheaf of $\Spa(K_0)$ along the map $X_{\eta,\et}\to \Spa(K_0)_\et$.

\begin{corollary}
	\label{cor: coh of horizontal period sheaves}
	Let $X$ be a smooth $p$-adic formal scheme over $\mathcal{O}_K$.
	\begin{enumerate}[label=\upshape{(\roman*)}]
		\item The natural map $\mathcal{O}_{\Spa(K_0)}\to  \nu_* \mathbb{B}_\crys$ is a Frobenius equivariant isomorphism.
		\item The natural map $\mathcal{O}_{\Spa(K_0)}\to  \nu_* \mathbb{B}_\st$ is a Frobenius equivariant isomorphism. 
		\item The natural map $\mathcal{O}_{\Spa(K_0)}\to  \nu_* \mathbb{B}_\dR$ is an isomorphism. 
	\end{enumerate}
\end{corollary}
\begin{proof}
	As the maps above are canonically defined, to check the isomorphisms it suffices to assume that $X$ is smooth and affine, where the claims follow from \Cref{prop: relation of all period sheaves}, \Cref{thm: coh of period sheaves}, and the observation that $\mathcal{O}_{X_{0,\eta}}^{\nabla=0}=K_0$.
\end{proof}

\section{Crystalline and semi-stable local systems}
\label{sec def local system}
	In this section, we give two definitions of crystalline and semi-stable local systems over the generic fiber of a $p$-adic formal scheme $X$.
	The first definition is in the spirit of Faltings (cf. \cite{Fal89}), which is defined using horizontal period sheaves in \Cref{sub horizontal period sheaves} and in particular allows a general setup where $X$ is not regular.
	In the special case when $X$ is smooth, affine and when an unramified model is chosen (cf. \Cref{conv of smooth affine}), we give the second definition by introducing the functors $D_\crys$ and $D_\st$ on the category of $p$-adic local systems over $X_\eta$, using period sheaves with connections introduced in \Cref{sub period sheaves with conn}.
	The second definition in particular comes with a \emph{functoriality} for the $F$-isocrystals that are associated with the given crystalline or semi-stable local systems, which will be crucial for our eventual application.
	
\subsection{Crystallinity and semi-stability a la Faltings}
\label{sub def local system Faltings}
	This subsection is devoted to the definitions of crystalline local systems and semi-stable local systems for general $p$-adic formal scheme $X$, extending that of Faltings (cf. \cite[p.~67]{Fal89}).
	We fix a topologically finite type $p$-adic formal schemes $X$ over $\mathcal{O}_K$, and as in \Cref{log of framed semi-stable reduction} we let $M_X$ be its standard log structure. 
	
	We start with a construction of the vector bundles over horizontal period sheaves that arises from a given (locally free) $F$-isocrystal on the (log-)crystalline site (cf. \cite[Def.~2.25]{GR22}).

\begin{construction}
		\label{const: Bcrys(E)}
		\begin{enumerate}[label=\upshape{(\roman*)}]
			\item\label{const: Bcrys(E) crys}  Let $(\mathcal{E},\varphi_{\mathcal{E}}) \in \Isoc^\varphi(X_{k,\crys})$ be an $F$-isocrystal on $X_{k,\crys}$, which naturally extends to an $F$-isocrystal on $X_{p=0, \crys}$ by \Cref{Dwork's trick}.\ref{Dwork's trick crys}.
			Then we define the sheaves $\mathbb{B}^+_\crys(\mathcal{E})$ and $\mathbb{B}_\crys(\mathcal{E})$ in $\Shv(X_{\eta, \pe})$ as follows:
			Let $U \in \Perfd/(X_{\eta,C})_\pe$ be any object and set $\widehat{U} = \Spa(S[1/p],S)$. As the morphism $\rAcrys(S) \to S/p$ is a pro-PD-thickening with a Frobenius endomorphism in $X_{p=0,\crys}$, it makes sense to set
			\begin{itemize}
				\item $\mathbb{B}^+_\crys(\mathcal{E})(U) \colonequals \mathcal{E}(\rAcrys(S), S/p)[1/p] = \bigl(\lim_r \mathcal{E}(\rAcrys(S)/p^r,S/p,\gamma)\bigr)[1/p]$, where $\gamma$ is the canonical PD-structure and
				\item $\mathbb{B}_\crys(\mathcal{E})(U) \colonequals \mathbb{B}^+_\crys(\mathcal{E})(U)[1/\mu]$.
			\end{itemize}
			Here as in \cite[Lem.~2.26]{GR22}, $\mathbb{B}^+_\crys(\mathcal{E})$ (resp. $\mathbb{B}_\crys(\mathcal{E})$) forms a sheaf of modules over $\mathbb{B}^+_\crys$ (resp. $\mathbb{B}_\crys$) and the sheaf condition follows from the crystal property together with the sheaf condition of $\mathbb{B}^+_\crys$ (resp. $\mathbb{B}_\crys$). 
   One checks that $\IB_\crys^+(-)$ (resp. $\IB_\crys(-)$) defines a functor $\Isoc^\varphi(X_{k, \crys})$
			Moreover, the Frobenius structure $\varphi_\mathcal{E}$ naturally equips $\mathbb{B}^+_\crys(\mathcal{E})$ and $\mathbb{B}_\crys(\mathcal{E})$ with isomorphisms
			\[
			\mathbb{B}^+_\crys(\mathcal{E}) \otimes_{\mathbb{B}^+_\crys, \varphi} \mathbb{B}^+_\crys \xrightarrow{\sim} \mathbb{B}^+_\crys(\mathcal{E}),\quad \mathbb{B}_\crys(\mathcal{E}) \otimes_{\mathbb{B}_\crys, \varphi} \mathbb{B}_\crys \xrightarrow{\sim} \mathbb{B}_\crys(\mathcal{E}).
			\]
			
			\item\label{const: Bcrys(E) lcrys} Similarly, given a locally free $F$-isocrystal $(\mathcal{E},\varphi_{\mathcal{E}}) \in \Isoc^\varphi((X_k,M_{X_k})_\lcrys)$ (which by the equivalence in \Cref{Dwork's trick}.\ref{Dwork's trick log-crys} can be identified with an $F$-isocrystal on $(X_{p=0}, M_{X_{p=0}})_\lcrys$), we can consider the evaluation of $\mathcal{E}$ at the pro-log-PD-thickening $(\rAcrys(S),S/p,M_{\rAcrys(S)})$ in $(X_{p=0}, M_{X_{p=0}})_\lcrys$.
			This defines vector bundles $\mathbb{B}^+_\crys(\mathcal{E})$ and $\mathbb{B}_\crys(\mathcal{E})$ in $\Shv(X_{\eta,\pe})$ over $\mathbb{B}^+_\crys$ and $\mathbb{B}_\crys$ respectively, and both are equipped with Frobenius structures.
			
			\item\label{const: Bcrys(E) st and dR}  In the above two situations, we define $\mathbb{B}^+_\st(\mathcal{E})$ and $\mathbb{B}_\st(\mathcal{E})$ by
			\[
			\mathbb{B}^+_\st(\mathcal{E})\colonequals \mathbb{B}^+_\crys(\mathcal{E})\otimes_{\mathbb{B}^+_\crys} \mathbb{B}^+_\st, \quad  \mathbb{B}_\st(\mathcal{E})\colonequals \mathbb{B}_\crys(\mathcal{E})\otimes_{\mathbb{B}_\crys} \mathbb{B}_\st.
			\]
			By construction, $\mathbb{B}^+_\st(\mathcal{E})$ is equipped with the tensor product Frobenius structure $\varphi_{\IB^+_\crys(\sE)}\otimes \varphi_{\IB^+_\st}$ and the monodromy operator $\mathrm{id}\otimes N_{\mathbb{B}^+_\st}$, and similarly for $\mathbb{B}_\st(\mathcal{E})$.
			We also define $\mathbb{B}^+_\dR(\mathcal{E})$ and $\mathbb{B}_\dR(\mathcal{E})$ by 
			\[
			\mathbb{B}^+_\dR(\mathcal{E}) \colonequals \mathbb{B}^+_\crys(\mathcal{E}) \otimes_{\mathbb{B}^+_\crys} \mathbb{B}^+_\dR, \quad \mathbb{B}_\dR(\mathcal{E}) \colonequals \mathbb{B}_\crys(\mathcal{E}) \otimes_{\mathbb{B}_\crys} \mathbb{B}_\dR.
			\]
		\end{enumerate}
\end{construction}
\begin{remark}
\label{rmk: Bcrys and forgetful functor}
	Here we note that the natural forgetful functor of sites $(X_k,M_{X_k})_\lcrys \to X_{k,\crys}$ sends the log PD-thickening $(\rAcrys(S),S/p,M_{\rAcrys(S)})$ onto the PD-thickening $(\rAcrys(S),S/p)$.
	In particular, for a given $F$-isocrystal $(\mathcal{E},\varphi_\mathcal{E})\in \Isoc^\varphi(X_{k,\crys})$, the associated $\mathbb{B}^+_\crys$-vector bundle $\mathbb{B}^+_\crys(\mathcal{E})$ defined via \Cref{const: Bcrys(E)}.\ref{const: Bcrys(E) crys} coincides with the $\mathbb{B}^+_\crys$-vector bundle $\mathbb{B}^+_\crys(\mathcal{E}')$ defined via \Cref{const: Bcrys(E)}.\ref{const: Bcrys(E) lcrys}, 
	where $(\mathcal{E}',\varphi_{\mathcal{E}'})\in \Isoc^\varphi((X_k,M_{X_k})_\lcrys)$ is the $F$-isocrystal over the log crystalline site that is induced from $(\mathcal{E},\varphi_\mathcal{E})$. 
\end{remark}

\begin{remark}[Site-theoretic interpretation]
	\label{rmk:site-theoretic_construction_of_Bcrys(E)}
	A more conceptual way of constructing $\mathbb{B}_\crys(\mathcal{E})$ in \Cref{const: Bcrys(E)} is to consider a functor of ringed topoi 
	\[
	c_X\colon (X_{\eta,\pe}^\sim, \mathbb{B}_\crys) \longrightarrow (X_{p=0,\crys}^\sim, \mathcal{O}_\crys),
	\]
	sending an affinoid perfectoid space $\Spa(S[1/p],S)$ onto the pro-PD-thickening $(\rAcrys(S), S/p)$.
	The functor is cocontinuous and induces a morphism of ringed topoi $c_X$ in the same direction (\cite[\href{https://stacks.math.columbia.edu/tag/00XN}{Tag 00XN}]{stacks-project}).
	In particular, for an $F$-isocrystal $\mathcal{E}$ over $X_{p=0}$ (or equivalently over $X_k$ by \Cref{Dwork's trick}), the pullback $c_X^*\mathcal{E}$ is by construction equal to the vector bundle $\mathbb{B}_\crys(\mathcal{E})$ over the pro-\'etale site $X_{\eta,\pe}$.
	The same applies to the logarithmic setting, which we shall not repeat.
\end{remark}

Now we are ready to define the crystallinity and the semi-stability.
\begin{definition}\label{def crys/st Faltings}
	Let $X$ be a topologically finite type $p$-adic formal scheme over $\mathcal{O}_K$, let $M_X$ be its standard log structure as in \Cref{log of framed semi-stable reduction}, and let $T$ be a $\mathbb{Q}_p$-local system over $X_\eta$ that admits a $\mathbb{Z}_p$-lattice.
	\begin{enumerate}[label=\upshape{(\roman*)}]
		\item\label{def crys Faltings} We call $T$ is a \emph{crystalline local system} (with respect to $X$) if there is a locally free $F$-isocrystal $(\mathcal{E},\varphi_\mathcal{E})\in \Isoc^\varphi(X_{k,\crys})$, together with a Frobenius equivariant isomorphism of $\mathbb{B}_\crys$-vector bundles
		\[
		\vartheta:\mathbb{B}_\crys(\mathcal{E}) \xrightarrow{\sim} \mathbb{B}_\crys\otimes_{\mathbb{Z}_p} T.
		\]
		\item\label{def st Faltings} We call $T$ is a \emph{semi-stable local system} (with respect to $X$) if there is a locally free $F$-isocrystal $(\mathcal{E},\varphi_\mathcal{E})\in \Isoc^\varphi((X_k,M_{X_k})_\lcrys)$, together with a Frobenius equivariant isomorphism of $\mathbb{B}_\crys$-vector bundles
		\[
		\vartheta:\mathbb{B}_\crys(\mathcal{E}) \xrightarrow{\sim} \mathbb{B}_\crys\otimes_{\mathbb{Z}_p} T.
		\]
\end{enumerate}
We use $\Loc_{\mathbb{Z}_p}^\crys(X_\eta)$ or $\Loc_{\mathbb{Z}_p}^\st(X_\eta)$ to denote the full subcategory of $\mathbb{Z}_p$-local systems over $X_\eta$ that are crystalline or semi-stable with respect to the $X$ separately.
\end{definition}
From \Cref{def crys/st Faltings}, the only difference between the crystallinity and the semi-stability is that the $F$-isocrystal $(\mathcal{E},\varphi_\mathcal{E})$ comes from either the crystalline site or the log crystalline site of the special fiber.
In particular, under the natural inclusion functor (\Cref{equiv def log isoc}), it is clear that a crystalline local system is semi-stable.
\begin{remark}[Equivalent definition using $\mathbb{B}_\st$]
	\label{rmk: def st Faltings alternative}
	Assume the same setup as in \Cref{def crys/st Faltings}, the local system $T$ is semi-stable if and only if there is an $F$-isocrystal $(\mathcal{E},\varphi_\mathcal{E})\in \Isoc^\varphi((X_k,M_{X_k})_\lcrys)$, together with a Frobenius and monodromy equivariant isomorphism of $\mathbb{B}_\st$-vector bundles
\[
    \vartheta':\mathbb{B}_\st(\mathcal{E}) \xrightarrow{\sim} \mathbb{B}_\st\otimes_{\mathbb{Z}_p} T.
\]
	This follows as the monodromy operator on $\mathbb{B}_\st(\mathcal{E})=\mathbb{B}_\crys(\mathcal{E})\otimes_{\mathbb{B}_\crys} \mathbb{B}_\st$ is equal to $\mathrm{id}\otimes N_{\mathbb{B}_\st}$ (\Cref{const: Bcrys(E)}.\ref{const: Bcrys(E) st and dR}), and the kernel of $N_{\mathbb{B}_\st}:\mathbb{B}_\st\to \mathbb{B}_\st$ is $\mathbb{B}_\crys$ (by \Cref{def: Ast} and \Cref{const: monodromy on hAst}).
	\footnote{We thank Heng Du for explaining this equivalence to us.}
    We refer the reader to \Cref{lem:two monodromy on D times OBst} for a related discussion.
\end{remark}
\begin{remark}[Unnecessity of filtration]
	\label{rmk: def crys/st minimal version}
	The careful reader might notice that even when $X$ is smooth, our definition of crystalline or semi-stable local systems above does not mention anything about filtration, which is different from either Faltings (\cite{Fal89}) or Tan--Tong (\cite{TT19}).
	The idea is that the datum of filtration should be completely encoded and determined by the de Rhamness of the local system (where the latter is a purely generic condition and makes no reference to the choice of the integral model $X$), which is implied by the crystallinity or the semi-stability, as we will see very soon.
	In fact when $X$ is smooth over $\mathcal{O}_K$, it is shown in \cite[Prop.~2.38]{GR22} that the two versions of crystallinity coincide, and there is a canonical filtration on the $K$-linear base extension of the $F$-isocrystal such that it is unique in an appropriate sense.
\end{remark}

\begin{remark}[Independence of models]
	We emphasize that the definitions of the crystallinity and the semi-stability in \Cref{def crys/st Faltings} a priori depend on the choice of the integral model $X$.
	However, one readily tells from our criteria \cref{intro:thm pc} that  for local systems over a smooth rigid space, the crystallinity (resp. the semi-stability) is independent of the choices of smooth (resp. semi-stable) models. 
 In fact, this independence statement can already be deduced from purity results, for which we refer the reader to \cite[Cor.~5.5]{DLMS2}.
\end{remark}
We also comment on a natural variant for the de Rhamness.
\begin{remark}[de Rham local system a la Faltings]
	One can give a similar definition for the de Rham local system as in \Cref{def crys/st Faltings} and \Cref{const: Bcrys(E)} as well, without referring to $\OB_\dR$: Namely, a vector bundle $\mathcal{E}$ with an integrable connection over $X_\eta$ can be identified with a crystal $\mathcal{E}'$ over the infinitesimal site $(X_\eta/K)_{\inf}$ (cf. \cite{Guo21}), and since $(\mathbb{B}^+_\dR, \widehat{\mathcal{O}}_{X_\eta})$ is a pro-infinitesimal thickening (of sheaves) over $(X_\eta/K)_{\inf}$, one can naturally evaluate a crystal $\mathcal{E}'$ at the thickening to get a $\mathbb{B}^+_\dR$-vector bundle $\mathbb{B}^+_\dR(\mathcal{E}')$.
	When $\mathcal{E}$ is equipped with a locally split filtration satisfying Griffiths transversality, one can further enhance $\mathbb{B}_\dR(\mathcal{E}')$ with a natural filtration as well.
    \footnote{A better formulation, in light of recent advance of the stacky approach in $p$-adic geometry, is to identify $\mathcal{E}$ with a coherent sheaf over the \emph{analytic filtered de Rham stack} $X_\eta^{\dR+}$, where the latter naturally admits a $\mathbb{B}^+_\dR$-valued point.
	Taking the pullback functor onto the $\mathbb{B}^+_\dR$-point gives a filtered vector bundle over $\mathbb{B}^+_\dR$ (hence over $\mathbb{B}_\dR$).
    We leave the details to interested readers.}
	Then the de Rhamness can be defined by requiring a filtered $\mathbb{B}_\dR$-linear isomorphism $\mathbb{B}_\dR(\mathcal{E}') \simeq T\otimes_{\mathbb{Z}_p} \mathbb{B}_\dR$.
	One can check using the rigidity of de Rham local systems (\cite{LZ17}) that this notion coincides with the usual one, even when $X_\eta$ is not smooth over $K$.
	For our purposes, we would not pursue this definition further and will instead recall an alternative definition of Brinon and Scholze in next subsection.
\end{remark}

\subsection{$D_\crys$ and $D_\st$ functors: local formulae}
\label{sub D functors}
In this subsection, using the period sheaves with connection, we construct the $D_\crys$ and the $D_\st$ functors from $p$-adic local systems to their associated $F$-isocrystals, for a smooth affine $p$-adic formal scheme $X$.
This produces a local formula for the associated (log-) $F$-isocrystal in terms of a given $p$-adic local system $T$, and in particular is functorial with respect to $T$.
	
We start by recalling the following result of Liu--Zhu on $p$-adic Riemann--Hilbert correspondence.
Recall the natural map of sites $\nu:X_{\eta,\pe}\to X_{\eta,\et}$ as in \cite{Sch13}. 
\begin{theorem}\cite[Thm.~1.5,~Thm.~3.9.(iv)]{LZ17}
		\label{thm: LZ}
		Let $X_\eta$ be a smooth rigid space over $K$.
		\begin{enumerate}[label=\upshape{(\roman*)}]
			\item\label{thm: LZ functor} There is a natural functor
			\[
			D_\dR: \mathrm{Loc}_{\mathbb{Z}_p}(X_\eta) \longrightarrow \Vect^\nabla(X_\eta),\quad T\longmapsto \nu_*(T\otimes_{\mathbb{Z}_p} \OB_\dR),
			\]
			where $\Vect^\nabla(X_\eta)$ is the category of vector bundles with flat connection over $X_\eta$, and the connection of $D_\dR$ is equal to $\nu_*\bigl( T\otimes_{\mathbb{Z}_p} \OB_\dR\xrightarrow{\mathrm{id}_T\otimes \nabla_{\OB_\dR}} T\otimes_{\mathbb{Z}_p} \OB_\dR \otimes_{\mathcal{O}_{X_\eta}} \Omega^1_{X_\eta/K} \bigr)$.
			\item\label{thm: LZ rank} \emph{\cite[Prop.\ 8.2.5]{Bri08}} There is a natural injective linearization map of $\OB_\dR$-vector bundles that is compatible with their $\mathbb{B}_\dR$-linear connections 
			\[
			\alpha_{\dR,T}:D_\dR(T)\otimes_{\mathcal{O}_{X_\eta}} \OB_\dR \longrightarrow T\otimes_{\mathbb{Z}_p} \OB_\dR.
			\]
			Moreover, when $\mathrm{rank}_{\mathbb{Z}_p}(T)=\mathrm{rank}_{\mathcal{O}_Z}(D_\dR(T))$, the map is an isomorphism.
			\item\label{thm: LZ pullback} The functor $D_\dR$ is compatible with the pullback along any map of smooth rigid spaces.
   In particular, $D_\dR(T)$ is a coherent sheaf over $X_\eta$.
		\end{enumerate}
\end{theorem}
The above in particular gives an equivalent definition for a $p$-adic local system to be de Rham (cf. \cite[Def.~8.3]{Scholze}), which we recall as below.
\begin{definition}
	\label{def dR}
	Let $X_\eta$ be a smooth rigid space over $K$.
	A $\mathbb{Z}_p$-local system $T$ over $X_\eta$ is called \textbf{de Rham} if the rank of $D_\dR(T)$ is equal to $\mathrm{rank}_{\mathbb{Z}_p}(T)$.
\end{definition}

\begin{corollary}[Rigidity of de Rham local system]
	\label{cor:LZ rigidity}
	Let $X_\eta$ be a smooth connected rigid space over $K$.
	A $\mathbb{Z}_p$-local system $T$ over $X_\eta$ is de Rham if and only if there is a closed point $x\in X_\eta$ such that the restriction $T|_x$ is a de Rham representation.
\end{corollary}
We then consider variants of the $D_\dR$ functor that can be used to test the crystallinity and the semi-stability.
Below we use the equivalences in \Cref{equiv def isoc} and \Cref{equiv def log isoc} to interpret an $F$-isocrystal as a finitely projective $R_0[1/p]$-module with integrable connection, Frobenius structure, and nilpotent endomorphism, and we recall the notations therein.

\begin{theorem}[$D_\crys$ and $D_\st$ functors]
\label{thm: D functors}
Assume the setup of a smooth affine $p$-adic formal scheme $X=\Spf(R)$ as in \Cref{conv of smooth affine}.
\begin{enumerate}[label=\upshape{(\roman*)}]
	\item\label{thm: D functors crys} There is a natural left exact tensor functor
	\[
	D_{\crys,R_0}\colon : \Loc_{\mathbb{Z}_p}(X_\eta) \longrightarrow \mathrm{Vect}^{\varphi, \nabla_{X_0}}(X_{0, \eta}),\quad T\longmapsto \bigl(T\otimes_{\mathbb{Z}_p} \OB_{\crys,R_0} \bigr)(X_\eta),
	\]
	together with a functorial injection $\alpha_{\crys,T}: D_{\crys,R_0}(T)\otimes_{R_0[1/p]} \OB_{\crys,R_0}\to T\otimes_{\mathbb{Z}_p} \OB_{\crys,R_0}$ that is compatible with Frobenius structures and connections.
	\item\label{thm: D functors log crys} There is a natural left exact tensor functor
	\[
	D_{\st,R_0}\colon \Loc_{\mathbb{Z}_p}(X_\eta) \longrightarrow \mathrm{Vect}^{\varphi, \nabla_{X_0},N}(X_{0, \eta}),\quad T\longmapsto \bigl( T\otimes_{\mathbb{Z}_p} \OB_{\st,R_0} \bigr) (X_\eta),
	\]
	together with a functorial injection $\alpha_{\st,T}: D_{\st,R_0}(T)\otimes_{R_0[1/p]} \OB_{\st,R_0}\to T\otimes_{\mathbb{Z}_p} \OB_{\st,R_0}$ that is compatible with Frobenius structures, connections and nilpotent structures.
	\item\label{thm: D functors relation} There is a functorial Frobenius-equivariant inclusion of flat connections over $R_0[1/p]$
	\[
	D_{\crys,R_0}(T) \longrightarrow D_{\st,R_0}(T),
	\]
	which identifies $D_{\crys,R_0}(T)$ as the kernel of the nilpotent operator $N_{D_{\st,R_0}(T)}$ on $D_{\st,R_0}(T)$.
    Moreover, their base extensions along $K_0\to K$ naturally embed into the flat connection $D_\dR(T)(X_\eta)$.
\end{enumerate}
\end{theorem}
\begin{remark}
	Before we discuss the proof, we warn the reader that for general $p$-adic local system $T$, unlike $D_\dR(T)$, which is a \emph{coherent sheaf} over the rigid space $X_\eta$, neither $D_{\crys,R_0}(T)$ or $D_{\st,R_0}(T)$ do not glue to a coherent sheaf in general. 
	However, we will see in \Cref{thm:pullback for D functors} that the ranks of $D_{\crys,R_0}(T)$ and $D_{\st,R_0}(T)$ will not decrease when restricted onto an open subspace, and both $D_{\crys,R_0}(T)$ and $D_{\st,R_0}(T)$ are analytic local if their ranks are equal to $\mathrm{rank}_{\mathbb{Z}_p}(T)$.
\end{remark}

\begin{proof}
Since $D_\dR(T)$ is a coherent sheaf and $X$ is assumed to be affine, below we do not distinguish between the sheaf $D_\dR(T)$ with its module of global sections $D_\dR(T)(X_\eta)$.

	We start with Part \ref{thm: D functors crys}.
	First we notice that there is a map $\nabla:D_{\crys,R_0}(T)\to D_{\crys,R_0}(T)\otimes_{R_0[1/p]} \Omega^1_{R_0[1/p]/K_0}$ defined by 
    $\nu_*\bigl( T\otimes_{\mathbb{Z}_p} \OB_{\crys,R_0} \to T\otimes_{\mathbb{Z}_p} \OB_{\crys,R_0}\otimes_{R_0[1/p]}\Omega^1_{R_0[1/p]/K_0} \bigr)(X_\eta)$.
    Indeed, since the coherent module $\Omega^1_{R_0[1/p]/K_0}$ over $R_0[1/p]$ is locally free, it follows from \Cref{thm: coh of period sheaves}.\ref{thm: coh of period sheaves crys} that there is a natural isomorphism 
    \begin{equation}
        \nu_*\bigl( T\otimes_{\mathbb{Z}_p} \OB_{\crys,R_0}\otimes_{R_0[1/p]}\Omega^1_{R_0[1/p]/K_0} \bigr)(X_\eta) \simeq D_{\crys,R_0}(T)\otimes_{R_0[1/p]} \Omega^1_{R_0[1/p]/K_0}.
    \end{equation}
	So we get a $K_0$-linear continuous map $\nabla:D_{\crys,R_0}(T)\to D_{\crys,R_0}(T)\otimes_{R_0[1/p]} \Omega^1_{R_0[1/p]/K_0}$.
    
	By the flatness of $T$ over $\mathbb{Z}_p$, the construction in \Cref{thm: LZ}.\ref{thm: LZ functor}, and the fact that the injection $\OB_{\crys,R_0}\to \OB_\dR$ is compatible with connections (\Cref{prop: relation of all period sheaves}), we see the map $\nabla$ is a $K_0$-linear integrable connection and $D_{\crys,R_0}(T)$ is a $R_0[1/p]$-submodule of $D_\dR(T)$ where the connections are compatible.
	Hence by the finiteness of $D_\dR(T)$ over $R_0[1/p]$, the submodule together with the integrable connection $(D_{\crys,R_0}(T), \nabla)$ is a flat connection over $R_0[1/p]$. 
    Note that since $\Spa(R_0[1/p])$ is a smooth rigid space over $K_0$, $D_{\crys, R_0}(T)$ is necessarily locally free, i.e., is a vector bundle.
	
	In addition, the Frobenius endomorphism of $\OB_{\crys,R_0}$ induces on $D_{\crys,R_0}(T) \colonequals \bigl( \nu_*(T\otimes_{\mathbb{Z}_p} \OB_{\crys,R_0})\bigr) (X_\eta)$ a $\varphi_{R_0}$-linear endomorphism $\wt{\varphi}_{D_{\crys, R_0}(T)}$, and we denote the linearization of which by $\varphi_{D_{\crys, R_0}(T)} : \varphi_{R_0}^* D_{\crys, R_0}(T) \to D_{\crys, R_0}(T)$. 
	As the Frobenius structure of  $\OB_{\crys,R_0}$ is compatible with its connection, we see the linearized Frobenius structure $\varphi_{D_{\crys,R_0}(T)}$ commutes with the connection of $D_{\crys,R_0}(T)$ as well. 
 Note that since $\varphi_{D_{\crys,R_0}(T)}$  is an injective horizontal morphism between two flat vector bundles of the same rank, it must be an isomorphism. 
    Hence the data $(D_{\crys,R_0}(T), \nabla, \varphi_{D_{\crys,R_0}(T)})$ defines an object in $\mathrm{Vect}^{\varphi, \nabla_{X_0}}(X_{0, \eta})$.
	Here we also remark that the construction of $D_{\crys,R_0}(T)$ together with its additional structures is functorial in $T$.
	
	By taking adjunction with respect to the map of ringed sites $\nu$,
 we get a natural diagram of sheaves over $X_{\eta,\pe}$, which is functorial with respect to $T$:
	\[
	\begin{tikzcd}
			D_{\crys,R_0}(T)\otimes_{R_0[1/p]} \OB_{\crys,R_0} \arrow[r, "\alpha_{\crys,T}"] \ar[d] & T\otimes_{\mathbb{Z}_p} \OB_{\crys,R_0} \ar[d]\\
			D_\dR(T)\otimes_{R[1/p]} \OB_\dR \arrow[r,"\alpha_{\dR,T}"] & T\otimes_{\mathbb{Z}_p} \OB_\dR.
	\end{tikzcd}
    \]
    So to show the injectiveness of $\alpha_{\crys,T}$, it suffices to show the other three arrows are injective.
    The right vertical map is an injection by \Cref{prop: relation of all period sheaves} and the flatness of $T$ over $\mathbb{Z}_p$.
    On the other hand, the injectiveness of $\alpha_{\dR,T}\colon D_\dR(T)\otimes_R \OB_\dR \to T\otimes_{\mathbb{Z}_p} \OB_\dR$ follows from \Cref{thm: LZ}.\ref{thm: LZ rank}.
    Finally, the left vertical map $D_{\crys,R_0}(T)\otimes_{R_0[1/p]} \OB_{\crys,R_0}\to D_\dR(T)\otimes_R \OB_\dR$ can be factored as 
    \[
    D_{\crys,R_0}(T)\otimes_{R_0[1/p]} \OB_{\crys,R_0}\longrightarrow D_{\crys,R_0}(T)\otimes_{R_0[1/p]} \OB_\dR  \longrightarrow  D_\dR(T)\otimes_{R[1/p]} \OB_\dR,
    \]
    where the first map is injective by the local freeness of $D_{\crys,R_0}(T)$ over $\mathcal{O}_{X_{0,\eta}}$ and \Cref{prop: relation of all period sheaves}, and the second map is injective since $D_{\crys,R_0}(T)\otimes_{K_0} K\to D_\dR(T)$ is an injective and $\OB_\dR$ is flat over $\mathcal{O}_{X_\eta}$.

	For Part \ref{thm: D functors log crys}, the $R_0[1/p]$-module $D_{\st,R_0}(T) \colonequals \nu_*(T\otimes_{\mathbb{Z}_p} \OB_{\st,R_0})$ is equipped with a Frobenius structure $\varphi_{D_{\st,R_0}(T)}$ by that of $\OB_{\crys,R_0}$, together with a map commuting with $\varphi_{D_{\st,R_0}(T)}$:
	\[
	\nabla \colonequals \nu_*\bigl( T\otimes_{\mathbb{Z}_p} \OB_{\st,R_0} \to T\otimes_{\mathbb{Z}_p} \OB_{\st,R_0}\otimes_{R_0[1/p]}\Omega^1_{R_0[1/p]/K_0} \bigr).
	\]
	By the same reasoning as in \ref{thm: D functors crys}, the last map above is identified with a natural $K_0$-linear flat connection $\nabla: D_{\st,R_0}(T) \to  D_{\st,R_0}(T)\otimes_{R_0[1/p]}\Omega^1_{R_0[1/p]/K_0}$ that commutes with $\varphi_{D_{\st,R_0}(T)}$.
	Moreover, the $\OB_{\crys,R_0}$-linear monodromy operator $N_{\OB_{\st,R_0}}:\OB_{\st,R_0}\to \OB_{\st,R_0}$ induces a $\nu_*\OB_{\crys,R_0} = R_0[1/p]$-linear nilpotent
 \footnote{Here the endomorphism is nilpotent because $D_{\st,R_0}(T)$ is finitely generated over $R_0[1/p]$ and $N_{\OB_{\st,R_0}}$ acts nilpotently on any section.}
 endomorphism $N_{D_{\st,R_0}(T)}$ on $D_{\st,R_0}(T)$, which commutes with the connection and the Frobenius structure.
	So we get an object in $\mathrm{Vect}^{\varphi, \nabla_{X_0},N}(X_{0, \eta})$, which by construction is functorial with respect to the local system $T$.
	Furthermore, similar to the proof for $D_{\crys,R_0}(T)$ above, the canonical map $\alpha_{\st,T}:D_{\st,R_0}(T)\otimes_{R_0[1/p]} \OB_{\st,R_0}\to T\otimes_{\mathbb{Z}_p} \OB_{\st,R_0}$ is an injection.

    Finally, we prove \ref{thm: D functors relation}. Since $T$ is locally free, the functor $T \tensor -$ is exact. As the functor $\nu_*$ is left exact, the exact sequence (cf. \cref{prop: relation of all period sheaves})
    \[\begin{tikzcd}
	0 & {\OB_{\crys, R_0}} & {\OB_{\st, R_0}} & {\OB_{\st, R_0}}
	\arrow[from=1-1, to=1-2]
	\arrow[from=1-2, to=1-3]
	\arrow["{N_{\OB_{\st, R_0}}}", from=1-3, to=1-4]
    \end{tikzcd}\]
    gives rise to another exact sequence 
    \[\begin{tikzcd}
	0 & {D_{\crys,R_0}(T)} & {D_{\st,R_0}(T)} & {D_{\st,R_0}(T)}
	\arrow[from=1-1, to=1-2]
	\arrow[from=1-2, to=1-3]
	\arrow["{N_{D_{\st,R_0}(T)}}", from=1-3, to=1-4].
    \end{tikzcd}\]
    This proves the first statement. The second statement follows from a similar argument together with the fact that the canonical $K$-linear map $\OB_{\st,R_0}\otimes_{K_0} K \to \OB_\dR$ is injective (\Cref{prop: relation of all period sheaves}). 
\end{proof}

\subsection{Crystallinity and semi-stability via $D$-functors}
\label{sub:crys_and_st_via_D_functor}

After introducing the $D_\crys$ and the $D_\st$ functors, we show that the crystallinity and the semi-stability, in the sense of Faltings, can be translated into a rank condition on the value of the $D$-functors.
This extends the well-known formula from the case of the point to the smooth affine setting.

To facilitate the discussion of crystalline and semi-stable local systems, we need the following observations on vector bundles over $\OB_{\crys,R_0}$ and $\OB_{\st,R_0}$.

\begin{lemma}
	\label{lem:two monodromy on D times OBst}
	Assume the setup of a smooth affine $p$-adic formal scheme $X=\Spf(R)$ as in \Cref{conv of smooth affine}.
	Let $\sM_0$ be a finite projective module over $R_0[1/p]$ equipped with a $R_0[1/p]$-linear nilpotent endomorphism $N_{\sM_0}$. 
    Define $\sM := \OB_{\st,R_0}\otimes_{R_0[1/p]} {\sM_0}$ and let $N : \sM \to \sM$ be the additive map defined by $N_{\OB_{\st,R_0}} \otimes \mathrm{id}_{\sM_0} + \mathrm{id}_{\OB_{\st,R_0}} \otimes N_{\sM_0}$. 
    Then $\sM$ is generated as an $\OB_{\st, R_0}$-module by $\ker(N_\sM)$. 
\end{lemma}
\begin{proof}
	To simplify the notations, we let $N_1$ be the endomorphism $\mathrm{id}_{\OB_{\crys,R_0}} \otimes N_{\sM_0}$ on $\sM_1 := \OB_{\crys, R_0} \tensor_{R_0[1/p]} \sM_0$. 
	Then by \Cref{cor: property of OAst}, pro-\'etale locally the sheaf of ring $\OB_{\st,R_0}$ is naturally identified with the polynomial ring of one variable $\OB_{\crys,R_0}[\logu]$ over $\OB_{\crys,R_0}$. 
    This allows us to write any element $x\in \sM$ as a sum $\sum_{i=0}^n x_i \logu^i$, where $x_i \in \sM_1$.
    Choose $m \gg 0$ such that $N_{\sM_0}^m = 0$. 
    Since $N_{\OB_{\st,R_0}}$ kills $\OB_{\crys,R_0}$ and sends $\logu^i$ onto $i \logu^{i-1}$, by solving equations explicitly we find 
    \begin{equation}
        \label{eqn: kernel of N explicit}
        \ker(N) = \mathrm{im}(\psi : \sM_1 \to \sM), \textrm{ where } \psi : x_0 \mapsto \sum_{i=0}^m \frac{(-N_1)^i(x_0)}{i!} \logu^i . 
    \end{equation}
	
	Now we let $\sM'$ be the $\OB_{\st,R_0}$-submodule of $\sM$ generated by $\ker(N)$.
	To show the equality $\sM' = \sM$, it suffices to show that any $x_0\in \sM_1$  in contained in $\sM'$.
	We prove this by induction on the minimal number $n(x_0) \in \mathbb{N}$ such that $N_1^{n(x_0)+1}(x_0)=0$.
	If $n(x_0) =0$, then by definition $x_0$ is in $\sM'$.
    For general $x_0 \in \sM_1$, we have $\psi(x_0) \in \sM'$ and 
    $$ x_0 = \psi(x_0) - \sum_{i=1}^m \frac{(-N_1)^i(x_0)}{i!} \logu^i. $$
    Notice that $n(N_1^i (x_0)) < n(x_0)$ for each $i \ge 1$.
    So by induction, $N_1^i(x_0) \logu^{i} \in \sM'$ for each $i \ge 1$. Hence the above equality implies that $x_0 \in \sM'$. 
\end{proof}

\begin{proposition}
	\label{prop:OB vb from F-iso}
	Assume the setup of a smooth affine $p$-adic formal scheme $X=\Spf(R)$ as in \Cref{conv of smooth affine}. Let $(\mathcal{E},\varphi_\mathcal{E})$ be an $F$-isocrystal over $(X_k, (0^\IN)^a)_\lcrys$ and $(\sM_0, \nabla_{\sM_0}, \varphi_{\sM_0}, N_{\sM_0})$ be the associated object in $\Vect^{\varphi, \nabla_{R_0}, N}(X_{0, \eta})$ provided by \Cref{equiv def log isoc}. 
	\begin{enumerate}[label=\upshape{(\roman*)}]
		\item\label{prop:OB vb from F-iso crys} There is a natural isomorphism of the following $\OB_{\crys,R_0}$-vector bundles with connection and Frobenius structure
		\begin{itemize}
			\item $\bigl( \mathbb{B}_\crys(\mathcal{E})\otimes_{\mathbb{B}_\crys} \OB_{\crys,R_0},\ \mathrm{id}_{\mathbb{B}_\crys(\mathcal{E})}\otimes \nabla_{\OB_{\crys,R_0}},\ \varphi_{\mathbb{B}_\crys(\mathcal{E})}\otimes \varphi_{\OB_{\crys,R_0}} \bigr)$;
			\item $\bigl(\OB_{\crys,R_0}\otimes_{R_0[1/p]} \sM_0,\ \nabla_{\OB_{\crys,R_0}}\otimes \mathrm{id}_{\sM_0} + \mathrm{id}_{\OB_{\crys,R_0}}\otimes\nabla_{\sM_0},\ \varphi_{\OB_{\crys,R_0}}\otimes \varphi_{\sM_0} \bigr)$.
		\end{itemize}
		In particular, by taking the horizontal sections, we get
		\[
		\mathbb{B}_\crys(\mathcal{E}) \simeq (\OB_{\crys,R_0}\otimes_{R_0[1/p]} \sM_0)^{\nabla=0}.
		\]
		\item\label{prop:OB vb from F-iso st} There is a natural isomorphism of the following $\OB_{\st,R_0}$-vector bundles with connection, Frobenius structure, and monodromy operator
		\begin{itemize}
			\item $\bigl( \mathbb{B}_\crys(\mathcal{E})\otimes_{\mathbb{B}_\crys} \OB_{\st,R_0},\ \mathrm{id}_{\mathbb{B}_\crys(\mathcal{E})}\otimes \nabla_{\OB_{\st,R_0}},\ \varphi_{\mathbb{B}_\crys(\mathcal{E})}\otimes \varphi_{\OB_{\st,R_0}},\ \mathrm{id}_{\mathbb{B}_\crys(\mathcal{E})} \otimes N_{\OB_{\st,R_0}} \bigr)$;
			\item $\bigl(\OB_{\st,R_0}\otimes_{R_0[1/p]} \sM_0,\ \nabla_{\OB_{\st,R_0}}\otimes \mathrm{id}_{\sM_0} + \mathrm{id}_{\OB_{\st,R_0}}\otimes\nabla_{\sM_0},\ \varphi_{\OB_{\st,R_0}}\otimes \varphi_{\sM_0},\ N_{\OB_{\st,R_0}}\otimes \mathrm{id}_{\sM_0} + \mathrm{id}_{\OB_{\st,R_0}}\otimes N_{\sM_0} \bigr).$
		\end{itemize}
		In particular, by taking the horizontal sections for the kernel of the monodromy operators, we get
		\[
		\mathbb{B}_\crys(\mathcal{E}) \simeq (\OB_{\st,R_0}\otimes_{R_0[1/p]} \sM_0)^{\nabla=0,N=0}.
		\]
	\end{enumerate}
\end{proposition}
\begin{proof}
	When the $F$-isocrystal comes from $X_{k,\crys}$ (i.e., when $N_{\sM_0} = 0$), Part \ref{prop:OB vb from F-iso crys} was shown in the proof of \cite[Prop.~2.36.(i)]{GR22}, replacing the ring $R$ in loc. cit. by our $R_0$ (which is in particular smooth over $W$).
	Part \ref{prop:OB vb from F-iso st} follows by taking the base change along $\OB_{\crys,R_0}\to \OB_{\st,R_0}$.
	For $F$-isocrystal over $(X_k,M_{X_k})_\lcrys$, except for the monodromy operators, the proof for Part \ref{prop:OB vb from F-iso crys} follows from the equivalent description of the log $F$-isocrystals in \Cref{equiv def log isoc}.
 Alternatively, one can also prove it similarly to \cite[Prop.\ 2.36.(i)]{GR22}, with the only difference being that we regard Equation (6) in loc.\ cit.  as a diagram of log divided power thickenings: for $D(1)$ and $R$ their log structures are defined by the monoid $p^\mathbb{N}$, and for $\mathbb{A}_\crys$ and $\OA_{\crys,R_0}$ the log structures are defined in \Cref{def: horizontal period sheaves} and \Cref{def OAcrys} respectively.
	To see the identification of the two monodromy operators, we know from \Cref{lem:two monodromy on D times OBst} that $\OB_{\st,R_0}\otimes_{R_0[1/p]} \sM_0$ is generated by the kernel of the nilpotent endomorphism $N=N_{\OB_{\st,R_0}}\otimes \mathrm{id}_{\sM_0} + \mathrm{id}_{\OB_{\st,R_0}}\otimes N_{\sM_0}$ on  $\OB_{\st,R_0}\otimes_{R_0[1/p]} \sM_0$.
	In particular, since $\mathbb{B}_\crys(\mathcal{E})$ is sent into $\ker(N)$ under the natural morphism $\mathbb{B}_\crys(\mathcal{E})\to \OB_{\st,R_0}\otimes_{R_0[1/p]} \sM_0$, the restriction of $N$ on the $\OB_{\st,R_0}$-submodule generated by $\mathbb{B}_\crys(\mathcal{E})$ is exactly $\mathrm{id}_{\mathbb{B}_\crys(\mathcal{E})}\otimes N_{\OB_{\st,R_0}}$.
\end{proof}

As a consequence of \Cref{prop:OB vb from F-iso}, there is a natural commutative diagram of categories as below (the categories $\Vect^\varphi(\mathbb{B}_\crys)$ and $\Vect^{\varphi,\nabla_{R_0},N}(\OB_{\st,R_0})$ are defined in the obvious way):
\begin{equation}
	\label{eq:diag_of_vb_over_OBst}
	\begin{tikzcd}
		\Vect^\varphi(\mathbb{B}_\crys) \arrow[rr, "-\otimes_{\mathbb{B}_\crys} \OB_{\st,R_0}"] && \Vect^{\varphi,\nabla_{R_0},N}(\OB_{\st,R_0}) \\
		\Isoc^\varphi((X_k, (0^\mathbb{N})^a)_\lcrys) \arrow[u, "\mathbb{B}_\crys(-)"] \arrow[rr, "\sim", "\textrm{\Cref{equiv def log isoc}}"'] && \Vect^{\varphi, \nabla_{R_0},N}(R_0[1/p]), \arrow[u, "-\otimes_{R_0[1/p]} \OB_{\st,R_0}"']
	\end{tikzcd}
\end{equation}
where the left column does not depend on the choice of the unramified model $R_0$, and the rest of the diagram depends on the choice of $R_0$ together with the isomorphism $R_0[1/p]\otimes_{K_0} K\simeq R$.

In the following, we let $\Vect^\varphi(\mathbb{B}_\crys)^\circ$ and $\Vect^{\varphi,\nabla_{R_0},N}(\OB_{\st,R_0})^\circ$) be the full subcategories of $\Vect^\varphi(\mathbb{B}_\crys)$ and $\Vect^{\varphi,\nabla_{R_0},N}(\OB_{\st,R_0})$ defined by the essential image of the objects in the category $\Isoc^\varphi((X_k, (0^\mathbb{N})^a)_\lcrys)$ along the arrows in (\ref{eq:diag_of_vb_over_OBst}).
\begin{proposition}
	\label{prop:Isoc_to_Vect(Bcrys)_equiv}
	Let $X$ be a smooth $p$-adic formal scheme over $\mathcal{O}_K$.
	\begin{enumerate}[label=\upshape{(\roman*)}]
		\item\label{prop:Isoc_to_Vect(Bcrys)_equiv_1} The induced functor $\mathbb{B}_\crys(-)$ 
		\[
		\Isoc^\varphi((X_k, (0^\mathbb{N})^a)_\lcrys) \longrightarrow \Vect^\varphi(\mathbb{B}_\crys)^\circ
		\]
		is an exact equivalence.
		\item\label{prop:Isoc_to_Vect(Bcrys)_equiv_2} Assume $X$ is affine and admits an unramified model $R_0$ together with a Frobenius endomorphism.
		The tensor product functor $\mathbb{M}\mapsto \mathbb{M}\otimes_{\mathbb{B}_\crys} \OB_{\st,R_0}$ induces an exact equivalence of subcategories
		\[
		\Vect^\varphi(\mathbb{B}_\crys)^\circ \xrightarrow{\sim} \Vect^{\varphi,\nabla_{R_0},N}(\OB_{\st,R_0})^\circ.
		\]
	\end{enumerate}
\end{proposition}
\begin{proof}
	For Part \ref{prop:Isoc_to_Vect(Bcrys)_equiv_2}, since the connection of an object in $\Vect^{\varphi,\nabla_{R_0},N}(\OB_{\st,R_0})^\circ$ is solvable, by \Cref{lem:two monodromy on D times OBst}, the tensor product functor $\mathbb{M}\mapsto \mathbb{M}\otimes_{\mathbb{B}_\crys} \OB_{\st,R_0}$ admits an inverse functor defined by taking the horizontal section, namely
	\[
	\mathcal{M} \longmapsto \mathcal{M}^{\nabla=0}.
	\]
	
	For Part \ref{prop:Isoc_to_Vect(Bcrys)_equiv_1}, since the functor is local with respect to the Zariski topology of $X$, to check that it is an equivalence, it suffices to assume $X$ is affine and admits an unramified model with a Frobenius endomorphism.
	Then we notice that by the projection formula of finite projective modules and by \Cref{thm: coh of period sheaves}.\ref{thm: coh of period sheaves st}, for a flat connection $\mathcal{M}_0$ over $R_0[1/p]$, the direct image functor along the map of sites $\nu: X_{\eta, \pe}\to X_{\eta,\et}$ induces a natural isomorphism of flat connections
	\[
	\mathcal{M}_0 \simeq \nu_*(\mathcal{M}_0\otimes_{R_0[1/p]} \OB_{\st,R_0}).
	\]
	This gives an inverse to the right vertical arrow in (\ref{eq:diag_of_vb_over_OBst}), namely the functor
	\[
	\Vect^{\varphi, \nabla_{R_0},N}(R_0[1/p]) \longrightarrow \Vect^{\varphi,\nabla_{R_0},N}(\OB_{\st,R_0})^\circ.
	\]
	As a consequence, the diagram (\ref{eq:diag_of_vb_over_OBst}) can be refined into the following commutative diagram of equivalences
	\begin{equation}
		\label{eq:diag_of_vb_over_OBst_2}
		\begin{tikzcd}
			\Vect^\varphi(\mathbb{B}_\crys)^\circ \arrow[rr, "-\otimes_{\mathbb{B}_\crys} \OB_{\st,R_0}","\sim"'] && \Vect^{\varphi,\nabla_{R_0},N}(\OB_{\st,R_0})^\circ  \arrow[d, "(-)(Y_\eta)", bend left=50]\\
			\Isoc^\varphi((X_k, (0^\mathbb{N})^a)_\lcrys) \arrow[u, "\mathbb{B}_\crys(-)"] \arrow[rr, "\sim", "\textrm{\Cref{equiv def log isoc}}"'] && \Vect^{\varphi, \nabla_{R_0},N}(R_0[1/p]), \arrow[u, "-\otimes_{R_0[1/p]} \OB_{\st,R_0}"]
		\end{tikzcd}
	\end{equation}
	where the two vertical arrows on the right are equivalences.
	So Part \ref{prop:Isoc_to_Vect(Bcrys)_equiv_1} follows.
\end{proof}

Another useful observation is that the maps $\alpha_{\crys,T}$ and $\alpha_{\st,T}$ are automatically isomorphic when the ranks of $F$-isocrystals are maximal, extending the well-known result in the case of a point.
\begin{theorem}
	\label{thm:rank max implies alpha is iso}
	Assume the setup of a smooth affine $p$-adic formal scheme $X=\Spf(R)$ that is small as in \Cref{conv of smooth affine}, and let $T\in \Loc_{\mathbb{Z}_p}(X_\eta)$ be of rank $d$.
	\begin{enumerate}[label=\upshape{(\roman*)}]
		\item If the rank of $D_{\crys,R_0}(T)$ is $d$, then the map $\alpha_{\crys,T}$ is an isomorphism.
		\item If the rank of $D_{\st,R_0}(T)$ is $d$, then the map $\alpha_{\st,T}$ is an isomorphism.
	\end{enumerate}
\end{theorem}
Here we note that if either of the two assumptions above is true, then the $K$-linear extension of $D_{\crys,R_0}(T)$ (resp. $D_{\st,R_0}(T)$) is naturally isomorphic to $D_\dR(T)(X_\eta)$, thanks to the injectivity in \Cref{thm: D functors}.\ref{thm: D functors relation} and the fact that $D_\dR(T)$ has at most rank $d$ over $X_\eta$.

For the convenience of the discussion, we recall a construction of the absolute Galois group for rigid spaces as below.

\begin{construction}
	\label{const:Galois group of rigid space}
	Let $X = \Spf(R)$ be a connected smooth affine $p$-adic formal scheme over $\sO_K$. 
    Let $\overline{R}$ be the integral closure of $R$ in a fixed connected maximal pro-finite-\'etale extension of $R[1/p]$ and let $S = \overline{R}^\wedge$ be its $p$-completion. 
    Let $\tilde{X}$ be the formal scheme $\Spf(S)$.
	By assumption, the generic fiber $\tilde{X}_\eta=\Spa(S[1/p],S)$ is an affinoid perfectoid space and is a pro-Galois cover of $X_\eta$, and we denote by $G_{X_\eta} = \Gal(\overline{R}/R)$ the Galois group,\footnote{In \cite{Bri08} (resp. \cite{Shi22}), $G_{X_\eta}$ is denoted by $\mathcal{G}_R$ (resp. $\mathcal{G}_{R_K}$).}
 which is naturally isomorphic to the algebraic fundamental group of $X_\eta$ (cf. \Cref{const: pi_1}).
 Moreover, we let $S_0$ be the $p$-completion of the connected maximal pro-finite-\'etale extension of $R$ inside of $S$.
	Then $S_0\to R$ is a $p$-complete pro-Galois cover whose Galois group, which we denote as $G_{X_k}$.
    We let $I_X$ be the kernel of the surjection $G_{X_\eta}\to G_{X_k}$ and call it the inertial subgroup.
\end{construction}

\begin{proof}[Proof of \Cref{thm:rank max implies alpha is iso}]
	
	Following \Cref{const:Galois group of rigid space}, since the local system $T$ and the period sheaves $\OB_{\crys,R_0}$ and $\OB_{\st,R_0}$ are all trivializable when restricted onto the pro-finite-\'etale cover $\tilde{X}_\eta$, we reduce the question to showing that the natural injection of free continuous $G_{X_\eta}$-representations of the same rank over $\OB_{\crys,R_0}(\tilde{X}_\eta)$ below is an isomorphism
	\begin{equation}
 \label{eq:evaluation of D at universal cover}
		D_{\crys,R_0}(T)\otimes_{R_0[1/p]} \OB_{\crys,R_0}(\tilde{X}_\eta) \longrightarrow T(\tilde{X}_\eta)\otimes_{\mathbb{Z}_p} \OB_{\crys,R_0}(\tilde{X}_\eta);
	\end{equation}
	Similarly for $D_{\st,R_0}(T)$.
	
	We let $e_1,\ldots, e_d$ be a basis of $T(\tilde{X}_\eta)[1/p]$ over $\mathbb{Q}_p$, and let $f_1,\ldots, f_d$ be a basis of $D_{\crys,R_0}(T)$ over $R_0[1/p]$.
	Then there is a $d\times d$ matrix $M$ in $\OB_{\crys,R_0}(\tilde{X}_\eta)$ (resp. $\OB_{\st,R_0}(\tilde{X}_\eta)$) such that $(f_i)=(e_i)M$.
	We let $e=\wedge_{i=1}^d e_i$, let $f=\wedge_{i=1}^d f_i$, and let $m=\det M$, then we have the equality $f=me$, and it suffices to show that $m$ is an invertible element in $\OB_{\crys,R_0}(\tilde{X}_\eta)$ (resp. $\OB_{\st,R_0}(\tilde{X}_\eta)$); equivalently, the top wedge product of \Cref{eq:evaluation of D at universal cover} is an isomorphism of rank one representations of $G_{X_\eta}$.
	Here we note that by construction, since the group $G_{X_\eta}$ acts trivially on $f$, we have 
	\[
	g(m)g(e)=g(me)=g(f)=f=me,~\forall g\in G_{X_\eta}.
	\]
	In particular, since the group $G_{X_\eta}$ preserves $\wedge^d T(\tilde{X}_\eta)[1/p]=\mathbb{Q}_p\cdot e$, the equality above implies that the $\mathbb{Q}_p$-vector space $\mathbb{Q}_p\cdot m$ is a $G_{X_\eta}$-sub-representation of $\OB_{\crys,R_0}(\tilde{X}_\eta)$ (resp. $\OB_{\st,R_0}(\tilde{X}_\eta)$).

	To proceed, we notice that since the element $t \in \mathrm{B}_\crys(\mathcal{O}_C)$ is invertible in both $\OB_{\crys,R_0}(\tilde{X}_\eta)$ and $\OB_{\st,R_0}(\tilde{X}_\eta)$, for some $r \in \IN$, $b\colonequals mt^r$ is contained in the $\Fil^0$ part of $\OB_\dR(\tilde{X}_\eta)$ but not in the $\Fil^1$ part. 
 We then want to prove the following. 
\begin{lemma}
The element $b$ is contained in $S_0[1/p]$. 
\end{lemma}

\begin{proof}
By assumption, there is a $p$-complete \'etale morphism from $X$ onto $\Spf(\mathcal{O}_K\langle x_1^{\pm1},\ldots,x_n^{\pm1} \rangle)$.
    We start by taking the reduction $\bar{b}$ of the element $b$ in $$\gr^0\OB_\dR(\tilde{X}_\eta)=\mathcal{O}\mathbb{C}(\tilde{X}_\eta)\simeq S[1/p][\frac{v_1}{\overline{\xi}},\ldots,\frac{v_n}{\overline{\xi}}]$$ (\cite[Cor.\ 6.15]{Sch13}, cf. \Cref{prop: BdR and OBdR}.\ref{prop: BdR and OBdR explicit}).
 Then the representation $\mathbb{Q}_p\cdot b$ of $G_{X_\eta}$ becomes trivial after tensoring with $\sO \IC(\tilde{X}_\eta)$ (i.e., after reducing the $\Fil^1$-part).
	So by \cite[proof of Thm.\ 5.15]{Shi18}, the $\mathbb{Q}_p$-local system that corresponds to $\mathbb{Q}_p\cdot b$ is a Hodge--Tate representation of zero Hodge--Tate weight, and the inertial subgroup $I_X$ acts on $\mathbb{Q}_p\cdot b$ via a finite quotient. 
	In particular, there is a connected finite \'etale extension $S_1[1/p]$ of the maximal unramified extension $S_0[1/p]$ within $\mathcal{O}\mathbb{C}(\tilde{X}_\eta)$ such that $\bar{b}$ is contained in $S_1[1/p]$.
	On the other hand, by Hensel's lemma, the finite \'etale extension $S_1[1/p]$ is naturally a subring of $\OB^+_\dR(\tilde{X}_\eta)$.

 Now we let $b'$ be the element in $S_1[1/p]$ that maps onto $\bar{b}$ along the surjection $\Fil^0\OB_\dR(\tilde{X}_\eta)\to \mathcal{O}\mathbb{C}(\tilde{X}_\eta)$.
	Now we claim that the element $b$ is in fact equal to $b'\in S_1[1/p]$.
	This is because otherwise $I_X$ acts on $\mathbb{Q}_p\cdot (b-b')$ through finite quotient as well. 
	If $b - b' \neq 0$, then for $i \ge 1$, $b - b'$ lies in $\Fil^i \sO \IB_\dR(\wt{X}_\eta)$ but not $\Fil^{i + 1} \sO \IB_\dR(\wt{X}_\eta)$, and hence has nonzero image in $\sO\IC(\wt{X}_\eta)(i)$. However, by relative Sen theory \cite{Shi18}, $I_X$ cannot act on a nonzero element of $\sO\IC(\wt{X}_\eta)(i)$ through a finite quotient, which yields a contradiction. Indeed, by \cite[Lem.\ 3.3.1]{Bri08}, there is a $G_{X_\eta}$-equivariant embedding $\sO\IC(\wt{X}_\eta)(i) \hookrightarrow \prod_{\mathfrak{p}\in T_{X_\eta}} \sO\IC(\Spa(\overline{L}^\wedge_p))(i)$, where $T_{X_\eta}$ is the set of prime ideals of $S$ that are above the ideal $\pi R\subset R$, and $\Spa(\overline{L}^\wedge_p)$ is a complete algebraic extension of the Shilov point $\Spa(L)$ of $X_\eta$.
				If there is an element $c\in \sO\IC(\wt{X}_\eta)(i)$ such that $I_{X_\eta}$ acts through the finite quotient, for each $\mathfrak{p}\in T_{X_\eta}$, the initial subgroup of $\Gal_L$ would act on the $\mathfrak{p}$-th factor $c_\mathfrak{p}\in \sO\IC(\Spa(\overline{L}^\wedge_p))(i)$ through a finite quotient as well.
				However, the latter was impossible for $i\neq 0$, by Sen's theory on Hodge--Tate representations of weight $i$, which in this generality was proved in \cite[Prop.\ 6]{Bri03}.

    The preceeding paragraph implies that the element $b$ in $\OB_\dR(\tilde{X}_\eta)$ is contained in the subring $S_1[1/p]$.
	Moreover, thanks to \Cref{thm: coh of period sheaves}.\ref{thm: coh of period sheaves crys} and \Cref{thm: coh of period sheaves}.\ref{thm: coh of period sheaves st}, we know the intersection of $\OB_{\crys,R_0}(\tilde{X}_\eta)$ (res. $\OB_{\st,R_0}$) and $S_1[1/p]$ within the ring $\OB_\dR(\tilde{X}_\eta)$ is the maximal unramified subring $S_0[1/p]$.
	Hence the element $b$ is contained in $S_0[1/p] \subset \OB_{\crys,R_0}(\tilde{X}_\eta)$.
\end{proof}
	
	By the preceeding lemma, the determinant $m$ of the transition matrix $M$ is an element in $S_0[1/p]\cdot t^r$.
	So we have the following commutative diagram of $G_{X_\eta}$-equivariant injections
	\[
	\begin{tikzcd}
		\wedge^d D_{\crys,R_0}(T)\otimes_{R_0[1/p]} S_0[1/p] \cdot t^r \arrow[rr,"\iota"] \ar[d] && \wedge^d T(\tilde{X}) \otimes_{\mathbb{Z}_p} S_0[1/p]\cdot t^r \ar[d]\\
		\wedge^d D_{\crys,R_0}(T)\otimes_{R_0[1/p]} \OB_{\crys,R_0}(\tilde{X}_\eta) \arrow[rr, "\alpha_{\crys,\wedge^d T}(\tilde{X}_\eta)"]&& T(\tilde{X}) \otimes_{\mathbb{Z}_p} \OB_{\crys,R_0}(\tilde{X}_\eta),
	\end{tikzcd} 
    \]
    where the second line is the base change of the first line along $S_0[1/p]\to \OB_{\crys,R_0}(\tilde{X}_\eta)$ and one has a completely similar diagram for $D_{\st,R_0}$.
    
    Now we let $\mathcal{O}_{K_{\infty}}$ be the $p$-complete extension of $\mathcal{O}_K$ by adding a compatible system of $p^{1/n}$-roots of $\pi$, and let $S_2$ be the $p$-completion of the ring $S_1\otimes_{\mathcal{O}_K} \mathcal{O}_{K_\infty} [x_1^{\pm1/p^\infty}, \ldots, x_n^{\pm1/p^\infty}]$.
    Then there is a natural injection of rings $S_0\to S_2$ such that it splits as a $S_0$-module.
    So to show that $\iota$ is an isomorphism, by the split injectivity of $$S_0[1/p]\to S_2[1/p][\frac{v_1}{\bar{\xi}},\ldots, \frac{v_d}{\bar{\xi}}]=\mathcal{O}\mathbb{C}(\Spf(S_1)_\eta), $$ it suffices to show the Galois equivariant tensor product $\iota\otimes_{S_0[1/p]} \mathcal{O}\mathbb{C}(\Spf(S_1)_\eta)$ is an isomorphism.
    This then follows from Liu--Zhu's p-adic Simpson correspondence \cite[Thm.\ 2.1.(ii)]{LZ17}, since the underlying module of the rank one flat connection $\wedge^d D_{\crys,R_0}(T)\otimes_{R_0[1/p]} R[1/p]=D_{\crys,R_0}(\wedge^d T)\simeq D_\dR(\wedge^d T)(X_\eta)$ coincides with the underlying module of the rank one Higgs field $\mathcal{H}(\wedge^d T)(X_\eta)$ by \cite[Cor.\ 3.12.(ii), Thm.\ 3.8.(ii)]{LZ17}.
    One treats $D_{\st,R_0}(\wedge^d T)$ in a completely analogous way.
\end{proof}
\begin{remark}
\Cref{thm:rank max implies alpha is iso} should not be surprising: 
  the classical case for $X=\Spf(\mathcal{O}_K)$ was known long time ago and follows from an algebraic fact that the period rings $\mathrm{B}_\crys(\mathcal{O}_C)$ and $\mathrm{B}_\st(\mathcal{O}_C)$ are $(K_0, \Gal(K))$-regular (cf.\ \cite[Def.\ 5.1.1, Prop.\ 9.1.6]{BC09}), and we learned the general strategy from loc.~cit.
  Actually, by considering the extra data of connection, one can try to use a modified argument of loc.~cit. to prove the higher-dimensional analogues, and it would follow if the category of continuous flat connections over $S_0[1/p]$ is abelian, which we do not know due to the non-noetherianity of $S_0[1/p]$.

  In fact, the careful reader might notice that in many past treatments of the relative $p$-adic Hodge theory before Liu--Zhu (for example \cite{Bri08}, \cite{AI12}, and \cite{TT19}), it was not clear if the map $\alpha_{\crys,T}: D_{\crys,R_0}(T)\otimes \OB_{\crys,R_0} \to T\otimes \OB_{\crys,R_0}$ is automatically an isomorphism once the source and target have the same rank. 
\end{remark}
Using the above constructions, we are able to give an equivalent definition of crystalline and semi-stable local systems over an affine base with good reduction.
Precisely, we have the following statement.
\begin{theorem}
	\label{thm:equivalent def of crys and st}
    Let $X = \Spf(R)$ be as in \cref{conv of smooth affine}. 
	\begin{enumerate}[label=\upshape{(\roman*)}]
		\item\label{thm:equivalent def of crys} A $\mathbb{Z}_p$-local system $T$ of rank $d$ over $X_\eta$ is crystalline with respect to $X$ if and only if $D_{\crys,R_0}(T)$ is of rank $d$ over $R_0[1/p]$. 
        Moreover, if $T$ is crystalline, then $T$ is naturally associated with $D_{\crys,R_0}(T)$ via the map $\alpha_{\crys,T}^{\nabla=0}$ and the equivalence in \Cref{prop:OB vb from F-iso}.
		\item\label{thm:equivalent def of st} A $\mathbb{Z}_p$-local system $T$ of rank $d$ over $X_\eta$ is semi-stable with respect to $X$ if and only if $D_{\st,R_0}(T)$ is of rank $d$ over $R_0[1/p]$.
        Moreover, if $T$ is semi-stable, then $T$ is naturally associated with $D_{\st,R_0}(T)$ via the map $\alpha_{\st,T}^{\nabla=0, N=0}$.
	\end{enumerate}
\end{theorem}
\begin{proof}
 Assuming that $K$ is unramified, it was proved in \cite[Prop.\ 3,21]{TT19} that the local system $T$ is crystalline if and only if $D_{\crys,R_0}(T)$ is of rank $d$ and the map $\alpha_\crys$ is an isomorphism. 
 Despite the assumption of loc.\ cit. being stronger, the same strategy can be carried to our setting, as we shall explain.
 
	We first assume $T$ is crystalline and is associated with an $F$-isocrystal $(\mathcal{E},\varphi_\mathcal{E})$ via the map $\vartheta:\mathbb{B}_\crys(\mathcal{E})\simeq T\otimes_{\mathbb{Z}_p} \mathbb{B}_\crys$.
	By taking the base change along $\mathbb{B}_\crys\to \OB_{\crys,R_0}$ and 
 then the direct image along $\nu:X_{\eta,\pe}\to X_{\eta,\et}$, the map $\vartheta$ induces an isomorphism 
	\[
	D_{\crys,R_0}(T)\simeq \bigl( \nu_*(\mathbb{B}_\crys(\mathcal{E})\otimes_{\mathbb{B}_\crys} \OB_{\crys,R_0})\bigr)(X_\eta),
	\]
	where by \Cref{prop:OB vb from F-iso}.\ref{prop:OB vb from F-iso crys}, the latter together with its Frobenius structure and connection is naturally isomorphic to
	\[
	\bigl( \nu_*(\sM_0 \otimes_{R_0[1/p]} \OB_{\crys,R_0}) \bigr)(X_\eta),
	\]
	where $\sM_0$ is the associated $R_0[1/p]$-vector bundle corresponding to $(\mathcal{E},\varphi_\mathcal{E})$ under \Cref{equiv def isoc}.
	Moreover, by the calculation in \Cref{thm: coh of period sheaves}.\ref{thm: coh of period sheaves crys} and the local freeness of $\sM_0$ over $R_0[1/p]$, we have $\bigl( \nu_*(\sM_0 \otimes_{R_0[1/p]} \OB_{\crys,R_0}) \bigr) (X_\eta)\simeq \sM_0$.
	This shows that the $F$-isocrystal $D_{\crys,R_0}(T)$ is naturally isomorphic to $\sM_0$.
	Furthermore, by \Cref{prop:OB vb from F-iso}.\ref{prop:OB vb from F-iso crys} again, the induced map $\alpha_{\crys,T}: \sM_0 \otimes_{R_0[1/p]} \OB_{\crys,R_0} \to T\otimes_{\mathbb{Z}_p} \OB_{\crys,R_0}$ is naturally identified with $\vartheta\otimes_{\mathbb{B}_\crys} \OB_{\crys,R_0}\colon \mathbb{B}_\crys(\mathcal{E})\otimes_{\mathbb{B}_\crys}\OB_{\crys,R_0} \to T\otimes_{\mathbb{Z}_p} \OB_{\crys,R_0}$, which by assumption is an isomorphism.
	
	Conversely, assume $D_{\crys,R_0}(T)$ is of the same rank as that of $T$. 
	By \Cref{thm:rank max implies alpha is iso}, the map $\alpha_{\crys,T}: \sM_0 \otimes_{R_0[1/p]} \OB_{\crys,R_0} \to T\otimes_{\mathbb{Z}_p} \OB_{\crys,R_0}$ is an isomorphism.
	We let $(\mathcal{E},\varphi_\mathcal{E})$ be the $F$-isocrystal over $X_{k,\crys}$ associated with $D_{\crys,R_0}(T)$ under \Cref{equiv def isoc}.
	Then by \Cref{prop:OB vb from F-iso}, we may replace $\sM_0 \otimes_{R_0[1/p]} \OB_{\crys,R_0}$ by $\mathbb{B}_\crys(\mathcal{E})\otimes_{\mathbb{B}_\crys} \OB_{\crys,R_0}$ to get a Frobenius-equivariant isomorphism of $\OB_{\crys,R_0}$-vector bundles
	\[
	\mathbb{B}_\crys(\mathcal{E})\otimes_{\mathbb{B}_\crys} \OB_{\crys,R_0} \simeq  T\otimes_{\mathbb{Z}_p} \OB_{\crys,R_0}.
	\]
	In particular, by taking the horizontal sections and \Cref{prop: Acrys and OAcrys}.\ref{prop: Acrys and OAcrys conn} we get a Frobenius-equivariant isomorphism of $\mathbb{B}_\crys$-vector bundles
	\[
	\mathbb{B}_\crys(\mathcal{E})=(\mathbb{B}_\crys(\mathcal{E})\otimes_{\mathbb{B}_\crys} \OB_{\crys,R_0})^{\nabla=0}\simeq (T\otimes_{\mathbb{Z}_p} \OB_{\crys,R_0})^{\nabla=0} =T\otimes_{\mathbb{Z}_p} \mathbb{B}_\crys.
	\]
	Hence the $p$-adic local system $T$ is associated with $(\mathcal{E},\varphi_\mathcal{E})$, in the sense of \Cref{def crys/st Faltings}.
	
	To check \ref{thm:equivalent def of st}, the above argument goes without change, except that we use the input from \Cref{prop:OB vb from F-iso}.\ref{prop:OB vb from F-iso st} instead.	
\end{proof}

As a quick consequence of the local formulae for the associated $F$-isocrystals above, we see the association of Faltings in \Cref{def crys/st Faltings} can be made functorial.
\begin{corollary}\label{functoriality of being associated}
	Assume the setup of a smooth affine $p$-adic formal scheme $X=\Spf(R)$ as in \Cref{conv of smooth affine}.
	For a $\mathbb{Z}_p$-local system $T$ over $X_\eta$, we let $\mathcal{E}_{\crys,R_0,T}$ (resp. $\mathcal{E}_{\st,R_0,T}$) be the $F$-isocrystal over $X_{k,\crys}$ (resp. $(X_k,M_{X_k})_\lcrys$) corresponding to $D_{\crys,R_0}(T)$ via \Cref{equiv def isoc} (resp. $D_{\st,R_0}(T)$ via \Cref{equiv def log isoc}).
	\begin{enumerate}[label=\upshape{(\roman*)}]
		\item There are natural Frobenius equivariant injections of $\mathbb{B}_\crys$-vector bundles as below that are functorial in $T\in \Loc_{\mathbb{Z}_p}(X_\eta)$
		\[
		\mathbb{B}_\crys(\mathcal{E}_{\crys,R_0,T}) \hookrightarrow \mathbb{B}_\crys(\mathcal{E}_{\st,R_0,T}) \hookrightarrow T\otimes_{\mathbb{Z}_p} \mathbb{B}_\crys.
		\]
		\item The second arrow, which we denote as $\vartheta_{\st,R_0,T}$, is an isomorphism when $T\in \Loc_{\mathbb{Z}_p}^\st(X_\eta)$.
		\item The composition, which we denote as $\vartheta_{\crys,R_0,T}$, is an isomorphism when $T\in \Loc_{\mathbb{Z}_p}^\crys(X_\eta)$.
	\end{enumerate}
\end{corollary}

\begin{remark}
	\label{rmk:filtration on D functors}
	As in Tan-Tong \cite{TT19} and Andreatta--Iovita \cite{AI12}, we can enhance $D_{\crys,R_0}(T)$ and $D_{\st,R_0}(T)$ by an additional data of filtration (after base changing to $K$).
	One can then check that the isomorphisms $\alpha_{\crys,T}$ and $\alpha_{\st,T}$ in \Cref{thm:equivalent def of crys and st} are in fact filtered isomorphisms after base changing to $K$.
	We refer the reader to \Cref{rmk: def crys/st minimal version} and \cite[\S 2]{GR22} for discussions.
\end{remark}
\begin{remark}
	As a consequence of \Cref{thm:equivalent def of crys and st}, we get an explicit local candidate for the associated $F$-isocrystal of a crystalline or semi-stable local system.
	If one is willing to consider the additional filtration as in \Cref{rmk:filtration on D functors}, it can be shown as in \cite[Cor.\ 2.35]{GR22} that the associated filtered $F$-isocrystal of a given crystalline or semi-stable local system is unique up to unique isomorphism.
\end{remark}

\section{Global Crystalline Riemann--Hilbert functors}
\label{sec pullback}
In this section, we analyze the behavior of the crystallinity, the semi-stability, and the functors $D_{\crys}$ and $D_\st$ for arbitrary $p$-adic local systems with respect to the pullback along a map of $p$-adic formal schemes.
As consequences, we introduce crystalline analogues of the $p$-adic Riemann--Hilbert correspondence on $\mathbb{Z}_p$-local systems over rigid spaces with good reduction, and give a quick proof of the purity result for crystalline and semi-stable local systems.

\subsection{Pullback compatibility for general $p$-adic formal schemes}
\label{sub general pullback}
The first subsection is devoted to showing that the crystallinity or semi-stability of a local system in the sense of Faltings is preserved under pullback along a map of $p$-adic formal schemes, without any regularity assumption.

We start with a preparation on the pullback of an $F$-isocrystal versus its associated $\mathbb{B}_\crys$-vector bundle defined in \Cref{const: Bcrys(E)}. 
\begin{lemma}\label{lem: pullback formula on crystalline side}
	Let $f:X'\to X$ be a map of $p$-adic formal schemes equipped with standard log structures (\Cref{log of framed semi-stable reduction}), and let $\mathcal{E}$ be an isocrystal over either $X_{p=0,\crys}$ or $(X_{p=0},M_{X_{p=0}})_\lcrys$.
	Then there is a natural isomorphism of $\mathbb{B}_{\crys,X'}$-vector bundles over $X'_{\eta,\pe}$:
 \begin{equation}
     \label{eqn: lem: pullback formula on crystalline side}
     f_{\eta,\mathbb{B}_\crys}^*\bigl( \mathbb{B}_{\crys,X}(\mathcal{E}) \bigr) \simeq  \mathbb{B}_{\crys,X'}(f_{p=0}^*\mathcal{E}).
 \end{equation}
	The isomorphism is functorial in $\mathcal{E}$, and is Frobenius equivariant when $\mathcal{E}$ has a Frobenius structure.
\end{lemma}
For later applications, here we note that the $F$-isocrystal $f_{p=0}^*\mathcal{E}$ is naturally identified with $f_k^*(\mathcal{E})$ under the equivalence in \Cref{Dwork's trick}.
In particular, the statement applies to both the setting of the mod $p$ special fibers and that of the reduced special fibers as well.
\begin{proof}
	We first notice that the map of mod $p$ reduction $X'_{p=0}\to X_{p=0}$ naturally induces a functor of sites $X'_{p=0,\crys} \to X_{p=0, \crys}$, sending a $p$-complete divided power thickening $(A,A/I,g:\Spec(A/I)\to X'_{p=0})$ over $X'_{p=0}$ onto the same divided power thickening with the structure morphism composed by $f_{p=0}$, namely
	\[
	(A,A/I,f_{p=0}\circ g:\Spec(A/I)\to X_{p=0}) \in X_{p=0, \crys},
	\]
	and is compatible with the log structures when $(A,A/I,g)$ is equipped with one.
	On the other hand, the pullback isocrystal $f_{p=0}^*\mathcal{E}$ is defined by sending a $p$-complete divided power thickening $(A,A/I,g)\in X'_{p=0,\crys}$ onto the $A$-module $\mathcal{E}(A,A/I, f_{p=0}\circ g)$.
	So for any perfectoid algebra $(S'[1/p],S')$ over $X'_\eta$, since the divided power thickening $(\rAcrys(S'), S'/p)$ is naturally an object over  $X'_{p=0,\crys}$, we see by \Cref{const: Bcrys(E)} that 
	\[
	\mathbb{B}^+_{\crys,X'_\eta}(f_{p=0}^*\mathcal{E}) (S'[1/p],S') = \mathcal{E}(\rAcrys(S'), S'/p),
	\]
	where we regard $(\rAcrys(S'), S'/p)$ as a divided power thickening over $X_{p=0}$ by composition with the map $f_{p=0}$ as above.
	Moreover, if there is another perfectoid algebra $(S[1/p],S)$ over $X_\eta$ together with a morphism $(S[1/p],S)\to (S'[1/p],S')$ that is compatible with $f$, then the above formula implies a natural base change isomorphism
	\[
	\mathbb{B}^+_{\crys,X_\eta}(\mathcal{E})(S[1/p],S)\otimes_{\mathrm{A}_{\crys}(S)} \mathrm{A}_\crys(S') \xrightarrow{\sim} \mathbb{B}^+_{\crys,X'_\eta}(f_{p=0}^*\mathcal{E})(S'[1/p],S'),
	\]
	which by construction is compatible with the Frobenius structure when $\mathcal{E}$ underlies an $F$-isocrystal, and holds without change for ($F$-)isocrystals over the log crystalline sites.
	Hence by considering the above constructions over all affinoid perfectoid objects over $X_\eta$ and $X'_\eta$, and inverting the elements $p$ and $\mu$ locally, we get a natural pullback formula
	\[
	f_{\eta, \mathbb{B}_\crys}^* \mathbb{B}_{\crys,X_\eta}(\mathcal{E}) \xrightarrow{\sim} \mathbb{B}_{\crys,X'_\eta}(f_{p=0}^*\mathcal{E}),
	\]
	which is functorial with respect to $\mathcal{E}$ and is compatible with Frobenius structures.
\end{proof}

Similarly, we have the following compatibility of \'etale pullbacks.
\begin{lemma}\label{lem: pullback formula on etale side}
	Let $f:X'\to X$ be a map of $p$-adic formal schemes, and let $T$ be a $\mathbb{Z}_p$-local system over $X_{\eta,\pe}$.
	Then there is a natural isomorphism of $\mathbb{B}_{\crys,X'_\eta}$-vector bundles over $X'_{\eta,\pe}$:
	\begin{equation}
    \label{eqn: pullback formula on etale side}
	    f_{\eta,\mathbb{B}_{\crys}}^* (T\otimes_{\mathbb{Z}_p} \mathbb{B}_{\crys,X_\eta}) \simeq f_\eta^{-1} T\otimes_{\mathbb{Z}_p} \mathbb{B}_{\crys,X'_\eta}.
	\end{equation}
\end{lemma}
\begin{proof}
	By definition, the sheaf $T\otimes_{\mathbb{Z}_p} \mathbb{B}_{\crys,X_\eta}$ is defined as the pro-\'etale sheaf associated to the presheaf which sends an affinoid perfectoid object $U=\Spa(S[1/p],S)$ in $X_{\eta,\pe}$ to
	\[
	T(U)\otimes_{\mathbb{Z}_p(U)} \mathbb{B}_\crys(S[1/p],S).
	\]
	Its $\mathbb{B}_\crys$-linear pullback, when restricted onto the object $U$, is the sheaf associated to the presheaf which sends an affinoid perfectoid object $V=\Spa(S'[1/p],S')$ in $X'_{\eta, \pe}|_U$ to 
	\[
	\bigl(T(U)\otimes_{\mathbb{Z}_p(U)} \mathbb{B}_\crys(S[1/p],S) \bigr) \otimes_{\mathbb{B}_\crys(S[1/p],S)} \mathbb{B}_\crys(S'[1/p],S'),
	\]
	which is naturally isomorphic to
	\[
	\bigl( T(U)\otimes_{\mathbb{Z}_p(U)} \mathbb{Z}_p(V)  \bigr) \otimes_{\mathbb{Z}_p(V)} \mathbb{B}_\crys(S'[1/p],S').
	\]
    Hence (\ref{eqn: pullback formula on etale side}) holds because it holds when evaluated on affinoid perfectoid objects.
\end{proof}

The main statement this subsection is the following.
\begin{theorem}\label{thm: pullback formula Faltings}
	Let $f:X'\to X$ be a map of topologically finite type $p$-adic formal schemes over $\mathcal{O}_K$.
	Assume $T$ is a $\mathbb{Z}_p$-local system over $X_\eta$ that is crystalline  (resp. semi-stable) with respect to $X$, and is associated to the $F$-isocrystal $(\mathcal{E}, \varphi_\mathcal{E})$ over $X_{k,\crys}$ (resp. over $(X_k,M_{X_k})_\lcrys$) via an isomorphism of vector bundles over $\mathbb{B}_{\crys,X_\eta}$
	\[
	\vartheta:\mathbb{B}_{\crys,X_\eta}(\mathcal{E}) \xrightarrow{\sim} T\otimes_{\mathbb{Z}_p} \mathbb{B}_{\crys,X_\eta}.
	\]
	Then the \'etale preimage $f_\eta^{-1} T=T|_{X'_\eta}$ is  crystalline  (resp. semi-stable) with respect to $X'$, and is associated to the $F$-isocrystal $f_k^*(\mathcal{E}, \varphi_\mathcal{E})$ via the isomorphism of vector bundles over $\mathbb{B}_{\crys,X'_\eta}$ obtained by the following composition
	\[
	 \mathbb{B}_{\crys,X'_\eta}(f_k^*\mathcal{E}) \stackrel{(\ref{eqn: lem: pullback formula on crystalline side})}{\simeq} f_{\eta,\mathbb{B}_\crys}^*\mathbb{B}_{\crys,X_\eta}(\mathcal{E}) \stackrel{f_{\eta,\mathbb{B}_\crys}^*\vartheta}{\simeq} f_{\eta,\mathbb{B}_\crys}^*(T\otimes_{\mathbb{Z}_p} \mathbb{B}_{\crys,X_\eta}) \stackrel{(\ref{eqn: pullback formula on etale side})}{\simeq} f_\eta^{-1} T \otimes_{\mathbb{Z}_p} \mathbb{B}_{\crys,X_\eta}.
	\]
\end{theorem}
\begin{proof}
	We take the $\mathbb{B}_{\crys}$-linear pullback of the isomorphism $\vartheta$ along the map of ringed sites $(X'_{\eta, \pe}, \mathbb{B}_{\crys,X'}) \to (X_{\eta, \pe},\mathbb{B}_{\crys,X})$, to get
	\[
	f_{\eta,\mathbb{B}_{\crys}}^*\mathbb{B}_{\crys,X}(\mathcal{E}) \simeq f_{\eta,\mathbb{B}_{\crys}}^* (T\otimes_{\mathbb{Z}_p} \mathbb{B}_{\crys,X}).
	\]
	The rest follows from the isomorphisms in \Cref{lem: pullback formula on crystalline side} and \Cref{lem: pullback formula on etale side}.
\end{proof}

As a quick consequence, we see the restriction at closed (i.e., classical) points always preserves the crystallinity and the semi-stability. 
\begin{corollary}
	\label{cor:pull back to closed points}
	Let $X$ be a topologically finite type $p$-adic formal scheme over $\mathcal{O}_K$, and let $T$ be a $\mathbb{Z}_p$-local system over $X_\eta$ that is crystalline (resp. semi-stable) with respect to $X$.

	Then for any closed point $x\in X_\eta$, the restriction $T|_x$ is a crystalline (resp. semi-stable) representation of $\Gal_{K(x)}$.
\end{corollary}
Here for a finite field extension $K'/K$, we implicitly identify a $\mathbb{Z}_p$-local system over a point $\Spa(K',\mathcal{O}_{K'})$ with a continuous $\mathbb{Z}_p$-representation of the Galois group $\Gal_{K'}$.
\begin{proof}
	By replacing $X$ by an affine open subscheme, the closed immersion of rigid space $x\to X_\eta$ can be obtained from a map of $p$-adic formal schemes $\tilde{x}:\Spf(\mathcal{O}_{K(x)}) \to X$.
	The rest follows from \Cref{thm: pullback formula Faltings}.
\end{proof}

\subsection{Pullback injectivity of $D_\crys$ and $D_\st$}
\label{sub pullback of D functors}
Next we study the behavior of the pullbacks of the $D_\crys$ and the $D_\st$ functors for general $p$-adic local systems, along a map of smooth affine $p$-adic formal schemes with a compatible choice of Frobenius structures.
We will show that the functors $D_\crys$ and $D_\st$, when viewed as taking values in $F$-isocrystals over the special fiber, are compatible with the pullback functors in an appropriate sense (\Cref{thm:pullback for D functors}).

\begin{theorem}
	\label{thm:pullback for D functors}
	We assume the following setup:
	\begin{enumerate}[label=\upshape{(\alph*)}]
		\item Let $K'/K$ be a finite extension with perfect residue fields $k'/k$.
		\item\label{item: pullback for D functors assump b} Let $X_0 = \Spf(R_0)$ and $X_0' = \Spf(R_0')$ be as in \cref{conv of smooth affine} and $f_0:X'_0 \to X_0$ be a Frobenius equivariant map over $\sO_{K_0}$. 
		\item\label{item: pullback for D functors assump c} Let $X'=X'_0\times_{\mathcal{O}_{K_0}} \mathcal{O}_{K'}$, let $X=X_0\times_{\mathcal{O}_{K_0}} \mathcal{O}_{K}$, and let $f:X'=\Spf(R')\to X=\Spf(R)$ be the natural extension of the map $f_0$.
	\end{enumerate}
Then we have the following. 
\begin{enumerate}[label=\upshape{(\roman*)}]
	\item\label{thm:pullback for D functors crys} There is a natural Frobenius equivariant injection of flat connections over $R'_0[1/p]$ that is functorial with respect to $T\in \Loc_{\mathbb{Z}_p}(X_\eta)$
	\[
	\iota_{\crys,f}\colon (f_0)_\eta^*D_{\crys,R_0}(T) \longrightarrow D_{\crys,R'_0}(f_\eta^{-1}T).
	\]
	\item\label{thm:pullback for D functors st} There is a natural Frobenius and monodromy equivariant injection of flat connections over $R'_0[1/p]$ that is functorial with respect to $T\in \Loc_{\mathbb{Z}_p}(X_\eta)$
	\[
	\iota_{\st,f}\colon (f_0)_\eta^*D_{\st,R_0}(T) \longrightarrow D_{\st,R'_0}(f_\eta^{-1}T).
	\]
 \item\label{thm:pullback for D functors fet}
			Assume the map $f$ (thus the map $f_0$) is finite \'etale.
			Then the above injections are isomorphisms.
\end{enumerate}
Moreover, the injections above are compatible with the composition, in the sense that if there is a further map $g:X''\to X'$ arising from a Frobenius equivariant map of smooth $p$-adic formal schemes $g'_0:\Spf(R''_0)\to X'_0$, then we have an equality of maps
\[
\iota_{\crys,g\circ f}=\iota_{\crys,g}\circ (g_0)_\eta^* \iota_{\crys,f},
\] and similarly for the semi-stable analogues.
\end{theorem}
Below we give the proof of the theorem for $D_\crys$; the analogues for $D_\st$ can be proved by simply replacing $\OB_{\crys}$ and $D_\crys$ by $\OB_\st$ and $D_\st$ respectively in the following arguments.
By the coherence of $D_\dR(T)$ and the affineness of the rigid spaces involved, we also temporarily abuse the notation $D_\dR(T)$ (similarly for $D_\dR(f_\eta^{-1}(T))$) with its global section $D_\dR(T)(X_\eta)$.

We start with the case for an arbitrary map $f_0$.
\begin{proof}[Proof of \Cref{thm:pullback for D functors}.\ref{thm:pullback for D functors crys}]	
	We first assume $K'=K$ for the moment, and thus $X'=X'_0\times_{\Spf(\mathcal{O}_{K_0})} \Spf(\mathcal{O}_K)$.
	By assumption, there is a natural commutative diagram of ringed sites that are equivariant with respect to Frobenius structures and connections:
	\[
	\begin{tikzcd}
		(X'_{\eta,\pe}, \OB_{\crys,R'_0}) \arrow[d, "{\nu'}"] \arrow[r, "f_{\eta,\pe}"] & (X_{\eta,\pe}, \OB_{\crys,R_0}) \arrow[d, "{\nu}"]\\
		(X'_{\eta,\et}, R'_0) \arrow[r, "f_0"] & (X_{\eta,\et}, R_0).
	\end{tikzcd}
    \]
    The diagram then induces a natural transformation 
    \[
    (f_0)_\eta^*\nu_*(\OB_{\crys,R_0}\otimes_{\mathbb{Z}_p} T) \longrightarrow \nu'_* f_{\eta,\pe}^*(\OB_{\crys,R_0}\otimes_{\mathbb{Z}_p} T).
    \]
    Similar to the proof of \Cref{lem: pullback formula on etale side}, the target sheaf $\nu'_* f_{\eta,\pe}^*(\OB_{\crys,R_0}\otimes_{\mathbb{Z}_p} T)$ is canonically isomorphic to $\nu'_* (\OB_{\crys,R'_0}\otimes_{\mathbb{Z}_p} f_\eta^{-1}T)$.
    So by the definition of the functor $D_\crys$, the above map can be identified with a natural Frobenius-equivariant map of flat connections over $R'_0[1/p]$ 
    \[
    \iota_\crys : (f_0)_\eta^*D_{\crys,R_0}(T) \longrightarrow D_{\crys,R'_0}(f_\eta^{-1}T),
    \]
    which is functorial with respect to $T\in \Loc_{\mathbb{Z}_p}(X_\eta)$.
    Analogously, we have a natural map of flat connections over $X'_\eta$
    \[
    \iota_\dR :f_\eta^*D_{\dR,X_\eta}(T) \longrightarrow D_{\dR,X'_\eta}(f_\eta^{-1}T),
    \]
    which fits into a naturally commutative diagram with $\iota\otimes_{K_0} K$, namely
    \begin{equation}
    	\label{eq:Dcrys and DdR}
    	    \begin{tikzcd}
    		\bigl((f_0)_\eta^*D_{\crys,R_0}(T)\bigr)\otimes_{K_0} K \arrow[r, "\iota_\crys \otimes_{K_0} K"]  \arrow[d] & D_{\crys,R'_0}(f_\eta^{-1}T)\otimes_{K_0} K \ar[d]\\
    		f_\eta^*D_{\dR,X_\eta}(T) \arrow[r, "\iota_\dR"] & D_{\dR,X'_\eta}(f_\eta^{-1}T).
    	\end{tikzcd}
    \end{equation}
    Here we note that by \Cref{thm: D functors}.\ref{thm: D functors relation} the right vertical map is an injection, and by \Cref{thm: LZ}.\ref{thm: LZ pullback} the bottom map is an isomorphism.
    
    To show the injectivity of the map $\iota_\crys$, it suffices to show the injectivity of the base change $\iota_\crys \otimes_{K_0} K$.
    Notice that since we assumed $K=K'$, we have $X'=X'_0\times_{\Spf(\mathcal{O}_{K_0})} \Spf(\mathcal{O}_K)$.
    So the flat connection $\bigl((f_0)_\eta^*D_{\crys,R_0}(T)\bigr)\otimes_{K_0} K$ is naturally isomorphic to $f_\eta^*(D_{\crys,R_0}(T)\otimes_{K_0} K)$, and the map $\iota_\crys \otimes_{K_0} K$ is identified with the following map of flat connections over $R[1/p]$
    \[
    f_\eta^*(D_{\crys,R_0}(T)\otimes_{K_0} K) \longrightarrow D_{\crys,R'_0}(f_\eta^{-1}T) \otimes_{K_0} K.
    \] 
    So to show the injectivity of $\iota\otimes_{K_0} K$, by the commutative diagram (\ref{eq:Dcrys and DdR}) above, it suffices to show the injectivity of the following map of flat connections 
    \[
    f_\eta^*(D_{\crys,R_0}(T)\otimes_{K_0} K) \longrightarrow f_\eta^*D_{\dR,X_\eta}(T).
    \]
    This then follows from \Cref{thm: D functors}.\ref{thm: D functors relation}, where we showed that $ D_{\crys,R_0}(T)\otimes_{K_0} K\to D_{\dR,X_\eta}(T)$ is an injection of flat connections over $X_\eta$.

    Next, to extend the proof of \ref{thm:pullback for D functors crys} to the case of a general base extension $K'/K$, we can now reduce the original question to the setting when $K'/K$ is totally ramified, $R_0=R_0'$, and $R'=R\otimes_{\mathcal{O}_K} \mathcal{O}_{K'}$. In particular, $f_\eta:X'_\eta=X_\eta\times_{\Spa(K)} \Spa(K')\to X_\eta$ is nothing but a finite \'etale cover given by extension of scalars.
    In this case, the question amounts to showing the injectivity of the map
\begin{equation}
	\label{eq:Dcrys pullback K'/K}
	    \bigl( \nu_*(\OB_{\crys,R_0}\otimes_{\mathbb{Z}_p} T) \bigr)(X_\eta) \longrightarrow \bigl( \nu'_*(\OB_{\crys,R_0}\otimes_{\mathbb{Z}_p} f_\eta^{-1}T) \bigr)(X'_\eta).
\end{equation}
    We let $\mathcal{F}$ be the \'etale sheaf $\nu_*(\OB_{\crys,R_0}\otimes_{\mathbb{Z}_p} T)$ on $X_{\eta,\et}$.
    Then the injectivity of the map in Equation (\ref{eq:Dcrys pullback K'/K}) is equivalent to the injectivity of 
    \[
    \mathcal{F}(X_\eta) \longrightarrow (f_\eta^{-1}\mathcal{F})(X'_\eta),
    \]
    where the latter is nothing but $\mathcal{F}(X'_\eta)$.
    Hence the claim follows from the assumption that $X'_\eta\to X_\eta$ is an \'etale cover.

    Finally, we check that the injections $\iota_{\crys,f}$ are compatible with the composition $X''\xrightarrow{g} X'\xrightarrow{f} X$, namely $\iota_{\crys,g\circ f}= \iota_{\crys,g}\circ (g_0)_\eta^* \iota_{\crys,f}$. As the natural map $D_{\crys,R_0}(T) \to D_{\dR,X_\eta}(T)$ is injective, the equality $\iota_{\crys,g\circ f}= \iota_{\crys,g}\circ (g_0)_\eta^* \iota_{\crys,f}$ follows from the commutativity of the diagram (\ref{eq:Dcrys and DdR}) and the equality $\iota_{\dR ,g\circ f}= \iota_{\dR ,g}\circ (g_0)_\eta^* \iota_{\dR,f}$. 
\end{proof}

\begin{proof}[Proof of \Cref{thm:pullback for D functors}.\ref{thm:pullback for D functors fet}]
		Thanks to the injectivity shown in \ref{thm:pullback for D functors crys}, in order to prove the base change isomorphism, we are free to replace $X_0'$ by a further finite \'etale cover.
		So we assum $f_0:X_0'\to X_0$ is a finite Galois cover with the Galois group $H$.
		Before we discuss the two cases separately, we recall some general construction.
		
		As in \Cref{const:Galois group of rigid space}, we let $S$ be the $p$-completion of the integral closure of $R$ in a fixed connected maximal pro-finite-\'etale extension of $R[1/p]$, and let $\tilde{X}$ be the formal scheme $\Spf(S)$.
		By assumption, the generic fiber $\tilde{X}_\eta=\Spa(S[1/p],S)$ is an affinoid perfectoid space and is a pro-Galois cover of $X_\eta$ with the Galois group $G_{X_\eta}$.
		Then both the $p$-adic local system $T$ and the sheaf $T\otimes_{\mathbb{Z}_p} \OB_{\crys,R_0}$ are trivialized at $\tilde{X}_\eta$, and their sections at $\tilde{X}_\eta$ admit continuous actions of $G_{X_\eta}$.
		So by translating the global section of the pro-\'etale sheaf into the continuous group cohomology with respect to the $G_{X_\eta}$-cover $\tilde{X}_\eta\to X_\eta$ (cf.\ \cite{Sch13}), we have
		\begin{equation}
			\label{eq:Dcrys via G-inv}
			D_{\crys,R_0}(T)=\bigl(T\otimes_{\mathbb{Z}_p} \OB_{\crys,R_0} \bigr)(X_\eta) \simeq \bigl( T(\tilde{X}_\eta)\otimes_{\mathbb{Z}_p} \OB_{\crys,R_0}(\tilde{X}_\eta) \bigr)^{G_{X_\eta}}.
		\end{equation}
		
		In this case, $X'_\eta\to X_\eta$ is a finite Galois cover of rigid spaces with $H\simeq G_{X_\eta}/G_{X'_\eta}$.
		We let $\tilde{X}=\Spf(S)$ be as in \Cref{const:Galois group of rigid space}, then the affinoid perfectoid space $\tilde{X}_\eta$ is covering both $X'_\eta$ and $X_\eta$ under the pro-\'etale topology, with the Galois group being $G_{X'_\eta}$ and $G_{X_\eta}$ respectively.
		Moreover, we have the following formulae for $D_{\crys,R_0}(T)$ and $D_{\crys,R_0'}(f^{-1}T)$:
		\begin{align}
			\label{eq:formulae for Dcrys fet}
			D_{\crys,R_0}(T)  &\simeq \bigl( T(\tilde{X}_\eta)\otimes_{\mathbb{Z}_p} \OB_{\crys,R_0}(\tilde{X}_\eta) \bigr)^{G_{X_\eta}};\\
			D_{\crys,R'_0}(f_\eta^{-1}T)  \simeq &\bigl( T(\tilde{X}_\eta)\otimes_{\mathbb{Z}_p} \OB_{\crys,R'_0}(\tilde{X}_\eta) \bigr)^{G_{X'_\eta}} \simeq \bigl( T(\tilde{X}_\eta)\otimes_{\mathbb{Z}_p} \OB_{\crys,R_0}(\tilde{X}_\eta) \bigr)^{G_{X'_\eta}},
		\end{align}
		where the last isomorphism follows from the equality that $\OA_{\crys,R_0}|_{X'_\eta} \simeq \OA_{\crys,R'_0}$ in \Cref{prop: Acrys and OAcrys}.\ref{prop: Acrys and OAcrys et}.
		We let $M$ be the finite $R'_0[1/p]$-module 
		\[
		\bigl( T(\tilde{X}_\eta)\otimes_{\mathbb{Z}_p} \OB_{\crys,R_0}(\tilde{X}_\eta) \bigr)^{G_{X'_\eta}}.
		\]
		Then by Equation (\ref{eq:formulae for Dcrys fet}) we get
		\[
		D_{\crys,R_0}(T)\simeq M^H; \quad D_{\crys,R'_0}(f_\eta^{-1}T) \simeq M.
		\]
		Hence to show the surjection of $f_{0,\eta}^* D_{\crys,R_0}(T) \longrightarrow D_{\crys,R'_0}(f_\eta^{-1}T)$, it suffices to prove that for a finite $R'_0[1/p]$-module $M$ with a semi-linear action by the finite group $H$, the following natural map is an isomorphism
		\[
		M^H\otimes_{R_0[1/p]} R'_0[1/p] \longrightarrow M.
		\]
		The latter follows from Galois descent, so we are done.
\end{proof}

	\subsection{Weak $F$-isocrystals}
	\label{sub: weak F-isoc}
	
	To prepare for the construction of thecrystalline Riemann--Hilbert functor later, in this subsection, we introduce the \emph{weak $F$-isocrystal}, defined as the sheaf of $F$-isocrystals that satisfies an injectivity condition.
	
	\begin{definition}
		\label{def: sheaf of F-isocrystals}
		Let $Z$ be a scheme of finite type over $k$. 
		A \textit{sheaf} of $(F$-)isocrystals $\mathcal{E}$ on $Z$ is the datum which assigns to each open subscheme $U \subseteq Z$ an object $\mathcal{E}_U \in \Isoc^\varphi(Z_{\crys})$, and assigns to each inclusion of open subscheme $U' \into U$ a restriction morphism $\mathrm{res}_{U', U} : \mathcal{E}_U|_{U'} \to \mathcal{E}_{U'}$, such that the following two conditions hold: 
		\begin{enumerate}[label=\upshape{(\alph*)}]
			\item For a sequence of inclusions $U'' \into U' \into U$, we have $\mathrm{res}_{U'', U'} \circ \mathrm{res}_{U', U} = \mathrm{res}_{U'', U}$. 
			\item If $\{ U_\alpha \}_{\alpha \in I}$ is a Zariski cover of $U$, and $(T,V)$ is a $p$-torsionfree pro-PD-thickening in $Z_{\crys}$ such that $T_{p=0}$ is flat over $Z$, then the following natural sequence 
			\begin{equation}
				\label{eqn: sheaf axioms}
				0 \to \mathcal{E}_U(T,V) \to \prod_{\alpha \in I} \mathcal{E}_{U_\alpha}(T|_{V_\alpha},V_\alpha) \rightrightarrows \prod_{\alpha, \beta \in I} \mathcal{E}_{U_{\alpha, \beta}}(T|_{V_{\alpha, \beta}}, V_{\alpha, \beta})
			\end{equation}
			of $\sO_T$-modules is exact, where $U_{\alpha, \beta} = U_{\alpha} \cap U_\beta$, $V_\alpha = V \cap U_\alpha$ and $V_{\alpha, \beta} = V \cap U_{\alpha, \beta}$. 
		\end{enumerate}
		Denote the category of sheaves of crystalline ($F$-)isocrystals by $\Shv(Z_{\mathrm{Zar}}, \Isoc^{(\varphi)})$, whose morphisms are defined in the obvious way.
		If $M$ is a log structure on $Z$, in a completely analogous way we define the sheaf of log crystalline $F$-isocrystals on $(Z, M)_\lcrys$,  and in the above notations we replace $Z$ by $(Z, M)$.
		Similarly for the category $\Shv(Z_{\mathrm{Zar}}, \Isoc^{\varphi,N})$ of objects that are equipped with nilpotent endomorphisms.
	\end{definition}
	
	\begin{remark}
 		\label{rmk:Sheaf_of_isoc_bases_change}
		Let $Z_\mathrm{Zar}'$ be a basis of the Zariski topology of $Z$; for example, the subcategory of affine open subschemes in $Z$.
		Then the inclusion functor of sites $Z_\mathrm{Zar}'\to Z_\mathrm{Zar}$ naturally defines an equivalence of categories
		\[
		\Shv(Z'_\mathrm{Zar}, \Isoc^\varphi) \simeq \Shv(Z_\mathrm{Zar}, \Isoc^\varphi).
		\]
		So to give a sheaf of $F$-isocrystal, it suffices to produce the data on a slightly smaller subcategory.
	\end{remark}
	\begin{remark}
	The usual notion of the $F$-isocrystal satisfies the Zariski descent.
	In particular, the category $\Isoc^{(\varphi)}(Z)$ is naturally viewed as the full subcategory of $\Shv(Z_{\mathrm{Zar}}, \Isoc^{(\varphi)})$ of \textit{coherent} objects, i.e., those whose restriction morphisms are all isomorphisms. 
	\end{remark}
	
 We also note that the pullback functor naturally extends to the notion of weak isocrystals.
 
 \begin{definition}
		\label{def:pullback_of_wIsoc}
		Let $f : Z' \to Z$ be a morphism between finite type $k$-schemes. 
		Define the pullback functor $f^*\colon \Shv(Z_\mathrm{Zar}, \Isoc^{(\varphi)}) \to \Shv(Z'_\mathrm{Zar}, \Isoc^{(\varphi)})$ by the Zariski sheafification of the crystalline pullback functor for isocrystals.
		Define the logarithmic analogue similarly. 
	\end{definition}

	For the sake of constructing the crystalline Riemann-Hilbert functors, it is helpful to consider a slightly weaker notion: 
	\begin{definition}
		\label{def:weak_F-isoc}
		We say that $\mathcal{E} \in \Shv(Z_{\mathrm{Zar}}, \Isoc^{(\varphi)})$ is a \emph{weak ($F$-)isocrystal} if all restriction morphisms are injective, and we denote the full subcategory of weak ($F$-)isocrystals by $\mathrm{wIsoc}^{(\varphi)}(Z_\crys)$. 
		The log crystalline analogues are defined similarly, and we denote the corresponding category by $\mathrm{wIsoc}^{(\varphi)}((Z,M)_\lcrys)$.
	\end{definition}
	
	Note that $\Isoc^{(\varphi)}(Z_\crys)$ is a full subcategory of $\mathrm{wIsoc}^{(\varphi)}(Z_\crys)$ and there is a natural projection $\mathrm{wIsoc}^{(\varphi)}(Z_\crys) \to \Isoc^{(\varphi)}(Z_\crys)$ given by ``taking global sections'' 
	\[
	\mathcal{E} \mapsto \mathcal{E}_Z.
	\]
	
	To analyze the discrepancy between a weak isocrystal and an actual isocrystal, we introduce the following notions.  
	\begin{definition}
 		\label{def:various_notions_of_wIsoc}
		Let $Z$ be a connected smooth variety over $k$. 
		Let $\mathcal{E}$ be either in $\wIsoc^\varphi(Z_\crys)$ or in $\wIsoc^\varphi((Z, (0^\IN)^a)_\lcrys)$.\footnote{Recall that by \cref{thm: log crys = crys + N} $\sE$ is always locally free.}
		\begin{itemize}
			\item We let $\mathrm{rank}^+ (\mathcal{E}) \colonequals \max_U \mathrm{rank}(\mathcal{E}_U)$ and $\mathrm{rank}^-(\mathcal{E}) \colonequals \min_U \mathrm{rank} ( \mathcal{E}_U)$, where $U$ runs through all affine open subschemes of $Z$.  
			\item We let $U_i$ be the union of all affine open subschemes $V \in Z_\mathrm{Zar}$ such that $\mathrm{rank}(\mathcal{E}_V)\geq i$.
			Then the \emph{rank stratification} of $\mathcal{E}$ is defined as the finite chain of open immersions
			\[
			Z=U_{\mathrm{rank}^-(\mathcal{E})} \supseteq U_{\mathrm{rank}^-(\mathcal{E})+1} \supseteq \cdots \supseteq U_{\mathrm{rank}^+(\mathcal{E})}.
			\]
			\item We call $U_\mathcal{E}\colonequals U_{\mathrm{rank}^+(\mathcal{E})}$ the \emph{pure locus} of $\mathcal{E}$. 
			\item Given a point $z \in Z$, we say that a connected open $U$ containing $z$ is a \textit{stable open neighborhood} if $\mathrm{rank} (\mathcal{E}_U)$ achieves the maximum among all open neighborhoods of $z$. 
		\end{itemize}
	\end{definition}

We give an explicit description of the pullback functor for weak F-isocrystals.

	\begin{remark}
		Let $f : Z' \to Z$ be a morphism between smooth connected varieties over $k$. Given $\mathcal{E} \in \wIsoc^\varphi(Z_\crys)$, the pullback $f^* \mathcal{E} \in \wIsoc^\varphi(Z')$ can be explicitly described as follows: 
		For every point $z' \in Z'$, choose a stable open neighborhood $U_{z'}$ of $f(z')$. 
		Then $f^* \mathcal{E} \in \wIsoc^\varphi(Z'_\crys)$ is the unique Zariski sheaf of isocrystals such that $(f^* \mathcal{E})_{f^{-1}(U)} = g_\crys^*\mathcal{E}_U$, where $g$ is the map of $k$-schemes $f^{-1}(U)\to U$, and $g_\crys^*$ is the usual pullback functor of isocrystals.
		It is not hard to check to that $f^* \mathcal{E}$ lies in $\wIsoc^\varphi(Z)$ and is independent of the choice of the stable open neighborhoods $U_{z'}$. 
	\end{remark}
Notice that for every $\mathcal{E} \in \wIsoc^\varphi(Z_\crys)$, the restriction $i^* \mathcal{E} $ is an object in $\Isoc^\varphi(
	U_\mathcal{E})$ for the open immersion  $i : U_\mathcal{E} \into Z$, where $U_\mathcal{E}$ is the pure locus.  
	In particular, a weak isocrystal $\mathcal{E}$ is an isocrystal if and only if $\mathrm{rank}^+ (\mathcal{E})=\mathrm{rank}^- (\mathcal{E})$.
	Below we provide a simple example that is not an isocrystal.
	\begin{example}
		If $i : U \into Z$ is the inclusion of an open subscheme, then there is a natural functor $i_! : \wIsoc^\varphi(U) \to \wIsoc^\varphi(Z_\crys)$ such that for every $\mathcal{E} \in \wIsoc^\varphi(U_\crys)$, we define $(i_! \mathcal{E})_V = \mathcal{E}_V$ when $V \subseteq U$ and $(i_! \mathcal{E})_V = 0$ otherwise. 
		If $U \neq Z$, then for every nonzero $\mathcal{E} \in \Isoc^\varphi(U_\crys)$, the pure locus $i_! \mathcal{E}$ is equal to $U$.
		In particular, the object $i_! \mathcal{E}$ is always in $\wIsoc^\varphi (Z_\crys)$ but not in $\Isoc^\varphi(Z_\crys)$. 
	\end{example}

\subsection{Global Riemann--Hilbert functor}
\label{sec: global RH}
In this subsection, we show that the $D_\crys$ and the $D_\st$ functors, once translated into $F$-isocrystals over the special fiber, are independent of the choice of the unramified models.
This will in particular allows us to obtain a global version of the Riemann--Hilbert functor, sending $\mathbb{Z}_p$-local systems to weak $F$-isocrystals.

We start with the independence of models for the $D$-functors.
\begin{theorem}
		\label{cor:D is independent of R_0}
		Let $X=\Spf(R)$ be a smooth affine $p$-adic formal scheme over $\mathcal{O}_K$ as in \Cref{conv of smooth affine}, and assume $X$ is small.
		Then the $F$-isocrystals corresponding to $D_{\crys,R_0}(T)$ and $D_{\st,R_0}(T)$ are independent of the choice of the unramified model $R_0$ and its Frobenius structure.
\end{theorem}
\begin{proof}
		Let $R_0$ and $R_0'$ be two unramified models of $X$ that are equipped with Frobenius endomorphisms.
		We let the ring $A$ be the $p$-complete divided power envelope of the surjection
		\[
		(R_0\otimes_W R_0')^\wedge_p \longrightarrow R_k,
		\]
		and let $p_1:R_0\to A$ and $p_2:R'_0\to A$ be the two structure maps.
		Then the ring $A$ is equipped with a continuous differential $\nabla:A\to A\otimes_{(R_0\otimes_W R_0')^\wedge_p} \Omega^1_{(R_0\otimes_W R_0')^\wedge_p/W}$ over $W$ and a Frobenius endomorphism $\varphi_A$, which are compatible with that of $R_0$ and $R_0'$.
		In addition, by the smallness of $X$, the map $p_1$ identifies $A$ with the $p$-complete divided power polynomial ring $R_0\{u_1,\ldots,u_r\}$.
  Here $u_i$ is the image of the element $x_i\otimes 1 - 1\otimes x_i'\in R_0\otimes_W R_0'$, and $x_i$ and $X_i'$ are local coordinates in $R_0$ and $R_0'$ respectively whose images in $R_k$ coincide.
            
		Here we note that the explicit formula above in particular implies that the continuous relative differential over $R_0$, namely the natural map
		\[
		\nabla_{A/R_0}:A\to A\otimes_{(R_0\otimes_W R_0')^\wedge_p} \Omega^1_{(R_0\otimes_W R_0')^\wedge_p/R_0},
		\]
		has the kernel equal to $R_0$.
		The same holds for the map $p_2$ and $\nabla_{A/R_0'}$ as well.
		
		Next, we define the the pro-\'etale sheaf $\OA_{\crys,A}$ over $X_{\eta,\pe}$ to be the $p$-complete divided power envelope for the surjection
		\[
		(\mathbb{A}_{\inf}\otimes_W A)^\wedge_p \longrightarrow \widehat{\mathcal{O}}^+,
		\]
		and let $\OB_{\crys,A}$ be its localization defined by pro-\'etale locally inverting the element $\mu$, where the latter is naturally a sheaf of algebras over $\mathbb{B}_\crys$.
		Analogous to \Cref{def OAcrys}, there are a natural Frobenius endomorphism and an $\mathbb{A}_\crys$-linear connection on $\OA_{\crys,A}$ that are compatible with that of $\OA_{\crys,R_0}$ and $\OA_{\crys,R_0'}$.
		Moreover, the relative differential $\nabla_{A/R_0}$ on $A$ induces a differential morphism on $\OA_{\crys,A}$ that is linear over $\OA_{\crys,R_0}$, for which we use the same notation $\nabla_{A/R_0}$.
		Here by the smallness and the explicit construction of the divided power envelope, the sheaf $\OA_{\crys,A}$ is naturally identified with the complete divided power polynomial rings, namely we have
	\begin{equation}
\label{eq:OBcrys_A_formula}
\OA_{\crys,R_0} \{u_1,\ldots,u_r\} \xrightarrow{\sim} \OA_{\crys,A}.
	\end{equation}
		The latter implies that the kernel of the relative differential $\nabla_{A/R_0}: \OA_{\crys,A} \to \OA_{\crys,A} \otimes_{(R_0\otimes_W R_0')^\wedge_p} \Omega^1_{(R_0\otimes_W R_0')^\wedge_p/R_0}$ is equal to $\OA_{\crys,R_0}$.
		The same applies to $R_0'$ as well, which we shall not repeat.
		We also let $\OA_{\st,A}$ be the tensor product $\OA_{\crys,A}\otimes_{\mathbb{A}_\crys} \mathbb{A}_\st$, and similarly define its localization $\OB_{\st,A}$.
		By \Cref{def OAst} and the discussion analogous to above, one naturally obtains an absolute connection, two relative connections, and a Frobenius structure on $\OA_{\st,A}$ that are compatible with the previous constructions.
		Furthermore, the nil operator on $\Ast$ induces a nil operator on $\OA_{\st,A}$ whose kernel is $\OA_{\crys,A}$.
		
		We then let $T$ be a $\mathbb{Z}_p$-local system over $X_\eta$.
		Similar to what we did for \Cref{thm: D functors}, we define the following Frobenius equivariant connections over $A[1/p]$ 
		\[
		D_{\crys,A}(T)\colonequals (T\otimes_{\mathbb{Z}_p} \OB_{\crys,A})(X_\eta),\quad D_{\st,A}(T)\colonequals (T\otimes_{\mathbb{Z}_p} \OB_{\st,A})(X_\eta),
		\]
		where the second term admits a natural nil endomorphism $N_{D_{\st,A}(T)}$ that is induced from $N_{\OB_{\st,A}}$, and there is a natural injection $D_{\crys,A}(T) \to  D_{\st,A}(T)$ identifies the source as the kernel of $N_{D_{\st,A}(T)}$.
		Moreover, the canonical injection of period sheaves $\OB_{\crys,R_0}\to \OB_{\crys,A}$ induces the equivariant map of abelian groups
		\[
		D_{\crys,R_0}(T) \longrightarrow D_{\crys,A}(T),
		\]
		and similar ones for $R_0'$, and $D_{\st,*}$ for $*\in\{R_0,R_0'\}$.
		  
		Now we claim that the linearizations below are isomorphisms 
		\[
		D_{\crys,R_0}(T)\otimes_{R_0[1/p]} A[1/p] \longrightarrow D_{\crys,A}(T), \quad D_{\st,R_0}(T)\otimes_{R_0[1/p]} A[1/p] \longrightarrow D_{\st,A}(T).
		\]
		For the first map, since $\OB_{\crys,R_0}$ is the union of the subsheaves $\cup_{n \in \IN} \OA_{\crys,R_0}\cdot \mu^{-n}$, and the same is true for $\OB_{\crys,A}$, it suffices to show that the following map is an isomorphism
		\footnote{Here by the finiteness of $T\otimes_{\mathbb{Z}_p} \OA_{\crys,R_0}\cdot \mu^{-n}(X_\eta)$ over $R_0$, we implicitly identify the left hand side below with the non-complete tensor product.}
            \begin{equation}
                \label{eqn: check iso 1}
                \bigl( (T\otimes_{\mathbb{Z}_p} \OA_{\crys,R_0}\cdot\frac{1}{\mu^n})(X_\eta) \otimes_{R_0} A \bigr)^\wedge_p \longrightarrow (T\otimes_{\mathbb{Z}_p} \OA_{\crys,A}\cdot\frac{1}{\mu^n})(X_\eta).
            \end{equation}
		Similarly, for the second map, since $\OA_{\st,*}$ is locally the union of the direct sum $\bigoplus_{i=0}^m\OA_{\crys,*}\cdot \logu^i$ for $m\in \mathbb{N}$, it suffices to show that the following map is an isomorphism
        \begin{equation}
        \label{eqn: check iso 2}
            \bigl( \bigoplus_{i=0} (T\otimes_{\mathbb{Z}_p} \OA_{\crys,R_0}\cdot\frac{1}{\mu^n}\cdot \log(\frac{u}{[\pi^\flat]})^i )(X_\eta) \otimes_{R_0} A \bigr)^\wedge_p \longrightarrow \bigl(  \bigoplus_{i=0} (T\otimes_{\mathbb{Z}_p} \OA_{\crys,A}\cdot\frac{1}{\mu^n}\cdot \log(\frac{u}{[\pi^\flat]})^i) \bigr)(X_\eta)
        \end{equation}
		By the explicit formula of $\OA_{\crys,A}$, the pro-\'etale sheaf is $p$-completely free over $\OA_{\crys,R_0}$; namely there is a free $R_0$-module $M=\bigoplus_I R_0$, where $I$ is some index set, such that 
		\[
		\OA_{\crys,A} \simeq \bigl( \OA_{\crys,R_0} \otimes_{R_0} M \bigr)^\wedge_p.
		\]
		We then make the observation, so that by setting the $\mathcal{F}$ below to be the pro-\'etale sheaf $\bigoplus_{i=0} T\otimes_{\mathbb{Z}_p} \OA_{\crys,R_0}\cdot \mu^{-n} \cdot \logu^i$, one checks that (\ref{eqn: check iso 1}) and (\ref{eqn: check iso 2}) are both isomorphisms. 
\begin{lemma}
			\label{lem:p-complete_direct_sum_and_coh}
			We let $\mathcal{F}$ be a $p$-complete and $p$-torsionfree $R_0$-linear pro-\'etale sheaf over $X_{\eta,\pe}$, and let $M=\oplus_I R_0$ be a free $R_0$-module. 
			There is a natural isomorphism of $p$-complete $R_0$-modules
			\[
			\bigl( \mathcal{F}(X_\eta)\otimes_{R_0} M \bigr)^\wedge_p \simeq \bigl( \mathcal{F}\otimes_{R_0} M \bigr)^\wedge_p (X_\eta).
			\]
\end{lemma}
\begin{proof}
		By translating the two terms into limits, we reduce to establishing the following isomorphism
		\[
		\lim_n \bigl( \mathcal{F}(X_\eta)/p^n\otimes_{R_0/p^n} M/p^n \bigr) \simeq \lim_n \bigl( (\mathcal{F}/p^n\otimes_{R_0/p^n} M/p^n) (X_\eta) \bigr).
		\]
		By the universal coefficient theorem, the above is an injection where the cokernel is the limit 
		\[
		\lim_n \mathrm{H}^1(X_\eta, \mathcal{F}\otimes_{R_0} M)[p^n].
		\]
		Since the cohomology commutes with the filtered colimit, the cokernel is isomorphic to 
		\[
		\lim_n \bigl(\mathrm{H}^1(X_\eta, \mathcal{F})\otimes_{R_0} M \bigr)[p^n].
		\]
		To show the vanishing of the latter, we note that the direct sum $\bigoplus_I R_0$ is contained in the direct product $\prod_I R_0$. Hence it suffices to show the vanishing of the following
		\[
		\lim_n \bigl( \prod_I \mathrm{H}^1(X_\eta, \mathcal{F}) \bigr)[p^n].
		\]
		The last term is further isomorphic to the limit
		\[
		\lim_n \prod_I \bigl(  \mathrm{H}^1(X_\eta, \mathcal{F}) [p^n] \bigr),
		\]
		which, by exchanging the order of the limit functors, is isomorphic to 
		\[
		\prod_I \lim_n \bigl(  \mathrm{H}^1(X_\eta, \mathcal{F}) [p^n] \bigr).
		\]
		Finally, since the pro-\'etale sheaf $\mathcal{F}$ is $p$-complete and $p$-torsionfree, we know $\lim_n \bigl(  \mathrm{H}^1(X_\eta, \mathcal{F}) [p^n] \bigr)$ vanishes, which finishes the proof. \end{proof}
		We let $\mathcal{E}_{\crys,*} \in \Isoc^\varphi(X_{k,\crys})$ for $*\in\{R_0,R_0',A\}$ be the $F$-isocrystals defined by the Frobenius equivariant flat connections $D_{\crys,*}$.
		Then by the Frobenius equivariant isomorphisms of the flat connections above, we obtain the isomorphism of $F$-isocrystals
		\[
		q_{R_0,R_0'}: \mathcal{E}_{\crys,R_0} \xrightarrow{\sim} \mathcal{E}_{\crys,A} \xrightarrow{\sim} \mathcal{E}_{\crys,R_0'},
		\]
		and similarly for $\mathcal{E}_{\st,*}$.
		This shows the independence of the $F$-isocrystals associated to the $p$-adic local system $T$.
		
		Finally, the isomorphism $q_{R_0,R_0'}$ constructed above in fact satisfies the cocycle condition. 
		To see this, let $R_{i}$ for $i=0,1,2$ be three unramified models with Frobenius structures.
		Similarly to the first paragraph above, for any $i\neq j$ in $\{0,1,2\}$, we form the coproduct of the divided power thickenings $A_{i,j}$, and let $A_{0,1,2}$ be the coproduct of the three.
		By the previous discussions, we get isomorphic $F$-isocrystals
		\[
		q_{R_{i},R_{j}} \colon \mathcal{E}_{\crys,R_{i}} \xrightarrow{\sim} \mathcal{E}_{\crys,A_{i,j}} \xrightarrow{\sim} \mathcal{E}_{\crys,R_{j}}.
		\]
		In addition, by the equality of the compositions 
		\[
		\begin{tikzcd}[column sep=small,row sep=small]
		& A_{i,j} \ar[rd] &\\
		R_{j}  \ar[ru] \ar[rd] && A_{0,1,2},\\
		& A_{j,k} \ar[ru] &
		\end{tikzcd}
	    \]
	    the two compositions of isomorphisms among $F$-isocrystals below are naturally identified
	    \[
	    \begin{tikzcd}[column sep=small,row sep=small]
	    	& \mathcal{E}_{\crys,A_{i,j}}  \arrow[rd, "\sim"] &\\
	    	\mathcal{E}_{\crys,R_{j}}   \arrow[ru, "\sim"] \arrow[rd, "\sim"] && \mathcal{E}_{\crys,A_{0,1,2}}.\\
	    	& \mathcal{E}_{\crys,A_{j,k}}  \arrow[ru, "\sim"] &
	    \end{tikzcd}
        \]
        Here in a way completely analogous to the first several paragraphs above, we define $\mathcal{E}_{\crys,A_{0,1,2}}$ to be the $F$-isocrystal associated to the flat connection $D_{\crys,A_{0,1,2}}(T)$ over $A_{0,1,2}[1/p]$.
        As a consequence, by replacing the middle arrows with the isomorphism towards and outwards $\mathcal{E}_{\crys,A_{0,1,2}}$, we see that the composition $q_{R_{i},R_{j}} \circ q_{R_{j},R_{k}}$ is equal to the isomorphism $q_{R_{i},R_{k}}$, which finishes the proof.
\end{proof}

\begin{remark}
	Analogous to \cite[\S 13]{BMS1}, another approach to prove \Cref{cor:D is independent of R_0} is to consider the collection $\mathrm{Frame}_R$ of all finite subsets $\Sigma\subset R^\times$, such that 
	\begin{itemize}
			\item the map $k[(x_u^{\pm1})_{u\in \Sigma}]\xrightarrow{x_u\mapsto u\text{~mod~}\pi} R_k$ is surjective;
			\item there exists a subset $\{u_1,\ldots,u_r\}\subset \Sigma$ such that the map $k[x_{u_1}^{\pm1},\ldots,x_{u_r}^{\pm1}] \to R_k$ is \'etale.
	\end{itemize}
	Here the collection $\mathrm{Frame}_R$ is non-empty after replacing $R$ by some ($p$-complete) Zariski localization, is closed under finite unions, and in particular has no reference to any specific coordinates.
	For each $\Sigma\in \mathrm{Frame}_R$, we can then let $A\coloneq A_\Sigma$ be the $p$-complete $p$-adic divided power envelope for the surjection 
	\[
	W\langle (x_u^{\pm1})_{u\in \Sigma} \rangle \longrightarrow R_k.
	\]
	As in the proof of \Cref{cor:D is independent of R_0}, one can produce the period sheaf $\OA_{\crys,A_\Sigma}$, the flat connection $D_{\crys,A_\Sigma}(T)$, and the $F$-isocrystal $\mathcal{E}_{\crys,A_\Sigma}$, for each $\Sigma\in \mathrm{Frame}_R$.
	Moreover, any unramified model $R_0$ admits a map to some object $A_\Sigma$, identifying the latter as a $p$-complete divided power polynomial ring over $R_0$.
	Hence arguing similarly as in the last paragraphs of \Cref{cor:D is independent of R_0}, we get the equivalences of $F$-isocrystals $\mathcal{E}_{\crys,A_\Sigma}$ for $\Sigma\in \mathrm{Frame}_R$.	
\end{remark}

	We let $X_\mathrm{Zar}^w$ be the category of small affine open subsets for a given smooth $p$-adic formal scheme $X$, and equip it with the $p$-complete Zariski topology.
	Then $X_\mathrm{Zar}^w$ forms a basis of the Zariski site $X_\mathrm{Zar}$, and by \Cref{rmk:Sheaf_of_isoc_bases_change} the inclusion functor of the sites induces an equivalence of the categories
	\[
	\Shv(X_{\mathrm{Zar}}, \Isoc^\varphi) \simeq \Shv(X_\mathrm{Zar}^w, \Isoc^\varphi),
	\]
	and a similar statement holds for the logarithmic analogue, or the one with nilpotent endomorphisms.
	For a $\mathbb{Z}_p$-local system $T$ over $X_\eta$, the construction in \Cref{sub D functors} (cf. \Cref{functoriality of being associated}) and \Cref{cor:D is independent of R_0} allow us to functorially define two presheaves of $F$-isocrystals on $X_\mathrm{Zar}^w$, whose values at an object $U$ are given by
	\begin{align*}
			\mathcal{E}_{\crys,T} \colon  U &\longmapsto \mathcal{E}_{\crys,U,T} \in \Isoc^\varphi(U_{k,\crys});\\
			\mathcal{E}_{\st,T} \colon U &\longmapsto \mathcal{E}_{\st,U,T} \in \Isoc^{\varphi,N}(U_{k,\crys}).
	\end{align*}
    Here $\mathcal{E}_{\crys,U,T}$ canonically embeds into $\mathcal{E}_{\st,U,T}$ as the kernel of the nilpotent endomorphism on the target.
    Our next result shows that in fact the above presheaves are weak $F$-isocrystals.
    
\begin{proposition}
    	\label{thm:E_* are weak F-isoc}
    	Let $X$ be a smooth $p$-adic formal scheme over $\mathcal{O}_K$, and let $T$ be a $\mathbb{Z}_p$-local system over $X_\eta$.
    	Then the presheaves of $F$-isocrystals (resp. with a nilpotent morphism) $\mathcal{E}_{\crys,T}$ (resp. $\mathcal{E}_{\st,T}$) is a weak $F$-isocrystal over $X_{k,\crys}$.
\end{proposition}
\begin{proof}
	In the following we focus on the proof for $\mathcal{E}_{\crys,T}$; the argument showing $\mathcal{E}_{\st,T}\in \Shv(Z_{\mathrm{Zar}}, \Isoc^{\varphi,N})$ is the same and we will not repeat.
	Let $U$ be an object in $X_\mathrm{Zar}^w$, and let $\{U_i\}_{i\in I}$ be a Zariski cover of $U$ in $X_\mathrm{Zar}^w$.
	By the deformation theory, we let $\Spf(R_0)$ be an unramified model of $U$ that is equipped with a Frobenius structure $\varphi_{R_0}$.
	Then by the $p$-complete \'etaleness of the maps $U_i\to U$, the given data $(R_0,\varphi_{R_0})$ uniquely induces the unramified models with Frobenius structures $(R_{0,i}, \varphi_{R_{0,i}})$ for each $U_i$, such that the latter are compatible with $(R_0,\varphi_{R_0})$; similarly for the fiber produces $U_i\times_U U_j$.
	Here we also note that by taking the generic fiber, the collection of rigid spaces $\{(U_i)_\eta\}_{i\in I}$ forms an analytic cover of $U_\eta$.
	
	Now we recall from the construction in \Cref{thm: D functors} that we have  
	\begin{align*}
		D_{\crys,R_0}(T) &\colonequals \bigl( T\otimes_{\mathbb{Z}_p} \OB_{\crys,R_0} \bigr)(U_\eta),\\
		D_{\crys,R_{0,i}}(T) &\colonequals \bigl( T\otimes_{\mathbb{Z}_p} \OB_{\crys,R_{0,i}} \bigr)((U_i)_\eta),\\
		D_{\crys,R_{0,i,j}}(T) &\colonequals \bigl( T\otimes_{\mathbb{Z}_p} \OB_{\crys,R_{0,i,j}} \bigr)((U_{i,j})_\eta).
	\end{align*}
    On the other hand, by the localization formula in \Cref{prop: Acrys and OAcrys}.\ref{prop: Acrys and OAcrys et}, the pro-\'etale sheaf $ T\otimes_{\mathbb{Z}_p} \OB_{\crys,R_{0,i}}$ on $(U_i)_{\eta,\pe}$ is naturally identified with the restriction $ \bigl(T\otimes_{\mathbb{Z}_p} \OB_{\crys,R_0}\bigr)|_{(U_i)_\eta}$, and similarly for each $T\otimes_{\mathbb{Z}_p} \OB_{\crys,R_{0,i,j}}$.
    We let $\mathcal{F}$ be the pro-\'etale sheaf $T\otimes_{\mathbb{Z}_p} \OB_{\crys,R_0}$ over $U_{\eta,\pe}$.
    Then for each open subspace $V\subset U$, by the tautological formula $\mathcal{F}|_V(V)=\mathcal{F}(V)$, we get
    \[
    D_{\crys,R_0}(T)=\mathcal{F}(U_\eta),\quad D_{\crys,R_{0,i}}(T) = \mathcal{F}((U_i)_\eta), \quad D_{\crys,R_{0,i,j}}(T) =\mathcal{F}((U_i\times_U U_j)_\eta).
    \]
    
    As a consequence, since a pro-\'etale sheaf over $U_\eta$ in particular satisfies the descent with respect to the analytic cover $\{(U_i)_\eta \to U_\eta\}_{i\in I}$, we get a short exact sequence of abelian groups that are compatible with Frobenius and the connection structures
    \[
    \mathcal{F}(U_\eta) \rightarrow \prod_{i\in I} \mathcal{F}((U_i)_\eta) \rightrightarrows \prod_{i,j\in I} \mathcal{F}((U_i\times_U U_j)_\eta).
    \]
	Finally, by the weak initial property of the unramified model $(R_0,R_k)$ in the crystalline site $U_{k,\crys}$, we get the sheaf condition for general flat pro-PD-thickenings in $U_{k,\crys}$ (cf. \Cref{def: sheaf of F-isocrystals}).
	This finishes the proof of the sheaf property of $\mathcal{E}_{\crys,T}$.
	
	For the injectivity in \Cref{def:weak_F-isoc}, we let $U'\subseteq U$ be any open immersion of small affine open subspaces in $X$.
	By the fact that $\mathcal{E}_{\crys,U,T}$ and $\mathcal{E}_{\crys,U',T}$ are independent of the unramified models (\Cref{cor:D is independent of R_0}), we may translate the problem into the injectivity of a map between flat connections.
	We let $\Spf(R_0)$ be a Frobenius equivariant unramified model of $U$.
	Similarly to the above, the open subspace $U'$ admits a Frobenius equivariant unramified model $R_0'$ that is compatible with $R_0$.
	So the claim follows from \Cref{thm:pullback for D functors}.\ref{thm:pullback for D functors crys}.
\end{proof}

By taking the pullback along the morphism of the ringed topoi $c_X$ as in \Cref{rmk:site-theoretic_construction_of_Bcrys(E)}, a weak $F$-isocrystal $\mathcal{E}$ over $X_{k,\crys}$ naturally produces a $\mathbb{B}_\crys$-linear sheaf over $X_{\eta,\pe}$
\[
\mathbb{B}_\crys(\mathcal{E})\colonequals c_X^* \mathcal{E},
\]
and naturally extends to objects in $\Isoc^{\varphi,N}(X_{k,\crys})$ or $\Isoc^\varphi((X_k,(0^\mathbb{N})^a)_\lcrys)$ in a compatible way.
Moreover, by \Cref{def:various_notions_of_wIsoc} the sheaf $\mathbb{B}_\crys(\mathcal{E})$ is analytic locally (on $X_\eta$) given by locally free sheaves: namely given a closed point $x\in X$ and a stable open neighbohrood $U_x$ of $x$ with respect to $\mathcal{E}$, the restriction of $\mathbb{B}_\crys(\mathcal{E})$ onto $(U_x)_\eta$ is naturally isomorphic to the $\mathbb{B}_\crys|_{(U_x)_\eta}$-vector bundle $\mathbb{B}_\crys(\mathcal{E}_{U_x})$, defined for the actual $F$-isocrystal $\mathcal{E}_{U_x}$ over $(U_x)_{k,\crys}$ as in \Cref{const: Bcrys(E)}.

Specified to the case when $\mathcal{E}$ is $\mathcal{E}_{\crys,T}$ (resp. $\mathcal{E}_{\st,T}$) for a $\mathbb{Z}_p$-local system $T\in \Loc_{\mathbb{Z}_p}(X_\eta)$, we have the following injectivity result, extending the local case of \Cref{functoriality of being associated}.
\begin{proposition}
	\label{prop:inj_of_Bcrys_vb_in_general}
	Let $X$ be a smooth $p$-adic formal scheme over $\mathcal{O}_K$, and let $T\in \Loc_{\mathbb{Z}_p}(X_\eta)$.
	There are natural injections of $\mathbb{B}_\crys$-linear sheaves on $X_{\eta,\pe}$
	\[
	\mathbb{B}_\crys(\mathcal{E}_{\crys,T}) \hookrightarrow \mathbb{B}_\crys(\mathcal{E}_{\st,T}) \hookrightarrow T\otimes_{\mathbb{Z}_p} \mathbb{B}_\crys,
	\]
	where the first arrow identifies the source as the kernel of the endomorphism $N_{\mathbb{B}_\crys(\mathcal{E}_{\st,T})}$ on the target.
\end{proposition}
\begin{proof}
	We let $x$ be a closed point of $X$, and let $U$ be a small affine open subspace of $X$ that is stable with respect to the both weak $F$-isocrystals $\mathcal{E}_{\crys,T}$ and $\mathcal{E}_{\st,T}$.
	Under the assumption, the restriction of the $\mathbb{B}_\crys$-linear sheaves onto $U_{\eta}$ are all vector bundles.
	Moreover, for a given choice of an unramified model $R_0$ of $U$ and a Frobenius structure $\varphi_{R_0}$, by \Cref{functoriality of being associated} we have the injective maps
	\begin{equation}
		\label{eq:inj_diagram_of_B_crys-sheaves_0}
		\mathbb{B}_\crys(\mathcal{E}_{\crys,U,T}) \xhookrightarrow{u_{R_0}} \mathbb{B}_\crys(\mathcal{E}_{\st,U,T}) \xhookrightarrow{v_{R_0}} T\otimes_{\mathbb{Z}_p} \mathbb{B}_\crys|_{U_\eta}.
	\end{equation}
	We now show that the maps defined in loc.\ cit. do not depend on the choice of the data $(R_0,\varphi_{R_0})$, using the strategy proving \Cref{cor:D is independent of R_0}.
	We will refer the reader for the notations introduced in the proof of loc.\ cit. therefore.
	
	We let $R_0$ and $R_0'$ be any two unramified models of $U$ that are equipped with the Frobenius endomorphisms. 
	We let divided power envelope ring $A=A_{R_0,R_0'}$ and the period sheaf $\OB_{\crys,A}$ be as in the first paragraphs of the proof of \Cref{cor:D is independent of R_0}, so that both $\OB_{\crys,R_0}$ and $\OB_{\crys,R'_0}$ admit injections into $\OB_{\crys,A}$ that are compatible with the additional structures.
	As in loc.\ cit., the structural map $R_0\to A$ induces an isomorphism of flat connections $D_{\crys,R_0}(T)\otimes_{R_0[1/p]} A[1/p]\xrightarrow{\sim} D_{\crys,A}(T)$, and similarly for $D_{\st,*}(T)$.
	So by the construction of $D_{\crys,*}(T)$ and $D_{\st,*}(T)$, there is an equivariant commutative diagram of pro-\'etale sheaves 
	\begin{equation}
		\label{eq:inj_diagram_of_OB_crys-sheaves}
	\begin{tikzcd}
		D_{\crys,R_0}(T)\otimes_{R_0[1/p]} \OB_{\crys,R_0} \ar[r] \ar[d] &D_{\st,R_0}(T)\otimes_{R_0[1/p]} \OB_{\st,R_0} \ar[r] \ar[d] & T\otimes \OB_{\st,R_0} \ar[d]\\
		D_{\crys,A}(T)\otimes_{A[1/p]} \OB_{\crys,A} \ar[r] & D_{\st,A}(T)\otimes_{A[1/p]} \OB_{\st,A} \ar[r] &T\otimes \OB_{\crys,A},
	\end{tikzcd}
	\end{equation}
    such that the vertical maps become isomorphic after tensoring the sources with $\OB_{\crys,A}$.
    As a consequence, by \Cref{prop: Acrys and OAcrys}.\ref{prop: Acrys and OAcrys conn}, the explicit formula in (\ref{eq:OBcrys_A_formula}), together with \Cref{prop:OB vb from F-iso}, the horizontal sections of the diagram (\ref{eq:inj_diagram_of_OB_crys-sheaves}) gives a commutative diagram
    \begin{equation}
    	\label{eq:inj_diagram_of_B_crys-sheaves}
    	\begin{tikzcd}
    	\mathbb{B}_\crys(\mathcal{E}_{\crys,R_0,T}) \arrow[r,"u_{R_0}"] \arrow[d,"\sim"'] &\mathbb{B}_\crys(\mathcal{E}_{\st,R_0,T}) \arrow[r,"v_{R_0}"] \arrow[d,"\sim"'] & T\otimes \mathbb{B}_\crys|_{U_\eta} \arrow[d,equal]\\
    	\mathbb{B}_\crys(\mathcal{E}_{\crys,A,T}) \arrow[r,"u_{A}"] & \mathbb{B}_\crys(\mathcal{E}_{\st,A,T}) \arrow[r,"u_{A}"] & T\otimes \mathbb{B}_\crys|_{U_\eta},\\
    	\mathbb{B}_\crys(\mathcal{E}_{\crys,R_0,T}) \arrow[r,"{u_{R_0'}}"] \arrow[u,"\sim"] &\mathbb{B}_\crys(\mathcal{E}_{\st,R_0,T}) \arrow[r,"{v_{R_0'}}"] \arrow[u,"\sim"] & T\otimes \mathbb{B}_\crys|_{U_\eta} \arrow[u,equal],
    \end{tikzcd}
    \end{equation}
    where each vertical arrows are isomorphisms.
    The arrows in the first and the last rows by \Cref{functoriality of being associated} are all injective, hence ones in the second row.
    As a consequence, for $*\in \{R_0,R_0',A\}$, the images of $v_{*}\circ u_{*}$ (resp. $u_{*}$) in (\ref{eq:inj_diagram_of_B_crys-sheaves}) are equal to each other, hence are independent of the choice of the unramified models.
    This in particular allows us to write the arrows in (\ref{eq:inj_diagram_of_B_crys-sheaves_0}) as $u_U$ and $v_U$ respectively.
	
	Finally, to see that the maps $u_U$ (resp. $v_U$) for different $U$ indeed glue, it suffices to show that for any affine open subspace $V\subset U$, the maps $u_U$ and $u_V$ (resp. $v_U$ and $v_V$) are compatible under the restriction.
	To this end, since the above maps can be constructed using any unramified models, we may fix an unramified model $(R_0,\varphi_{R_0})$ of $U$ and let $(S_0,\varphi_{S_0})$ be the unique unramified model of $V$ that is compatible with the former.
	It is then left to check the compatibility of $u_{R_0}$ and $u_{S_0}$ (resp. the maps $v_*$), which follows from \Cref{thm:pullback for D functors}.
\end{proof}
	
We are now ready to show the pullback compatibility for a map of smooth $p$-adic formal schemes, without any assumption on the unramified models.
\begin{proposition}
	\label{prop:pullback_in_general}
	Let $f:Y\to X$ be a map of smooth $p$-adic formal schemes over $\mathcal{O}_K$,  with $f_s$ the induced map on reduced special fibers, and let $T\in \Loc_{\IZ_p}(X_\eta)$.
	There are natural injections of weak $F$-isocrystals
	\[
	\iota_{\crys, f} : f_s^*\mathcal{E}_{\crys,T} \hookrightarrow \mathcal{E}_{\crys,f_\eta^{-1}T},\quad \iota_{\st, f} : f_s^*\mathcal{E}_{\st,T} \hookrightarrow \mathcal{E}_{\st,f_\eta^{-1}T},
	\]
	that are compatible to each other via the nilpotent operators.
\end{proposition}
\begin{proof}
	We first notice that by taking the pullback of the injections of $\mathbb{B}_\crys$-sheaves in \Cref{prop:inj_of_Bcrys_vb_in_general} along the map $f_\eta:Y_\eta\to X_\eta$, and by the canonical identifications in \Cref{lem: pullback formula on crystalline side} and \Cref{lem: pullback formula on etale side}, we get natural injections of $\mathbb{B}_\crys$-sheaves over $Y_{\eta,\pe}$
	\begin{equation}
		\label{eq:B_crys-pullback_injection_1}
		\mathbb{B}_{\crys,Y}(f_s^* \mathcal{E}_{\crys,T}) \hookrightarrow \mathbb{B}_{\crys,Y}(f_s^* \mathcal{E}_{\st,T}) \hookrightarrow  (f_\eta^{-1}T)\otimes \mathbb{B}_{\crys,Y}.
	\end{equation}
    On the other hand, applying \Cref{prop:inj_of_Bcrys_vb_in_general} at the local system $f_\eta^{-1}T$ over $Y_\eta$, we have the injections
    \begin{equation}
    	\label{eq:B_crys-pullback_injection_2}
    	\mathbb{B}_{\crys,Y}(\mathcal{E}_{\crys,f_\eta^{-1} T}) \hookrightarrow \mathbb{B}_{\crys,Y}(\mathcal{E}_{\st,f_\eta^{-1} T}) \hookrightarrow  (f_\eta^{-1}T)\otimes \mathbb{B}_{\crys,Y}.
    \end{equation}
    So to show the claim, it suffices to show that the subsheaf $\mathbb{B}_{\crys,Y}(f_s^* \mathcal{E}_{*,T})$ of $f_\eta^{-1}T\otimes \mathbb{B}_{*,Y}$ is contained in another subsheaf $ \mathbb{B}_{\crys,Y}(\mathcal{E}_{\crys,f_\eta^{-1} T})$, for $*\in \{\crys,\st\}$.
    
    As the containment can be checked locally, we may assume $Y$ is affine and admits an unramified model $R_0$ with a Frobenius endomorphism.
    In addition, by shrinking $Y$ onto an open neighborhood of a fixed point $y\in Y$ if necessary, we assume that $Y$ is the stable neighborhood of all the weak $F$-isocrystals discussed above, so that each of them are actual $F$-isocrystals over $Y_s$.
    Under the assumption, by the equivalence of categories in \Cref{prop:Isoc_to_Vect(Bcrys)_equiv}, since both $\mathbb{B}_{\crys,Y}(f_s^* \mathcal{E}_{*,T})$ and $ \mathbb{B}_{\crys,Y}(\mathcal{E}_{\crys,f_\eta^{-1} T})$ come from $F$-isocrystals, we may use the commutative diagram (\ref{eq:diag_of_vb_over_OBst_2}) and translate the question into flat connections.    
    
    Now we apply the tensor product of the maps in (\ref{eq:B_crys-pullback_injection_1}) with $\OB_{\st,R_0}$, to get a natural equivariant injections of $\OB_{\st,R_0}$-vector bundles over $Y_{\eta,\pe}$
    \[
    \mathbb{B}_{\crys,Y}(f_s^* \mathcal{E}_{\st,T}) \otimes_{\mathbb{B}_{\crys,Y}} \OB_{\st,R_0} \hookrightarrow (f_\eta^{-1}T)\otimes \OB_{\st,R_0}.
    \]
    By further taking the section of this injection at $Y_\eta$, and by the definition of $D_{\st,R_0}(f_\eta^{-1}T)$ in \Cref{thm: D functors}, we get an injection of flat connections over $R_0[1/p]$ 
    \[
    \bigl( \mathbb{B}_{\crys,Y}(f_s^* \mathcal{E}_{\st,T}) \otimes_{\mathbb{B}_{\crys,Y}} \OB_{\st,R_0} \bigr)(Y_\eta) \hookrightarrow D_{\st,R_0}(f_\eta^{-1}T),
    \]
    where the left hand side is equal to the flat connection $f_s^* \mathcal{E}_{\st,T}(R_0,R_k)$ with additional structures (cf. the commutative diagram (\ref{eq:diag_of_vb_over_OBst_2})).
    This gives the required injection, thanks again to the equivalences of various categories of loc.\ cit..
\end{proof}

Assembling the results in this subsection, we obtain the \emph{crystalline Riemann--Hilbert functors} in general.
\begin{theorem}[Crystalline Riemann--Hilbert functor]
		\label{thm:gluing of D functors}
		Let $X$ be a smooth $p$-adic formal scheme over $\mathcal{O}_K$, and let $T\in \Loc_{\mathbb{Z}_p}(X_\eta)$ be of rank $d$.
		\begin{enumerate}[label=\upshape{(\roman*)}]
			\item\label{thm:gluing of D functors modules} There is a natural weak $F$-isocrystal $\mathcal{E}_{\crys,T}$ (resp. $\mathcal{E}_{\st, T}$) of rank $\leq d$ over $X_{k,\crys}$ (resp. over $(X_k, (0^\IN)^a)_\lcrys$) that is functorial in $T$, such that $\mathcal{E}_{\crys,T}$ is the kernel of the nilpotent endomorphism $N_{\mathcal{E}_{\st,T}}$ on $\mathcal{E}_{\st,T}$, under the equivalence in \Cref{thm: log crys = crys + N}.
			\item\label{thm:gluing of D functors modules local formula}
			For any affine open subscheme $U=\Spf(R)\subset X$ with an unramified model $\Spf(R_0)$ and a Frobenius endomorphism, the evaluation induces isomorphisms of $R_0[1/p]$-vector bundles that are compatible with connections, Frobenius structures, and monodromy operators
			\[
			\mathcal{E}_{\crys,T}(U,U_k)\simeq D_{\crys,R_0}(T),\quad \mathcal{E}_{\st,T}(U,U_k, (0^\IN)^a) \simeq D_{\st,R_0}(T).
			\]
			\item\label{thm:gluing of D functors pullbacks} Assume that $f:Y\to X$ is a map of smooth $p$-adic formal schemes over $\mathcal{O}_K$, with $f_s$ the induced map on reduced special fibers.
			Then we have natural injections of weak $F$-isocrystals 
			\[
			\iota_{\crys, f} : f_s^*\mathcal{E}_{\crys,T} \hookrightarrow \mathcal{E}_{\crys,f_\eta^{-1}T},\quad \iota_{\st, f} : f_s^*\mathcal{E}_{\st,T} \hookrightarrow \mathcal{E}_{\st,f_\eta^{-1}T}.
			\]
   The maps are isomorphisms if $f$ is finite \'etale.
			
   			\item\label{thm:gluing of D functors smooth base change} In the setting of \ref{thm:gluing of D functors pullbacks}, assume $f$ is $p$-completely smooth.
			Then both $\iota_{\crys, f}$ and $\iota_{\st, f}$ are isomorphisms of weak $F$-isocrystals.
   \item\label{thm:gluing of D functors crys/st} The local system $T$ is crystalline (resp. semi-stable) if and only if the weak $F$-isocrystal $\mathcal{E}_{\crys,T}$ (resp. $\mathcal{E}_{\st,T}$) is an $F$-isocrystal of rank $d$. 

   Under the assumptions, the map $\iota_{\crys, f}$ (resp. $\iota_{\st, f}$) in \ref{thm:gluing of D functors pullbacks} is an isomorphism for any morphism $f:Y\to X$ of smooth $p$-adic formal schemes.
		\end{enumerate}
\end{theorem}
\begin{proof}
	Part \ref{thm:gluing of D functors modules} and Part \ref{thm:gluing of D functors modules local formula} follow from \Cref{thm:E_* are weak F-isoc} and its local construction.
	On the other hand, Part \ref{thm:gluing of D functors pullbacks} follows from \ref{prop:pullback_in_general}, and Part \ref{thm:gluing of D functors crys/st} follows from \Cref{thm:equivalent def of crys and st}.
	So it is left to check Part \ref{thm:gluing of D functors smooth base change}.
	
	We first notice that if the map $f$ is $p$-completely \'etale, to check that the injections $\iota_{*, f}$ are isomorphic, by deformation theory, it suffices to assume $X$ is smooth affine as in \Cref{conv of smooth affine} and $f$ is induced from a map of Frobenius equivariant unramified models $f_0:\Spf(R_0')\to \Spf(R_0)$.
	Under the assumption, when the map $f$ is finite \'etale, the claim follows from \Cref{thm:pullback for D functors}.\ref{thm:pullback for D functors fet}.
	So to extend the claim to general $p$-complete \'etale morphisms, we may assume the map $f$ to be a ($p$-complete) open immersion.
	In this case, we let $x$ be a fixed closed point of $X$.
	By shrinking $X$ (and hence $Y$) to a neighborhood of $x$ if necessary, since the statement is Zariski local with respect to $X$, we may assume $X$ and $Y$ are stable open neighborhoods of $x$ for all the weak $F$-isocrystals in the statement (cf. \Cref{def:various_notions_of_wIsoc}).
	Under the assumption, the claim follows since by definition, the stable open neighborhood means that the $F$-isocrystals above satisfy the restriction formula with respect to further open immersions.
	
	For a general $p$-completely smooth morphism $f$, thanks to the cases above, we may reduce to considering the case when the map $f$ admits an Frobenius equivariant unramified model $f_0$, and the latter is given by the projection $X_0 \times \what{\IA}^{n}_W \to X_0$, where $\what{\IA}^{n}_W = \Spf(W\< x_1, \cdots, x_n \>)$. 
	We let $s_0:X_0\to X_0\times \what{\IA}^{n}_W$ be the fiber at the origin. 
	By shrinking around any fixed point of $X$ if necessary, we assume again that $X$ is the stable open neighborhood of that point for all the weak $F$-isocrystals in the statement.
	So all the weak $F$-isocrystals are now actual $F$-isocrystals, and the pullback functors in the statement are the usual pullback functors for $F$-isocrystals.
	Here we also notice that the composition $f_0\circ s_0$ is equal to the identity map $\mathrm{id}_{X_0}$.
	If the injective map of flat connections $\iota_{\crys,f}$ is not an isomorphism, by \cite[Thm.\ 4.1]{Ogu84} so is the pullback of $\iota_{\crys,f}$ along $(s_0)_\eta^*$.
	On the other hand,
	notice that the map of $F$-isocrystals $\iota_{\crys,\mathrm{id}_X}$ is an isomorphism.
	The latter forces the map $\iota_{\crys,f}$ to be surjective, thanks to the fact that the map of isocrystals $\iota_{\crys}$ is compatible with the composition of $p$-adic formal schemes.
	This finishes the proof of $\iota_{\crys,f}$, and the case for $\iota_{\st, f}$ follows similarly.
\end{proof}

\begin{remark}
	Here different from \cite{TT19}, we make no assumption on the ramification degree of $K$.
	In particular, the formal scheme $X$ does not have to be defined over an unramified model.
	Moreover, we emphasize that we make no assumption on the $p$-adic local system $T$.
\end{remark}

\begin{remark}
\label{rmk: base change diagram}
    Consider the setting of \cref{thm:gluing of D functors}.\ref{thm:gluing of D functors pullbacks} (or \Cref{prop:pullback_in_general}). 
    We refer the morphisms $\iota_{*, f}$ ($* = \crys, \st$) as the \emph{base change morphisms} for $T$ along the map $f$. This terminology is inspired by \cite[\href{https://stacks.math.columbia.edu/tag/02N6}{Tag 02N6}]{stacks-project}.  
    Note that thanks to the identification offered by \cref{lem: pullback formula on crystalline side}, $\IB_{\crys, Y}(\iota_{\st, f})$ is naturally viewed as an injection $f_{\eta, \IB_{\crys}}^*(\IB_{\crys, X}(\sE_{\st, T})) \into \IB_{\crys,Y}(\sE_{\st, f^{-1}_\eta(T')})$. 
    Moreover, in the special case when the map $f$ comes from a map of unramified models as in \Cref{thm:pullback for D functors}, we then obtain a commutative diagram: 
    \begin{equation}
    \label{eqn: base change diagram}
    \begin{tikzcd}
	{f_{\eta, \IB_{\crys}}^*(\IB_{\crys, X}(\sE_{\st, T}))} && {f^*_{\eta, \IB_\crys}(T \tensor_{\IZ_p} \IB_{\crys, X})} \\
	{\IB_{\crys,Y}(\sE_{\st, f^{-1}_\eta(T')})} && {f_\eta^{-1}(T) \tensor_{\IZ_p} \IB_{\crys, Y}}
	\arrow["{f_\eta^* \alpha_{\st, T}^{\nabla = 0, N= 0}}", hook, from=1-1, to=1-3]
	\arrow["{\IB_{\crys, Y}(\iota_{\st, f})}"', hook, from=1-1, to=2-1]
	\arrow["{\textrm{Lem.~\ref{lem: pullback formula on etale side}}}", "\simeq"', from=1-3, to=2-3]
	\arrow["{\alpha_{\st, f_{\eta}^{-1}(T)}^{\nabla = 0, N= 0}}", hook, from=2-1, to=2-3]
    \end{tikzcd}
    \end{equation}
    There is a completely analogous diagram for $\iota_{\crys, f}$, which we do not spell out.
\end{remark}

\section{Normalizations and finite covers}
\label{sec nor}
In the proof of the pointwise criteria later, one of the key ingredients is to descend the crystallinity or the semi-stability along suitable finite covers. 
As a preparation, we discuss normalizations of rigid analytic spaces and formal schemes in this section. Then we will introduce the notion of effective sets and prove an $\ell$-adic analogue of the main theorems. 

\subsection{Preliminaries}
\label{sub pril}
We first recall some basic facts from \cite{Conrad99}. 
As before, we let $K$ be a $p$-adic field with perfect residue field $k$ and the uniformizer $\pi$. 
Unless otherwise noted, in this section we reserve the notation $(-)^\wedge$ for completion with respect to $p$-adic topology. 

To begin, recall that a topologically of finite type $K$-algebra $A$ is an excellent ring.
Moreover, if $Z = \Spa(A)$ and $z \in Z$ corresponds to a maximal ideal $\mathfrak{m} \subseteq A$, then the analytic localization $\mathcal{O}_{Z, z}$ is also excellent (\cite[Thm~1.1.3]{Conrad99}) and there is a natural local morphism $A_{(\mathfrak{m})} \to \mathcal{O}_{Z, z}$ of Noetherian rings which becomes an isomorphism after completion. 

We let $\mathsf{FS}_{\mathcal{O}_K}$ be the category of locally Noetherian $p$-adic formal schemes $X$ over $\Spf(\mathcal{O}_K)$ such that the reduced special fiber $X_{s}$ is a scheme locally of finite type over $k$.
Then $X\in \mathsf{FS}_{\mathcal{O}_K}$ if and only if it is a $p$-adic formal scheme that is locally topologically of finite type over $\mathcal{O}_K$, and an affine formal scheme $\Spf(R)$ is an object of $\mathsf{FS}_{\mathcal{O}_K}$ if and only if $R$ is of a quotient of $\mathcal{O}_K \< x_1, \cdots, x_m \>$ for some $m, n \ge 0$ (\cite[\S 7.1]{deJong}). 
As a classical convention (cf. \cite[\S7.4]{BoschLectures}), we call a formal scheme \emph{admissible} if it is in $\mathsf{FS}_{\mathcal{O}_K}$ and is flat over $\mathbb{Z}_p$.
We also let $\mathsf{Rig}_K$ be the category of rigid analytic spaces over $K$. 
As usual, we denote the generic fiber functor $\mathsf{FS}_{\mathcal{O}_K} \to \mathsf{Rig}_K$ by $X \mapsto X_\eta$. 

\begin{proposition}
\label{prop: extend classical points}
    Let $X \in \mathsf{FS}_{\sO_K}$ be an admissible formal scheme over $\Spf(\sO_K)$. Then every classical point $x_\eta \in X_\eta$ extends to a morphism $x : \Spf(\sO_{K(x_\eta)}) \to X$ such that $x_\eta$ is given by the rigid generic fiber of $x$. 
\end{proposition}
\begin{proof}
    Assume that $X = \Spf(A)$ is affine and let $\fm \subseteq A_K$ be the maximal ideal which gives $x_\eta$, i.e., $K(x_\eta) = A_K / \fm$. Let $B$ denote the image of $A$ under the reduction map $A_K \to K(x_\eta) =: K'$. Then by \cite[\S8.3~Prop.~7]{BoschLectures} and its proof, $B$ is a local integral domain of dimension $1$ such that $\sO_K \subseteq B \subseteq \sO_{K'}$. The point $x$ is given by the composite $A \to B \into \sO_{K'}$. 
\end{proof}

\begin{definition}
    \label{def: reduction of points}
    In the setting of \cref{prop: extend classical points}, we call the closed point on $X_s$ given by the image of $x_s$ the \textbf{reduction} of $x_\eta$. 
\end{definition}

Let $X$ be an object of $\mathsf{FS}_{\mathcal{O}_K}$ or of $\mathsf{Rig}_K$. 
Then $X$ is said to have property $P$ at a closed point $x\in X$ for
\[
P\in \{\text{regular, reduced, normal, Gorenstein, Cohen-MacCaulay}\}
\]
if the local ring $\mathcal{O}_{X, x}$ has property P. 
If $X = \Spf(A) \in \mathsf{FS}_{\mathcal{O}_K}$ is affine  (resp.\ $X = \mathrm{Spa}(A) \in \mathsf{Rig}_K$ is affinoid), then $X$ has property $P$ at every closed point if and only if the ring $A$ does (\cite[Lem.~1.2.1]{Conrad99}).
Moreover, by the \textit{normalization} of an affine (resp. affinoid) object $X$, we mean $\Spf(A^\nu)$ (resp.\ $\mathrm{Spa}(A^\nu)$), where $A^\nu$ is the integral closure of the ring $A/\mathrm{Nil}(A)$ in its total ring of fractions, where $\mathrm{Nil}(A)$ is the nilpotent radical of the ring $A$.

We will freely make use of the following fact: 
\begin{proposition}\cite[Cor.~1.2.3, Thm~2.1.3]{Conrad99}
	\label{prop: Conrad} 
	If the map $\Spa(A') \to \Spa(A)$ (resp. $\Spf(A') \to \Spf(A)$) is an open immersion, then the fiber product $\Spa(A' \otimes_A A^\nu)$ (resp.\ $\Spf(A' \otimes_A A^\nu)$) is the normalization of $\Spa(A')$ (resp.\ $\Spf(A)$). 
If $X \in \mathsf{FS}_{\mathcal{O}_K}$ and $X^\nu \to X$ is the normalization map of $X$, then $(X^\nu)_\eta \to X_\eta$ is the normalization of $X_\eta$.
\end{proposition}

Let $\sfC_K$ denote the category of admissible \textit{normal} affine formal schemes $X = \Spf(R)$ over $\mathcal{O}_K$ such that the reduced special fiber $(X_k)_{\mathrm{red}}$ is irreducible and $X_\eta$ is a smooth rigid space over $K$,
and we require the morphism to be a morphism of $p$-adic formal schemes whose mod $p$ reduction is generically finite.
Let $\sfC^{\mathrm{sm}}_K \subseteq \sfC_K$ be the subcategory of those $X$ that are smooth over $\mathcal{O}_K$. 
For each $X \in \sfC_K$, we define the following objects: 
\begin{itemize}
	\item $F \colonequals \,$the fraction field of the reduced special fiber $(X_k)_{\mathrm{red}}$. 
	\item $\fp\colonequals \,$the prime ideal of $R$ corresponding to the generic point of $(X_k)_{\mathrm{red}}$, i.e., $(X_k)_{\mathrm{red}}$ is the integral affine scheme $\Spec(R/\fp)$. 
	Note that if $X \in \sfC^{\mathrm{sm}}_K$, then $\fp(X) = \pi R$. 
	\item $\mathcal{O}_L \colonequals (R_{\fp})^\wedge$, which is a discrete valuation ring with residue field $F$. 
	\item $L \colonequals \,$the fraction field of $\mathcal{O}_L$, i.e., $\mathcal{O}_L[\pi^{-1}]$. 
\end{itemize}
Note that as $R$ is normal, Serre's criterion \cite[031S]{stacks-project} implies that $R_{\fp}$ is a discrete valuation ring; moreover, if $\varpi$ is its uniformizer of $R_{\fp}$, then we have $\pi = \varpi^m$ for some $m \in \mathbb{N}_{\ge 1}$, and $m = 1$ if and only if the mod $\pi$ reduction $X_k$ is generically reduced. 
The notation $R^\wedge_\fp$ can be viewed as either $p$-adic or $\varpi$-adic completion. 
When we need to emphasize the dependency on $X$, we shall write $F(X), \fp(X), \mathcal{O}_L(X), L(X)$ for $F, \fp, \mathcal{O}_L, L$ respectively. 
It is then clear that $\mathcal{O}_L(-)$ (resp. $L(-)$) defines a functor from $\sfC_K$ to the category of complete discrete valuation rings (resp. complete discrete valuation fields).

\begin{convention}
	\label{conv:shrink}
	Given $X = \Spf(R) \in \sfC_K$ and $f \in R \smallsetminus \fp$, we set $R_f$ to be the localization $R[1/f]$ and let $R_\bf$ be the $p$-comopletion $(R_f)^\wedge$. 
	By \textbf{shrinking $X$}, we mean that we replace $X$ by $X_f \colonequals \Spf(R_\bf)$ for some $f\in R\smallsetminus \fp$. 
	If there is any other formal scheme $X'$ that admits a map to $X$ in the context, then we replace $X'$ by the open sub formal scheme $X' \times_X X_f$ accordingly. 
\end{convention} 

\begin{remark}
	We shall often implicitly use the fact that the functors $\sO_L(-)$, and hence $L(-)$, are invariant under shrinking: 
	Given $X = \Spf(R) \in \sfC_K$ and $X_f$ for $f \in R \smallsetminus \fp$, the open immersion $X_f \to X$ induces an identification $\sO_L(X_f) = \sO_L(X)$. 
	To see this, as $\fp R_\bf$ is a prime ideal in $R_\bf$, it suffices to show that the natural morphism $R \to R_\bf$ induces an isomorphism $$R_\fp / \pi^n \to (R_\bf)_{\fp R_\bf} / \pi^n$$ for every $n \in \IN$. 
	The latter holds since the mod $\pi^n$ reduction of $f$ does not belong to the ideal $\fp R/\pi^n$.
\end{remark}

\subsection{A Reduced fiber theorem}
\label{sub reduced fiber} 
We give a variant of the reduced fiber theorem of \cite{BLR95} for formal schemes in terms of normalizations, which is more similar to the classical version for schemes in  \cite[\href{https://stacks.math.columbia.edu/tag/09IL}{Tag 09IL}]{stacks-project}. 
The benefit of doing so is that normalizations have universal properties which allow us to construct morphisms to them. 

We start with an elementary lemma: 
\begin{lemma}
	\label{lem: pc reduced special fiber} 
	Let $R$ be a normal integral domain over $\sO_K$ and set $X \colonequals \Spec(R)$. 
	Suppose that the mod $\pi$ reduction $X_k$ is irreducible and the set 
	$$\{ x \in X_k \mid x = (\wt{x})_k \textit{ for some } \wt{x} \in X(\sO_{K'}) \textit{ where $K'/K$ is a finite unramified extension}  \}$$ 
	is Zariski dense in $X_k$. 
	Then $X_k$ is generically reduced. 
\end{lemma}
\begin{proof}
	It suffices to assume that $k$ is algebraically closed.
	Let $\eta_k \in X$ be the generic point of the reduced special fiber $(X_k)_{\mathrm{red}}$, which corresponds to a prime ideal $\fp \subseteq A$. 
	By the normality assumption, $X$ is regular in codimension $1$, so the local ring $A_{(\fp)}$ is a discrete valuation ring, and we denote its uniformizer by $\varpi$.
	Let $i : \sO_K \into A$ be the structural morphism. 
	Then $i(\pi) = u \varpi^m$ for some $m \in \IN$ and $u \in A_{(\fp)}$. 
	Our goal is to show that $m = 1$. 
	
	Suppose for the sake of contradiction that $m > 1$. 
	Up to replacing $X$ by an affine open neighborhood of $\eta_k$, we may assume that the elements $u, u^{-1}$ and $\varpi$ all lie in $A$. 
	By the density assumption, after shrinking there still exists a morphism $\wt{x} : A \to \sO_K$ such that $\wt{x} \circ i  = \mathrm{id}_{\sO_K}$. 
	Then by assumption, we have $\wt{x}(\varpi) = v \pi^n$ for some $n \ge 0$ and $v \in \sO_K^\times$, and 
	\[
	\pi = \wt{x}(i(\pi)) = \wt{x}(u) \wt{x}(\varpi)^m = \wt{x}(u) v \pi^{mn}.
	\] 
	If $n = 0$, then $\pi = \wt{x}(u) v$ becomes a unit in $\sO_K$, which is a contradiction. 
	But if $n  > 1$, then $mn > 1$ and $1 = \wt{x}(u) v \pi^{mn - 1}$, so that $\pi$ again becomes a unit in $\sO_K$, which is a contradiction. 
	Therefore, we must have $m= 1$ as desired. 
\end{proof}
Now we state the version of the reduced fiber theorem that will be used later.
\begin{theorem}
	\emph{(Reduced Fiber Theorem)}
	\label{thm: reduced fiber theorem}
	Let $X$ be an admissible $p$-adic formal scheme over $\Spf(\sO_K)$, with $X_\eta$ a smooth rigid space and $(X_k)_{\mathrm{red}}$ an irreducible scheme. 
	Then there exists a finite extension $K'$ of $K$ with residue field $k'$, such that the special fiber $((X_{\sO_{K'}})^\nu)_{k'}$ of the normalization $(X_{\sO_{K'}})^\nu$ is generically reduced along every irreducible component that dominates $(X_k)_{\mathrm{red}}$.  
\end{theorem}

\begin{proof} 
	The main theorem of \cite{BLR95} tells us that there exists some finite extension $K'$ and a flat topologically finite type $p$-adic $\sO_{K'}$-formal scheme $Y$ that has the reduced special fiber, together with a finite morphism $Y \to X_{\sO_{K'}}$ inducing an isomorphism on the rigid generic fibers. 
	We fix this $K'$ for the moment. 
	For simplicity, below we use $\wt{X}$ to abbreviate the normalization $(X_{\sO_{K'}})^\nu$. 
	Without loss of generality, we may assume that $K'$ is totally ramified over $\sO_K$ so that the residue field $k'$ of $\sO_{K'}$ is the same as $k$. 
	
	Note that the natural morphism $\wt{X}_\eta \to X_{\eta,K'}$ is an isomorphism as it is a normalization map (\cite[Thm~2.1.3]{Conrad99}) with the target being already normal. 
	In particular, the rigid generic fibers of $X_{\sO_{K'}}, Y$ and $\wt{X}$ are naturally identified. 
	Let $Y^{\mathrm{sm}}$ be the maximal open formal subscheme of $Y$ which is smooth over $\sO_{K'}$. 
	The reducedness of $Y_{k}$ implies that it is generically smooth along every irreducible component, so that $(Y^{\mathrm{sm}})_{k} \colonequals Y^{\mathrm{sm}} \times_{\Spf(\sO_{K'})} \Spec(k)$ is Zariski dense in $Y_{k}$. 
	As $Y^{\mathrm{sm}}$ is normal, its morphism to $X_{\sO_{K'}}$ factors through a morphism $\epsilon : Y^{\mathrm{sm}} \to \wt{X}$.
	In summary, we get a commutative diagram of $p$-adic formal schemes as below
	\[\begin{tikzcd}
	{Y^{\mathrm{sm}}} & {\wt{X}} \\
	Y & {X_{\sO_{K'}}}.
	\arrow["\epsilon", from=1-1, to=1-2]
	\arrow[hook, from=1-1, to=2-1]
	\arrow[from=1-2, to=2-2]
	\arrow[from=2-1, to=2-2]
\end{tikzcd}\]
	
	In the following, we call an irreducible component of $(\wt{X}_k)_{\mathrm{red}}$ \textit{good} if it dominates $(X_k)_{\mathrm{red}}$. 
	We then claim that the special fiber $\epsilon_k : (Y^{\mathrm{sm}})_k \to \wt{X}_{k}$ of the map $\epsilon$ is dominant on each good irreducible component of $\wt{X}_k$. 
	First, we note that the map $Y_k \to X_k$ is surjective on closed points. 
	Recall that for an admissible $p$-adic formal scheme, the specialization map from classical points on the rigid-analytic generic fiber to the closed points on the special fiber is surjective (\cite[\S 8.3~Prop.~8]{BoschLectures}).
	Let $\wt{x} \in X_{\eta,K'}$ be a classical point which specializes to $x\in X_k$. 
	By identifying $\wt{x}$ as a closed point in $Y_{\eta,K'}\simeq X_{\eta,K'}$, the point $\wt{x}$ specializes to a point $y \in Y_{k}$. 
	Since specialization maps are compatible with morphisms between admissible $p$-adic formal schemes, by contemplating the generic and the special fibers of $Y\to X_{\mathcal{O}_{K'}}$, we see the image of $y$ under the map $Y_{k} \to X_{k}$ is equal to $x$. 
	Moreover, as the map $Y_k \to X_k$ is finite and $Y^{\mathrm{sm}}_k \subseteq Y_k$ is Zariski dense, there exists an open dense subscheme $U \subseteq X_k$ such that $U \times_{X_k} Y_k \subseteq Y^{\mathrm{sm}}_k$. 
	Since $\wt{X}_k \to X_k$ is also finite, on each good irreducible component $\wt{X}_0$ of $(\wt{X}_k)_\mathrm{red}$, the image of a general closed point $x_0 \in \wt{X}_0$ along the map $\wt{X}_k\to X_k$ lies in $U$. 
	So to show the dominance, it suffices to find a closed point $y \in Y^{\mathrm{sm}}_k$ such that $\epsilon_k(y) = x_0$, for which we apply the above argument again: 
	Choose a classical point $\wt{x}_0$ on $\wt{X}_{\eta,K'} \simeq Y_{\eta,K'}$ which specializes to $x_0$, and we let $y$ be the specialization of $\wt{x}_0\in Y_{\eta,K'}$ to $Y_k$. 
	Then we claim that $y$ is contained in $Y^\mathrm{sm}_k$ with $\epsilon_k(y)=x_0$.
	To see this, we notice that by the compatibility of specialization maps, as the image of $x_0$ (denoted as $x$) along $\wt{X}_k\to X_k$ is in $U$, the image of $y$ along $Y_k\to X_k$, which is equal to $x$, is contained in $U$ as well.
	In particular, the point $y$ is inside of $U\times_{X_k} Y_k$ and thus in $Y^{\mathrm{sm}}_k$.
	As a consequence, by the choice of $y_0$ and the compatibility of specialization maps again, we see $\epsilon_k(y)=x_0$.
    The relation of various points in the special fiber is summarized in the following diagram:
 \[
 \begin{tikzcd}
 & x_0 \in \wt{X}_0 = \text{reduction of $\wt{x}_0\in \wt{X}_{\eta, K'}$} \arrow[d, mapsto]\\
 y_0 \in Y_k^\mathrm{sm} = \text{reduction of $\wt{x}_0\in \wt{Y}_{\eta, K'}$} \arrow[ur, dashed, mapsto] \arrow[r, mapsto] & x \in U \subset X_k.
 \end{tikzcd}
 \]
	
	By the smoothness, every closed point on $Y^{\mathrm{sm}}_k$ lifts to an $\sO_{K''}$-point on $Y^{\mathrm{sm}}$ for some $K''/K'$ finite and \textit{unramified}. 
	The preceding paragraph then implies that the same can be said about general closed points on a good irreducible component of $(\wt{X}_k)_{\mathrm{red}}$. 
	Hence we can conclude by \cref{lem: pc reduced special fiber}. 
\end{proof}
\begin{remark}
	\label{rmk: of reduced fiber thm}
	Let $K''/K'$ be any finite extension of $p$-adic fields.
	As in the first paragraph of the proof for \Cref{thm: reduced fiber theorem}, the conditions on $Y$ relative to $\sO_{K'}$ remains unchanged if we replace $Y\to X_{\sO_{K'}}$ by their base change along $\sO_{K'}\to \sO_{K''}$.
	In particular, \Cref{thm: reduced fiber theorem} still holds if we replace $K'$ by its any finite extension, including those $K''$ that are Galois over $K$.
\end{remark}

\subsection{Normal covers} 
\label{sub normal curve}
\begin{definition}
	\label{def: normal covers}
	Given $X \in \sfC_K$, we say that a finite morphism $X' \to X$ is a \emph{normal cover} if $X' \in \sfC_K$ and the map of generic fibers $X'_\eta \to X_\eta$ is a finite \'etale cover. 
	Denote the category of normal covers of $X$ by $\mathsf{N}(X)$, where the morphisms between objects are morphisms between $p$-adic formal $X$-schemes. 
\end{definition}

\begin{lemma}
	\label{lem: simple observation}
	Let $X = \Spf(R) \in \sfC_K$ and $X' = \Spf(R') \in \sfN(X)$. 
	Let $L \colonequals L(X)$, $L' \colonequals L(X')$, $\fp \colonequals \fp(X)$ and $\fp' \colonequals \fp(X')$, and let $\sO_{L'}$ be the integral closure of $\sO_L$ in $L'$. 
	Then we have: 
	\begin{enumerate}[label=\upshape{(\roman*)}]
		\item\label{lem: simple observation alg} There are canonical identifications 
		\begin{equation*}
			R' \tensor_R \sO_L = (R' \tensor_R R_{\fp})^\wedge = (R'_{\fp'})^\wedge = \sO_L(X') = \sO_{L'}.
		\end{equation*}
		\item\label{lem: simple observation localization} For every $f \in R \smallsetminus \fp$, the map $X' \to X'|_{X_f} \colonequals X' \times_{X} X_f$ forms an object of $\sfN(X_f)$. In particular, the restriction $(-)|_{X_f}$ defines a natural functor $\sfN(X) \to \sfN(X_f)$. 
		\item\label{lem: simple observation integral closure} $R'$ is the integral closure of $R$ in both $R'_{\fp'}[\pi^{-1}]$ and $R'[\pi^{-1}]$.
            \item\label{rmk: Galois action on integral model} The natural morphism $\Aut(X'/X) \to \Aut(X'_\eta/X_\eta)$ is an isomorphism. 
	\end{enumerate}
\end{lemma}
\begin{proof}
	We start with \ref{lem: simple observation alg}.
	Recall by definition that $\sO_L = (R_\fp)^\wedge$. 
	Consider the natural map $\varepsilon : R' \tensor_R R_\fp \to R' \tensor_R (R_\fp)^\wedge$ obtained by base changing $R_\fp \to R_\fp^\wedge$ along $R \to R'$. Since $(R_\fp)^\wedge$ is a Noetherian local ring and $R' \tensor_R (R_\fp)^\wedge$ is finite over $(R_\fp)^\wedge$, Krull's intersection theorem \cite[00IP]{stacks-project} ensures that $R' \tensor_R (R_\fp)^\wedge$ is still $p$-adically separated. 
	Then we apply \cite[\href{https://stacks.math.columbia.edu/tag/031B}{Tag 031B}]{stacks-project} to deduce that $R' \tensor_R (R_\fp)^\wedge$ is already $p$-complete. 
	Moreover, by taking completion, we obtain the map $\varepsilon^\wedge : (R' \tensor_R R_\fp)^\wedge \to R' \tensor_R (R_\fp)^\wedge $. 
	It is easy to see that for each $n \in \IN$, $\varepsilon^\wedge / \pi^n$ is an isomorphism, and hence so is $\varepsilon^\wedge$. 
	This together with the completeness above gives the first identification in \ref{lem: simple observation alg}. 
	For the second, we note that as $\fp' \subseteq R'$ is the unique prime above $\fp$, the multiplicative set $R' \smallsetminus \fp'$ is the saturation of the multiplicative set $R \smallsetminus \fp$ inside $R'$. 
	Therefore, $R' \tensor_R R_{\fp} = R'_{\fp'}$, whose $p$-completion is $\sO_L(X')$ by definition. 
	Finally, to see that $\sO_L(X') = \sO_{L'}$, recall that the normality of $X'$ implies that $R'_{\fp'}$ (and hence $\sO_L(X')$) is a discrete valuation ring. 
	In particular, $\sO_L(X')$ is integrally closed in its faction field $L'$. 
	Since $\sO_L(X')$ is finite over $\sO_L$, the ring $\sO_L(X')$ is in fact the integral closure of $\sO_L$ in $L'$, i.e., the subring $\sO_{L'}$. 
	
	Part \ref{lem: simple observation localization} follows from \cref{prop: Conrad}. 
	For \ref{lem: simple observation integral closure}, simply note that by \cite[\href{https://stacks.math.columbia.edu/tag/03GH}{Tag 03GH}]{stacks-project}, the integral closure of $R$ in either $R'_{\fp'}[\pi^{-1}]$ or $R'[\pi^{-1}]$ is finite over $R$, and hence finite over $R'$.
	But notice that $R'$ is already integral closed, hence it coincides with the integral closure of $R$. In fact, both \ref{lem: simple observation localization}  and  \ref{lem: simple observation integral closure} eventually reduce to the technical statement that $R$ is an excellent and thus Nagata ring. 
    Finally, \ref{rmk: Galois action on integral model} is a direct consequnece of \ref{lem: simple observation integral closure}. 
\end{proof}

As an application of Gabber--Ramero's approximation result, we have the following spreading-out equivalence on finite extensions of the generic point. 
Below we use $\FEt(*)$ to denote the category of finite \'etale $*$-algebras.
\begin{theorem}
	\label{thm: normal covers}
	Let $X \in \sfC_K$, and let $L \colonequals L(X)$.
	Denote by $\mathsf{DVF}$ the category of fraction fields of discrete valuation rings.
	The functor $L(-) : \sfC_K \to \mathsf{DVF}$ then induces an equivalence of categories 
	$$ L_X : \varinjlim \sfN(X_f) \stackrel{\sim}{\to} \FEt(L)^\circ  $$
	where the colimit is taken over $f \in R \smallsetminus \fp$ and $\FEt(L)^\circ$ denotes the category of the connected objects of $\FEt(L)$, i.e., finite field extensions of $L$.
\end{theorem}
We view $R \smallsetminus \fp$ as a filtered direct system such that $f \le g$ if the open immersion $X_\bg \to X$ factors through $X_\bf$ (e.g., when $f \mid g$). 
We refer the reader to \cite[III.3]{ArtinGT} for the definition of a filtered colimit of categories that is used here. 
\begin{proof}
	We first prove \textit{essential surjectivity}. 
	Let $L \subseteq L'$ be a finite field extension. 
	We need to show that, up to shrinking $X$, there exists a normal cover $X'$ of $X$ such that $L' \cong L(X')$. 
	
	First, one quickly checks that $\sO_L$, which by definition is $(\varinjlim R_f)^\wedge$, is also equal to $(\varinjlim R_\bf)^\wedge$. 
	Indeed, the natural morphism $\varepsilon : \varinjlim R_f \to \varinjlim R_\bf$ induces a morphism between their completions, and the latter is an isomorphism since it is true for each reduction $\varepsilon / \pi^n$ and colimits commute with tensor products. 
	Below we write $T$ for the filtered colimit $\varinjlim R_\bf$, so that $T^\wedge = \sO_L$. 
	The reason to consider $T$ instead of $\varinjlim R_f$ is that $(T, (\pi))$ is now a Henselian pair by \cite[\href{https://stacks.math.columbia.edu/tag/0FWT}{Tag 0FWT}]{stacks-project}, so that we we may apply Gabber--Ramero's approximation result (extending that of Elkik) \cite[Prop.~5.4.53]{GR03}, which tells us the following: Base change along the map $T[\pi^{-1}] \to T^\wedge [\pi^{-1}] = L$ induces an equivalence of categories $\FEt(T[\pi^{-1}]) \simeq \FEt(L)$.
	Since $L' \in \FEt(L)$, we can find some $V \in \FEt(T[\pi^{-1}])$ such that $L \tensor_{T[\pi^{-1}]} V = L'$. 
	As colimits commute with localization, $T[\pi^{-1}] = \varinjlim (R_\bf [\pi^{-1}])$. 
	Now we have 
	\begin{equation}
		\label{eqn: apply Gabber-Ramero}
		\FEt(L) \simeq \FEt(T[\pi^{-1}]) = \FEt(\varinjlim (R_\bf[\pi^{-1}])) \simeq \varinjlim \, \FEt(R_\bf[\pi^{-1}]).
	\end{equation}
	For a reference of the second isomorphism, see \cite[III.3]{ArtinGT}. 
	This implies that for some $f \in R \smallsetminus \fp$, there exists an object $V_f \in \FEt(R_\bf[\pi^{-1}])$ such that $V = V_f \tensor_{R_\bf[\pi^{-1}]} T[\pi^{-1}]$ (so that $L' = L \tensor_{R_\bf[\pi^{-1}]} V_f$). 
	
	We now shrink $X$ to $\Spf(R_\bf)$ (i.e., replace $R$ by $R_\bf$), after which we have a finite \'etale $R[\pi^{-1}]$-algebra $V$ inside $L'$ such that $L' = L \tensor_R V$. 
	Let $R'$ be the integral closure of $R$ in $V$. 
	Then we have $R'[\pi^{-1}] = V$ and that $R'$ is finite over $R$. 
	Note also that by construction $R'$ is flat (as $R' \subseteq L'$) and topologically of finite type over $\sO_K$, so $X' \colonequals \Spf(R')$ is admissible. 
	Moreover, we claim that $(X_k')_{\mathrm{red}}$ has at most a single irreducible component that dominates $X_k$. 
	Suppose not. Then there exists an $f_0 \in R \smallsetminus \pi R$ such that for all $f_0 \mid f$, the fiber product $X' \times_X X_\bf$ is disconnected, which would imply $\varinjlim (R' \tensor_R R_\bf) = R' \tensor_R R_\fp$ is not an integral domain. 
	On the other hand, as $R_\fp$ is a discrete valuation ring, the morphism $R_\fp \to R_\fp^\wedge$ is faithfully flat and in particular injective, and hence so is its base change along $R \to R'$, i.e., $R' \tensor_R R_\fp \to R' \tensor_R R_\fp^\wedge$. 
	But by \cref{lem: simple observation}.\ref{lem: simple observation alg} $R' \tensor_R R_\fp^\wedge = R' \tensor_R \sO_L$ is a subalgebra of $L'$, which must be an integral domain. 
	This gives the desired contradiction to affirm the claim. Therefore, we may shrink $X$ again to make $(X_k')_{\mathrm{red}}$ integral. 
	Now $X'$ becomes an object of $\sfN(X)$ such that $L(X') = L'$, which is what we are looking for. 
	
	Finally we show that the functor $L_X$ is \textit{fully faithful}. 
	Faithfulness follows from \cref{lem: simple observation}.\ref{lem: simple observation alg}. 
	The fullness amounts to the following statement: Suppose for $i = 1, 2$ we are given the data $X'_{i} \in \sfN(X_{f_i})$, $L_i \colonequals L(X'_{i})$ and $\tau \in \mathrm{Hom}_L(L_1, L_2)$. Then there exists $g \in R \smallsetminus \fp$ divisible by both $f_1$ and $f_2$ such that $\tau$ is induced by a morphism $\wt{\tau} : X'_{1}|_{X_g} \to X'_{2}|_{X_g}$ in $\sfN(X_g)$. 
	To prove this statement, we are allowed to shrink $X$, so we may simply assume that $X = X_{f_1} = X_{f_2}$. 
	Suppose that $X_i = \Spf(R_i)$ and $\fp_i \colonequals \fp(X_i) \subseteq R_i$. 
	By considering the equivalence of categories in (\ref{eqn: apply Gabber-Ramero}) again, we find some $g \in R \smallsetminus \fp$ such that $\tau$, which is a morphism in $\FEt(L)$, comes from a morphism $\wt{\tau} : (R_1 \tensor_R R_\bg)[\pi^{-1}] \to (R_2 \tensor_R R_\bg)[\pi^{-1}]$. 
	As $X_i|_{X_g}$ is nothing but $\Spf((R_i \tensor_R R_\bg)^\wedge)$, by \cref{lem: simple observation}.\ref{lem: simple observation localization} $\wt{\tau}$ restricts to a morphism $X_1|_{X_g} \to X_2|_{X_g}$ as desired. 
\end{proof}

\subsection{Primitive inseparable covers}
\label{sub primitive insep cover} 
Let $f : Z' \to Z$ be a morphism of schemes in characteristic $p$. 
We say that $f$ is \textit{purely inseparable} if it is injective on the underlying topological spaces and for every point $z \in Z'$, the residue field extension $k(z) \into k(f(z))$ is purely inseparable. 
If $Z$ is normal, Noetherian, and integral and $Z'$ is integral, then a finite morphism $f : Z' \to Z$ is purely inseparable if and only if the induced extension between their fraction fields is purely inseparable (\cite[p.\ 208]{LiuAG} together with Nagata's compactification result).

\begin{definition}
	\label{def: purely inseparable cover}
	Let $X, Y \in \sfC^{\mathrm{sm}}_K$. We say that a finite flat morphism $f : Y \to X$ is a \textit{primitive inseparable cover} if the special fiber $f_k$ is purely inseparable of degree $p$ and the generic fiber $f_\eta$ is a Galois cover with group $G \cong \IZ / p \IZ$. 
\end{definition}  

We recall a simple consequence of the Cohen structure theorem. 
\begin{lemma}
	\label{lem: Cohen}
	Let $(V, (\pi))$ be a complete discrete valuation ring with residue field $k$. 
	Let $(R, \fm)$ be a complete local $V$-algebra with Krull dimension $n + 1$ for some $n\geq 0$. 
	Suppose that the natural morphism of residue fields $V / \pi V \to R/ \fm$ is an isomorphism, and the ring $R/ \pi R$ is regular of dimension $n$. 
	Then we have $R \cong V[\![t_1, \cdots, t_n]\!]$. 
\end{lemma}
\begin{proof}
	By Cohen structure theorem \cite[\href{https://stacks.math.columbia.edu/tag/0C0S}{Tag 0C0S}]{stacks-project}, we know $R / \pi R \cong k[\![x_1, \cdots, x_n]\!]$. 
	Fix such an isomorphism and lift $x_i$'s to $\wt{x}_i \in \fm$. 
	Then there is a natural morphism $V[\![t_1, \cdots, t_n]\!] \to R$ sending $t_i$ to $\wt{x}_i$, which is surjective by applying \cite[\href{https://stacks.math.columbia.edu/tag/0315}{Tag 0315}]{stacks-project} onto the ideal $(\pi, t_1, \cdots, t_n)$. 
	On the other hand, as $V[\![t_1, \cdots, t_n]\!]$ is an integral domain and has the same dimension as $R$, the morphism must also be injective---otherwise the quotient would have strictly smaller dimension. 
\end{proof}

\begin{proposition}
	\label{lem: find point}
	Let $X, Y \in \sfC^{\mathrm{sm}}_K$. 
	Suppose that $f : Y \to X$ be a primitive inseparable cover. 
	Then up to shrinking $X$ and replacing $Y$ by its restriction, there exists for every closed point $x_0 \in X_s$, 
	\begin{itemize}
		\item a finite unramified extension $K'/K$,
		\item a totally ramified extension $K''/K'$ of degree $p$,
		\item a point $x \in X(\sO_{K'})$ lifting $x_0$, 
        \end{itemize}
	such that the fiber product $Y\times_X x$ is a single $\mathcal{O}_{K''}$-point of $Y$.
\end{proposition}
\begin{proof}
	Suppose that $X = \Spf(A)$ and $Y = \Spf(B)$. 
	Then we may view $f$ as a map between $\sO_K$-algebras $A \to B$. 
	By abuse of notation, we shall also use $f$ to denote other morphisms between algebras which are naturally induced by $A \to B$. 
	
	By assumption, the fraction field $F(Y) = \mathrm{Frac}(B_k)$ is generated over $F(X) = \mathrm{Frac}(A_k)$ via the embedding $f : F(X) \into F(Y)$ by an element $g$ such that $g^p = t$ for some $t \in F(X)$. 
	Up to shrinking $X$ and $Y$, below we shall assume that $t \in A_k$ and $g \in B_k$.
	Then we note that $dt \in \Omega^1_{A_k/k}$ is necessarily nonzero. 
	Indeed, Cartier isomorphism tells us that $$\ker(F_*(d) : F_*(A_k) \to F_*(\Omega_{X_k/k})) = \sO_{X^{(p)}_k},$$ where $F : X_k \to X_k^{(p)}$ is the relative Frobenius of $X_k$. 
	Hence if $dt = 0$, then $t$ has to be a $p$th power, which implies that $g \in f(F(X))$ and contradicts our assumption. 
	Moreover, by the noetherian assumption on $A$, we know the $p$-completely flat morphism $A\to B$ is actually flat (\cite[\href{https://stacks.math.columbia.edu/tag/0912}{Tag 0912}]{stacks-project}).
	
	Up to further shrinking $X$, we shall assume that $dt$ is nowhere vanishing at closed points of $X$. 
	Let $x_0 \in X(k')$ be any point for a finite extension $k'/k$ of residue fields, and let $K'$ be the unramified extension of $K$ lifting $k'/k$. 
	Let $\fm \subseteq A$ be the maximal ideal defined by $x_0$. 
	Let $A^\wedge_\fm$ denote the completion of $A$ at $\fm$.
	Let $n$ be the relative dimension of $X$ and $Y$ over $\sO_K$. 
	Then by \cref{lem: Cohen}, we have an isomorphism of complete local rings $A^\wedge_\fm \cong \sO_{K'}[\![t_0, \cdots, t_n]\!]$, which we fix for now. 
	Similarly, let $y_0 \in Y(k')$ be the unique pre-image of $x_0$ and $\fn \subseteq B$ be the maximal ideal given by $y_0$. 
	We also fix an isomorphism $B^\wedge_\fn \cong \sO_{K'}[\![s_0, \cdots, s_n]\!]$. 
	Below we shall view $A$ and $B$ as subalgebras of $A^\wedge_\fm$ and $B^\wedge_\fn$ respectively and the same applies to their mod $\pi$ reductions. 
	Up to replacing $g$ by $g - g(y_0)$ and $t$ by $t - g(y_0)^p$, we assume that $t(x_0) = 0$ and $g(y_0) = 0$. 
	This implies that when viewed as elements in $k'[\![t_0, \cdots, t_n]\!]$ and $k'[\![s_0, \cdots, s_n]\!]$ respectively, $t$ and $g$ have no constant terms. 
	
	We then claim that there exists an automorphism of $\sO_{K'}[\![t_0, \cdots, t_n]\!]$ which sends $t_0$ to an element $\wt{t}$ such that $\wt{t} \mod \pi  = t$. 
	By the assumption that $dt$ does not vanish at $x_0$, we know $\partial t / \partial t_i \neq 0$ for some $i$. 
	Up to renaming variables we assume $i = 0$. 
	Let $\wt{t} \in \sO_{K'}[\![t_0, \cdots, t_n]\!]$ be any lift of $t$ with no constant term.
	Then we consider the endomorphism $\sO_{K'}[\![t_0, \cdots, t_n]\!]$ which sends $t_0$ to $\wt{t}$ and fixes $t_i$ for each $i > 0$. 
	As a consequence, since $\partial t /\partial t_0\neq 0$, the endomorphism has an invertible Jacobian (as a matrix with $\sO_{K'}$-entries) and hence is an automorphism by the inverse function theorem on power series rings. 
	
	The above allows us to assume that the mod $\pi$ reduction of $t_0$ is $t$. 
	Similarly, let $\wt{g} \in B^\wedge_\fn$ be an element with no constant term such that $\wt{g} \mod \pi = g$. 
	Note that $f$ induces a morphism $\sO_{K'}[\![t_0, \cdots, t_n]\!] \to \sO_{K'}[\![s_0, \cdots, s_n]\!]$ and we shall denote by $f_i$ the image of $t_i$. 
	Then our assumption that $g^{p} = t$ translates to $\wt{g}^{p} + \pi h = f_0$ for some $h \in \sO_{K'}[\![s_0, \cdots, s_n]\!]$. 
	Let $h_0$ be the constant term of $h$.

	Let $u\in \sO_{K'}$ be an element such that $u \in \sO_{K'}^\times$ and $u \neq h_0$ if $h_0 \in \sO_{K'}^\times$. 
	We set $x$ to be the $\mathcal{O}_{K'}$-point of $X$ given by sending $t_0$ to $u \pi$ and $t_i$ to $0$ for $i > 0$, and write $x_\eta \in X_\eta(K)$ for its generic fiber. 
	Then by assumption, the point $x$ specializes to $x_0$ in the special fiber.
	Let $y_\eta \in Y_\eta$ be any closed point in the preimage of $x_\eta$ and set $K'' \colonequals K(y)$. Then the surjection $y_\eta : B[1/p] \to K''$ restricts to a surjection $y : B \to \sO_{K''}$, i.e., $y_\eta$ extends to a point $y \in Y(\sO_{K''})$ (cf. \cite[\S8.3,~Prop.~8]{BoschLectures}).
	As the point $y$ specializes to $y_0$, the corresponding map of $y_0$ lifts to a morphism $y\colon \sO_{K'}[\![s_0, \cdots, s_n]\!] \to \sO_{K''}$. 
	Hence we get a commutative diagram 
	\[\begin{tikzcd}
		{\sO_{K'}[\![t_0, \cdots, t_n]\!]} & {\sO_{K'}[\![s_0, \cdots, s_n]\!]} \\
		{\sO_{K'}} & {\sO_{K''}}
		\arrow["{t_i \mapsto f_i}", from=1-1, to=1-2]
		\arrow["x"', from=1-1, to=2-1]
		\arrow[hook, from=2-1, to=2-2]
		\arrow["y", from=1-2, to=2-2]
	\end{tikzcd}\]
	Suppose that $y$ is given by $s_i \mapsto \beta_i$ for some pesudo-uniformizer $\beta_i \in \sO_{K''}$. 
	Then by looking at the images of $t_0\in\sO_{K'}[\![t_0, \cdots, t_n]\!]$ under the maps above, the diagram yields the following equation in $\sO_{K''}$: 
	\begin{equation}
		\label{eqn: pi-adic val}
		\wt{g}(\beta_1, \cdots, \beta_n)^{p} + \pi h(\beta_1, \cdots, \beta_n) = f_0(\beta_1, \cdots, \beta_n) = u \pi. 
	\end{equation}
	As $\wt{g}$ has no constant term, $u \neq h_0$, and $\mathrm{val}_\pi(\beta_i) > 0$ for each $i$, the above equation tells us that for some $i$, $\mathrm{val}_\pi(\beta_i) \le p^{-1}$. 
	Indeed, otherwise $\mathrm{val}_\pi(\wt{g}(\beta_1, \cdots, \beta_n)^{p}) > 1$, but by the choice of the element $u$, we know $\mathrm{val}_\pi(u \pi - \pi h(\beta_1, \cdots, \beta_n)) = 1$. 
	As a consequence, we see $[K'' : K'] \ge p$. 
	On the other hand, as $f_\eta$ is a finite Galois cover of degree $p$, we have $[K'':K'] \le p$. 
	So the above equality must be achieved, i.e., $[K'' : K'] = p$ and we have $\mathrm{val}_\pi(\beta_i) = p^{-1}$. 
	
	Finally, to see that the point $y$ is the fiber product $Y\times_X x$, we first notice that this is clear for the generic fiber: because $Y_\eta\to X_\eta$ is a finite \'etale cover of degree $p$, and $K(y_\eta)=K'$ is a degree $p$ ramified cover of $K=K(x_\eta)$.
	The claim in the level of formal schemes then follows from the flatness assumption of $A\to B$: the base change $Y\times_X x$ is flat of degree $p$ over $x=\Spf(\mathcal{O}_{K'})$ and contains $y=\Spf(\mathcal{O}_{K''})$ as a closed subscheme, hence must equal to $y$ itself.
 \end{proof}

\begin{remark}
	\label{ex: basic example}
	Here is a prototypical example for \cref{lem: find point}: Suppose that $\sO_K$ contains all $p$th roots of unity and consider the one-dimensional torus $X \colonequals \Spf( \sO_K \< t^\pm\>)$. Its endomorphism given by $t \mapsto t^p$ satisfies the hypothesis of \cref{lem: find point}. Note that $X(K)$ is identified with $\{ x \in \sO_K \mid \mathrm{val}_\pi(x) = 0 \}$. The proof suggests that, to find a point $x\in X(K)$ such that the cover is ramified at $x$,  we may take $x = 1 + \pi$. Indeed, the unique $y \in X_\eta$ over $x$ has the generic residue field $K''\colonequals K(\alpha)$ for $\alpha^p = 1 + \pi$, and the minimal polynomial of $\alpha - 1$ is a degree $p$ Eisenstein polynomial, so $[K'' : K] = p$ as desired. 
\end{remark}

\subsection{Towers of covers} 
\label{sub towers}
In this subsection, we show that up to replacing the base field by a finite extension and shrinking the base $p$-adic formal scheme $X$, a finite extension of the field $L(X)$ always comes from a tower of finite covers of $X$ consisting of a finite Galois cover followed by a tower of primitive inseparable covers.

\begin{proposition}
	\label{prop: tower}
	Suppose that $X, X' \in \sfC^{\mathrm{sm}}_K$ are of the same relative dimension over $\sO_K$ and $X' \to X$ is a finite morphism such that the induced morphism $L(X) \to L(X')$ is a Galois extension.
 Then up to shrinking $X$, there exists a tower of finite covers $$X' = X_m \to X_{m - 1} \to \cdots \to X_0 \to X$$ for some $m \in \IN$, such that $X_0 \to X$ is a Galois \'etale cover and $X_{i+1}\to X_i$ are primitive inseparable covers for $i\geq 1$. 
\end{proposition}
\begin{proof}
	Let $L\colonequals L(X)$, let $L'\colonequals L(X')$, and denote their residue fields by $F$ and $F'$. 
	By the smoothness assumptions on $X$ and $X'$, the element $\pi$ is an uniformizer for both $\sO_L$ and $\sO_{L'}$. 
	In particular, the extension $\sO_L \subseteq \sO_{L'}$ is weakly unramified in the sense of \cite[09E4]{stacks-project}. Let $\sO_{L_0} \subseteq \sO_{L'}$ be the maximal subextension unramified over $\sO_L$ and $F_0$ be its residue field. Then $F_0$ is the maximal separable extension of $F$ in $F'$, so that $F'/F$ is purely inseparable. 
	We briefly recall the argument for the reader's convenience: Since $\sO_L$ is a complete discrete valuation ring, any finite separable extension $F''$ of $F$ gives an unramified extension $\sO_{L''}$ of $\sO_L$. 
	Moreover, $\sO_{L''}$ is necessarily of the form $\sO_L[\alpha]/ (f(\alpha))$ where $f(x) \in \sO_L[x]$ is a polynomial whose reduction in $F''[x]$ is separable. 
	By Hensel's lemma, if $F''$ embeds into $F'$, then the equation $f(x) = 0$ is solvable in $\sO_{L'}$ and hence $\sO_{L''}$ embeds into $\sO_{L'}$. 
	
	Let $G \colonequals \Gal(L'/ L_0)$. Since $L'/L_0$ is weakly unramified and $F'/F_0$ is purely inseparable, $G$ is a $p$-group. 
	By Jordan-H\"older theorem, we may take a composition series $$1 \unlhd G_{m -1} \unlhd \cdots \unlhd G = G_0$$ for some $m \in \IN$. Then being a simple $p$-group, $H_i \colonequals G_{i} / G_{i+1}$ is necessarily isomorphic to $\IZ / p \IZ$. 
	The compoistion series gives a filtration $$L_0 \subseteq L_1 \subseteq L_2 \subseteq \cdots \subseteq L_m = L'$$ such that $\Gal(L'/L_i) = G_i$. 
	Let $F_i$ be the residue field of $\sO_{L_i}$. 
	Then since each $\sO_{L_i} \subseteq \sO_{L_{i + 1}}$ is weakly unramified, the purely inseparable extension $F_{i + 1}/F_i$ has degree exactly $p$. 
	
	Finally, we note that for any $X'' \in \sfN(X)$ such that $\sO_L(X'')$ is weakly unramified over $\sO_L$, the special fiber $X''_k$ is generically reduced, and hence generically smooth over $k$. 
	After shrinking $X$, the cover $X''$ becomes smooth over $\sO_K$. 
 	Moreover, by the generic smoothness of the maps of special fibers, we can assume that the map $X''_k\to X_k$ is flat.
  	This implies that the map $X''\to X$ is $p$-completely flat, and by the noetherian assumption on $X$, we know the $p$-completely flat morphism $X''\to X$ is actually flat (\cite[\href{https://stacks.math.columbia.edu/tag/0912}{Tag 0912}]{stacks-project}).
  Now the conclusion follows easily from \cref{thm: normal covers} and the preceding paragraph. 
\end{proof} 

In practice, \cref{prop: tower} will be used in conjunction with \cref{thm: pre-tower} below.   
\begin{theorem}
	\label{thm: pre-tower} Suppose that $X \in \sfC^{\mathrm{sm}}_K$ has a geometrically connected special fiber and let $L'$ be any finite Galois extension of $L\colonequals L(X)$. 
	Then up to shrinking $X$, there exists
	\begin{itemize}
		\item a finite Galois extension $\sK/K$,
		\item a commutative diagram of finite Galois field extensions as below
		\begin{equation}
			\label{eqn: diagram of Gal ext}
			\begin{tikzcd}
				{L \tensor_K \sK} & {\sL'} \\
				L & {L'},
				\arrow[hook, from=2-1, to=2-2]
				\arrow[hook, from=2-1, to=1-1]
				\arrow[hook, from=1-1, to=1-2]
				\arrow[hook, from=2-2, to=1-2]
			\end{tikzcd}
		\end{equation}
	\item 	and a normal cover $\sX'$ of $\sX \colonequals X \tensor_{\sO_K} \sO_\sK \in \sfC^{\mathrm{sm}}_\sK$,
	\end{itemize}
 such that $\sX'$ is smooth over $\sO_\sK$ and $\sO_L(\sX') = \sO_{\sL'}$. 
\end{theorem}
It will be clear in the proof that $L \tensor_K \sK$ is indeed a field and is nothing but $L(\sX)$. 
Before proving \Cref{thm: pre-tower}, we recall an elementary fact on field extensions.
\begin{lemma}
\label{lem: split Galois extensions}
    Let $E \subseteq F$ be a finite Galois field extension and $E \subseteq E'$ be any other finite separable extension. Set $d \colonequals [F : E]$. Then there is a decomposition $F \tensor_E E' \simeq \prod_{i = 0}^{m - 1} F_i$ for some $m \mid d$, such that each $F_i$ is a Galois extension of $E'$ of degree $d/m$. 
\end{lemma}
\begin{proof}
    Let $f(x) \in E[x]$ be a monic irreducible polynomial such that $E' \simeq E[x]/ (f(x))$. 
    Let $f_0(x)$ be a factor of $f(x)$ in $F[x]$ and let $H$ be the stabilizer of $f_0(x)$ under the natural action of $G \colonequals \Gal(F/E)$ on $F[x]$.
    Let $\sigma_0 = 1, \cdots, \sigma_{m-1}$ be representatives of $G/H$. Then we must have
    $$ f(x) = \prod_{i = 0}^{m - 1} f_i(x) \text{ where $f_i(x) \colonequals \sigma_i(f(x))$.} $$
    Indeed, the right hand side is defined over $E$ and divides $f(x)$. But $f(x)$ is monic and irreducible in $E[x]$, so they must be equal. 
    Since $E'/E$ is separable, roots of $f(x)$ (in some algebraic closure of $E$) are distinct, so that $f_i(x)$'s are pairwise coprime in $F[x]$. 
    Therefore, the Chinese remainder theorem gives a canonical decomposition $F \tensor_E E' \simeq \prod_{i = 0}^{m - 1} F_i$ for $F_i \colonequals F[x]/(f_i(x))$. 
    The action of $H_i$ on $F_i$ fixes $E'$, so that $H_i \le \Aut(F_i/ E')$. However, one easily checks that $|H_i| = [F_i : E'] = d/m$. 
    Hence $F_i / E'$ is Galois with $\Gal(F_i/E') = H_i$. 
\end{proof}

\begin{proof}[Proof of \Cref{thm: pre-tower}]
	By \cref{thm: normal covers}, after shrinking $X$ we can find a normal cover $X'$ over $X$ such that $L'$ can be identified with $L(X')$. 
	By \cref{thm: reduced fiber theorem}, there exists a finite extension $\sK$ over $K$ such that the special fiber $(X'_{\sO_\sK})^\nu_\kappa$ of the normalization $(X'_{\sO_\sK})^\nu$ of $X'_{\sO_\sK} \colonequals X' \tensor_{\sO_K} \sO_\sK$ is generically reduced along every irreducible component which dominates $X_k$, where $\kappa$ is the residue field of $\sK$. 
	Note that by \cref{rmk: of reduced fiber thm}, we are allowed to replace $\sK$ by its normal closure over $K$ and assume that $\sK$ is Galois over $K$. 
	After shrinking $X$ again, we may then assume that $(X'_{\sO_\sK})^\nu$ is smooth over $\mathcal{O}_\sK$. 
	Let $(X'_{\sO_\sK})^\nu = \coprod_{i = 0}^{m - 1} \sX_i$ be its decomposition into connected components. Then each $\sX_i$ is an object of $\sfC^{\mathrm{sm}}_\sK$. 
Below we shall suppose that $X = \Spf(R)$, and  $X'=\Spf(R')$, so that $\sX \colonequals X_{\sO_\sK} = \Spf(\sR)$ is a smooth $p$-adic formal scheme over $\sO_\sK$ for $\sR \colonequals R \tensor_{\sO_K} \sO_\sK$. 
	We also let $\varpi$ be the uniformizer of $\sO_\sK$. 

Recall that $X_k$ is assumed to be geometrically connected, so that $\sX$ is also connected. 
By the smoothness of $\sX$ over $\sO_\sK$, 
 the ring $\sR$ is an integral domain and $\sR_{(\varpi)} = \sR \tensor_R R_{(\pi)}$ is a discrete valuation ring. 
	Moreover, we have (cf. \cref{lem: simple observation}.(a)) $$ \sO_L \tensor_{\sO_K} \sO_\sK = (R_{(\pi)})^\wedge \tensor_{\sO_K} \sO_\sK = (R_{(\pi)} \tensor_{\sO_K} \sO_\sK)^\wedge = (\sR_{(\varpi)})^\wedge. $$
	This implies that the tensor product $\sL \colonequals L \tensor_K \sK$ is isomorphic to $L(\sX)$. 
In particular, $\sL \colonequals L \tensor_K \sK$ is a field. 
 Let $\sR_i$ be the finite $\sR$-algebra such that $\sX_i = \Spf(\sR_i)$. Set $\sO_{\sL_i} = \sO_L(\sX_i) = (\sR'_{i, (\varpi)})^\wedge$ and $\sL_i = L(\sX_i) = \sO_{\sL_i}[\varpi^{-1}]$. 
 As one would expect, we have: 
\begin{lemma}
		$L' \tensor_K \sK = \prod_{i = 0}^m \sL_i$. 
\end{lemma}
\begin{proof}
		Let $T \colonequals R'_{\sO_\sK} \tensor_{R} R_{(\pi)}$. 
		As the ring $R'_{\sO_\sK}$ is finite over $R$, we have $T^\wedge = R'_{\sO_\sK} \tensor_{R} (R_{(\pi)})^\wedge = R'_{\sO_\sK} \tensor_{R} \sO_L$. 
		Since $T$ is a finite $R_{(\pi)}$-algebra and $T[\pi^{-1}]$ is normal (in fact a product of fields), we may switch the order of normalization and $p$-completion thanks to \cite[\href{https://stacks.math.columbia.edu/tag/0BG9}{Tag 0BG9}]{stacks-project}
		\footnote{Here to apply the result, we implicitly use the fact that for an $R_{(\pi)}$-algebra, the completion with respect to the maximal ideal of $R_{(\pi)}$ is just $p$-adic completion. }, and get $$ (T^\wedge)^\nu \simeq T^\nu \tensor_{R_{(\pi)}} (R_{(\pi)})^\wedge = (T^\nu)^\wedge.$$

		On the other hand, as normalization commutes with localization \cite[\href{https://stacks.math.columbia.edu/tag/0307}{Tag 0307}]{stacks-project}, we have $$T^\nu = (R'_{\sO_\sK})^\nu \tensor_R R_{(\pi)} = \prod_{i = 1}^m \sR_{i, (\varpi)}. $$ 
		Combined with the previous isomorphism that $(T^\wedge)^\nu \simeq (T^\nu)^\wedge$, we then get 
		\begin{equation}
			\label{eqn: T wedge nu}
			(R'_{\sO_\sK} \tensor_{R} \sO_L)^\nu = (T^\wedge)^\nu \simeq  (T^\nu)^\wedge = \prod_{i = 0}^{m - 1} (\sR_{i, (\varpi)})^\wedge = \prod_{i = 0}^{m - 1} \sO_{\sL_i}.
		\end{equation}
		Inverting $\varpi$ and using that normalization commutes with localization again, the left end term in \Cref{eqn: T wedge nu} becomes
		\begin{equation}
			\label{eqn: invert varpi}
			(R'_{\sO_\sK} \tensor_{R} \sO_L)^\nu [\varpi^{-1}] = (\sO_\sK[\varpi^{-1}] \tensor_{\sO_K} R' \tensor_{R} \sO_L)^\nu = (\sK \tensor_K L')^\nu = \sK \tensor_K L',
		\end{equation}
		where we used $R' \tensor_{R} \sO_L = \sO_{L'}$ in the second equality. 
		Now combine (\ref{eqn: T wedge nu}) and (\ref{eqn: invert varpi}), we get the result.
\end{proof}

Note that the extensions $L \subseteq L'$ and $\sL \subseteq \sL \tensor_K \sK$ are both finite and Galois.
    Therefore, by \cref{lem: split Galois extensions}, $\sL_0$ is a Galois extension of both $L'$ and $\sL \tensor_K \sK$. Hence we obtain the diagram (\ref{eqn: diagram of Gal ext}) by setting $\sL' \colonequals \sL_0$.
    Accordingly, we set $\sX' \colonequals \sX_0$. Then being a connected component of $(X'_{\mathcal{O}_{\sK}})^\nu$, $\sX'$ is smooth over $\mathcal{O}_{\sK}$. 
    Moreover, \cref{eqn: T wedge nu} implies that $\sO_L(\sX') = \sO_{\sL'}$, so we are done.
\end{proof}

\begin{remark}
\label{rmk: pi_0}
    It is sometimes convenient to apply \cref{thm: pre-tower} together with the following observation: Suppose that $X$ is a connected $p$-adic formal scheme of finite type over $\Spf(\sO_K)$. Then there exists a finite unramified extension $K'$ of $K$ such that the structure morphism $X \to \Spf(\sO_K)$ factors through $\Spf(\sO_{K'}) \to \Spf(\sO_K)$, where $X$ has geometrically connected special fiber as a $p$-adic formal scheme over $\Spf(\sO_{K'})$. To see this, since $X_s$ is a connected $k$-variety, $\pi_0(X_s) = \Spec(k')$ for some finite extension $k'$ of $k$, so that $X_s$ is geometrically connected over $k'$ (\cite[\S10.2~Prop.~2.18]{LiuAG}). 
    Let $K'$ be the corresponding finite unramified extension of $K$. Then the formal \'etaleness of $\Spf(\sO_{K'}) \to \Spf(\sO_K)$ implies that the structure morphism $X_s \to \Spec(k')$ lifts unqiuely to a morphism  $X \to \Spf(\sO_{K'})$. 
\end{remark}

\subsection{Effective sets and an $\ell$-adic Analogue} 
\label{sub:effective}
In this section, we introduce the notion of the \emph{effective set} and prove the $\ell$-adic analogue of the pointwise criterion.

We start with the following two lemmas that check that the purity result of branch locus for schemes also hold true for suitable formal schemes.
\begin{lemma}
    Let $X$ be a connected and admissible $p$-adic formal scheme over $\Spf(\sO_K)$. If the mod $\pi$ special fiber $X_k$ is reduced, then $X_\eta$ is connected. 
\end{lemma}
\begin{proof}
    It suffices to consider the case when $X = \Spf(R)$ is affine. Then $X_\eta = \Spa(R_K)$ is connected if and only if $\Spec(R_K)$ is connected (\cite[Ex.~2.4.9]{ConradNotes}). Therefore, we reduce to showing that the scheme $\Spec(R)$ has connected generic fiber, which follows from \cite[055J]{stacks-project}. 
\end{proof}

We remark that the reducedness assumption on $X_k$ is important. Otherwise $\Spf(\sO_K[x]/(x^2 - \pi^2))$ would produce a counterexample.

\begin{lemma}[Purity of branch locus]
\label{lem: purity of branch locus}
    Let $X$ be a regular admissible formal scheme over $\sO_K$ with an irreducible and normal reduced special fiber $X_s$, and let $U \subseteq X$ be a non-empty formal open subscheme. Then we have the following:
    \begin{enumerate}[label=\upshape{(\roman*)}]
        \item\label{item: purity 1} The restriction functor $\FEt(X) \to \FEt(U)$ is fully faithful. 
        \item\label{item: purity 2} The restriction functor $\FEt(X) \to \FEt(X_\eta)$ is faithful. 
        \item\label{item: purity 3} The restriction functor $\FEt(X_\eta) \to \FEt(U_\eta)$ is faithful. 
        \item\label{item: purity 4} The induced functor 
        \begin{equation}
        \label{eqn: purity 4}
        \FEt(X) \to \FEt(X_\eta) \times_{\FEt(U_\eta)} \FEt(U)
        \end{equation}
        is an equivalence of categories. 
    \end{enumerate}
\end{lemma}
\begin{proof}
    Part \ref{item: purity 1} follows from the simple facts that $\FEt(X) = \FEt(X_s)$, $\FEt(U) = \FEt(U_s)$ and $\FEt(X_s) \to \FEt(U_s)$ is fully faithful (\cite[\href{https://stacks.math.columbia.edu/tag/0BQI}{Tag 0BQI}]{stacks-project}). 
    Then we consider \ref{item: purity 2} and \ref{item: purity 3} simultanenously. This can be checked Zariski-locally on $X$, so we may assume that $X = \Spf(R)$ is affine. Take some $U$ open in $R$ and $X' = \Spf(R') \in \FEt(X)$. 
    We may assume that $U = \Spf(R_{ \{ f \} })$ for some $f \in R$ which does not vanish on $X_s$. Now \ref{item: purity 2} (resp. \ref{item: purity 3}) follows from the fact that $R' \to R'_{\{ f \}}$ (resp. $R' \to R'[1/p]$) is injective.  

    Finally, we consider \ref{item: purity 4}. The full-faithfulness of (\ref{eqn: purity 4}) follows from \ref{item: purity 1}, \ref{item: purity 2}, and \ref{item: purity 3}. 
    Then, by standard gluing arguments, to prove essential surjectivity we reduce to considering the case when $X = \Spf(R)$ is affine as before.
    Let $X'_\eta \in \FEt(X_\eta)$ be any object, which we may assume to be connected. 
    Let $X'$ be the normalization of $X$ in $X'_\eta$, so that $X' = \Spf(R')$ for some finite $R$-algebra $R'$.
    By \cref{prop: Conrad}, the given hypotheses tells us that $R[1/p] \to R'[1/p]$ is \'etale and for some $f$ which does not vanish on $X_s$, the restriction $R_{\{f\}} \to R'_{\{f\}}$ is \'etale. We shall view $X'$ and $X$ as formal completions of the schemes $\wt{X}' := \Spec(R')$ and $\wt{X} := \Spec(R)$ along their corresponding special fibers.
    Then it suffices to show that the morphism between the schemes $\wt{X}' \to \wt{X}$ is \'etale. 
    As both schemes are locally noetherian, we are allowed to apply \cite[\href{https://stacks.math.columbia.edu/tag/0BMB}{Tag 0BMB}]{stacks-project}, for which the assumptions (1)-(4) in \textit{loc. cit.} are clear, and we check item (5) below.
    
    Let $x' \in \wt{X}'_k = X'_k$ be a closed point such that $f(x') = 0$ and let $y \in \wt{X}'$ be any point which specializes to $x'$ and is of codimension $1$. 
    If the characteristic of the residue field of $y$ is zero, then we are done because $R[1/p] \to R'[1/p]$ is \'etale. Otherwise, $y$ must be the generic point of an irreducible component of $\wt{X}'_k = X'_k$, and since the function $f$ does not vanish on $X_s$, we know $f(y) \neq 0$. 
    Since the map $R \to R'$ is finite, the ring $R' \tensor_{R} R_{\{ f \}}$ is thus $p$-complete and the natural morphism $R' \tensor_{R} R_{\{ f \}} \to R'_{\{ f \}}$ is an isomorphism. 
    Finally, as the unramified locus of a morphism locally of finite type is invariant under arbitrary base change \cite[\href{https://stacks.math.columbia.edu/tag/0475}{Tag 0475}]{stacks-project}, the fact that $R_{\{f\}} \to R'_{\{f\}}$ is unramifed at $y$ implies that $R \to R'$ is also unramified at $y$.
\end{proof}

\begin{construction}
\label{const: pi_1}
    Let $X$ be a connected smooth $p$-adic formal scheme over $\sO_K$. Let $b \in X_\eta$ be a classical point. Let $b_0 \in X_k$ be the closed point given by the reduction of $b$ and choose geometric points $\bar{b}, \bar{b}_0$ over $b, b_0$ respectively. We write $\pi_1^\mathrm{alg}(X_\eta, \bar{b})$ for the \textbf{algebraic fundamental group} of the rigid analytic space (cf. \cite{deJong_fund}). Namely, $\pi_1^\mathrm{alg}(X_\eta, \bar{b})$ is the automorphism group of the fiber functor from the category of finite \'etale covers of $X_\eta$ to that of sets. Then there is a natural specialization map $\pi_1^\mathrm{alg}(X_\eta, \bar{b}) \to \pi_1^\et(X_k, \bar{b}_0)$, and we define the kernel of which to be the \textbf{global inertia group} and denote it by $I_X$ as before. Now for each classical point $x \in X_\eta$, we choose a geometric point $\bar{x}$ over $x$. Moreover, we choose an \'etale path connecting $\bar{x}$ and $\bar{b}$, thereby obtaining an isomorphism $\pi_1^\mathrm{alg}(X_\eta, \bar{x}) \simeq \pi_1^\mathrm{alg}(X_\eta, \bar{b})$. Then we define $I_x$ to be the inertia subgroup of $\Gal_{K(x)} = \pi_1^\mathrm{alg}(x, \bar{x})$, which up to change of base points admits a map to $I_X$. 
    Let $\sI_x$ denote the union of conjugates of $\mathrm{im}(I_x \to I_X)$. It is clear that $\sI_x$ depends only on $x$ and not on the choice of $\bar{x}$ or the \'etale path connecting $\bar{x}$ and $\bar{b}$. 
\end{construction}

\begin{lemma}
\label{lem: surj of pi_1}
    In \cref{const: pi_1}, the specialization map $\pi_1^{\mathrm{alg}}(X_\eta, \bar{b}) \to \pi_1(X_k, \bar{b}_0)$ is surjective. As a consequence, a finite \'etale cover $X'_\eta$ over $X_\eta$ extends to a finite \'etale cover of $X$ if and only if the monodromy representation $\pi_1^{\mathrm{alg}}(X_\eta, \bar{b}) \to \mathrm{Aut}(X'_{\eta, \bar{b}})$ vanishes on $I_X$. 
\end{lemma}
\begin{proof}
        It suffices to prove the first statement, which easily implies the second. Let $U \subseteq X$ be any non-empty open formal subscheme which contains the geometric point $\bar{b}$ (and hence $\bar{b}_0$). Then we have a commutative diagram 
        \[\begin{tikzcd}
	{\pi_1^{\mathrm{alg}}(U_\eta, \bar{b})} & {\pi_1^{\mathrm{alg}}(X_\eta, \bar{b})} \\
	{\pi_1^\et(U_k, \bar{b}_0)} & {\pi_1^\et(X_k, \bar{b}_0)}
	\arrow[from=1-1, to=1-2]
	\arrow[from=1-1, to=2-1]
	\arrow[from=1-2, to=2-2]
	\arrow[from=2-1, to=2-2] 
    \end{tikzcd}. \]
    As the bottom arrow is surjective by \cite[\href{https://stacks.math.columbia.edu/tag/0BQI}{Tag 0BQI}]{stacks-project}, it suffices to show that the left vertical arrow is surjective. Therefore, we may shrink $X$ and assume that $X = \Spf(A)$ is affine. In this case, $\pi_1^{\mathrm{alg}}(X_\eta, \bar{b}) = \pi_1^\et(\Spec(A), \bar{b})$, so that the conclusion follows from \cite[\href{https://stacks.math.columbia.edu/tag/0BQI}{Tag 0BQI}]{stacks-project} again. 
\end{proof}

\begin{convention}
    We use $\sC(X)$ to the denote the set of classical points on the rigid generic fiber of a $p$-adic formal scheme $X$. 
    Let $\sC^{\mathrm{ur}} \subseteq \sC$ denote the subset of points defined over unramified extensions of $K$.
\end{convention}

\begin{definition}
    Let $X$ be a smooth $p$-adic formal scheme over $\sO_K$. We say that a morphism $Y \to X$ is \textit{semi-\'etale} if the induced map on generic map $Y_\eta \to X_\eta$ is \'etale, and the map $Y \to X$ can be factored as $Y \to U_{\sO_{K'}} \to X$, such that 
    \begin{itemize}
        \item $U$ is an open sub formal scheme of $X$;
        \item $K'$ is a finite extension of $K$; 
        \item the map $Y\to  U_{\sO_{K'}}$ is finite;
        \item the composition $Y\to U_{\sO_{K'}}\to \Spf(\sO_{K'})$ makes $Y$ a smooth $p$-adic formal scheme over $\sO_{K'}$.
    \end{itemize}
    
    The various properties are summarized in the diagram below.
    \[\begin{tikzcd}
	{Y_\eta} && {X_\eta} \\
	Y & {U_{\sO_{K'}}} & X \\
	& {\Spf(\sO_{K'})} & {\Spf(\sO_K)}
	\arrow[from=2-2, to=2-3]
	\arrow[from=2-2, to=3-2]
	\arrow[from=2-3, to=3-3]
	\arrow[from=3-2, to=3-3]
	\arrow[draw=none, from=2-2, to=3-3]
	\arrow["{\mathrm{finite}}", from=2-1, to=2-2]
	\arrow["{\mathrm{smooth}}"', from=2-1, to=3-2]
	\arrow[from=1-1, to=2-1]
	\arrow["{\text{{\'e}tale}}", from=1-1, to=1-3]
	\arrow[from=1-3, to=2-3]
    \end{tikzcd}\]
    We shall call the intermediate field extension $K'$ the \textit{field of definition} for $Y$. If $Y$ is connected, then $K'$ is the compositum of all subalgebras of $\Gamma(\sO_Y)$ which are finite \'etale over $K$.
\end{definition}

\begin{definition}
\label{def:eff set}
    We say that a subset $\sC \subseteq \sC(X)$ is \textit{dense in inertia} if the subgroup generated by $\cup_{x \in \sC} \sI_x$ is dense in the topological group $I_X$ (i.e., $\cup_{x \in \sC} \sI_x$ topologically generates $I_X$). 
    We say that $\sC$ is \textit{effective} if it is semi-\'etale-locally dense in inertia, i.e., for every semi-\'etale morphism $Y \to X$, the pullback $\sC|_{Y_\eta}$ of $\sC$ to $Y_\eta$ is dense in inertia as a subset of $\sC(Y)$. 
\end{definition}

Recall that when $G$ is a profinite group, a (normal) subgroup $H \le G$ is dense in $G$ if and only if the only open (normal) subgroup of $G$ which contains $H$ is $G$ itself. This immediately implies the following lemma. 

\begin{lemma}
\label{lem: dense in inertia}
Let $X$ be a connected smooth $p$-adic formal scheme over $\sO_K$. A subset $\sC \subseteq \sC(X)$ is dense in inertia if and only if it has the following property: For every connected finite Galois cover $X'_\eta \to X_\eta$, the normalization $X'$ of $X$ in $X'_\eta$ is \'etale over $X$ if and only if $X'_\eta|_x$ is unramified over $K(x)$ (i.e., extends to a finite \'etale $\sO_{K(x)}$-scheme) for every $x \in \sC$.
\end{lemma}
\begin{proof}
    Assume that $X$ is connected and choose base points $\bar{b}, \bar{b}_0$ as above. Then $X_\eta'$ is defined by an open subgroup $H \le \pi_1^\mathrm{alg}(X_\eta, \bar{b})$. It extends to a connected finite \'etale cover $X' \to X$ if and only if there exists an open subgroup $H_0 \le \pi_1^\mathrm{alg}(X_k, \bar{b}_0)$ whose pre-image in $\pi_1^\mathrm{alg}(X_\eta, \bar{b})$ is $H$. Clearly, this is satisfied if and only if $\sI \le H$, thanks to \cref{lem: surj of pi_1}. 
\end{proof}

\begin{theorem}
\label{thm: effectivity}
     Let $X$ be a connected smooth $p$-adic formal scheme over $\sO_K$.
     If $\sC \subseteq \sC(X)$ is effective, then for every semi-\'etale morphism $Y \to X$ with a primitive inseparable cover $f : Y' \to Y$, there exists some $y \in \sC|_{Y_\eta}$ such that $f^{-1}(y)$ is a single point defined over a totally ramified extension of $K(y)$. 
     If the reduction of $\sC^{\mathrm{ur}}$ is Zariski dense, the converse is also true. 
\end{theorem}
\begin{proof}
    Assume that $\sC$ is effective and let $Y' \stackrel{f}{\to} Y \to X$ be as in the first statement. By assumption $\sC|_{Y_\eta}$ is dense in inertia. 
    By replacing the base point of the fundamental group if necessary, we may choose a lift of $\bar{b}$ to $Y$, which we still denote by $\bar{b}$. Now $f_\eta$ is defined by an index $p$ open normal subgroup $H \le \pi_1^\mathrm{alg}(Y_\eta, \bar{b})$. 
    We first claim that $I_Y$ is not contained in $H$. 
    Otherwise $Y'_\eta$ would extend to a finite \'etale cover $\wt{Y}'$ of $Y$, and we must have $\wt{Y}' = Y'$, which is not possible.
    Note that the fiber $Y'_\eta|_y$, being a Galois \'etale cover of $y$ of degree $p$, is either a single point defined over a degree $p$ totally ramified extension of $K(y)$, or unramified over $K(y)$.
    In the latter case, $\sI_y$ is contained in $H$. Therefore, if there is no closed point $y$ in $\sC|_{Y_\eta}$ such that $Y'_\eta|_y$ is a single ramified point, then $\cup_{y \in \sC|_{Y_\eta}} \sI_y$ is contained in $H \cap I_Y$, which is properly contained in $I_Y$---this contradicts our assumption. 

    Now we prove the converse assuming that the reduction of $\sC^{\mathrm{ur}}$ is Zariski dense in the special fiber. 
    Note that by \cref{rmk: pi_0} we may assume that the special fiber $X_k$ is geometrically connected.
    Let $Y \to X$ be a semi-\'etale morphism and $Y'_\eta \to Y_\eta$ be an arbitrary Galois \'etale cover. Define $Y'$ to be the normalization of $Y$ in $Y'_\eta$.  By \cref{lem: dense in inertia}, it suffices to show that if $Y'_\eta |_{y}$ is unramified for every $y \in \sC|_{Y_\eta}$, then $Y' \to Y$ is \'etale. Using the assumption that the reduction of $\sC^{\mathrm{ur}}$ is Zariski dense, one quickly infers that $\sC^{\mathrm{ur}}|_{Y'_\eta}$ (as a subset of $\sC(Y')$) has also Zariski dense reduction. 
    If $K'$ is the field of definition of $Y$, then points in $\sC^{\mathrm{ur}}|_{Y'_\eta}$ are by assumption defined over unramified extensions of $K'$. Therefore, by \cref{lem: pc reduced special fiber}, the special fiber of the $p$-adic formal scheme $Y'$ over $\Spf(\sO_{K'})$ is generically reduced. 
    After we shrink $X$ (or $Y$), $Y'$ becomes smooth but possibly disconnected, in which case we shall replace $Y'$ be a connected component and assume that $Y' \in \sfC^\mathrm{sm}_{K'}$; moreover, applying \cref{prop: tower}, we may further assume that $Y' \to Y$ admits a factorization $Y' = Y_m \stackrel{f_m}{\to}\cdots \stackrel{f_1}{\to} Y_0 \stackrel{f_0}{\to} Y$ such that $f_0$ is \'etale and $f_j$ is a primitive inseparable cover for every $j > 0$. However, here our hypothesis forces $m = 0$, i.e., $Y'= Y_0$; otherwise by assumption, we can find some $y \in f_0^{-1}(\sC|_{Y_\eta})$ such that $f_1^{-1}(y)$ is ramified over $K(y)$, but then $Y'_\eta|_y$ cannot be unramified over $K(y)$. 
\end{proof}

\begin{corollary}
\label{cor: effectivity of all points}
    Let $X$ be a smooth $p$-adic formal scheme over $\sO_K$. If there exists a Zariski dense subset $\sC_s$ of closed points of $X_k$ such that $\sC \subseteq \sC(X)$ contains all the classical points which specialize to a point in $\sC_s$, then $\sC$ is effective. In particular, $\sC(X)$ itself is effective.
\end{corollary}
\begin{proof}
    This follows directly from \cref{lem: find point} and \cref{thm: effectivity}. 
\end{proof}
We immediately infer the $\ell$-adic analogue of \cref{intro:thm pc}. 
\begin{theorem}[$\ell$-adic pointwise criterion]
\label{Thm:l-adic PC}
    Let $X$ be a smooth $p$-adic formal scheme over $\sO_K$. 
    Let $\ell$ be a prime ($\ell = p$ is allowed). If an $\IZ_\ell$-local system $T$ on $X_\eta$ is unramified at all classical points, then $T$ is unramified with respect to $X$. 
\end{theorem}
\begin{proof}
    We may assume that $X$ is connected and use notations in \cref{const: pi_1}. From \cref{cor: effectivity of all points}, we know in particular that $\sC(X)$ itself is dense in inertia. Then we conclude by \cref{lem: surj of pi_1}. 
\end{proof}

\begin{remark}
    \label{rmk: two views on pi_1}
In the previous sections, we had taken a more algebraic point of view on $\pi_1$. 
Suppose that $X = \Spf(R) \in \sfC^{\mathrm{sm}}_K$.
As mentioned in \cref{const:Galois group of rigid space}, the algebraic fundamental group $\pi_1^\mathrm{alg}(X_\eta, \overline{L})$ is isomorphic to the Galois group $G_{X_\eta}$ of the maximal connected pro-finite-\'etale cover $\Spf((\overline{R})^\wedge_p)_\eta\to X_\eta$.
This identification amounts to choosing an embedding $\overline{R} \into \overline{L}$ as the base point
and in particular induces a choice of an object $\fp \in T_X$, where $T_X$ is the set of prime ideals of $\overline{R}$ above $(\pi)$. 
More concretely, the choice of the prime ideal $\fp$ is equivalent to choosing, for each connected finite \'etale cover $X'_\eta \to X_\eta$, an irreducible component of the reduced special fiber $X'_s$ in a compatible way (where $X'$ is the normalization of $X$ in $X'_\eta$). 
Likewise, one checks that the choice of $\overline{R} \into \overline{L}$ also induces an identification $G_{X_k} \simeq \pi_1^\et(X, \overline{L})$. 

Suppose now $U = \Spf(V)$ is an affine open formal subscheme of $X$. 
Then we may choose compatible embeddings $\overline{R} \into \overline{V}$ and $\overline{V} \into \overline{L}$ whose composition is the chosen $\overline{R} \into \overline{L}$. 
Note that $\overline{R} \into \overline{V}$ in particular induces a map $T_U \to T_X$. 
It is known that $G_{X_\eta}$ acts transitively on $T_X$, and the natural map $\Gal_L \to G_{X_\eta}$ surjects onto the stabilizer $G_{X_\eta}(\fp)$ of $\fp$ in $G_{X_\eta}$ (\cite[Lem.~3.3.1]{Bri08}, cf. \cite[\S6.1]{Shi22}). 
As the same is true for $U$, one quickly checks that $\pi_1^{\mathrm{alg}}(X_\eta, \overline{L})$ is the \textit{normal closure} of the image of $\pi_1^{\mathrm{alg}}(U_\eta, \overline{L})$. 
Alternatively, one can see this using the fact that $\sC(U)$ is an effective subset of $\sC(X)$. 

We caution the reader that the image of $\pi_1^{\mathrm{alg}}(U_\eta, \overline{L})$ is in general not the full $\pi_1^{\mathrm{alg}}(X_\eta, \overline{L})$. The reason is that the natural map $T_U \to T_X$ is in general not surjective, which happens whenever there exists some connected $X'_\eta \in \FEt(X_\eta)$ such that the restriction $U'_\eta := U_\eta \times_{X_\eta} X'_\eta$ is not connected. 
To illustrate this, we present below an example when $\sO_K$ is replaced by $\IC[\![t]\!]$, but the reader should find no difficulty in finding a similar example over $\sO_K$. 
\end{remark}

\begin{example}
    Let $E$ be a union of three lines on $\IP^2$ over $\IC$ which intersect transversely at nodes $Q_1, Q_2, Q_3$, and let $E'$ be a general (smooth) plane cubic. Suppose that $E$ and $E'$ are defined by $F$ and $F'$ in $\mathrm{H}^0(\sO_{\IP^2}(3))$ respectively. 
    Set $B := \mathrm{Proj}(\IC[t, s])$ and consider the pencil $\sE \subseteq \IP^2 \times B$ defined by $t F + s F'$, so that $E = \sE_b$ for $b = [1 : 0]$. 
    Let $\what{B}$ be the formal complection of $B$ at $b$, with (rigid) generic point $\eta$. Set $\what{\sE} := \sE \times_B \what{B}$, i.e., the formal completion of $\sE$ along $E = \sE_b$. One checks that $\what{\sE} \to \what{B}$ has semi-stable reduction. In particular, $\sE_\eta$ is smooth over $\eta$.
    
    Pick a point $\infty \in \IP^2$ and a line $L \subseteq \IP^2$ in general positions. Then projection from $\infty$ to $L$ defines a finite morphism $\psi : \what{\sE} \to \what{L} := L \times \what{B}$ of degree $3$. 
    There exist formal open subschemes $X \subseteq \what{L}$ and $X' \subseteq \what{\sE}$ such that $\psi$ restricts to an \'etale cover $X'_\eta \to X_\eta$, and $X_b$ contains $Q_i$'s. By standard Bertini-type arguments, one finds an open formal subscheme $U \subseteq X$ an open formal subscheme such that $U_b$ does not contain the images of $Q_i$'s.
    Then $X'_\eta$ is connected, whereas $U_\eta' := U \times_X X'_\eta$ is not. In fact, $U_\eta$ has $3$ connected components. 
\end{example}

\section{Pointwise criteria}
\label{sec pc}
	In this section, we prove the pointwise criteria for the crystallinity and the semi-stability.
	Namely, for a smooth rigid space $X_\eta$ that has good (resp. semi-stable) reduction, and a $p$-adic local system over $X_\eta$, the property of being crystalline (resp. semi-stable) with respect to a given semi-stable model can be checked by its restrictions on an effective subset $\mathcal{C}$ of classical points in $X_\eta$. 
	
	Throughout the section, we use $K$ to denote a discretely valued $p$-adic field with perfect residue field $k$ as before.
 We also recall that for a $p$-adic formal scheme $X$ over $\mathcal{O}_K$, we use $X_s$ to denote the reduced special fiber $(X_k)_\mathrm{red}$, and use $X_k$ to denote the mod $\pi$ fiber.
\subsection{Statement and first reductions}
\label{sub pc first reduction}
	
In the following, we recall that by a \emph{classical point} of a rigid space $X_\eta$, we mean a point $x_\eta\in X_\eta$ that locally defined by a maximal ideal of $A$, where $\Spa(A)\subset X_\eta$ is an affinoid neighborhood of $x_\eta$.
Moreover, given such a classical point, we use $\Gal_{x_\eta}$ to denote the absolute Galois group of the residue field $K(x_\eta)$ of $X_\eta$ at $x_\eta$, where $K(x_\eta)$ is a finite extension of $K$. 
As explained in \cref{prop: extend classical points}, if $X_\eta$ admits an integral model $X$ which is flat and topologically of finite type over $\sO_K$, then $x_\eta$ also admits an integral model $x \in X(\sO_{K(x_\eta)})$.

	To begin, we extend the notion of the effective set to the semi-stable reduction.
	\begin{definition}
		Let $X$ be a semi-stable formal scheme over $\sO_K$. We say that a subset $\sC \subseteq \sC(X)$ is effective if the restriction of $\sC$ to the smooth locus of $X$ is effective. 
	\end{definition}
	Now we can state the exact form of the statement we want to prove
		\begin{theorem}
		\label{thm:PC for semi-stable reduction}
		Let $X$ be a smooth (resp. semi-stable) $p$-adic formal scheme over $\mathcal{O}_K$, and let $\mathcal{C}\subset X_\eta$ be an effective subset of classical points.
		Then a $\mathbb{Z}_p$-local system $T$ over $X_\eta$ is crystalline (resp. semi-stable) with respect to $X$ if and only if the restriction $T|_{x_\eta}$ is crystalline (resp. semi-stable) for each $x_\eta\in \mathcal{C}$.
	\end{theorem}

\begin{proof}
By \Cref{cor:pull back to closed points}, it suffices to prove the ``if'' statement. 
First we notice that thanks to the recent purity result for semi-stable local systems in \cite[Thm.\ 1.1]{DLMS2}, since the set of Shilov points is contained in the good reduction locus, it suffices to assume $X$ is smooth, affine, and connected, with an unramified model $R_0$ as in \Cref{conv of smooth affine}.
By \Cref{thm:gluing of D functors}.\ref{thm:gluing of D functors crys/st}, we only need to show that the associated weak $F$-isocrystal $\mathcal{E}_{\crys,T}$ (resp. $\mathcal{E}_{\st,T}$) is of rank $d \colonequals \mathrm{rank\,}T$ everywhere. 
Here we note that by the purity result again,  we are allowed to shrink $X$ further by an open formal subscheme whenever needed.
In particular, we may assume the weak $F$-isocrystals above are actual $F$-isocrystals.
Moreover, since for any finite unramified extension of the base field $K'/K$ the canonical morphism $X_{\sO_{K'}} \to X$ is finite \'etale, \Cref{thm:pullback for D functors}.\ref{thm:pullback for D functors fet} allows us to replace $X$ by a connected component of $X_{\sO_{K'}}$. 
Therefore, we may assume that $X_s$ is geometrically connected.

Following the convention and notation as in \Cref{sub pril}, we let $\mathcal{O}_L$ be the complete discrete valuation ring with imperfect residue field, defined by the $p$-complete localization of $R$ at the ideal $\pi R$, with $L=L(X)$ being the fraction field of $\sO_L$.
There is then a natural map of $p$-adic formal scheme $\Spf(\mathcal{O}_L)\to X$ that induces a map of adic spaces $\Spa(L)\to X_\eta$, identifying $\Spa(L)$ as the Shilov point of $X_\eta$ (\cite[Prop.\ 2.2]{BH22}).

Next we note that since $T|_{x_\eta}$ is a crystalline and hence de Rham representation of $\Gal_{x_\eta}$ for some $x_\eta\in \mathcal{C}$, by Liu--Zhu's rigidity theorem of de Rham local systems (\cite{LZ17}; cf. \Cref{cor:LZ rigidity}), we know that $T$ is a de Rham local system over $X_\eta$.
In particular, the restriction $T|_{\Spa(L)}$ corresponds to a de Rham representation of $\Gal_L$, which by the work of Morita (\cite{Mor14}) and Ohkubo (\cite{Ohk13}) is potentially semi-stable.
Namely, there is a finite Galois extension $L''/L$ such that the restriction $T|_{\Spa(L'')}$ corresponds to a semi-stable representation of $\Gal_{L''}$.

Now thanks to \Cref{prop: tower} and \Cref{thm: pre-tower}, by shrinking $X$ and replacing $L''$ by a further finite extension if necessary, there exists a map of $p$-adic formal schemes $X''\to X$ which factors as a composite 
\[
X''=X_m\xrightarrow{g_m} \cdots \xrightarrow{g_2} X_1 \xrightarrow{g_1} X_0 \xrightarrow{g_0} X' \colonequals X_{\mathcal{O}_{K'}} \to X,
\]
such that 
\begin{itemize}
    \item the extension $L \into L''$ is identified with $L(X) \into L(X'')$;
    \item $K'$ is a finite Galois extension of $K$ and $X_{\sO_{K'}} \to X$ is the canonical map given by extension of scalars; 
    \item $g_0$ is finite \'etale, and for each $i > 0$, the map $g_i:X_i\to X_{i-1}$ is a primitive inseparable cover between objects in $\mathsf{C}^{\mathrm{sm}}_{K'}$ (cf. \Cref{def: purely inseparable cover}).
\end{itemize}

 Let us finish the proof of the pointwise criterion for crystalline (resp. semi-stable) local systems. 
 By the purity of semi-stable local systems in \cite[Thm.\ 1.1]{DLMS2}, we know the preimage $T|_{X''_\eta}$ is semi-stable with respect to $X''$. 
 Then by iteratively applying \cref{thm: effectivity} and \cref{thm:descend along primitive insep} below to each map $g_i$ for $i\geq 1$, we deduce that $T|_{X_{0, \eta}}$ is crystalline (resp. semi-stable). 
 In addition, by \Cref{thm:gluing of D functors}.\ref{thm:gluing of D functors pullbacks}, we know $T|_{X'_\eta}$ is crystalline (resp. semi-stable). 
 So it is left to descend the crystallinity (resp. the semi-stability) from $T|_{X'_\eta}$ to $T$.
 For the latter, we notice that $\cap_{x_\eta \in \sC} K(x_\eta)$ is a finite unramified extension of $K$: if not, then the image $\cup_{x_\eta \in \sC} \sI_{x_\eta}$ of the inertial subgroup of points along the surjection $I_X \to I_{K}$ would be contained in a proper open subgroup, contradicting our assumption that $\cup_{x_\eta \in \sC} \sI_{x_\eta}$ topologically generate $I_X$. 
 Therefore, we may conclude by \Cref{thm:descend along base extension}.\end{proof}

\subsection{Descending crystallinity and semi-stability}
\label{sub pc descend}
We prove the two descent results \cref{thm:descend along primitive insep} and \cref{thm:descend along base extension} on the crystallinity and the semi-stability used in the previous subsection. 

We start with a preparation about the Galois action on the associated $F$-isocrystal.
\begin{lemma}
	\label{lem:Galois on F-isoc}
    Let $K$ be a $p$-adic field, and let $Z$ be a smooth $p$-adic formal scheme over $\mathcal{O}_K$.
    Assume $g:Z'\to Z$ is a map satisfying either of the following two assumptions:
    \begin{enumerate}[label=\upshape{(\alph*)}]
    	\item\label{item: primitive assump a} It is a primitive inseparable cover between two smooth $p$-adic formal schemes over $\mathcal{O}_K$. 
    	\item\label{item: ramif ext assump b} There is a finite Galois and totally ramified extension $K'/K$ such that $Z'=Z_{\mathcal{O}_{K'}}$ and $g$ is the extension map. 
        \end{enumerate} In either case, the generic fiber $g_\eta: Z'_\eta \to Z_\eta$ is a finite Galois cover and we denote by $H$ the Galois group.
    Let $T\in \Loc_{\mathbb{Z}_p}(Z_\eta)$, and let $T'\colonequals g_\eta^{-1}T$ be the preimage.
    \begin{enumerate}[label=\upshape{(\roman*)}]
    	\item\label{lem:Galois on F-isoc general} There is a natural $H$-action on the $F$-isocrystal $\mathcal{E}_{\crys,T'} \in \Isoc^\varphi(Z'_{s,\crys})$ (resp. $\mathcal{E}_{\st,T'} \in \Isoc^\varphi((Z'_s, (0^\IN)^a )_\lcrys)$) that is functorial in $T$.
    	\item\label{lem:Galois on F-isoc affine} In case \ref{item: ramif ext assump b}, assume that $\Spf(R_0)$ is an unramified model of $Z$ (and hence $Z'$) as in \Cref{conv of smooth affine}.
    	Then the $H$-action in \ref{lem:Galois on F-isoc general} is compatible with the $H$-action on $D_{\crys,R_0}(T')$ and $D_{\st,R_0}(T')$ via \Cref{thm:gluing of D functors}.\ref{thm:gluing of D functors modules local formula}. 
     Moreover, $D_{*, R_0}(T) = D_{*, R_0}(T')^H$ for $*\in \{\st, \crys\}$. 
    \end{enumerate}
\end{lemma}

We remark that the common feature of assumptions \ref{item: primitive assump a} and \ref{item: ramif ext assump b} is that $g_s$ is a homeomorphism, so that it induces an equivalence $\Isoc^\varphi(Z_s) \simeq \Isoc^\varphi(Z'_s)$ by \cite[Thm.~4.6]{Ogu84}. 
The natural $H$-action on $D_{*, R_0}(T')$ will be explained in the proof.

\begin{proof}
	We first notice that for any element $h\in H=\Gal(Z'_\eta/Z_\eta)$, we have a natural identity of maps $h_\eta\circ g_\eta=g_\eta$.
	Since the local system $T'$ is defined as the preimage, we have the following isomorphisms as a part of the descent data:
	\[
	\sigma_{h,\eta}:h_\eta^{-1}T'\xrightarrow{\sim} T'.
	\]
	So by applying the crystalline Riemann--Hilbert functor, we get an isomorphism of $F$-isocrystals in $\Isoc^\varphi(X'_{s,\crys})$ that is compatible with the group structure on $H$
	\[
	\sigma_{h,s} \colonequals \mathcal{E}_{\crys, \sigma_{h,\eta}} \colon \mathcal{E}_{\crys,h_\eta^{-1}T'} \xrightarrow{\sim} \mathcal{E}_{\crys,T'}.
	\]
	On the other hand, by \cref{lem: simple observation}\ref{rmk: Galois action on integral model}, the Galois group $H=\Gal(Z'_\eta/Z_\eta)$ naturally acts on $Z'$ in the category of $p$-adic formal schemes over $Z$.
	So by identifying the elements $h\in H$ with the automorphisms $h:Z'\to Z'$ and by the base change thoerem of the crystalline Riemann--Hilbert functor in \Cref{thm:gluing of D functors}.\ref{thm:gluing of D functors pullbacks}, we get an isomorphism of $F$-isocrystals in $\Isoc^\varphi(Z'_{s,\crys})$:
	\[
	h_s^* \mathcal{E}_{\crys,T'} \simeq \mathcal{E}_{\crys,h_\eta^{-1}T'}.
	\]
	Notice that by assumption the map of the reduced special fiber $g_s:Z'_s\to Z_s$ is purely inseparable or equal to the identity.
	Hence the induced action of $H$ on $Z'_s$ is trivial, and we have identifications $h_s^* \mathcal{E}_{\crys,T'} = \mathrm{id}_{Z'_s}^* \mathcal{E}_{\crys,T'}= \mathcal{E}_{\crys,T'}$.
	As a consequence, $\sigma_{h,s}$ a natural action of $H$ on $\mathcal{E}_{\crys,T'} \in \Isoc^\varphi(Z'_{s,\crys})$:
	\begin{equation}
		\label{eq:group action on F-isoc}
			\sigma_{h,s}: \mathcal{E}_{\crys,T'} \simeq \mathcal{E}_{\crys,h_\eta^{-1}T'} \xrightarrow{\sim} \mathcal{E}_{\crys,T'}.
	\end{equation}

    Now we focus on the setup of \ref{lem:Galois on F-isoc affine}. We note that by taking the evaluation at the pro-PD-thickening $(R_0,R_k)$, the above induces an action of $H$ on $D_{\crys,R_0}(T')$.
        To see the $H$-action above is compatible with the $H$-action on $D_{\crys,R_0}(T')$ via \Cref{thm:gluing of D functors}.\ref{thm:gluing of D functors modules local formula}, we recall the $H$-action on the latter.
	For simplicity, we assume $Z=\Spf(R)$ is the base change of of $Z_0=\Spf(R_0)$.
	As in \Cref{const:Galois group of rigid space}, we let $\tilde{Z}=\Spf(S)$ be the integral perfectoid formal scheme over $R_0$ defined by the maximal connected pro-finite-\'etale cover of $Z'_\eta$ (and hence of $Z_\eta$).
	Then we have 
	\[
	D_{\crys,R_0}(T')\colonequals \bigl(T(\tilde{Z}_\eta)\otimes_{\mathbb{Z}_p} \OB_{\crys,R_0}(\tilde{Z}_\eta) \bigr)^{G_{Z'_\eta}},
	\]
	where the $H$-action is induced by the action of $G_{Z_\eta}$ on $T(\tilde{Z}_\eta)\otimes_{\mathbb{Z}_p} \OB_{\crys,R_0}(\tilde{Z}_\eta)$.
	So the compatibility of the two $H$-action follows from the translation between the Galois action and the descent data of pro-\'etale sheaves. 
 The claim that $D_{\crys, R_0}(T) = D_{\crys, R_0}(T')^H$ is clear.
	
	Finally, for $\mathcal{E}_{\st,T'}$, the proof works the same and we do not repeat it here.
\end{proof}

Another preparation is about the descent of $\mathbb{B}_\crys$-vector bundles.
In the following, for a given smooth rigid space $Z_\eta$, we let $\mathbb{B}_{\crys,Z_\eta}$ be the horizontal period sheaf $\mathbb{B}_\crys$ over the pro\'etale site $Z_{\eta,\pe}$ introduced in \cref{sec period sheaves}.
\begin{definition}
	Let $g:Z'_\eta \to Z_\eta$ be a finite \'etale cover of smooth affinoid rigid spaces over $K$.
	\begin{enumerate}
		\item Let $\sM'$ be a $\IB_{\crys, Z'_\eta}$-vector bundle. 
		Then a \emph{descent datum} for $\sM'$ with respect to $Z_\eta$ is an isomorphism $\epsilon : p_{1,\mathbb{B}_\crys}^* \sM' \sto p_{2,\mathbb{B}_\crys}^* \sM'$, where $p_i : Z'_\eta \times_{Z_\eta} Z'_\eta \to Z'_\eta$ are the two projections, such that $\epsilon$ satisfies the cocycle condition. 
		\item Define $\mathrm{Des}(\IB_{\crys, Z'_\eta})_{Z_\eta}$ to be the category whose objects are the pairs $(\sM', \epsilon)$ where $\sM' \in \Vect(\IB_{\crys, Z_\eta})$ and $\epsilon$ is a descent datum on $\sM'$ with respect to $Z_\eta$, and whose morphisms are defined in the natural way.
	\end{enumerate}
\end{definition}
Note that in the special case when $Z'_\eta\to Z_\eta$ is a finite Galois cover with the Galois group $H$, a descent datum can be translated as isomorphisms of $\mathbb{B}_{\crys,Z'_\eta}$-vector bundles
\[
\sigma_h\colon h_{\mathbb{B}_\crys}^* \sM'  \simeq \sM',~\forall h\in H,
\]
that are compatible with the group structure on $H$ (cf. \cite[\href{https://stacks.math.columbia.edu/tag/0CDQ}{Tag 0CDQ}]{stacks-project}).
\begin{lemma}
\label{lem: fully faithful}
Let $g: Z'_\eta \to Z_\eta$ be a finite \'etale cover of smooth rigid spaces over $K$.
The pullback functor $$ g_{\mathbb{B}_\crys}^* : \Vect(\IB_{\crys, Z_\eta}) \to \mathrm{Des}(\IB_{\crys, Z'_\eta})_{Z_\eta}$$ is fully faithful. 
\end{lemma}
\begin{proof}
By considering the internal Homs, we reduce the problem to showing that if $\sM \in \mathrm{Vect}(\IB_{\crys, Z_\eta})$, $\sM' = g_{\mathbb{B}_\crys}^* \sM$ and $\epsilon$ is the canonical descent datum, then an element $m' \in \Gamma(Z'_{\eta,\pe}, \sM')$ comes from $\Gamma(Z_{\eta,\pe}, \sM)$ if and only if $\epsilon(p_{1,\mathbb{B}_\crys}^*(m')) = p_{2,\mathbb{B}_\crys}^*(m')$. 

Let $\{ U_\alpha = \Spa(R_\alpha, R^+_\alpha) \}_\alpha$ be an affinoid perfectoid cover of $Z_\eta$ in $Z_{\eta,\text{pro\'et}}$ which trivializes $\sM$. 
Let $U'_\alpha \colonequals g_{\mathbb{B}_\crys}^* U_\alpha$ and let $m'_\alpha$ be the restriction of $m'$ to $U'_\alpha$. 
Now using the assumption that the map $g$ itself is finite \'etale, we see $\{ U'_\alpha \}$ is also a set of affinoid perfectoid objects in $Z_{\eta,\text{pro\'et}}$ that covers $\{ U_\alpha \}$ (cf. \cite[Lem.~4.5(i)]{Sch13}). 
So by the sheaf property of $\IB_{\crys, Z_\eta}$ on $Z_{\eta,\text{pro\'et}}$, for each index $\alpha$, the sequence  \begin{equation}
	0 \to {\IB_{\crys}(R_\alpha, R_\alpha^+)} \to {\IB_{\crys}(R'_\alpha, R'^+_\alpha)} \rightrightarrows {\IB_{\crys}(R'_\alpha \widehat{\tensor}_{R_\alpha} R'_\alpha, (R'_\alpha \widehat{\tensor}_{R_\alpha} R'_\alpha)^+)}
\end{equation}
is an equalizer diagram. 
By the local freeness of $\sM$ over $\IB_\crys$, the sequence remains exact after we tensor with $\sM$. 
This implies that 
\begin{equation}
	0 \to {\sM(U_\alpha)} \to {\sM(U'_\alpha)} \rightrightarrows {\sM(U'_\alpha \times_{U_\alpha} U'_\alpha)}
\end{equation}
is also an equalizer diagram. 

Now, assume the element $m'$ satisfies the condition $\epsilon(p_{1,\mathbb{B}_\crys}^*(m')) = p_{2,\mathbb{B}_\crys}^*(m')$.
Then the element $m'_\alpha$ lies in the equalizer of the double arrows, and in particular is the image of some $m_\alpha \in \sM(U_\alpha)$. 
Conversely, if we start with an element $m_\alpha\in \sM(U_\alpha)$, then by the short exact sequence again, we see its image $m'_\alpha \colonequals g_{\mathbb{B}_\crys}^*m_\alpha$ satisfies the equation $\epsilon(p_{1,\mathbb{B}_\crys}^*(m'_\alpha)) = p_{2,\mathbb{B}_\crys}^*(m'_\alpha)$.
Combining the observations above, we see the functor from $\IB_{\crys,Z_\eta}$-locally free sheaves to the category of descent data over $Z'_\pe$ is fully faithful when restricted to each $U_\alpha$.
Finally, since the gluing of $m'_\alpha$ for all $\alpha$ can be checked on $\{ U'_\alpha \}$, the claim for $m$ follows from the sheaf property of $\sM$ and $\sM'$ and the above naturally globalizes to entire $Z_\eta$. 
\end{proof}

Assembling the above preparations, we now prove the first descent result.
\begin{theorem}
	\label{thm:descend along primitive insep}
	Let $K$ be a $p$-adic field, and let $g:Z'\to Z$ be a primitive inseparable cover between connected smooth $p$-adic formal schemes over $\mathcal{O}_K$.
	Let $T\in \Loc_{\mathbb{Z}_p}(Z_\eta)$, and let $T'\colonequals g_\eta^{-1}T$ be the preimage.
	Assume we have
	\begin{itemize}
		\item the local system $T'$ is semi-stable with respect to $Z'$;
		\item there is a classical point $x_\eta \in Z_\eta$ such that $g_\eta^{-1}(x_\eta)$ is one point that is ramified over $x_\eta$, and the restriction $T|_{x_\eta}$ is a crystalline (resp. semi-stable) representation of $\Gal_{K(x_\eta)}$.
	\end{itemize}
    Then the local system $T$ is crystalline (resp. semi-stable) with respect to $Z$.
\end{theorem}
\begin{proof}
    As the statement is local with respect to the Zariski topology of $Z$ and $Z'$, we assume that both $Z$ and $Z'$ are affine and there is an unramified model $Z_0'$ of $Z'$ with a lift of Frobenius $\varphi_{Z_0'}$ as in \cref{conv of smooth affine}.
     We separate the proof into two steps. In step 1, we show that the natural Galois action on the $F$-isocrystal $\mathcal{E}_{\st,T'}$ is trivial. In step 2, we show that $\mathcal{E}_{\st,T'}$ descends to an object in $\Isoc^\varphi((Z_s, (0^\IN)^a)_\lcrys)$ and the descent is associated with the original local system $T$.

		\textbf{Step 1: }
		Assume the local system $T$ is of rank $d$.
		By \Cref{lem:Galois on F-isoc}, there is a natural action of the Galois group $H=\Gal(Z'_\eta/Z_\eta)$ on the log $F$-isocrystal $\mathcal{E}_{\st,T'}\in \Isoc^\varphi((Z'_s,(0^\IN)^a)_\lcrys)$.
		We first claim that the $H$ action is trivial.

		To see this, we let $x:\Spf(\mathcal{O}_{K(x_\eta)})\to Z$ be the point of the $p$-adic formal scheme $Z$ whose generic fiber is $x_\eta$.
		We let $y_\eta\colonequals g_\eta^{-1}(x_\eta)$ be the preimage of $x_\eta$, which by assumption is a single point with $K(y_\eta)$ a degree $p$ ramified Galois extension of $K(x_\eta)$.
		In particular the Galois group of the extension $K(y_\eta)/K(x_\eta)$ is the same as $H$.
		We let $y \colonequals \Spf(\mathcal{O}_{K(y_\eta)})\to Z'$ be the induced map of $p$-adic formal schemes.
		Then by construction, the $H$-action on the generic fiber naturally induces an action on $y$, and we get an $H$-equivariant map of $p$-adic formal schemes $y\rightarrow Z'$ that is compatible with $x\rightarrow Z$ and the generic fibers, namely the commutative diagram:
		\[
		\begin{tikzcd}
			\Spf(\mathcal{O}_{K(y_\eta)}) \arrow[rr, "{H\text{-equivariant}}", "y"'] \arrow[d]  && Z' \arrow[d, "g"]\\
			\Spf(\mathcal{O}_{K(x_\eta)}) \arrow[rr,"x"'] && Z.
		\end{tikzcd}
		\]
		
		Now we apply \Cref{thm:gluing of D functors}.\ref{thm:gluing of D functors crys/st}, to get an $H$-equivariant isomorphism
		\begin{equation}
			\label{eq:descend iso eq}
					y_s^*\mathcal{E}_{\st,T'} \simeq \mathcal{E}_{\st,y_\eta^{-1}T'},
		\end{equation}
		where the latter is an $F$-isocrystal over $(k,M_k)_\lcrys$.
		Moreover by \Cref{thm:gluing of D functors}.\ref{thm:gluing of D functors modules local formula}, the right hand side of \Cref{eq:descend iso eq} can be identified with the usual $K_0$-vector space $D_\st(V_{y_\eta})$ with Frobenius and nilpotent operators, where $V_{y_\eta}=y_\eta^{-1}T'(\Spa(C,\mathcal{O}_C))$ is the  $\Gal_{K(y_\eta)}$-representation over $\mathbb{Z}_p$ that corresponds to the local system $y_\eta^{-1}T'$.
		Here by the commutative diagram above, the local system $y_\eta^{-1}T'$ is equal to the restriction of $T|_{x_\eta}$ along the finite Galois cover $y_\eta\to x_\eta$.
		On the other hand, by the assumption that the restriction $T|_{x_\eta}$ is semi-stable, for each $h\in H$, the action of $h$ on $D_\st(V_{y_\eta})$ is trivial.
		By \Cref{eq:descend iso eq}, the isomorphism $\sigma_{h,s}: \mathcal{E}_{\st,T'} \to \mathcal{E}_{\st,T'}$ is an endomorphism of an the log $F$-isocrystal $\sE_{\st, T'}$ whose pullback along $y_s:\Spec(k)\to Z'_s$ is the identity map.
		This is only possible when the map $\sigma_{h,s}$ itself is the identity. To see this,  we can translate $\sigma_{h,s}$ to a Frobenius and monodromy equivariant isomorphism of flat connections over $Z'_{0,\eta}$ such that its pullback to a point is the identity (\Cref{prop:equiv def of crys vs pullback}).
		Since the category of flat connections over $Z'_{0,\eta}$ is abelian, both flat connections $\ker(\mathrm{id}-\sigma_{h,s})$ and $\mathrm{coker}(\mathrm{id}-\sigma_{h,s})  $ vanish when restricted to a point, and thus vanish themselves.
		Hence the automorphism $\sigma_{h,s}$ on $\mathcal{E}_{\st,T'}$ (and hence on $\mathcal{E}_{\crys,T'}$) is equal to the identity, for every $h\in H$.\footnote{Here the triviality of the action $\sigma_{h,s}$ can also be deduced from \cite[Thm~4.1]{Ogu84}.} 
		
		Suppose now that the restriction $T|_{x_\eta}$ is crystalline. Then the monodromy operator on $D_\st(V_{y_\eta})$ acts as zero.
		Notice that the monodromy operator commutes with the flat connection under the identification in \Cref{equiv def log isoc}, and is compatible with the pullback isomorphism in \Cref{eq:descend iso eq}.
		So by the fact that the category of flat connections over $Z'_{0,\eta}$ is abelian and using the same idea last paragraph, we see the monodromy operator on $\mathcal{E}_{\st,T'}(Z'_0, Z'_k, (0^\IN)^a)$ is trivial as well.
  Therefore, $\sE_{\st, T'}$ is the image of $\sE_{\crys, T'}$ under the morphism $\Isoc^\varphi(Z_{k, \crys}) \to \Isoc^\varphi((Z_k, (0^\IN)^a)_\lcrys)$ induced by the natural forgetful functor $(Z_k, (0^\IN)^a)_\lcrys \to Z_{k, \crys}$, so that $\IB_\crys(\sE_{\st, T'}) \simeq \IB_{\crys}(\sE_{\crys, T'})$ (cf. \cref{rmk: Bcrys and forgetful functor}). 

  		\textbf{Step 2: }
	We then claim that both $\mathcal{E}_{\crys,T'}$ and $\mathcal{E}_{\st,T'}$ descend to $Z'_s$: since $Z'_s\to Z_s$ is a purely inseparable map of varieties in characteristic $p$, we know from \cite[Thm.\ 4.6]{Ogu84}
		\footnote{More precisely, the self fiber products $(Z'_s)^n$ of $Z'_s$ over $Z_s$ are infinitesimal extensions of $Z'_s$, and in particular the pullback functor along $(Z'_s)^n\to Z'_s$ induces an equivalence on $F$-isocrystals (\cite[Rmk.\ 2.7.2]{Ogu84}). 
			Hence the category of $F$-isocrystals of the \v{C}ech nerve of the proper surjective map $Z'_s\to Z_s$ is equivalent to the constant category $\Isoc^\varphi(Z'_{s,\crys})$.} 
		that the pullback functor induces an equivalence of categories
		\[
		\Isoc^\varphi(Z_{s,\crys}) \xrightarrow{\sim} \Isoc^\varphi(Z'_{s,\crys}).
		\]
		Moreover, thanks to \Cref{equiv def log isoc}, we have an analogous equivalence
		\[
		\Isoc^\varphi((Z_s, (0^\IN)^a)_\lcrys) \xrightarrow{\sim} \Isoc^\varphi((Z'_s,(0^\IN)^a)_\lcrys).
		\]
        As a consequence, we obtain $\mathcal{F}_{\crys,T}\in \Isoc^\varphi(Z_{s,\crys})$ and $\mathcal{F}_{\st,T}\in \Isoc^\varphi((Z_s, (0^\IN)^a)_\lcrys)$ together with isomorphisms $f_s^* \mathcal{F}_{\crys, T} \simeq \sE_{\crys, T}$ and $f_s^* \mathcal{F}_{\st, T} \simeq \sE_{\st, T}$, which are unique up to unique isomorphism. 
        Here we equip $\mathcal{F}_{\crys,T}$ and $\mathcal{F}_{\st,T}$ with trivial $H$-actions.
		
		We then show that the $F$-isocrystal $\mathcal{F}_{\st,T}$ in $(Z_s, (0^\IN)^a)_\lcrys$ is in fact naturally associated to the local system $T$, which by \Cref{def crys/st Faltings} will imply that the local system $T$ is semi-stable with respect to $Z$. 
        We do this by descending the isomorphism of vector bundles over $\IB_{\crys, Z'_\eta}$
        $$ \alpha_{\st, T'}^{\nabla = 0, N = 0} : \IB_{\crys, Z'_\eta}(\sE_{\st, T'}) \stackrel{\sim}{\to} T' \tensor_{\IZ_p} \IB_{\crys, Z'_\eta} $$
        to an isomorphism $\IB_{\crys, Z_\eta}(\sF_{\st, T}) \simeq T \tensor_{\IZ_p} \IB_{\crys, Z_\eta}$ of vector bundles over $\IB_{\crys, Z_\eta}$. 
        By \cref{lem: fully faithful}, this amounts to comparing the following two descent data: 
        \begin{itemize}
            \item The descent data $\delta_{h, \eta}: h_\eta^* \bigr( T' \tensor_{\IZ_p} \IB_{\crys, Z'_\eta} \bigr) \stackrel{\sim}{\to} T' \tensor_{\IZ_p} \IB_{\crys, Z'_\eta}$ induced by $f_\eta^{-1} T \simeq T'$. 
            \item The descent data 
            $\delta_{h, s} : h_\eta^* \IB_{\crys}(\sE_{\st, T'}) \stackrel{\sim}{\to} \IB_{\crys}(\sE_{\st, T'})$ induced by $f_s^*\sF_{\st, T} \simeq \sE_{\st, T'}$. 
        \end{itemize}
        More precisely, what we need to show is the commutativity of the following diagram: 
        \begin{equation}
        \label{eqn: descent diagram to show}
            \begin{tikzcd}
	{h_\eta^* \IB_{\crys}(\sE_{\st, T'}) } & {h_\eta^* \bigr( T' \tensor_{\IZ_p} \IB_{\crys, Z'_\eta} \bigr)} \\
	{\IB_{\crys}(\sE_{\st, T'}) } & { T' \tensor_{\IZ_p} \IB_{\crys, Z'_\eta}}
	\arrow["{h_\eta^* \alpha_{\st, T'}^{\nabla = 0, N=0}}", from=1-1, to=1-2]
	\arrow["{\delta_{h, s}}"', from=1-1, to=2-1]
	\arrow["{\delta_{h, \eta}}", from=1-2, to=2-2]
	\arrow["{\alpha_{\st, T'}^{\nabla = 0, N=0}}", from=2-1, to=2-2]
            \end{tikzcd}
        \end{equation}

        To begin, we remind the reader that \cref{lem: pullback formula on crystalline side} and \cref{lem: pullback formula on etale side} provide us with the following identifications  
        $$ \IB_{\crys}(h_s^* \sE_{\st, T'}) \simeq h_\eta^*\IB_{\crys}(\sE_{\st, T'}) \textrm{ and }h_\eta^{-1}(T') \tensor_{\IZ_p} \IB_{\crys, Z'_\eta} \simeq h_\eta^*(T' \tensor_{\IZ_p} \IB_{\crys, Z'_\eta}), $$
        which we shall use implicitly for the content below.

        Now, let us describe the two sets of descent data. First we shall temporarily suppress the assumption that $f_s$ is purely inseparable and clarify some tautology. Let $\sigma_{h, \eta} : h_\eta^* T' \stackrel{\sim}{\to} T'$ be the descent data given by $f_\eta^{-1} (T) \simeq T'$. 
        Then $\delta_{h, \eta}$ is nothing but $\sigma_{h, \eta} \tensor_{\IZ_p} \IB_{\crys, Z'_\eta}$. 
        Similarly, let $\nu_{h, s} : h_s^* \sE_{\st, T'} \simeq \sE_{\st, T'}$ be the descent data given by $f_s^* \sF_{\st, T} \simeq \sE_{\st, T'}$. 
        Then $\delta_{h, s}$ is characterized as the unique isomorphism such that the following diagram commutes (cf. \cref{rmk: base change diagram}): 
        \begin{equation}
        \label{eqn: define delta_hs}
            \begin{tikzcd}
	{ \IB_{\crys}(h_s^* \sE_{\st, T'}) } & {h_\eta^{-1} (T') \tensor_{\IZ_p} \IB_{\crys, Z'_\eta} } \\
	{\IB_{\crys}(\sE_{\st, T'}) } & { T' \tensor_{\IZ_p} \IB_{\crys, Z'_\eta}}
	\arrow["{ h_\eta^* \alpha_{\st, T'}^{\nabla = 0, N=0}  }", from=1-1, to=1-2]
	\arrow["{\IB_{\crys}(\nu_{h, s})}"', from=1-1, to=2-1]
	\arrow["{\delta_{h, s}}", from=1-2, to=2-2]
	\arrow["{\alpha_{\st, T'}^{\nabla = 0, N=0}}", from=2-1, to=2-2]
            \end{tikzcd}
        \end{equation}
   
	    Recall that $\sigma_{h, s} = \mathcal{E}_{\st,\sigma_{h,\eta}}  : \mathcal{E}_{\st,h_\eta^{-1}T'} \xrightarrow{\sim} \mathcal{E}_{\st,T'}$ is the map induced by the functoriality of the crystalline Riemann--Hilbert functor. The functoriality of the injection $\alpha_{\st, T}$ defined in \cref{thm: D functors}.\ref{thm: D functors log crys} gives us a commutative diagram (cf. \cref{thm:equivalent def of crys and st}) 
	    \begin{equation}
	    	\label{eq:descend diagram associated}
	    	\begin{tikzcd}
	    		    	{\mathbb{B}_{\crys,Z'_\eta}(\mathcal{E}_{\st,h_\eta^{-1}T'})} && {h_\eta^{-1}T'\otimes_{\mathbb{Z}_p}\mathbb{B}_{\crys,Z'_\eta}} \\
	{\mathbb{B}_{\crys,Z'_\eta}(\mathcal{E}_{\st,T'})} && {T'\otimes_{\mathbb{Z}_p}\mathbb{B}_{\crys,Z'_\eta}.}
	\arrow["{\alpha_{\st,h_\eta^{-1}T'}^{\nabla=0, N=0}}", from=1-1, to=1-3]
	\arrow["{\mathbb{B}_\crys(\sigma_{h, s})}"', from=1-1, to=2-1]
	\arrow["{\sigma_{h,\eta}\otimes {\mathbb{B}_{\crys,Z'_\eta}}}", from=1-3, to=2-3]
	\arrow["{\alpha_{\st,T'}^{\nabla=0, N=0}}",  from=2-1, to=2-3]
	    	\end{tikzcd}
	    \end{equation}
     By comparing diagrams (\ref{eqn: descent diagram to show}) with (\ref{eqn: define delta_hs}) and (\ref{eq:descend diagram associated}), as well as the diagram in \cref{rmk: base change diagram}, we deduce that the commutativity of (\ref{eqn: descent diagram to show}) is equivalent to the equality $\sigma_{h, s} = \nu_{h, s}$. 

     Finally, we make use of the assumption that $f_s$ is purely inseparable, which forces $h_s$ and $\nu_{h, s}$ to be the identity maps, and turns the collection $\{ \sigma_{h, s} \}_{h \in H}$ into an $H$-action on $\sE_{\st, T'}$. 
     Therefore, we reduce to showing that the $H$-action on $\sE_{\st, T'}$ is trivial, which is precisely what we showed in Step 1. Now we have shown that $\sF_{\st, T}$ is associated to $T$, so that $T$ is semi-stable. If we additionally assume that $T|_{x_\eta}$ is crystalline for some $x_\eta \in Z_\eta$, then by the last paragraph of Step 1, $\sF_{\crys, T}$ is associated to $T$ and hence $T$ is crystalline. 
\end{proof}

We then consider the another descent result along an extension of the base ring. 
\begin{theorem}
	\label{thm:descend along base extension}
	Let $K'/K$ be a finite Galois and totally ramified extension of $p$-adic field, let $Z$ be a connected smooth $p$-adic formal schemes over $\mathcal{O}_K$, and let $g:Z_{\mathcal{O}_{K'}} \to Z$ be the extension map.
	Let $T\in \Loc_{\mathbb{Z}_p}(Z_\eta)$, and let $T'\colonequals g_\eta^{-1}T$ be the preimage.
	Assume we have
	\begin{enumerate}[label=\upshape{(\alph*)}]
	\item\label{item: descend thm assump a} the local system $T'$ is semi-stable with respect to $Z'$;
    \item\label{item: descend thm assump b} there exists a subset $\sC$ of classical points in $Z_\eta$ such that the intersection $\cap_{x_\eta \in \sC} K(x_\eta)$ is an unramified extension over $K$ and the restriction $T|_{x_\eta}$ is a crystalline (resp. semi-stable) representation of $\Gal_{K(x_\eta)}$. 
	\end{enumerate}
	Then the local system $T$ is crystalline (resp. semi-stable) with respect to $Z$.
\end{theorem}

\begin{proof}
	We let $H$ be the Galois group $\Gal(K'/K)$, and assume $T$ is of rank $d$. 
    Apply \cref{lem:Galois on F-isoc} under the assumption \ref{item: ramif ext assump b} in loc.\ cit.. 
    Then we know that for both $*\in \{\crys, \st\}$, there is a natural $H$-action on $\sE_{*, T'}$, and $\sE_{*, T}$ is the $H$-invariant part $\sE_{*, T'}^H$ of $\sE_{*, T'}$. The latter statement can be checked Zariski-locally on $Z$, and follows from \cref{lem:Galois on F-isoc}.\ref{lem:Galois on F-isoc affine}.

    We first show that $T$ is semi-stable, for which we may assume that $Z$ is affine. Using \cref{thm:rank max implies alpha is iso} and \cref{lem:Galois on F-isoc}.\ref{lem:Galois on F-isoc affine}, we reduce to showing that the action of $H$ on $\sE_{\st, T'}$ is trivial. 
    As before, we think of $\sE_{\st, T'}$ as an F-isocrystal on $Z_s = Z'_s$ with a nil endomorphism $N$. 
    Recall that we are allowed to replace $K$ by a finite unramified extension, so that we may assume that  $\cap_{x_\eta \in \sC} K(x_\eta) = K$, or equivalently, $\Gal_K$ is generated by $\cup_{x_\eta \in \sC} \Gal_{K(x_\eta)}$. 
    By \cref{lem: split Galois extensions}, for each closed point $y_\eta \in g_\eta^{-1}(x_\eta)$, $K(y_\eta)$ is a finite Galois extension of $K(x_\eta$) and $K' \tensor_K K(x_\eta)$ is the product of such $K(y_\eta)$'s. 
    Note that an element in $\Gal(K(y_\eta)/K(x_\eta))$ naturally restricts to one in $H = \Gal(K' / K)$. One quickly deduces that 
    \begin{equation}
        \label{eqn: generated H}
     \textrm{$H$ is generated by }\bigcup_{x_\eta \in \sC} \bigcup_{y_\eta \in g^{-1}_\eta(x_\eta)} \mathrm{im}(\Gal(K(y_\eta)/K(x_\eta)) \to H).
     \end{equation}
    Now, for each $x_\eta \in \sC$ and $y_\eta \in g^{-1}_\eta(x_\eta)$, the assumption that $T|_{x_\eta}$ is a semi-stable representation of $\Gal_{K(x_\eta)}$ implies that the residual $\Gal(K(y_\eta)/K(x_\eta))$-action on $D_\st(T|_{y_\eta})$ is trivial. As $\sE_{\st, T'}$ is already semi-stale, the base change theorem \cref{thm:gluing of D functors}.\ref{thm:gluing of D functors pullbacks} implies that the image of $\Gal(K(y_\eta)/K(x_\eta))$ in $H$ acts trivially on $\sE_{\st, T'}$. Hence the conclusion follows from (\ref{eqn: generated H}). 

    Finally, we need to argue that if in assumption \ref{item: descend thm assump b} we additionally assume that $T|_{x_\eta}$ is crystalline for all $x_\eta \in \sC$, then $T$ is moreover crystalline. Just as in the proof \cref{thm:descend along primitive insep}, this amounts to showing that $N = 0$, which one may check by restricting to any point in $\sC$. 
\end{proof}

\begin{example}
    \label{counterexample of Liu-Zhu}
    The example in \cref{ex: basic example} gives a prototypical counterexample for the naive analogue of Liu-Zhu's rigidity theorem, when the base is a smooth scheme over $\sO_K$. 
    Namely, consider $X = \Spf(\sO_K \< t^\pm \>)$ and the endomorphism $f$ on $X$ defined by $t \mapsto t^p$. Then we consider the local system $T$ given by $f_* \IZ_p$. 
    Note that for each classical point $x_\eta$ on $X_\eta$, $T|_{x_\eta}$ is crystalline, or equivalently, semi-stable\footnote{It is a coincidence when the inertia group acts through a finite quotient. See Prop.~7.17 of Fontaine and Ouyang's notes available at \url{www.imo.universite-paris-saclay.fr/~fontaine/galoisrep.pdf}}, if and only if $f_\eta^{-1}(x_\eta)$ is a union of $p$ distinct point each defined over $K(x_\eta)$. 
    One can check by hand that (a) for every closed point $x_s \in X_s$, $T$ is crystalline or semi-stale at some classical point $x_\eta \in X_\eta$ which specializes to $x_s$; and (b) for every $x_\eta \in X_\eta$ where $T$ is crystalline or semi-stable, $T$ remains crystalline or semi-stable at any $x'_\eta \in X_\eta$ which is sufficiently $p$-adically close to $x_\eta$. Nonetheless, $T$ is neither crystalline nor semi-stable. 
\end{example}

\subsection{Relative semi-stable comparison}
\label{sub application}
Finally, we prove the relative semi-stable comparison theorems in \cref{sec: intro rel Cst}. 
For the reader's convenience, we recall the statement below.

\begin{theorem}
    \label{thm:relative Cst}
    Let $X$ be a semi-stable $p$-adic formal scheme over $\sO_K$ with the standard log structure $M_X$, let $(Y, M_Y)$ be a fine and saturated $p$-adic formal log scheme over $\sO_K$.
    Let $f:(Y, M_Y) \to (X, M_X)$ be a proper and log smooth morphism with Cartier type mod $\pi$ reduction, and let $f_\eta$ be the induced morphism between generic fibers.
    
    Assume one of the following conditions is true. 
    \begin{enumerate}[label=\upshape{(\alph*)}]
        \item\label{thm: rel CN17} The map $f$ is algebrizable.
        \item\label{thm: rel CK19} The underlying map between formal schemes of $f$ is pointwise weakly semi-stable. 
    \end{enumerate}
    Then the higher direct image $R^i f_{\eta, \ket *} \IQ_p$ is a semi-stable local system over $X_\eta$.  
\end{theorem}
\begin{proof}
    By \cref{intro:thm pc}, it suffices to show that $T \colonequals R^i f_{\eta, \ket *} \IQ_p$ is a $p$-adic local system and for each classical point $x_\eta \in X_\eta$, $T|_{x_\eta}$ is semi-stable. Let $x \in X(\sO_{K(x_\eta)})$ be its integral extension and let $\overline{x}_\eta$ be a geometric point over $x_\eta$. 

    First, we note that under assumption \ref{thm: rel CK19}, since $M_Y|_{Y_\eta}$ is trivial, $Y_\eta$ is proper and smooth over $X_\eta$ and $T$ is the usual etale cohomology $R^i f_{\eta, \et*} \IQ_p$. In particular, by \cite[Thm.\ 10.5.1]{SW20} and the proper base change theorem \cite[Thm.\ 4.1.1]{Hub96}, $T$ is a local system and $T|_{x_\eta}$ can be identified with the $\Gal_{K(x_\eta)}$-representation $\mathrm{H}^i_\et(Y_{\overline{x}_\eta}, \IQ_p)$. 
    Then by \cite[Thm.~9.5]{CK19}, $\mathrm{H}^i_\et(Y_{\overline{x}_\eta}, \IQ_p)$ is semi-stable. 
    
    The above argument applies in situation \ref{thm: rel CN17} as well with a little modification. 
    As the map $f$ is assumed to be algebraic, we may apply \cite[Thm.~13.1]{Nak17} to deduce that $T$ is indeed a local system. 
    Then, the proper base change theorem in the logarithmic setting (\cite[Thm.~5.1]{Nak97}) allows us to identify $T|_{x_\eta}$ with the $\Gal_{K(x_\eta)}$-representation $\mathrm{H}^i_\ket((Y, M_Y)_{\overline{x}_\eta}, \IQ_p)$. 
    As the log structure on $x_\eta$ is trivial, the assumption that $(Y_{x_\eta}, M_{Y}|_{Y_{x_\eta}})$ being log smooth implies that $Y_{x_\eta}$ is smooth over $x_\eta$ and the log structure $M_{Y}|_{Y_{x_\eta}}$ is given by a normal crossing divisor (cf. \cite[Cor.~3.3.11]{DLLZ1}). 
    Then, by Cor.~6.3.4 in \textit{loc. cit.}, there is a canonial isomorphism $$\mathrm{H}^i_\ket((Y_{\overline{x}_\eta}, M_Y|_{Y_{\overline{x}_\eta}}), \IQ_p) \simeq \mathrm{H}^i_\et(U_{\overline{x}_\eta}, \IQ_p), $$
    where $U_{\overline{x}_\eta}$ is the open subvariety of $Y_{\overline{x}_\eta}$ on which the log structure is trivial (or equivalently, the open complement of the normal crossing divisor). 
    Finally, \cite[Cor.~5.15]{CN17} tells us that the right hand side is semi-stable. 
\end{proof}

\begin{remark}
    \label{rmk: saturated}
    We clarify how condition \ref{thm: rel CK19} above is related to log structures of the fiber at a point. 
    Denote $K(x_\eta)$ by $K'$ and let $\pi'$ be its uniformizer. 
    Then the morphism $x : \Spf(\sO_{K'}) \to X$ necessarily extends to a (not necessarily exact) closed immersion $(\Spf(\sO_{K'}), (\pi')^\IN) \to (X,M_X)$. 
    By \cite[Prop.~II.2.14]{Tsu19}, the special fiber $f_k$ is saturated. In particular, the fibers of morphism between the underlying shcemes of $f_k$ are reduced by \cite[Thm.~II.4.2]{Tsu19}. 
    As the latter statement is also true for the generic fiber $f_\eta$, by \cite[Thm.~II.4.2]{Tsu19} again the map $f$ is saturated. 
    So thanks to \cite[Prop.~II.2.13]{Tsu19}, the base change $(Y_x, M_{Y_x}) := x^* (Y, M_Y)$ is saturated. 
    Now, note that $(Y_x, M_{Y_x})$ is regular in the sense of Kato. Indeed, by \cite[Thm.~8.2]{Kat94b}, the fact that $(\sO_{K'}, (\pi')^\IN)$ is regular implies that $(Y_x, M_{Y_x})$ is also regular. 
    Then we know from \cite[Thm.~11.6]{Kat94b} that $M_{Y_x}$ is necessarily the standard log structure. 
\end{remark}

Before proving the association statement, we give a little lemma. 

\begin{lemma}
\label{lem: Galois descent at a point}
    Let $S$ be a smooth and connected rigid space over $K$. Let $M, N$ be two vector bundles with flat connections on $S$. Let $K'$ be a finite totally ramified extension of $K$ and set $S' \colonequals S \tensor_K K'$. Let $\alpha : M|_{S'} \to N|_{S'}$ be a morphism between flat vector bundles. Let $F$ be a finite unramified extension and set $F' \colonequals F \tensor_K K'$. 
    
    Suppose that for $s \in S(F)$ and $s' \colonequals s_{F'}$, the fiber $\alpha_{s'} : M_{s'} \to N_{s'}$ descends to a morphism $M_s \to N_s$. Then $\alpha$ descends to a morphism $M \to N$. 
\end{lemma}
\begin{proof}
    To prove the statement we may replace $K'$ by its normal closure over $K$. By Galois descent, we reduce to showing that $\sigma \alpha \sigma^{-1} = \alpha$ for all $\sigma \in \Gal(K'/K)$. Note that $\sigma \alpha \sigma^{-1}$ is also a morphism between vector bundles with flat connections. Since a totally ramified extension and an unramified extension of $K$ are linearly disjoint, $F'$ is a field and $\Gal(F'/F) = \Gal(K'/K)$. The hypothesis implies that $\sigma \alpha \sigma^{-1} - \alpha$ vanishes at $s'$, and hence everywhere, thanks to \cite[Thm.\ 4.1]{Ogu84}.
\end{proof}

Now we prove the association statement \cref{intro:thm:association} in a more precise form. 

\begin{theorem}
\label{thm:association}
    In the setting of \Cref{thm:relative Cst}.\ref{thm: rel CK19}, 
    assume $X$ is in addition smooth over $\sO_K$.
    Let $f_s : (Y_s, M_{Y_s}) \to (X_s, M_{X_s})$ be the reduced special fiber of $f$. Then there is a canonical isomorphism of $F$-isocrystals on $(X_s, M_{X_s})_\lcrys$
    \begin{equation}
        \label{eqn: identify log F-isocrystals}
        \sE_{\st, T} \simeq R^i f_{s, \crys*} \sO_{(Y_s, M_{Y_s})/W}[1/p]
    \end{equation}
    such that, when evaluated on the divided power thickeing $(X,X_{p=0})$, there is a commutative diagram of flat connections over $X_\eta$
    \begin{equation}
    \label{diag: relative Cst}
        \begin{tikzcd}
	{\sE_{\st, T}(X, X_{p = 0}, M_X)} & {(R^i f_{s, \crys} \sO_{(Y_s, M_{Y_s})/W}[1/p])(X, X_{p = 0}, M_X)} \\
	{D_\dR(R^i f_{\eta, *} \IZ_p)} & {R^i f_{\eta *} \Omega^\bullet_{Y_\eta/X_\eta}}
	\arrow["\sim","(\ref{eqn: identify log F-isocrystals})"', from=1-1, to=1-2]
	\arrow["\textrm{\Cref{thm: D functors}.\ref{thm: D functors relation}}"', "\simeq", from=1-1, to=2-1]
	\arrow["\textrm{\cref{thm:relative Hyodo-Kato}}", "\simeq"', from=1-2, to=2-2]
	\arrow["\sim", "\textrm{\cite{Sch13}}"', from=2-1, to=2-2].
        \end{tikzcd}.
    \end{equation}
    Here $f_{s, \crys}$ is the morphism of topoi $(Y_s, M_{Y_s})_\lcrys^\sim \to (X_s, M_{X_s})_\lcrys^\sim$. 
\end{theorem}
\begin{proof}
Recall that we already know that $T$ is semi-stable, so that $\sE_{\st, T}$ is an actural F-isocrystal on $(X_k, M_{X_s})_\lcrys$. 
It remains to construct the isomorphism (\ref{eqn: identify log F-isocrystals}) and show that (\ref{diag: relative Cst}) commutes.  In fact, we will do these tasks in reverse order. But first we shall explain why the verical arrows in (\ref{diag: relative Cst}) exist. For simplicity, below we write $\sE$ for $R^i f_{s, \crys} \sO_{(Y_s, M_{Y_s})/W}[1/p]$.

First we shall explain why the verical arrows in (\ref{diag: relative Cst}) exist. 
For simplicity, below we write $\sE$ for $R^i f_{s, \crys} \sO_{(Y_s, M_{Y_s})/W}[1/p]$. 
Zariski-locally there exists a $p$-completely smooth $W$-model $X_0$ such that $X \simeq X_0 \tensor_W \sO_K$ and the Frobenius endomorphism lifts to $X_0$.
So thanks to \cref{prop: relative log crys}, we know $\sE$ is a locally free F-isocrystal on $(X_s, M_{X_s})_\lcrys$. 
Note that by assumption we have $M_{X_s} = (0^\IN)^a$ and $M_{X_{p = 0}} = (\pi^\IN)^a$. 
By the equivalence $\Isoc^\varphi((X_s, (0^\IN)^a)_\lcrys) \simeq \Isoc^\varphi((X_{p = 0}, (\pi^\IN)^a)_\lcrys)$ (\cref{equiv def log isoc}) it makes sense to evaluate the log F-isocrystals $\sE_{\st, T}$ and $\sE$ on $(X, X_{p = 0}, M_X)$. 
Moreover, in the course of proving \cref{thm:relative Hyodo-Kato}, we constructed a canonical isomorphism (cf. \cref{eqn: Breuil to evaluate E})
\begin{equation}
    \label{eqn: K-linearization}
    \sE(X_0, X_s, (0^\IN)^a) \tensor_{K_0} K \simeq \sE(X, X_{p = 0}, M_X).
\end{equation}
This explains why \cref{thm:relative Hyodo-Kato}.\ref{thm:relative Hyodo-Kato map} gives the right vertical arrow in (\ref{diag: relative Cst}). 
Here we note that the construction of (\ref{eqn: K-linearization}) works for any object in $\Isoc^{\varphi}(X_s, (0^\IN)^a)$, and in particular for $\sE_{\st, T}$, so that $\sE_{\st, T}(X_0, X_s, (0^\IN)^a) \tensor_{K_0} K \simeq \sE_{\st, T}(X, X_{p = 0}, M_X)$. 
Recall that by construction of the crystalline Riemann-Hilbert functors (\Cref{thm:gluing of D functors}.\ref{thm:gluing of D functors modules}), we have $\sE_{\st,T}(X_{0}, X_s, (0^\IN)^a) = D_{\st,X_0}(T)$. 
Since $T$ is known to be semi-stable, the natural injection $D_{\st.X_0}(T) \tensor_{K_0} K \into D_\dR(T)$ is an isomorphism (cf. the comment below \cref{thm:rank max implies alpha is iso}). 
Therefore, \cref{thm: D functors}.\ref{thm: D functors relation} gives the left vertical arrow in (\ref{diag: relative Cst}). 
The independence of the choice of $X_0$ for the left vertical arrow was proved in \cref{cor:D is independent of R_0}, and that for the right vertical arrow can be easily deduced from the proof of \cref{thm:relative Hyodo-Kato}.

Let $\alpha : \sE_{\st, T}(X, X_s, M_X) \simeq \sE(X, X_s, M_X)$ be the unique isomorphism which makes (\ref{diag: relative Cst}) commute. 
We shall argue $\alpha$ decends to an isomorphism $\alpha_0 : \sE_{\st, T}(X_0, X_s, (0^\IN)^a) \simeq \sE(X_0, X_s, (0^\IN)^a)$. 
Let $F_0$ be a finite unramified extension of $K_0$ such that $X_0(\sO_{F_0}) \neq \emptyset$ and take some $x_0 \in X_0(\sO_{F_0})$. Let $F$ be the composite extension $ K \tensor_{K_0} F_0$ and let $x$ be the base change of $x_0$ to $\sO_F$. 
Let $x_s$ be the closed point of $x$ (or equivalently, $x_0$) and let $\overline{x}_\eta$ be the geometric point over $x_\eta$ given by an embedding $K \into C$.

Using suitable base change theorems, including \cref{thm:pullback for D functors}, \cref{thm:relative Hyodo-Kato} and \cite[Thm.~3.9]{LZ17}, we may identify the base change of (\ref{diag: relative Cst}) along the closed immersion $x\to X$ with the following diagram:  
\begin{equation}
    \label{diag: rel Cst bc}
    \begin{tikzcd}
	{D_\st(\mathrm{H}^i_\et(Y_{\overline{x}_\eta}, \IQ_p)) \tensor_{K_0} K } & {\mathrm{H}^i_{\lcris}((Y_{x_s}, M_{Y_x}|_{Y_{x_s}})/(W, (0^\IN)^a)) \tensor_{K_0} K} \\
	{D_\dR(\mathrm{H}^i_\et(Y_{\overline{x}_\eta}, \IQ_p))} & {\mathrm{H}^i_\dR(Y_{x_\eta}/F)},
	\arrow[from=1-1, to=1-2]
	\arrow["{\textrm{\Cref{thm: D functors}.\ref{thm: D functors relation}}}"',"\simeq", from=1-1, to=2-1]
	\arrow["{\textrm{\cite[Thm.~5.1]{HK94}}}","\simeq"', from=1-2, to=2-2]
	\arrow["{\textrm{\cite{Sch13}}}","\sim"', from=2-1, to=2-2]
\end{tikzcd}
\end{equation}
where we recall that the restriction of $M_Y$ to $Y_{x_0}$ is the standard log structure $M_{Y_x}$ on $Y_{x_0}$ (see \cref{rmk: saturated}).
On the other hand, applying \cite[Thm.~9.5]{CK19}, we obtain an isomorphism 
\begin{equation}
    \label{eqn: abs st comparison}
    D_\st(\mathrm{H}^i_\et(Y_{\overline{x}_\eta}, \IQ_p))  \simeq \mathrm{H}^i_{\lcris}((Y_x, M_{Y_x})/(W, (0^\IN)^a))
\end{equation}
which is compatible with the Frobenius and monodromy operators. 
Since the semi-stable comparison isomorphism in \cite[Thm.~9.5]{CK19} is compatible with the de Rham comparison isomorphism given by \cite{Sch13} (see \cite[Rmk~9.6]{CK19}), we see that the top horizontal arrow of (\ref{diag: rel Cst bc}) is necessarily given by (\ref{eqn: abs st comparison}) via extension of scalars along $K_0 \into K$.
As a consequence, by \cref{lem: Galois descent at a point} above, $\alpha$ descends to an isomorphism $\alpha_0 : \sE_{\st, T}(X_0, X_s, (0^\IN)^a) \simeq \sE(X_0, X_s, (0^\IN)^a)$ as desired. 
It remains to show that $\alpha$ respects the Frobenius and monodromy operators. However, these properties can also be checked after restriction to a point. Since $\alpha_x$ has these properties, we are done. 
\end{proof}

Below we give some remarks on the technical conditions in \cref{thm:relative Cst} and \cref{thm:association} (i.e., \cref{intro:thm:relative Cst} and \cref{intro:thm:association} respectively). 
\begin{remark}
\label{rmk: assumptions in rel Cst}
    \begin{enumerate}[label=\upshape{(\alph*)}]
        \item The strategy of proving \Cref{thm:relative Cst} is to combine our \Cref{intro:thm pc} with the semi-stable comparison theorems in the absolute setting, where we use \cite{CN17} for \Cref{thm:relative Cst}.\ref{thm: rel CN17} and \cite{CK19} for \Cref{thm:relative Cst}.\ref{thm: rel CK19}.
        In particular, the technical constraints in \Cref{thm:relative Cst} originate from the requirement for algebraicity in \cite[Cor.~5.15]{CN17}, and that for weak semi-stability in \cite[Thm.~9.5]{CK19} respectively. 

        We also remark that, using the prismatic methods, one can extend the strategy of the first named author and Reinecke in \cite{GR22} to remove the algebraicity assumption in \Cref{thm:relative Cst}.\ref{thm: rel CN17}. 
        This can achieved for example by combining the main results of \cite{KY23}
        \footnote{We thank Koshikawa for kindly drawing our attention to their work.}
        with \cite{DLMS2}. On the other hand, using the pointwise criteria, one can show that $R^i f_{\eta, \ket *} \IQ_p$ is semi-stable as long as the generic fiber $f_\eta$ is proper and log smooth, with \textit{each fiber} satisfying either of the assumptions in \Cref{thm:relative Cst}.
        \item We assume the setup of \Cref{thm:relative Cst}.\ref{thm: rel CK19} for \cref{thm:association}, only because our proof of the latter requires certain compatibility with Scholze's de Rham comparison isomorphism \cite{Sch13}, which in the absolute setting was explicitly checked in \cite[Thm.~9.5]{CK19}.

        \item Inspired by work of Kedlaya \cite[Thm.~6.4.5]{Ked1}, we find it plausible that when $X$ is semi-stable, the restriction functor from the category of F-isocrystals on $(X_k, M_X|_{X_k})_\lcrys$ (with suitable convergence properties) to those on $(X^{\mathrm{sm}}_k, M_X|_{X^{\mathrm{sm}}_k})_\lcrys$ is fully faithful, where $X^{\mathrm{sm}}_k$ is the maximal smooth locus of $X_k$. If we assume this claim, then for \cref{thm:association} the case of a semi-stable base follows rapidly from that of a smooth base. 
    \end{enumerate}
\end{remark}

Finally, we explain an example about the family of stable curves. 
Recall that $\sM_{g, n}$ (resp. $\overline{\sM}_{g, n}$) is the moduli stack over $\IZ$ of $n$-marked smooth (resp. stable) curves of genus $g$, i.e., $\overline{\sM}_{g, n}$ is the Deligne-Mumford compactification $\sM_{g, n}$. 
Let $\sD \colonequals \overline{\sM}_{g, n} \smallsetminus \sM_{g, n}$ be the boundary.  

\begin{proposition}
\label{eg:moduli_of_curve}
    Let $S$ be a semi-stable $p$-adic formal scheme over $\sO_K$. 
    Suppose that $f : C \to S$ is a stable curve given by a morphism $\gamma : S \to \overline{\sM}_{g, n}$ such that 
    \begin{itemize}
        \item the special fiber $\gamma_k$ factors through the boundary $\sD_k$;
        \item the generic fiber $\gamma_\eta$ factors through the analytification of the interior $\sM_{g, n, K}$.
    \end{itemize} 
    Let $\sigma_1, \cdots, \sigma_n: S \to C$ be the marked points, let $C^\circ \colonequals C \smallsetminus \cup_{i = 1}^n \sigma_i(S)$ be the open complement, and let $f^\circ: C^\circ \to S$ be restriction of $f$ on $C^\circ$. 
    
    Then $R^1 f^\circ_{\eta, \et *} \IQ_p$ is a semi-stable local system on $S_\eta$.
\end{proposition}
\begin{proof}
    We equip $C \to S$ with the basic log structure $(C, M^{\mathrm{bas}}_C) \to (S, M^{\mathrm{bas}}_S)$, as is characterized in \cite[Prop.~2.3]{FKato}. 
    It can be checked from the constructions in \textit{loc. cit.} that this family is proper, log smooth, and its mod $\pi$ reduction is of Carteier type. 
    
    Let $(X, D)$ be the $p$-completion of $(\overline{\shM}_{g, n}, \sD) \tensor \sO_K$.
    Let $M_X$ be the log structure on $X$ given by $D$. 
    By \cite[Thm.~4.5]{FKato}, the universal curve over $X$ can be equipped with the structure of a log stable curve, which is in particular proper and log smooth; moreover, $\gamma^* M_X$ is nothing but $M_S^{\mathrm{bas}}$. 

    Recall that locally $M_X$ is given by sections of $\sO_X$ which are units on the complement $X \smallsetminus D$. Similarly, $M_S$ on $S$ is given by sections of $\sO_S$ which are units on $S_\eta$. Since $\gamma_\eta^{-1} D$ is supported on $S_k$ by assumption, we know the image of $\gamma^* M_X$ in $\sO_S$ is contained in that of $M_S$.
    Hence we obtain a map of log pairs $(S, M_S) \to (S, \gamma^* M_X)$. 
    
    Finally, by pulling back $(C, M^{\mathrm{bas}}_C)$ along $(S, M_S) \to (S, \gamma^* M_X = M^{\mathrm{bas}}_S)$, we obtain a proper, log smooth, and algebraic morphism $(C, M_C) \to (S, M_S)$ which has Cartier type mod $\pi$ reduction\footnote{This can also be seen by applying \cite[Thm.~4.2]{Tsu19}---it is part of the definition of log stable curves that the geometric fibers of the underlying morphism between schemes are reduced.}. So the claim follows by applying \cref{thm:relative Cst}.\ref{thm: rel CN17}. 
\end{proof}

Below we also mention that one can easily construct a family of stable curves $\gamma$ as above that satisfies the assumptions in \Cref{eg:moduli_of_curve}.   
\begin{example}
\label{ex: lift boundary}
    Let $X$ be a smooth $p$-adic formal scheme over $\sO_K$, let $D \subseteq X$ be a relative normal crossing divisor that is flat over $\sO_K$, and let $M_X$ be the log structure on $X$ induced by $\sD$. 
    We claim that for every point $z \in D_s$, there is an open neighborhood $U_s$ of $z$ in $D_s$ such that $U_s$ lifts to a semi-stable $p$-adic formal scheme $U \subseteq X$ over $\sO_K$ that satisfies the property $U_\eta \subseteq X_\eta \smallsetminus D_\eta$. 
    
    The question is \'etale local, so without loss of generality we may assume that $X = \Spf(R)$ is an open formal subscheme of $\Spf(\sO_K\langle x_1, \cdots, x_r, x_{r+1}^\pm, \cdots, x_n^\pm \rangle)$, $D = V(x_1 x_2 \cdots x_r) \subseteq X$, and $z = (0, \cdots, 0, 1, \cdots, 1)$. 
    Then we can simply take $U_s = D_s$ and $U = V(x_1 x_2 \cdots x_r - \pi)$.  
\end{example}

\bibliographystyle{amsalpha}
\bibliography{newref}

\end{document}